\documentclass{amsart}
\usepackage{amssymb}
\usepackage{amsmath}
\usepackage{amsthm}
\usepackage{mathrsfs}
\usepackage{colortbl}
\usepackage{graphicx}
\newtheorem{theorem}{Theorem}
\newtheorem{lemma}[theorem]{Lemma}

\newtheorem{remark}[theorem]{Remark}
\usepackage[margin=35truemm]{geometry}

\theoremstyle{definition}
\newtheorem{definition}[theorem]{Definition}

\newtheorem{proposition}[theorem]{Proposition}
\newtheorem{corollary}[theorem]{Corollary}

\makeatletter

\@addtoreset{theorem}{section}
\makeatother

\makeatletter

\@addtoreset{equation}{section}
\makeatother

\newcommand{\dual}[1]{\langle #1 \rangle }
\begin{document}                                                 

\title[Heat equations in two unbounded domains and an interface]{Global solvability of the heat equations in two unbounded domains and the interface}                                   
\author[Hajime Koba]{Hajime Koba}                            
\address{Faculty of Advanced Science and Technology, Kumamoto University, 2-39-1 Kurokami, Chuo-ku, Kumamoto, 860-8555, Japan}
\email{koba-hajime@kumamoto-u.ac.jp}
                                                

\begin{abstract}
This paper considers the existence of a unique global-in-time strong solution to the heat equations in two unbounded domains $\Omega_A$, $\Omega_B$ and the interface $\Gamma (= \partial \Omega_A \cap \partial \Omega_B)$. We introduce and study some function spaces in the two unbounded domains and the interface. We apply our function spaces and maximal $L^p$-regularity for Hilbert space-valued functions to show the existence of a local-in-time strong solution to our heat equations. By using an energy equality of our heat system, we prove the existence of a unique global-in-time strong solution to the system with large initial data when the slope of the interface is gentle and our parameters are limited. The key ideas for showing the existence of our strong solutions are to transform our system into a system of equations in two half spaces $\mathbb{R}^3_+, \mathbb{R}^3_-$ and the whole space $\mathbb{R}^2$, and to make use of nice properties of the heat semigroups and kernels for $\mathbb{R}^3_+$, $\mathbb{R}^3_-$, and $\mathbb{R}^2$. In Appendix (I), we derive our heat equations in the two unbounded domains and the interface from an energetic point of view. In Appendix (II), we study representation formulas for differential operators on unbounded domains and surfaces.
\end{abstract}
\maketitle

\section{Introduction}\label{sect1}

\begin{figure}[htbp]
\includegraphics[width=13cm]{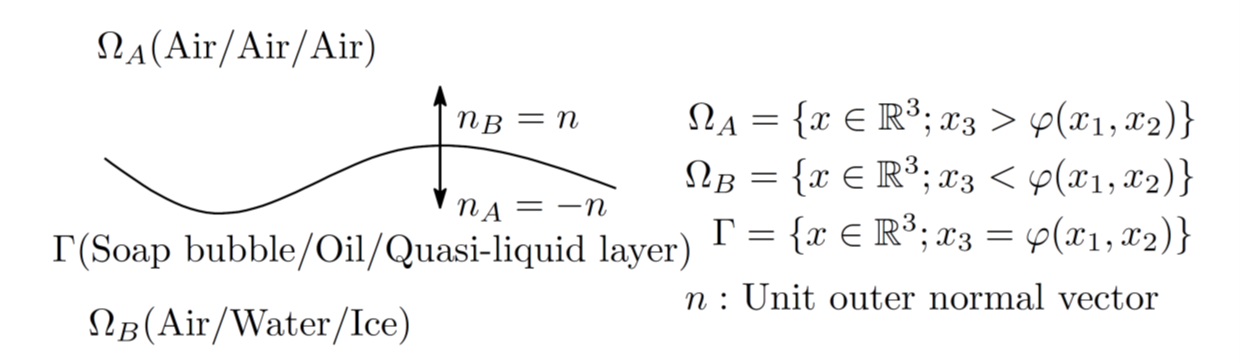}
\caption{Three-phase problems}
\label{Fig1}
\end{figure}

We are interested in the existence of a global-in-time strong solution to the heat equations in two unbounded domains $\Omega_A$, $\Omega_B$ and the interface $\Gamma (= \partial \Omega_A \cap \partial \Omega_B)$. The heat system is a simple heat transfer model for a soap bubble flying in the air, oil floating on water, or a melting ice (see Figure \ref{Fig1}). This paper considers soap bubbles, oil, and quasi-liquid layers as fluids in surfaces. One can see \cite{Gur93,GP01,SSO07,Kob20, Kob23b} for heat transfer on surfaces and three-phase problems, and \cite{KL82,FYK87,SZNYF10} for quasi-liquid layers.

This paper has two purposes. The first one is to introduce and study some function spaces in two unbounded domains $\Omega_A, \Omega_B$ and the interface $\Gamma (= \partial \Omega_A \cap \partial \Omega_B)$. The second one is to apply our function spaces and maximal $L^p$-regularity for Hilbert space-valued functions to construct a strong solution of our heat equations.

Let us first introduce basic notations. We use the characters $i,j, \ell, i',j',\ell'$ as $1,2,3$, and the characters $\alpha, \beta,\alpha', \beta'$ as $1,2$, that is, $i , j, \ell, i',j',\ell' \in \{ 1,2,3 \}$ and $\alpha, \beta, \alpha',\beta' \in \{ 1,2 \}$. Let $x = { }^t (x_1, x_2, x_3), y= { }^t(y_1,y_2,y_3), z = { }^t(z_1,z_2,z_3) \in \mathbb{R}^3$, $x_h = { }^t (x_1, x_2),y_h= { }^t(y_1,y_2), z_h = { }^t (z_1,z_2), X = { }^t (X_1,X_2) \in \mathbb{R}^2$ be the spatial variables, and $t , \tau, t_1 ,t_2 \geq 0$ be the time variables. The symbols $\nabla$, $\nabla_h$, $\nabla_X$ and $\Delta$, $\Delta_h$, $\Delta_X$ are three gradient and three Laplace operators defined by $\nabla = { }^t (\partial_1 , \partial_2 , \partial_3)$, $\nabla_h = { }^t ( \partial_1 , \partial_2)$, $\nabla_X = { }^t ( \partial_{X_1} , \partial_{X_2})$, $\Delta = \partial_1^2 + \partial_2^2 + \partial_3^2$, $\Delta_h = \partial_1^2 + \partial_2^2$, and $\Delta_X = \partial_{X_1}^2 + \partial_{X_2}^2$, where $\partial_j = \partial/{\partial x_j}$ and $\partial_{X_\alpha} = \partial/{\partial X_{\alpha}}$. Write $\Delta_y = \partial^2_{y_1} + \partial^2_{y_2} + \partial^2_{y_3}$, $\nabla_y = { }^t (\partial_{y_1}, \partial_{y_2}, \partial_{y_3})$, $\Delta_{y_h} = \partial_{y_1}^2 + \partial^2_{y_2}$, $\nabla_{y_h} = { }^t ( \partial_{y_1} , \partial_{y_2} )$, $\Delta_z = \partial^2_{z_1} + \partial^2_{z_2} + \partial^2_{z_3}$, $\Delta_{z_h} = \partial_{z_1}^2 + \partial^2_{z_2}$, $\nabla_z = { }^t (\partial_{z_1}, \partial_{z_2}, \partial_{z_3})$, and $\nabla_{z_h} = { }^t ( \partial_{z_1} , \partial_{z_2} )$. Set $\mathbb{R}^2 = \{ X = { }^t (X_1,X_2); X_1 , X_2 \in \mathbb{R} \}$,
\begin{equation*}
\mathbb{R}^3_+ = \{ y = { }^t (y_1,y_2,y_3) \in \mathbb{R}^3; { }y_3 >0 \},{ \ } \mathbb{R}^3_- = \{ z = { }^t(z_1,z_2,z_3) \in \mathbb{R}^3; { }z_3 < 0 \}.
\end{equation*}
It is clear that $\overline{\mathbb{R}^3_+} = \{ y \in \mathbb{R}^3; { }y_3 \geq 0 \},{ \ } \overline{\mathbb{R}^3_-} = \{ z \in \mathbb{R}^3; { }z_3 \leq 0 \}$, and $\partial \mathbb{R}^3_+ \cap \partial \mathbb{R}^3_- = \mathbb{R}^2 \times \{ 0 \}$. In this paper we equate $\mathbb{R}^2 \times \{ 0 \}$ to $\mathbb{R}^2$. For $1 \leq p \leq \infty$, $k \in \mathbb{N}$, and $U \in \{ \mathbb{R}^3_+, \mathbb{R}^3_- , \mathbb{R}^2 \}$, the symbols $L^p(U)$ and $W^{k,p}(U)$ denote the usual Lebesgue and Sobolev spaces, and $\Vert \cdot \Vert_{L^p(U)}$ and $\Vert \cdot \Vert_{W^{k,p} (U)}$ denote its norms. Let $\left< \cdot , \cdot \right>_{L^2(U)}$ be the $L^2$-inner product on $L^2(U)$. In particular, for all $\varphi \in BC^2( \mathbb{R}^2)$ we use the symbol $\Vert \nabla_X \varphi \Vert_{L^\infty (\mathbb{R}^2)}$ as follows:
\begin{equation*}
\Vert \nabla_X \varphi \Vert_{L^\infty ( \mathbb{R}^2)} = \sup_{X \in \mathbb{R}^2} \bigg\vert \frac{\partial \varphi}{\partial X_1}(X) \bigg\vert + \sup_{X \in \mathbb{R}^2} \bigg\vert \frac{\partial \varphi}{\partial X_2} (X) \bigg\vert.
\end{equation*}
Fix $\varphi \in BC^2( \mathbb{R}^2)$. Set $\Omega_A = \{ x \in \mathbb{R}^3;{ \ }x_3 > \varphi (x_1,x_2) \}$,
\begin{equation*}
\Omega_B = \{ x \in \mathbb{R}^3;{ \ }x_3 < \varphi (x_1, x_2) \},{ \ }\Gamma = \{ x \in \mathbb{R}^3;{ \ }x_3  = \varphi (x_1, x_2) \},
\end{equation*}
$\Omega_{A,T} = \Omega_A \times (0,T)$, $\Omega_{B,T} = \Omega_B \times (0,T)$, and $\Gamma_T = \Gamma \times (0,T)$ for some $T \in ( 0, \infty ]$. It is clear that $\overline{\Omega_A} = \{ x \in \mathbb{R}^3;{ \ }x_3 \geq \varphi (x_1,x_2) \}$, $\overline{\Omega_B} = \{ x \in \mathbb{R}^3;{ \ }x_3 \leq \varphi (x_1, x_2) \}$. For each $\alpha , \beta \in \{ 1,2 \}$,  we write
\begin{equation*}
\varphi_\alpha = \varphi_\alpha (X) = \frac{\partial \varphi}{\partial X_\alpha} \text{ and }\varphi_{\alpha \beta} = \varphi_{\alpha \beta} (X) = \frac{\partial^2 \varphi}{\partial X_\alpha \partial X_\beta}.
\end{equation*}
Note that $\varphi_\alpha (x_h) = \partial \varphi/{\partial x_\alpha}$, $\varphi_{\alpha \beta} (x_h) = \partial^2 \varphi/{\partial x_\alpha \partial x_\beta}$. Let $n= n(x) = { }^t (n_1,n_2,n_3)$ be the unit outer normal vector at $x \in \Gamma$ defined by
\begin{equation}\label{eq11}
n(x_1,x_2, x_3) = \frac{1}{\sqrt{1 + \varphi_1^2 + \varphi_2^2}} 
\begin{pmatrix}
- \varphi_1\\
- \varphi_2\\
1
\end{pmatrix}.
\end{equation}
Let $n_A = n_A (x)$ and $n_B = n_B (x)$ be the two unit outer normal vectors at $x \in \partial \Omega_A$ and $x \in \partial \Omega_B$, respectively. We easily find that $n_A = -n_B = -n$ (see Figure \ref{Fig1}). For smooth functions $f$, $x \in \Gamma$, and $j \in \{ 1 , 2, 3 \}$, we define
\begin{align*}
\partial_j^\Gamma f &= \partial_j f - n_j (n \cdot \nabla) f,\\
\nabla_\Gamma f & = { }^t (\partial_1^\Gamma f , \partial_2^\Gamma f, \partial_3^\Gamma f),\\
\Delta_\Gamma f  & = (\partial_1^\Gamma)^2 f + (\partial_2^\Gamma)^2 f + (\partial_3^\Gamma)^2 f.
\end{align*} 
We call $\Delta_\Gamma$ the \emph{Laplace-Beltrami operator}. Note that $\Delta_\Gamma f = \nabla_\Gamma \cdot \nabla_\Gamma f$. See Section \ref{sect7} for differential operators $\partial_j^\Gamma$, $\nabla_\Gamma$, and $\Delta_\Gamma$. The symbol $d \mathcal{H}_x^2$ denotes the $2$-dimensional Hausdorff measure.

This paper considers the existence of a global-in-time strong solution to the heat equations in two unbounded domains $\Omega_A$, $\Omega_B$ and the interface $\Gamma$:
\begin{equation}\label{eq12}
\begin{cases}
\rho_A \partial_t \theta_A = \mu_A \Delta \theta_A & \text{ in } \Omega_{A , T} ,\\
\rho_B \partial_t \theta_B = \mu_B \Delta \theta_B & \text{ in } \Omega_{B,T},\\
\rho_S \partial_t \theta_S = \mu_S \Delta_\Gamma \theta_S + \mu_A \gamma_A[ (n \cdot \nabla) \theta_A] - \mu_B \gamma_B[ (n \cdot \nabla) \theta_B] & \text{ in } \Gamma_T,\\
\gamma_A [\theta_A] = \gamma_B [\theta_B] = \theta_S & \text{ in } \Gamma_T,
\end{cases}
\end{equation}
with
\begin{equation*}
\begin{cases}
{\displaystyle{\lim_{ \vert x \vert \to \infty , { \ }x \in \Omega_A } \theta_A =0}},{ \ }{\displaystyle{\lim_{ \vert x \vert \to \infty,{ \ }x \in \Omega_B } \theta_B  =0}},{ \ }{\displaystyle{\lim_{ \vert x_h \vert \to \infty,{ \ }x \in \Gamma } \theta_S =0}} & \text{ on } (0, T),\\
\theta_A \vert_{t=0} = \theta^A_0{ \ } \text{ in } \Omega_A, { \ }\theta_B \vert_{t = 0} = \theta_0^B{ \ } \text{ in }\Omega_B, { \ }\theta_S \vert_{t=0} = \theta_0^S{ \ } \text{ in } \Gamma,
\end{cases}
\end{equation*}
where $T \in (0, \infty]$, and $\gamma_A$, $\gamma_B$ are two trace operators such that
\begin{align*}
\gamma_A[\cdot] : W^{1,2} (\Omega_A) \to L^2 (\Gamma)(= L^2 (\partial \Omega_A)),\\
\gamma_B[\cdot] : W^{1,2} (\Omega_B) \to L^2 (\Gamma)(= L^2(\partial \Omega_B)).
\end{align*}
The unknown functions $\theta_A = \theta_A (x,t)$, $\theta_B = \theta_B (x,t)$, and $\theta_S = \theta_S (x,t)$ are the temperatures of the fluid(or substance) in $\Omega_A$, $\Omega_B$, and $\Gamma$, respectively. The given positive constants $\mu_A$, $\mu_B$, and $\mu_S$ are the thermal conductivities of the fluid(or substance) in $\Omega_A$, $\Omega_B$, and $\Gamma$, respectively. The given functions $\theta^A_0 = \theta^A_0 (x)$, $\theta^B_0 = \theta^B_0 (x)$, $\theta^S_0 = \theta^S_0 (x)$ are initial data. The symbols $\rho_A$, $\rho_B$, and $\rho_S$ are given positive constants. One can see Section \ref{sect6} for the derivation of system \eqref{eq12} and parameters $(\rho_A,\rho_B,\rho_S)$, and Section \ref{sect3} for two trace operators $\gamma_A,\gamma_B$, function spaces $W^{1,2} (\Omega_A)$, $W^{1,2} (\Omega_B)$, $L^2(\Gamma)$, and terms $\gamma_A [(n \cdot \nabla) \theta_A]$, $\gamma_B[ (n \cdot \nabla ) \theta_B]$.

To construct a strong solution of system \eqref{eq12}, we have to deal with both
\begin{equation}\label{eq13}
\mu_A \gamma_A[ (n \cdot \nabla) \theta_A] - \mu_B \gamma_B[ (n \cdot \nabla) \theta_B] \text{ and }\gamma_A [\theta_A] = \gamma_B [\theta_B] = \theta_S.
\end{equation} 
However, it is not easy to handle \eqref{eq13} directly since trace operators $\gamma_A,\gamma_B$ are not closed operators in general. Moreover, it is difficult to derive the estimates necessary to construct a solution of our system from these terms. To overcome these difficulties, we apply our function spaces and maximal $L^p$-regularity for Hilbert space-valued functions. For that reason, we introduce new parameters $(\kappa_A , \kappa_B, \kappa_S)$ such that
\begin{equation}\label{eq14}
\mu_A = \rho_A \kappa_A , { \ }\mu_B = \rho_B \kappa_B ,{ \ } \mu_S = \rho_S \kappa_S.
\end{equation}
We make use of conditions \eqref{eq14} and our ideas (see the second half of this section) to construct a strong solution of our system. From \eqref{eq12} and \eqref{eq14}, we have
\begin{equation}\label{eq15}
\begin{cases}
\partial_t \theta_A = \kappa_A \Delta \theta_A & \text{ in } \Omega_{A,T},\\
\partial_t \theta_B = \kappa_B \Delta \theta_B & \text{ in } \Omega_{B,T},\\
\partial_t \theta_S = \kappa_S \Delta_\Gamma \theta_S + \frac{\rho_A \kappa_A}{\rho_S} \gamma_A[ (n \cdot \nabla) \theta_A] - \frac{ \rho_B \kappa_B}{\rho_S} \gamma_B[ (n \cdot \nabla) \theta_B] & \text{ in } \Gamma_T,\\
\gamma_A [\theta_A] = \gamma_B [\theta_B] = \theta_S & \text{ in } \Gamma_T,\\
\theta_A \vert_{t=0} = \theta^A_0{ \ } \text{ in } \Omega_A, { \ }\theta_B \vert_{t = 0} = \theta_0^B{ \ } \text{ in }\Omega_B, { \ }\theta_S \vert_{t=0} = \theta_0^S{ \ } \text{ in } \Gamma,&
\end{cases}
\end{equation}
with
\begin{equation*}
\begin{cases}
{\displaystyle{\lim_{ \vert x \vert \to \infty , { \ }x \in \Omega_A } \theta_A =0}},{ \ }{\displaystyle{\lim_{ \vert x \vert \to \infty,{ \ }x \in \Omega_B } \theta_B  =0}},{ \ }{\displaystyle{\lim_{ \vert x_h \vert \to \infty,{ \ }x \in \Gamma } \theta_S =0}} & \text{ on } (0, T).
\end{cases}
\end{equation*}

Now we state one of the main results of this paper.
\begin{theorem}\label{thm11}
Let $\varphi \in BC^2( \mathbb{R}^2)$ and $\rho_A, \rho_B, \rho_S, \kappa_A , \kappa_B , \kappa_S >0$. Assume that $\Vert \nabla_X \varphi \Vert_{L^\infty (\mathbb{R}^2)} \leq 1/2$. Then there is $\mathcal{M}_0 = \mathcal{M}_0 (\kappa_A , \kappa_B , \kappa_S) >0$ such that if
\begin{equation*}
\Vert \nabla_X \varphi \Vert_{L^\infty ( \mathbb{R}^2)} + \frac{\rho_A + \rho_B}{\rho_S} \leq \mathcal{M}_0
\end{equation*}
then for each $\theta_0^A \in W^{1,2} (\Omega_A)$, $\theta_0^B \in W^{1,2}( \Omega_B)$, $\theta_0^S \in W^{1,2} (\Gamma)$ satisfying $\gamma_A[\theta_0^A] = \theta_0^S$ and $\gamma_B[\theta_0^B] = \theta_0^S$, system \eqref{eq15} admits a unique global-in-time strong solution ${ }^t(\theta_A , \theta_B, \theta_S)$:
\begin{align*}
\theta_A & \in C ([0, \infty); L^2 (\Omega_A)) \cap L_{loc}^2 (0,\infty; W^{2,2}(\Omega_A)) \cap W_{loc}^{1,2}(0,\infty ; L^2(\Omega_A)) \cap L^2_{loc}( \Omega_A \times \mathbb{R}_+),\\
\theta_B & \in C ([0, \infty); L^2 (\Omega_B)) \cap L_{loc}^2 (0,\infty; W^{2,2}(\Omega_B)) \cap W_{loc}^{1,2}(0,\infty ; L^2(\Omega_B)) \cap L^2_{loc}( \Omega_B \times \mathbb{R}_+),\\
\theta_S & \in C ([0, \infty); L^2 (\Gamma)) \cap L_{loc}^2 (0,\infty; W^{2,2}(\Gamma)) \cap W_{loc}^{1,2}(0,\infty ; L^2(\Gamma)) \cap L_{loc}^2 (\Gamma \times \mathbb{R}_+),
\end{align*}
satisfying two properties that
\begin{equation*}
\lim_{t \to 0 + 0} ( \Vert \theta_A (t) - \theta_0^A \Vert_{L^2(\Omega_A)} + \Vert \theta_B (t) - \theta_0^B \Vert_{L^2(\Omega_B)} + \Vert \theta_S (t) - \theta_0^S \Vert_{L^2 (\Gamma)}  ) = 0,
\end{equation*}
and that for each fixed $T >0$
\begin{equation*}
\Vert \gamma_A [\theta_A] - \theta_S \Vert_{L^2 (0,T ; L^2(\Gamma))} = \Vert \gamma_B [\theta_B] - \theta_S \Vert_{L^2 (0,T ; L^2(\Gamma))} = 0.
\end{equation*}
Moreover, the solution ${ }^t(\theta_A , \theta_B , \theta_S)$ satisfies that for $0 \leq t_1 \leq t_2 < \infty$
\begin{multline}\label{eq16}
\rho_A \Vert \theta_A (t_2) \Vert^2_{L^2 (\Omega_A)} + \rho_B \Vert \theta_B (t_2) \Vert^2_{L^2 (\Omega_B)} + \rho_S \Vert \theta_S (t_2) \Vert^2_{L^2 (\Gamma)}\\
 + 2 \kappa_A \rho_A \int_{t_1}^{t_2} \Vert \nabla \theta_A ( \tau) \Vert^2_{L^2 (\Omega_A)}{ \ } d \tau + 2 \kappa_B \rho_B \int_{t_1}^{t_2} \Vert \nabla \theta_B ( \tau ) \Vert^2_{L^2 (\Omega_B)}{ \ }d \tau \\
 +  2 \kappa_S \rho_S \int_{t_1}^{t_2} \Vert \nabla_\Gamma \theta_S (\tau) \Vert^2_{L^2 (\Gamma)}{ \ }d \tau\\
  = \rho_A \Vert \theta_A (t_1) \Vert^2_{L^2 (\Omega_A)} + \rho_B \Vert \theta_B (t_1) \Vert^2_{L^2 (\Omega_B)} + \rho_S \Vert \theta_S (t_1) \Vert^2_{L^2 (\Gamma)}.
\end{multline}
Here
\begin{align*}
& L^2_{loc}(0,\infty; W^{2,2} (\Omega_\sharp)) := \{ \Phi ; \Phi \in L^2(0,T; W^{2,2} (\Omega_\sharp)) \text{ for each fixed } T >0  \},\\
& W^{1,2}_{loc}(0,\infty; L^2 (\Omega_\sharp)) := \{ \Phi ; \Phi \in W^{1,2}(0,T; L^2 (\Omega_\sharp)) \text{ for each fixed } T >0  \},
\end{align*}
where $\sharp \in \{ A , B , S\}$ and $\Omega_S := \Gamma$.
\end{theorem}
\noindent See Section \ref{sect4} for the definition of strong solutions to system \eqref{eq15} and Section \ref{sect3} for function spaces such as $L^2(\Omega_A)$, $W^{1,2} (\Omega_B)$, and $W^{2,2}(\Gamma)$. Note that the structure of the heat system is different for one-phase and multi-phase problems. Indeed, the solution $\theta_S$ is not necessarily equal to zero even if the initial value $\theta_0^S$ is zero.

Let us explain main difficulties and five key ideas for constructing a strong solution to our heat equations \eqref{eq15}. The main difficulty in proving the existence of a strong solution to our system is dealing with terms \eqref{eq13} involving trace operators $\gamma_A,\gamma_B$ since $\gamma_A, \gamma_B$ are not closed operators in general. Moreover, it is difficult to derive the estimates necessary to construct a solution of our system from these terms. To overcome these difficulties, we introduce parameters $(\rho_A , \rho_B ,\rho_S, \delta, \epsilon, \kappa_A, \kappa_B, \kappa_S)$. By restricting parameters $(\rho_A , \rho_B , \rho_S, \delta, \epsilon)$ and $\Vert \nabla_X \varphi \Vert_{L^\infty (\mathbb{R}^2)}$, we can derive the estimates necessary for showing the existence of solutions to our systems (see Propositions \ref{prop53} and \ref{prop55} in Section \ref{sect5} for details). Parameters $(\rho_A , \rho_B, \rho_S)$ are derived from mathematical modeling of system \eqref{eq12} in Section \ref{sect6}, $\delta$ from the transformation that transforms system \eqref{eq15} into system \eqref{eq24} in Section \ref{sect2}, $\epsilon$ from the construction of our solutions in Section \ref{sect5}, and $(\kappa_A , \kappa_B , \kappa_S)$ from \eqref{eq14}. This is the first key idea. The second idea is to transform equations \eqref{eq15} into a system of equations in two half spaces $\mathbb{R}^3_+$, $\mathbb{R}^3_-$ and the whole space $\mathbb{R}^2$. We apply a strong solution to the system in two half spaces $\mathbb{R}^3_+,\mathbb{R}^3_-$ and the whole space $\mathbb{R}^2$ to construct a strong solution to our heat equations in two unbounded domains $\Omega_A, \Omega_B$ and the interface $\Gamma$. The third idea is to use a function $\hat{\theta}_S$ in \eqref{eq21} and a solution $\hat{u}_S$ in system \eqref{eq28} to deal with $\gamma_A [\theta_A] = \gamma_B [\theta_B] = \theta_S$ (see Section \ref{sect2} and Proposition \ref{prop31} for details). Indeed, we find that $\hat{\theta}_S = \hat{u}_S$ from Section \ref{sect2}. We transform system \eqref{eq15} into system \eqref{eq28} using variable transformations with $\varphi$ and $\hat{\theta}_S (=\hat{u}_S)$. This allows us to make use of nice properties of the heat semigroups and kernels for $\mathbb{R}^3_+$, $\mathbb{R}^3_-$, and $\mathbb{R}^2$. The fourth idea is to apply our function spaces and the Laplace operator $\mathcal{A}$ for two half spaces $\mathbb{R}^3_+$, $\mathbb{R}^3_-$ and the whole space $\mathbb{R}^2$. Applying a strong solution to system \eqref{eq2010}, we construct a strong solution to system \eqref{eq15}. To show the existence of a strong solution to \eqref{eq2010}, we use the following function spaces:
\begin{align}
H & = \{ \psi = { }^t (\psi_A , \psi_B, \psi_S ) \in L^2 (\mathbb{R}^3_+) \times L^2 (\mathbb{R}^3_-) \times L^2 (\mathbb{R}^2); { \ } \Vert \psi \Vert_H < + \infty \},\label{eq17}\\
X_T & = \{ w \in C ([0,T]; H ); { \ } \Vert w \Vert_{X_T} < \infty \}\notag
\end{align}
with 
\begin{align*}
\Vert \psi \Vert_H & = (\Vert \psi_A \Vert_{L^2(\mathbb{R}^3_+)}^2 + \Vert \psi_B \Vert_{L^2 (\mathbb{R}^3_-)}^2 + \Vert \psi_S \Vert_{L^2 (\mathbb{R}^2)}^2)^{1/2},\\
\Vert w \Vert_{X_T} & = \sup_{0 \leq t \leq T}\Vert w \Vert_H + \Vert d w /{d t } \Vert_{L^2(0,T;H)} + \Vert  \mathcal{A} w \Vert_{L^2(0,T;H)},
\end{align*}
where $\mathcal{A}$ is the linear operator on $H$ defined by
\begin{equation*}
\begin{cases}
\mathcal{A} \psi = { }^t ( - \kappa_A \Delta_y \psi_A , - \kappa_B \Delta_z \psi_B , - \kappa_S \Delta_X \psi_S ),\\
D (\mathcal{A}) = [W_0^{1,2} (\mathbb{R}^3_+) \cap W^{2,2} (\mathbb{R}^3_+)] \times [W_0^{1,2} (\mathbb{R}^3_-) \cap W^{2,2} (\mathbb{R}^3_-)] \times W^{2,2} (\mathbb{R}^2).
\end{cases}
\end{equation*}
Using maximal $L^p$-regularity of $\mathcal{A}$, we show the existence of a local-in-time strong solution to \eqref{eq2010}; see Section \ref{sect5}. The fifth idea is to apply the uniqueness of the strong solutions to \eqref{eq15} in order to show the uniqueness of the strong solutions to system \eqref{eq29}. Using an energy equality for equations \eqref{eq15}, we discuss the uniqueness of the strong solutions to system \eqref{eq29} (see Proposition \ref{prop43} in Section \ref{sect4} for details).

Let us introduce some papers related to this paper, and compare this paper with those papers. Ukai \cite{Uka87} studied representation formulas for solutions of the heat and the Stokes equations in a half space, and derived $L^p$-$L^q$ estimates of the solutions. Davies \cite[Chapter 5]{Dav89} studied the heat kernel of a complete Riemannian manifold. They derived fundamental properties of their heat semigroup under the condition that the Ricci curvature of the manifold is bounded below by a negative constant. Desch-Hieber-Pr\"{u}ss \cite{DHP01} studied maxima $L^p$-$L^q$-regularity for the solutions of the heat and the Stokes equations in a half space. They applied spectral analysis to prove that their heat and Stokes operators admit $\mathcal{R}$-bounded $H^\infty$-calculus on $L^p$-spaces. Favini-Goldstein-Goldstein-Romanelli \cite{FGGR02} studied the heat equation in a bounded domain $\Omega$ with their generalized Wentzell boundary condition: $\Delta \theta + a_0(n \cdot \nabla ) \theta + b_0 \theta =0$ on $\partial \Omega$. They showed that their heat operator generates a semigroup on $L^p$-spaces when $a_0$ and $b_0$ are non-negative functions. Mazzucato-Nistor \cite{MN06} studied maximal $L^p$-$L^q$-regularity of the heat kernel on noncompact manifolds. They showed that their heat system admits a unique local-in-time strong $L^p$ solution when initial data is in an interpolation space. Koba \cite{Kob22} applied maximal $L^p$-regularity of the heat kernel for a bounded domain to show the existence of a unique global-in-time strong solution to their advection-diffusion equation on an evolving surface with a boundary. Koba \cite{Kob26} considered their heat equations in two half spaces $\mathbb{R}^3_+$, $\mathbb{R}^3_-$ and the interface $\mathbb{R}^2 \times \{ 0 \} ( \cong \mathbb{R}^2)$. They studied fundamental properties of their Laplace operator for the two half spaces and the interface to construct a unique global-in-time strong solution to their system. In this paper, we study the heat equations in two unbounded domains $\Omega_A$, $\Omega_B$ and the interface $\Gamma (= \partial \Omega_A \cap \partial \Omega_B)$ when the slope of the interface is gentle. To prove Theorem \ref{thm11}, we make use of results in \cite{Uka87, DHP01} and improve methods in \cite{Kob22,Kob26}. We apply maximal $L^p$-regularity of our Laplace operator $\mathcal{A}$ and nice properties of the semigroup ${\rm{e}}^{- t \mathcal{A}}$ generated by $\mathcal{A}$ to construct strong solutions to our system. The paper \cite{FGGR02} studied the heat equation in a bounded domain with their generalized Wentzell boundary condition, on the other hand, this paper considers an evolution equation in the interface $\Gamma (=\partial \Omega_A \cap \partial \Omega_B)$.

The outline of this paper is as follows: In Section \ref{sect2}, we first introduce four systems to prove Theorem \ref{thm11}, and then state one of the main results of this paper. In Section \ref{sect3}, we introduce and study function spaces in two unbounded domains $\Omega_A$, $\Omega_B$ and the interface $\Gamma$. In particular, we investigate properties of trace operators $\gamma_A$ and $\gamma_B$. In Section \ref{sect4}, we discuss the uniqueness and the existence of solutions to our systems. More precisely, we investigate the relationship between the solutions to system \eqref{eq15} and system \eqref{eq29}. In Section \ref{sect5}, we prove Theorem \ref{thm11}. We apply our function spaces, maximal $L^2$-regularity of operator $\mathcal{A}$, and an energy equality to show the existence of a global-in-time strong solution to our systems. In Appendix $(\rm{I})$ (Section \ref{sect6}), we derive our heat equations \eqref{eq12} in two unbounded domains $\Omega_A$, $\Omega_B$ and the interface $\Gamma$ from an energetic point of view. In Appendix $(\rm{II})$ (Section \ref{sect7}), we recall representation formulas for differential operators $\partial_j^\Gamma$, $\Delta_\Gamma$, and study the relationships between $\Vert \cdot \Vert_{\mathcal{W}^{2,2} (\Gamma)}$ and $\Vert \cdot \Vert_{W^{2,2}(\mathbb{R}^2)}$, between $\Vert \cdot \Vert_{\mathcal{W}^{2,2} (\Omega_A)}$ and $\Vert \cdot \Vert_{W^{2,2}(\mathbb{R}^3_+)}$, and between $\Vert \cdot \Vert_{\mathcal{W}^{2,2} (\Omega_B)}$ and $\Vert \cdot \Vert_{W^{2,2}(\mathbb{R}^3_-)}$. See Sections \ref{sect3} and \ref{sect7} for $\Vert \cdot \Vert_{\mathcal{W}^{2,2} (\Gamma)}$, $\Vert \cdot \Vert_{\mathcal{W}^{2,2} (\Omega_A)}$, and $\Vert \cdot \Vert_{\mathcal{W}^{2,2} (\Omega_B)}$.

\section{Four Systems and Main Results}\label{sect2}

In this section, we introduce four systems to construct strong solutions of equations \eqref{eq15}. We first deal with $\gamma_A [\theta_A] = \gamma_B[\theta_B] = \theta_S$ in \eqref{eq15}. Secondly, we transform a system of equations in two unbounded domains $\Omega_A,\Omega_B$ and the interface $\Gamma (= \partial \Omega_A \cap \partial \Omega_B)$ into a system of equations in two half spaces $\mathbb{R}^3_+, \mathbb{R}^3_-$ and the whole space $\mathbb{R}^2$. Thirdly, we change our equations to abstract systems. Finally, we state one of the main results of this paper.

Let $\varphi \in BC^2(\mathbb{R}^2)$ and $n= n(x) = { }^t (n_1,n_2,n_3)$ be the unit outer normal vector at $x \in \Gamma$ defined by \eqref{eq11}. Let $\rho_A, \rho_B, \rho_S, \kappa_A, \kappa_B, \kappa_S >0$. Let $\theta_0^A \in W^{1,2} (\Omega_A) \cap C_0^1 ( \overline{\Omega_A})$, $\theta_0^B \in W^{1,2}( \Omega_B) \cap C_0^1 ( \overline{\Omega_B})$, and $\theta_0^S \in W^{1,2} (\Gamma) \cap C_0^1 ( \Gamma)$ such that $\gamma_A [ \theta_0^A ] = \theta_0^A \vert_{\partial \Omega_A} = \theta_0^S$ and $\gamma_B[\theta_0^B] = \theta_0^B \vert_{\partial \Omega_B} = \theta_0^S$, where $\gamma_A$, $\gamma_B$ are two trace operators such that $\gamma_A:W^{1,2} (\Omega_A ) \to L^2(\Gamma) (=L^2(\partial \Omega_A))$ and $\gamma_B: W^{1,2} (\Omega_B) \to L^2(\Gamma)(=L^2(\partial \Omega_B))$. See Section \ref{sect3} for function spaces $W^{1,2} (\Omega_A)$, $C_0^1 ( \overline{\Omega_A})$, $W^{1,2} (\Omega_B)$, $C_0^1 ( \overline{\Omega_B})$, $W^{1,2} (\Gamma)$, $C_0^1 ( \Gamma)$, and two trace operators $\gamma_A$, $\gamma_B$. Assume that $\theta_A$, $\theta_B$, $\theta_S$ are smooth functions in $\mathbb{R}^4$, and that ${ }^t(\theta_A , \theta_B , \theta_S)$ is a strong solution to system \eqref{eq15} with initial data ${ }^t (\theta_0^A , \theta_0^B , \theta_0^S)$.

Let us first deal with $\gamma_A [\theta_A] = \gamma_B[\theta_B] = \theta_S$ in system \eqref{eq15}. For $x_h = { }^t (x_1,x_2) \in \mathbb{R}^2$ and $t>0$, we define
\begin{equation}\label{eq21}
\hat{\theta}_S = \hat{\theta}_S (x_h,t) = \theta_S (x_1 , x_2 , \varphi (x_1,x_2) ,t ). 
\end{equation}
Let $\delta >0$. Set
\begin{equation}\label{eq22}
\begin{cases}
u_A = u_A (x , t ) = \theta_A (x ,t ) - \hat{\theta}_S ( x_h , t) {\rm{e}}^{\delta \{ \varphi (x_h) -x_3 \} } { \ }(x \in \overline{\Omega_A},{ \ }t>0),\\
u_B = u_B (x , t ) = \theta_B (x ,t ) - \hat{\theta}_S ( x_h , t) {\rm{e}}^{ \delta \{ x_3 - \varphi (x_h )  \} } { \ }(x \in \overline{\Omega_B},{ \ }t>0),\\
u_S = u_S (x , t ) = \theta_S ( x , t ) { \ }(x \in \Gamma,{ \ }t>0),\\
\end{cases}
\end{equation}
\begin{equation}\label{eq23}
\begin{cases}
u_0^A = u_0^A (x) = \theta_0^A(x) - \hat{\theta}_0^S(x_h) {\rm{e}}^{ \delta \{ \varphi (x_h ) -x_3 \} } { \ }(x \in \overline{\Omega_A}),\\
u_0^B = u_0^B (x) = \theta_0^B(x)  - \hat{\theta}_0^S(x_h) {\rm{e}}^{ \delta \{ x_3 - \varphi (x_h )  \} } { \ }(x \in \overline{\Omega_B}),\\
u_0^S = u_0^S (x) = \theta_0^S(x) { \ }(x \in \Gamma).
\end{cases}
\end{equation}
From \eqref{eq22}, \eqref{eq23}, and \eqref{eq15}, we have the following equations.
\begin{equation}\label{eq24}
\begin{cases}
\partial_t u_A - \kappa_A \Delta u_A = \mathcal{F}_A(u_A, u_B, \hat{\theta}_S)  & \text{ in } \Omega_{A,T} ,\\
\partial_t u_B - \kappa_B \Delta u_B = \mathcal{F}_B (u_A, u_B, \hat{\theta}_S) & \text{ in } \Omega_{B,T}\\
\partial_t u_S - \kappa_S \Delta_\Gamma u_S = \mathcal{F}_S (u_A, u_B , \hat{\theta}_S) & \text{ in } \Gamma_T,\\
\gamma_A [u_A] = \gamma_B [u_B] = 0 & \text{ in } \Gamma_T,\\
u_A \vert_{t=0} = u^A_0{ \ } \text{ in } \Omega_A, { \ }u_B \vert_{t = 0} = u_0^B{ \ } \text{ in }\Omega_B, { \ }u_S \vert_{t=0} = u_0^S{ \ } \text{ in } \Gamma,&
\end{cases}
\end{equation}
with
\begin{equation*}
{\displaystyle{\lim_{ \vert x \vert \to \infty , { \ }x \in \Omega_A } u_A =0}},{ \ }{\displaystyle{\lim_{ \vert x \vert \to \infty,{ \ }x \in \Omega_B } u_B  =0}},{ \ }{\displaystyle{\lim_{ \vert x_h \vert \to \infty,{ \ }x \in \Gamma } u_S =0}} \text{ on } (0, T),
\end{equation*}
where
\begin{multline}\label{eq25}
\mathcal{F}_A (u_A ,u_B ,\hat{\theta}_S ) = - (\partial_t \hat{\theta}_S) {\rm{e}}^{ \delta \{  \varphi (x_h) - x_3 \}} + \kappa_A  (\Delta_h \hat{\theta}_S) {\rm{e}}^{ \delta \{ \varphi (x_h) - x_3 \}}\\ +  \kappa_A \hat{\theta}_S \{ \delta^2 + \delta \varphi_{11} + \delta \varphi_{22} + (\delta \varphi_1)^2 + (\delta \varphi_2)^2 \}  {\rm{e}}^{ \delta \{ \varphi (x_h) - x_3 \}}\\
+ 2 \kappa_A (\delta \varphi_1 \partial_1 \hat{\theta}_S + \delta \varphi_2 \partial_2 \hat{\theta}_S)  {\rm{e}}^{ \delta \{ \varphi (x_h) - x_3 \}} ,
\end{multline}
\begin{multline}\label{eq26}
\mathcal{F}_B (u_A ,u_B , \hat{\theta}_S ) = - (\partial_t \hat{\theta}_S) {\rm{e}}^{ \delta \{  x_3 - \varphi (x_h) \}} + \kappa_B  (\Delta_h \hat{\theta}_S) {\rm{e}}^{ \delta \{ x_3 - \varphi (x_h) \}}\\ +  \kappa_B \hat{\theta}_S \{ \delta^2 - \delta \varphi_{11} - \delta \varphi_{22} + (\delta \varphi_1)^2 + (\delta \varphi_2)^2 \}  {\rm{e}}^{ \delta \{ x_3 - \varphi (x_h) \}}\\
- 2 \kappa_B(\delta \varphi_1 \partial_1 \hat{\theta}_S + \delta \varphi_2 \partial_2 \hat{\theta}_S) {\rm{e}}^{ \delta \{ x_3 - \varphi (x_h) \}} ,
\end{multline}
\begin{multline}\label{eq27}
\mathcal{F}_S (u_A ,u_B , \hat{\theta}_S ) = \frac{\rho_A \kappa_A}{\rho_S} \gamma_A[  (n, \nabla ) u_A ] - \frac{\rho_B \kappa_B}{\rho_S} \gamma_B[  (n \cdot \nabla ) u_B ]\\
+ \frac{\rho_A \kappa_A}{\rho_S} \gamma_A[  (n, \nabla ) \{ \hat{\theta}_S {\rm{e}}^{ \delta \{ \varphi (x_h) - x_3  \}} \} ] - \frac{\rho_B \kappa_B}{\rho_S} \gamma_B[  (n, \nabla ) \{ \hat{\theta}_S  {\rm{e}}^{ \delta \{ x_3 - \varphi (x_h) \}} \} ].
\end{multline}
Here $\varphi_\alpha = \partial \varphi/{\partial x_\alpha}$ and $\varphi_{\alpha \beta} = \partial^2 \varphi/{\partial x_\alpha \partial x_\beta }$. See Proposition \ref{prop3017} and Remark \ref{rem7017} for $\gamma_A [(n \cdot \nabla) f]$ and $\gamma_B[ (n \cdot \nabla ) f]$.
\begin{remark}\label{rem21}
We check that for all $x \in \Gamma $, $0 < t < T$, and $\sharp \in \{ A , B \}$,
\begin{align*}
u_\sharp (x,t) &= \theta_\sharp (x,t) \vert_{\Gamma} - \hat{\theta}_S (x_h,t) = \theta_S (x,t ) - \hat{\theta}_S (x_h ,t)\\
 &= \theta_S (x_1,x_2, \varphi (x_h), t ) - \hat{\theta}_S (x_h ,t) = \hat{\theta}_S (x_h ,t) - \hat{\theta}_S (x_h ,t) = 0.
\end{align*}
We also find that $u_0^A \in W_0^{1,2} ( \Omega_A)$ and $u_0^B \in W_0^{1,2} ( \Omega_B)$. See also Proposition \ref{prop31}.
\end{remark}

To solve system \eqref{eq24}, we transform the system into a system of equations in two half spaces $\mathbb{R}^3_+$, $\mathbb{R}^3_-$ and the whole space $\mathbb{R}^2$. To this end, we use the tools.
\begin{lemma}\label{lem22}
$(\rm{i})$ For all $y= { }^t(y_1,y_2,y_3) \in \overline{\mathbb{R}^3_+}$ and $x = { }^t(x_1,x_2,x_3) \in \overline{\Omega_A}$, we set
\begin{equation*}
\tilde{x}_A = \tilde{x}_A (y)  =
\begin{pmatrix}
\tilde{x}_1^A\\
\tilde{x}_2^A\\
\tilde{x}_3^A
\end{pmatrix}
=
\begin{pmatrix}
y_1\\
y_2\\
y_3 + \varphi (y_1,y_2)
\end{pmatrix},{ \ } \tilde{x}_A^{-1} (x) = \begin{pmatrix}
x_1\\
x_2\\
x_3 - \varphi (x_1,x_2)
\end{pmatrix}.
\end{equation*}
Then each mapping $\tilde{x}_A: \mathbb{R}^3_+ \to \Omega_A$ and $\tilde{x}_A: \overline{\mathbb{R}^3_+} \to \overline{\Omega_A}$ is bijective, $\tilde{x}_A^{-1}$ is the inverse function of $\tilde{x}_A$, and
\begin{equation*}
\Omega_A = \{ x \in \mathbb{R}^3; x = \tilde{x}_A(y), y \in \mathbb{R}^3_+ \},{ \ }\mathbb{R}^3_+ = \{ y \in \mathbb{R}^3;{ \ }y = \tilde{x}^{-1}_A (x), { \ }x \in \Omega_A \}.
\end{equation*}
$(\rm{ii})$ For all $z= { }^t(z_1,z_2,z_3) \in \overline{\mathbb{R}^3_-}$ and $x = { }^t(x_1,x_2,x_3) \in \overline{\Omega_B}$, we set
\begin{equation*}
\check{x}_B = \check{x}_B (z) =
\begin{pmatrix}
\check{x}_1^B\\
\check{x}_2^B\\
\check{x}_3^B
\end{pmatrix} =
\begin{pmatrix}
z_1\\
z_2\\
z_3 + \varphi (z_1,z_2)
\end{pmatrix},{ \ }
\check{x}_B^{-1} (x) = \begin{pmatrix}
x_1\\
x_2\\
x_3 - \varphi (x_1,x_2)
\end{pmatrix}.
\end{equation*}
Then each mapping $\check{x}_B: \mathbb{R}^3_- \to \Omega_B$ and $\check{x}_B: \overline{\mathbb{R}^3_-} \to \overline{\Omega_B}$ is bijective, $\check{x}_B^{-1}$ is the inverse function of $\check{x}_B$, and
\begin{equation*}
\Omega_B = \{ x \in \mathbb{R}^3; x = \check{x}_B(z), z \in \mathbb{R}^3_- \},{ \ }\mathbb{R}^3_- = \{ z \in \mathbb{R}^3;{ \ }z = \check{x}^{-1}_B (x), { \ }x \in \Omega_B \}.
\end{equation*}
$(\rm{iii})$ For all $X= { }^t(X_1,X_2) \in \mathbb{R}^2$ and $x = { }^t(x_1,x_2,x_3) \in \Gamma$, we set
\begin{equation*}
\hat{x}_S = \hat{x}_S (X) =
\begin{pmatrix}
\hat{x}_1^S\\
\hat{x}_2^S\\
\hat{x}_3^S
\end{pmatrix} =
\begin{pmatrix}
X_1\\
X_2\\
\varphi (X_1,X_2)
\end{pmatrix},{ \ }\hat{x}_S^{-1} (x) = \begin{pmatrix}
x_1\\
x_2
\end{pmatrix}.
\end{equation*}
Then the mapping $\hat{x}_S: \mathbb{R}^2 \to \Gamma$ is bijective, $\hat{x}_S^{-1}$ is the inverse function of $\hat{x}_S$, and
\begin{equation*}
\Gamma = \{ x \in \mathbb{R}^3; x = \hat{x}_S(X), X \in \mathbb{R}^2 \},{ \ }\mathbb{R}^2 = \{ X \in \mathbb{R}^2;{ \ }X = \hat{x}^{-1}_S (x), { \ }x \in \Gamma \}.
\end{equation*}
\end{lemma}
\noindent The proof of Lemma \ref{lem22} is left for the readers. See Section \ref{sect7} for properties of $(\tilde{x}_A$, $\check{x}_B$, $\hat{x}_S$).

Let $\tilde{x}_A$, $\check{x}_B$, $\hat{x}_S$ be the three mappings defined by Lemma \ref{lem22}. Now we use $\tilde{x}_A$, $\check{x}_B$, $\hat{x}_S$ to change system \eqref{eq24} into a system of equations in two half spaces $\mathbb{R}^3_+,\mathbb{R}^3_-$ and the whole space $\mathbb{R}^2$. For $X \in \mathbb{R}^2$, we set
\begin{equation*}
G_S = G_S (X) = 1 + \varphi_1^2 + \varphi_2^2.
\end{equation*}
See subsection \ref{subsec71} for $G_S$ and differential operators $\partial_j^\Gamma$ and $\Delta_\Gamma$. Set $\tilde{u}_A = \tilde{u}_A (y,t) = u_A ( \tilde{x}_A (y),t)$, $\check{u}_B = \check{u}_B (z,t) = u_B ( \check{x}_B (z),t)$, $\hat{u}_S = \hat{u}_S (X,t) = u_S ( \hat{x}_S (X) , t)$, $\tilde{u}_0^A = \tilde{u}_0^A (y) = u_0^A ( \tilde{x}_A (y))$, $\check{u}_0^B = \check{u}_0^B (z) = u_0^B ( \check{x}_B (z))$, and $\hat{u}_0^S = \hat{u}_0^S (X) = u_0^S ( \hat{x}_S (X))$. Using change of variables with $(\tilde{x}_A$, $\check{x}_B$, $\hat{x}_S$), and then applying Lemma \ref{lem71}, assertion $(\rm{ii})$ in Remark \ref{rem72}, Lemma \ref{lem75}, $(\rm{ii})$ in Remark \ref{rem76}, Lemma \ref{lem79}, $(\rm{i})$ in Remark \ref{rem7010}, and Remark \ref{rem7017} (see also Remark \ref{rem23}) into system \eqref{eq24}, we have
\begin{equation}\label{eq28}
\begin{cases}
\partial_t \tilde{u}_A + \mathscr{L}_1 \tilde{u}_A = \mathring{\mathcal{F}}_1 (\tilde{u}_A, \check{u}_B, \hat{u}_S)  & \text{ in } \mathbb{R}^3_+ \times (0, T) ,\\
\partial_t \check{u}_B + \mathscr{L}_2 \check{u}_B = \mathring{\mathcal{F}}_2 (\tilde{u}_A, \check{u}_B, \hat{u}_S)  & \text{ in } \mathbb{R}^3_- \times (0, T ),\\
\partial_t \hat{u}_S + \mathscr{L}_3 \hat{u}_S = \mathring{\mathcal{F}}_3 (\tilde{u}_A, \check{u}_B, \hat{u}_S)  & \text{ in } \mathbb{R}^2 \times (0, T ),\\
\gamma_+ [\tilde{u}_A] = \gamma_- [\check{u}_B] = 0 & \text{ in } \mathbb{R}^2 \times (0, T),\\
\tilde{u}_A \vert_{t=0} = \tilde{u}^A_0{ \ } \text{ in } \mathbb{R}^3_+, { \ }\check{u}_B \vert_{t = 0} = \check{u}_0^B{ \ } \text{ in }\mathbb{R}^3_-,& { \ }\hat{u}_S \vert_{t=0} = \hat{u}_0^S{ \ } \text{ in } \mathbb{R}^2,
\end{cases}
\end{equation}
with
\begin{equation*}
{\displaystyle{\lim_{ \vert y \vert \to \infty , { \ }y \in \mathbb{R}^3_+ } \tilde{u}_A =0}},{ \ }{\displaystyle{\lim_{ \vert z \vert \to \infty,{ \ }z \in \mathbb{R}^3_- } \check{u}_B  =0}},{ \ }{\displaystyle{\lim_{ \vert X \vert \to \infty,{ \ }X \in \mathbb{R}^2 } \hat{u}_S =0}} \text{ on } (0, T),
\end{equation*}
where $\mathscr{L}_1 \tilde{u}_A = - \kappa_A \Delta_y \tilde{u}_A + \mathscr{B}_1 \tilde{u}_A$, $\mathscr{L}_2 \check{u}_B = - \kappa_B \Delta_z \check{u}_B + \mathscr{B}_2 \check{u}_B$, $\mathscr{L}_3 f_S = - \kappa_S \Delta_X \hat{u}_S + \mathscr{B}_3 \hat{u}_S$, 
\begin{align*}
\mathscr{B}_1 \tilde{u}_A = \kappa_A \bigg( 2 \varphi_1 \frac{\partial^2 \tilde{u}_A}{\partial y_1 \partial y_3} + 2 \varphi_2 \frac{\partial^2 \tilde{u}_A}{\partial y_2 \partial y_3} - (\varphi_1^2 + \varphi_2^2) \frac{\partial^2 \tilde{u}_A}{\partial y_3^2 } + (\varphi_{11} +\varphi_{22}) \frac{\partial \tilde{u}_A}{\partial y_3} \bigg),\\
\mathscr{B}_2 \check{u}_B = \kappa_B \bigg(  2 \varphi_1 \frac{\partial^2 \check{u}_B}{\partial z_1 \partial z_3} + 2 \varphi_2 \frac{\partial^2 \check{u}_B}{\partial z_2 \partial z_3} - (\varphi_1^2 + \varphi_2^2) \frac{\partial^2 \check{u}_B}{\partial z_3^2 } + ( \varphi_{11} + \varphi_{22} ) \frac{\partial \check{u}_B}{\partial z_3} \bigg),
\end{align*}
\begin{multline*}
\mathscr{B}_3 \hat{u}_S = \kappa_S \bigg( \frac{\varphi_1^2}{G_S} \frac{\partial^2 \hat{u}_S}{\partial X_1^2} + \frac{\varphi_2^2}{G_S}\frac{\partial^2 \hat{u}_S}{\partial X_2^2} + \frac{2 \varphi_1 \varphi_2}{G_S}\frac{\partial^2 \hat{u}_S}{\partial X_1 \partial X_2}\\
+ \frac{\varphi_1 ( \varphi_{11} + \varphi_{22} + \varphi_2^2 \varphi_{11} +  \varphi_1^2 \varphi_{22} - 2 \varphi_1 \varphi_2 \varphi_{12}) }{G^2_S} \frac{\partial \hat{u}_S}{\partial X_1}\\
+ \frac{ \varphi_2 ( \varphi_{11} + \varphi_{22} + \varphi_2^2 \varphi_{11} +  \varphi_1^2 \varphi_{22} - 2 \varphi_1 \varphi_2 \varphi_{12})}{G^2_S}  \frac{\partial \hat{u}_S}{\partial X_2} \bigg),
\end{multline*}
\begin{multline*}
\mathring{\mathcal{F}}_1 (\tilde{u}_A, \check{u}_B, \hat{u}_S)  = - (\partial_t \hat{u}_S) {\rm{e}}^{ - \delta y_3} + \kappa_A  (\Delta_{y_h} \hat{u}_S) {\rm{e}}^{ - \delta y_3}\\ +  \kappa_A \hat{u}_S \{ \delta^2 + \delta \varphi_{11} + \delta \varphi_{22} + (\delta \varphi_1)^2 + (\delta \varphi_2)^2 \}  {\rm{e}}^{ - \delta y_3}\\
+ 2 \kappa_A (\delta \varphi_1 \partial_{y_1} \hat{u}_S + \delta \varphi_2 \partial_{y_2} \hat{u}_S) {\rm{e}}^{ - \delta y_3},
\end{multline*}
\begin{multline*}
\mathring{\mathcal{F}}_2 (\tilde{u}_A, \check{u}_B, \hat{u}_S)  = - (\partial_t \hat{u}_S) {\rm{e}}^{ \delta z_3} + \kappa_B  (\Delta_{z_h} \hat{u}_S) {\rm{e}}^{ \delta z_3}\\ +  \kappa_B \hat{u}_S \{ \delta^2 - \delta \varphi_{11} - \delta \varphi_{22} + (\delta \varphi_1)^2 + (\delta \varphi_2)^2 \}  {\rm{e}}^{ \delta z_3}\\
- 2 \kappa_B(\delta \varphi_1 \partial_{z_1} \hat{u}_S + \delta \varphi_2 \partial_{z_2} \hat{u}_S) {\rm{e}}^{ \delta z_3},
\end{multline*}
\begin{multline*}
\mathring{\mathcal{F}}_3 (\tilde{u}_A, \check{u}_B, \hat{u}_S)  = \frac{\rho_A \kappa_A}{\rho_S} \gamma_+ \bigg[  - \frac{\varphi_1}{\sqrt{G_S}} \frac{\partial \tilde{u}_A }{\partial y_1} - \frac{\varphi_2}{\sqrt{G_S}} \frac{\partial \tilde{u}_A }{\partial y_2} +  \frac{ 1 + \varphi_1^2 + \varphi_2^2 }{\sqrt{G_S}} \frac{\partial \tilde{u}_A }{\partial y_3} \bigg](X,t)\\
 - \frac{\rho_B \kappa_B}{\rho_S} \gamma_- \bigg[ - \frac{\varphi_1}{\sqrt{G_S}} \frac{\partial \check{u}_B }{\partial z_1} - \frac{\varphi_2}{\sqrt{G_S}} \frac{\partial \check{u}_B }{\partial z_2} +  \frac{1 + \varphi_1^2 + \varphi_2^2 }{\sqrt{G_S}} \frac{\partial \check{u}_B }{\partial z_3} \bigg](X,t)\\
- \bigg( \frac{ \rho_A \kappa_A - \rho_B \kappa_B }{\rho_S} \bigg) \bigg( \frac{\varphi_1}{\sqrt{G_S}} \frac{\partial \hat{u}_S}{\partial X_1} + \frac{\varphi_2}{\sqrt{G_S}} \frac{\partial \hat{u}_S}{\partial X_2} \bigg)\\
- \bigg( \frac{ \rho_A \kappa_A + \rho_B \kappa_B  }{\rho_S} \bigg) \bigg( \frac{\delta ( 1 + \varphi_1^2 + \varphi_2^2 )}{\sqrt{G_S}} \hat{u}_S \bigg) .
\end{multline*}
Here $\gamma_+$, $\gamma_-$ are the two trace operators such that
\begin{align*}
\gamma_+[\cdot] : W^{1,2} (\mathbb{R}^3_+) \to L^2 (\mathbb{R}^2)(= L^2 (\partial \mathbb{R}^3_+)),\\
\gamma_-[\cdot] : W^{1,2} (\mathbb{R}^3_-) \to L^2 (\mathbb{R}^2)(= L^2(\partial \mathbb{R}^3_-)).
\end{align*}
\noindent Note that $\hat{u}_S = \hat{\theta}_S$ and that $\hat{\theta}_S$ and $\varphi$ do not depend on $x_3$.
\begin{remark}\label{rem23} $(\rm{i})$ Applying Remark \ref{rem7017}, we see that
\begin{multline*}
\int_\Gamma (n,\nabla) \{ \hat{\theta}_S {\rm{e}}^{\delta \{ \varphi (x_h ) -x_3 \}} \} { \ }d \mathcal{H}_x^2\\
 =\int_{\mathbb{R}^2} \bigg( - \frac{\varphi_1}{\sqrt{G_S}} \frac{\partial }{\partial y_1} - \frac{\varphi_2}{\sqrt{G_S}}  \frac{\partial }{\partial y_2} + \frac{ 1 + \varphi_1^2 + \varphi_2^2 }{\sqrt{G_S}} \frac{\partial }{\partial y_3}\bigg) \{  \hat{\theta}_S(y_h,t) {\rm{e}}^{- \delta y_3} \} \bigg\vert_{y_3 =0} \sqrt{G_S} { \ }dy_h\\
 = \int_{\mathbb{R}^2} \bigg( - \frac{\varphi_1}{\sqrt{G_S}} \frac{\partial \hat{\theta}_S }{\partial y_1} - \frac{\varphi_2}{\sqrt{G_S}}  \frac{\partial \hat{\theta}_S }{\partial y_2} - \delta \frac{ 1 + \varphi_1^2 + \varphi_2^2 }{\sqrt{G_S}} \hat{\theta}_S \bigg) \sqrt{G_S} { \ }d y_h\\
= \int_{\mathbb{R}^2 } \bigg( - \frac{\varphi_1}{\sqrt{G_S}} \frac{\partial \hat{u}_S }{\partial X_1} - \frac{\varphi_2}{\sqrt{G_S}}  \frac{\partial \hat{u}_S }{\partial X_2} - \delta \frac{ 1 + \varphi_1^2 + \varphi_2^2 }{\sqrt{G_S}} \hat{u}_S \bigg) \sqrt{G_S} { \ }dX.
\end{multline*}
Similarly, we see that
\begin{multline*}
\int_\Gamma (n,\nabla) \{ \hat{\theta}_S {\rm{e}}^{\delta \{ x_3 - \varphi (x_h ) \}} \} { \ }d \mathcal{H}_x^2\\
= \int_{\mathbb{R}^2 } \bigg( - \frac{\varphi_1}{\sqrt{G_S}} \frac{\partial \hat{u}_S }{\partial X_1} - \frac{\varphi_2}{\sqrt{G_S}}  \frac{\partial \hat{u}_S }{\partial X_2} + \delta \frac{ 1 + \varphi_1^2 + \varphi_2^2 }{\sqrt{G_S}} \hat{u}_S \bigg) \sqrt{G_S} { \ }dX.
\end{multline*}
$(\rm{ii})$ From Remark \ref{rem21} and Section \ref{sect3}, we find that
\begin{equation*}
\tilde{u}_0^A \in W_0^{1,2} (\mathbb{R}^3_+), { \ }\check{u}_0^B \in W_0^{1,2} (\mathbb{R}^3_-), { \ }\hat{u}_0^S \in W^{1,2} (\mathbb{R}^2).
\end{equation*}
\end{remark}

In this paper, we apply the semigroup and maximal regularity theories to construct strong solutions of our systems. To end this, we change system \eqref{eq28} into an abstract system. Define six operators $\mathcal{A}_1$, $\mathcal{A}_2$, $\mathcal{A}_3$, $\mathcal{B}_1$, $\mathcal{B}_2$, $\mathcal{B}_3$ as follows:
\begin{equation*}
\begin{cases}
\mathcal{A}_1 \psi_A= - \kappa_A \Delta_y \psi_A,\\
D (\mathcal{A}_1) = W_0^{1,2} (\mathbb{R}^3_+) \cap W^{2,2} ( \mathbb{R}^3_+),
\end{cases}{ \ }
\begin{cases}
\mathcal{B}_1 \psi_A = \mathscr{B}_1 \psi_A,\\
D (\mathcal{B}_1) = W_0^{1,2} (\mathbb{R}^3_+) \cap W^{2,2} ( \mathbb{R}^3_+),
\end{cases} 
\end{equation*}
\begin{equation*}
\begin{cases}
\mathcal{A}_2 \psi_B= - \kappa_B \Delta_z \psi_B,\\
D (\mathcal{A}_2) = W_0^{1,2} (\mathbb{R}^3_-) \cap W^{2,2} ( \mathbb{R}^3_-),
\end{cases}{ \ }
\begin{cases}
\mathcal{B}_2 \psi_B = \mathscr{B}_2 \psi_B,\\
D (\mathcal{B}_1) = W_0^{1,2} (\mathbb{R}^3_-) \cap W^{2,2} ( \mathbb{R}^3_-),
\end{cases} 
\end{equation*}
\begin{equation*}
\begin{cases}
\mathcal{A}_3 \psi_S = - \kappa_S \Delta_X \psi_S,\\
D (\mathcal{A}_3) = W^{2,2} ( \mathbb{R}^2),
\end{cases}{ \ }
\begin{cases}
\mathcal{B}_3 \psi_S = \mathscr{B}_3 \psi_S,\\
D (\mathcal{B}_3) = W^{2,2} ( \mathbb{R}^2).
\end{cases} 
\end{equation*}
See subsections \ref{subsec51} and \ref{subsec52} for details on the six operators. Applying these operators into equations \eqref{eq28}, we have
\begin{equation}\label{eq29}
\begin{cases}
\displaystyle{\frac{d}{dt} v_A + \mathcal{A}_1 v_A =  - \mathcal{B}_1 v_A + \mathcal{F}_1 (v_A, v_B, v_S) { \ \ }\text{ on } (0, T) } ,\\[10pt]
\displaystyle{\frac{d}{dt} v_B + \mathcal{A}_2 v_B = - \mathcal{B}_2 v_B + \mathcal{F}_2 (v_A, v_B, v_S) { \ \ }\text{ on } (0, T) } ,\\[10pt]
\displaystyle{\frac{d}{dt} v_S + \mathcal{A}_3 v_S = - \mathcal{B}_3 v_S + \mathcal{F}_3 (v_A, v_B, v_S) { \ \ }\text{ on } (0, T) } ,\\[10pt]
v_A \vert_{t=0} = v_0^A, { \ }v_B \vert_{t = 0} = v_0^B, { \ }v_S \vert_{t=0} = v_0^S,
\end{cases}
\end{equation}
where $v_0^A = \tilde{u}_0^A$, $v_0^B = \check{u}_0^B$, $v_0^S = \hat{u}_0^S$,
\begin{multline*}
\mathcal{F}_1 (v_A, v_B, v_S)  = - ( d v_S/{dt}) {\rm{e}}^{ - \delta y_3} + \kappa_A  (\Delta_{y_h} v_S) {\rm{e}}^{ - \delta y_3}\\ +  \kappa_A v_S \{ \delta^2 + \delta \varphi_{11} + \delta \varphi_{22} + (\delta \varphi_1)^2 + (\delta \varphi_2)^2 \}  {\rm{e}}^{ - \delta y_3}\\
+ 2 \kappa_A (\delta \varphi_1 \partial_{y_1} v_S + \delta \varphi_2 \partial_{y_2} v_S) {\rm{e}}^{ - \delta y_3},
\end{multline*}
\begin{multline*}
\mathcal{F}_2 (v_A, v_B, v_S)  = - (d v_S/{dt}) {\rm{e}}^{ \delta z_3} + \kappa_B  (\Delta_{z_h} v_S) {\rm{e}}^{ \delta z_3}\\ +  \kappa_B v_S \{ \delta^2 - \delta \varphi_{11} - \delta \varphi_{22} + (\delta \varphi_1)^2 + (\delta \varphi_2)^2 \}  {\rm{e}}^{ \delta z_3}\\
- 2 \kappa_B(\delta \varphi_1 \partial_{z_1} v_S + \delta \varphi_2 \partial_{z_2} v_S) {\rm{e}}^{ \delta z_3},
\end{multline*}
\begin{multline*}
\mathcal{F}_3 (v_A, v_B, v_S)  = \frac{\rho_A \kappa_A}{\rho_S} \gamma_+ \bigg[  - \frac{\varphi_1}{\sqrt{G_S}} \frac{\partial v_A }{\partial y_1} - \frac{\varphi_2}{\sqrt{G_S}} \frac{\partial v_A }{\partial y_2} +  \frac{1 + \varphi_1^2 + \varphi_2^2 }{\sqrt{G_S}} \frac{\partial v_A }{\partial y_3}  \bigg]\\
 - \frac{\rho_B \kappa_B}{\rho_S} \gamma_- \bigg[ - \frac{\varphi_1}{\sqrt{G_S}} \frac{\partial v_B }{\partial z_1} - \frac{\varphi_2}{\sqrt{G_S}} \frac{\partial v_B }{\partial z_2} +  \frac{1 + \varphi_1^2 + \varphi_2^2 }{\sqrt{G_S}} \frac{\partial v_B }{\partial z_3} \bigg]\\
- \bigg( \frac{ \rho_A \kappa_A - \rho_B \kappa_B  }{\rho_S} \bigg) \bigg( \frac{\varphi_1}{\sqrt{G_S}} \frac{\partial v_S}{\partial X_1} + \frac{\varphi_2}{\sqrt{G_S}} \frac{\partial v_S}{\partial X_2} \bigg)\\
- \bigg( \frac{ \rho_A \kappa_A + \rho_B \kappa_B }{\rho_S} \bigg) \bigg( \frac{\delta ( 1 + \varphi_1^2 + \varphi_2^2 )}{\sqrt{G_S}} v_S \bigg).
\end{multline*}
\noindent To solve system \eqref{eq29}, we transform the system into a more abstract system. Set $v = v(t) = { }^t(v_A,v_B,v_S)$, $\mathcal{F} = \mathcal{F} (v) = { }^t (\mathcal{F}_1 ,\mathcal{F}_2,\mathcal{F}_3)$, and $v_0 = { }^t ( v_0^A,v_0^B,v_0^S)$. Define operators $\mathcal{A}$, $\mathcal{B}$ on $H$ as follows:
\begin{equation*}
\begin{cases}
\mathcal{A} f=  { }^t ( \mathcal{A}_1 f_A, \mathcal{A}_2 f_B , \mathcal{A}_3 f_S )\\
D (\mathcal{A}) = D ( \mathcal{A}_1) \times D ( \mathcal{A}_2) \times D (\mathcal{A}_3),
\end{cases}{ \ }
\begin{cases}
\mathcal{B} f=  { }^t ( \mathcal{B}_1 f_A, \mathcal{B}_2 f_B , \mathcal{B}_3 f_S )\\
D (\mathcal{B}) = D ( \mathcal{B}_1) \times D ( \mathcal{B}_2) \times D (\mathcal{B}_3).
\end{cases} 
\end{equation*}
Here $H$ is the function space defined by \eqref{eq17}. From \eqref{eq29}, we obtain
\begin{equation}\label{eq2010}
\begin{cases}
\displaystyle{\frac{d}{dt} v + \mathcal{A} v =  - \mathcal{B} v + \mathcal{F} (v) { \ \ }\text{ on } (0, T) } ,\\
v \vert_{t=0} = v_0,
\end{cases}
\end{equation}
Therefore, we have four systems \eqref{eq24}, \eqref{eq28}, \eqref{eq29}, \eqref{eq2010} to construct a solution of equations \eqref{eq15}. In Section \ref{sect5}, we construct a strong solution to system \eqref{eq2010} by applying nice properties of operator $\mathcal{A}$ and restricting our parameters. We make use of a strong solution of \eqref{eq2010} to construct a strong solution to \eqref{eq15}.

Finally, we state one of the main results of this paper.
\begin{theorem}\label{thm24}
Let $\varphi \in BC^2(\mathbb{R}^2)$ and $\rho_A, \rho_B, \rho_S, \kappa_A , \kappa_B , \kappa_S >0$. Let $\mathcal{M}_0 = \mathcal{M}_0(\kappa_A , \kappa_B ,\kappa_S)$ and $\delta_0 = \delta_0(\kappa_A , \kappa_B, \kappa_S)$ be the two positive constants defined by \eqref{eq5044} and \eqref{eq5045}, respectively. Suppose that $\Vert \nabla_X \varphi \Vert_{L^\infty (\mathbb{R}^2)} \leq 1/2$ and that $\delta = \delta_0$. Assume that
\begin{equation*}
\Vert \nabla_X \varphi \Vert_{L^\infty ( \mathbb{R}^2)} + \frac{\rho_A + \rho_B}{\rho_S} \leq \mathcal{M}_0.
\end{equation*}
Then for each $v_0 \in W_0^{1,2} (\mathbb{R}^3_+) \times W_0^{1,2} (\mathbb{R}^3_-) \times W^{1,2} ( \mathbb{R}^2)$, system \eqref{eq2010} admits a unique global-in-time strong solution $v$:
\begin{equation*}
v \in C ([0, \infty); H) \cap L_{loc}^2 (0,\infty; H_0^1 \cap H^2) \cap W_{loc}^{1,2}(0,\infty ; H) \cap L^2_{loc}( \mathbb{R}_+ \times \mathbb{R}^3_{+,-,0}),
\end{equation*}
satisfying 
\begin{equation*}
\lim_{t \to 0 + 0} \Vert v (t) - v_0 \Vert_{H}  = 0.
\end{equation*}
Here
\begin{align*}
& L^2_{loc}(0,\infty; H_0^1 \cap H^2) := \{ w ; w \in L^2(0,T; H_0^1 \cap H^2) \text{ for each fixed } T >0  \},\\
& W^{1,2}_{loc}(0,\infty; H) := \{ w ; w \in W^{1,2}(0,T; H) \text{ for each fixed } T >0  \},\\
& L^2_{loc}( \mathbb{R}_+ \times \mathbb{R}^3_{+,-,0}) := \{ w; w \in L^2((0,T) \times \mathbb{R}^3_{+,-,0}) \text{ for each fixed } T >0 \},\\
& H_0^1 := W_0^{1,2}(\mathbb{R}^3_+) \times W_0^{1,2}(\mathbb{R}^3_-) \times W^{1,2}(\mathbb{R}^2),\\
& H^2 := W^{2,2}(\mathbb{R}^3_+) \times W^{2,2}(\mathbb{R}^3_-) \times W^{2,2}(\mathbb{R}^2).
\end{align*}
\end{theorem}
\noindent See Definition \ref{def54} for function spaces $L^2(0,T; H_0^1 \cap H^2)$ and $L^2((0,T) \times \mathbb{R}^3_{+,-,0})$. In Section \ref{sect5}, we prove Theorem \ref{thm24}. 
\section{Function Spaces and Trace Operators}\label{sect3}

In this section, we first introduce and study several function spaces such as $L^2(\Omega_A)$, $W_0^{1,2}(\Omega_B)$, and $W^{2,2} (\Gamma)$, and then characterize trace operators $\gamma_A,\gamma_B$ such that $\gamma_A: W^{1,2} (\Omega_A) \to L^2(\Gamma)$ and $\gamma_B: W^{1,2} (\Omega_B) \to L^2(\Gamma)$. Fix $\varphi \in BC^2( \mathbb{R}^2)$. Let $n= n(x) = { }^t (n_1,n_2,n_3)$ be the unit outer normal vector at $x \in \Gamma$ defined by \eqref{eq11}. Let $\tilde{x}_A$, $\check{x}_B$, $\hat{x}_S$, $g_A^{ij}$, $g_B^{ij}$, and $g_S^{\alpha \beta}$ be the functions appearing in Section \ref{sect7} (see Section \ref{sect7} for their properties of these functions and see also Lemma \ref{lem22}). Throughout this section, we assume that
\begin{equation*}
\Vert \nabla_X \varphi \Vert_{L^\infty ( \mathbb{R}^2)} = \sup_{X \in \mathbb{R}^2} \bigg\vert \frac{\partial \varphi}{\partial X_1}(X) \bigg\vert + \sup_{X \in \mathbb{R}^2} \bigg\vert \frac{\partial \varphi}{\partial X_2} (X) \bigg\vert \leq \frac{1}{2}.
\end{equation*}

The aim of this section is to prove the following key proposition.
\begin{proposition}\label{prop31}
Let $\delta>0$. Let $f_A \in W^{1,2}(\Omega_A)$, $f_B \in W^{1,2}(\Omega_B)$, $f_S \in W^{1,2}(\Gamma)$, and $\hat{f}_S \in W^{1,2} (\mathbb{R}^2)$ such that $f_S = \mathcal{P}_S [\hat{f}_S]$, where $\mathcal{P}_S$ is the pullback operator defined by Definition \ref{def32}. Assume that $\gamma_A[f_A] = f_S$ and $\gamma_B[f_B] = f_S$. Then, the following two assertions hold:\\ 
$(\rm{i})$ For almost all $x \in \Omega_A$, we define
\begin{equation*}
\mathfrak{u}_A = \mathfrak{u}_A (x) = f_A (x) - \hat{f}_S(x_h) {\rm{e}}^{ \delta \{ \varphi (x_h) - x_3 \}}.
\end{equation*}
Then $\mathfrak{u}_A \in W^{1,2} (\Omega_A)$ and
\begin{equation*}
\gamma_A [\mathfrak{u}_A] = 0. 
\end{equation*}
$(\rm{ii})$ For almost all $x \in \Omega_B$, we define
\begin{equation*}
\mathfrak{u}_B = \mathfrak{u}_B (x) = f_B (x) - \hat{f}_S(x_h) {\rm{e}}^{ \delta \{ x_3 - \varphi (x_h) \}}.
\end{equation*}
Then $\mathfrak{u}_B \in W^{1,2} (\Omega_B)$ and
\begin{equation*}
\gamma_B [\mathfrak{u}_B] = 0. 
\end{equation*}
Here $W^{1,2}(\Omega_A)$, $W^{1,2}(\Omega_B)$, $W^{1,2} (\Gamma)$ are the function spaces defined by Definitions \ref{def39}, \ref{def3010}, and $\gamma_A$, $\gamma_B$ are the two trace operators defined by Definition \ref{def3015}.
\end{proposition}
\noindent We prove Proposition \ref{prop31} at the second half of this section.

We define our function spaces by applying the following pullback operators.
\begin{definition}[Pullback operators]\label{def32}{ \ }\\
$(\rm{i})$ Let $\psi_A \in L^1_{loc}(\mathbb{R}^3_+)$. We define $\mathcal{P}_A[\psi_A]$ as follows: for almost all $x \in \Omega_A$
\begin{equation}\label{eq31}
\mathcal{P}_A [\psi_A] = \mathcal{P}_A [\psi_A] (x_1,x_2,x_3) := \psi_A (x_1 , x_2 , x_3 - \varphi (x_1,x_2)).
\end{equation}
In the case when $\psi_A \in C_0(\overline{ \mathbb{R}^3_+})$, for all $x \in \overline{\Omega_A}$, we define $\mathcal{P}_A[\psi_A]$ by \eqref{eq31}.\\
$(\rm{ii})$ Let $\psi_B \in L^1_{loc}( \mathbb{R}^3_-)$. We define $\mathcal{P}_B[\psi_B]$ as follows: for almost all $x \in \Omega_B$
\begin{equation}\label{eq32}
\mathcal{P}_B [\psi_B] = \mathcal{P}_B [\psi_B] (x_1,x_2,x_3) := \psi_B (x_1 , x_2 ,  x_3 - \varphi (x_1,x_2) ).
\end{equation}
In the case when $\psi_B \in C_0(\overline{ \mathbb{R}^3_-})$, for all $x \in \overline{\Omega_B}$, we define $\mathcal{P}_B[\psi_B]$ by \eqref{eq32}.\\
$(\rm{iii})$ Let $\psi_S \in L^1_{loc}(\mathbb{R}^2)$. We define $\mathcal{P}_S[\psi_S]$ as follows: for almost all $x \in \Gamma$
\begin{equation}\label{eq33}
\mathcal{P}_S [\psi_S] = \mathcal{P}_S [\psi_S] (x_1,x_2,x_3) := \psi_S (x_1 , x_2).
\end{equation}
In the case when $\psi_S \in C_0(\mathbb{R}^2 )$, for all $x \in \Gamma$, we define $\mathcal{P}_S[\psi_S]$ by \eqref{eq33}.
\end{definition}
\noindent From Lemma \ref{lem22}, we can define pullback operators $\mathcal{P}_A$, $\mathcal{P}_B$, and $\mathcal{P}_S$ on $L_{loc}^1(\mathbb{R}^3_+)$, $L^1_{loc}(\mathbb{R}^3_-)$, and $L^1_{loc} (\mathbb{R}^2)$, respectively.

Let us introduce several norms and semi-norms.
\begin{definition}[Norms and semi-norms (I)]\label{def33}
For $\psi_S \in L^1_{loc}(\mathbb{R}^2)$, we define
\begin{align*}
& \Vert \psi_S \Vert_{ \mathcal{L}^2 (\Gamma) } = \bigg( \int_{\mathbb{R}^2} \vert \psi_S (X) \vert^2 \sqrt{G_S} { \ }d X \bigg)^{1/2},\\
& \Vert \psi_S \Vert_{ \dot{\mathcal{W}}^{1,2} (\Gamma)} = \bigg( \int_{\mathbb{R}^2}  g_S^{\alpha \beta} \frac{\partial \psi_S }{\partial X_\alpha} \frac{ \partial \psi_S }{\partial X_\beta} \sqrt{G_S} { \ }dX \bigg)^{1/2},\\
& \Vert \psi_S \Vert_{ \dot{\mathcal{W}}^{2,2} (\Gamma) } = \bigg( \sum_{i,j=1}^3 \int_{\mathbb{R}^2} \bigg\vert g_S^{\alpha' \beta'} \frac{\partial \hat{x}^S_i }{\partial X_{\alpha'}}  \frac{\partial}{\partial X_{\beta'}}\bigg( g_S^{\alpha \beta} \frac{\partial \hat{x}^S_j }{\partial X_\alpha} \frac{ \partial \psi_S }{\partial X_\beta} \bigg) \bigg\vert^2 \sqrt{G_S} { \ }dX \bigg)^{1/2},\\
& \Vert \psi_S \Vert_{\mathring{\mathcal{W}}^{2,2} (\Gamma) }  = \bigg( \int_{\mathbb{R}^2} \bigg\vert \frac{1}{\sqrt{G_S}} \frac{\partial}{\partial X_\alpha} \bigg( \sqrt{G_S} g_S^{\alpha \beta} \frac{\partial \hat{f}}{\partial X_\beta} \bigg) \bigg\vert^2 \sqrt{G_S} { \ }dX \bigg)^{1/2},\\
& \Vert \psi_S \Vert_{ \mathcal{W}^{1,2} (\Gamma) } = ( \Vert \psi_S \Vert_{ \mathcal{L}^2 (\Gamma) }^2 + \Vert \psi_S \Vert_{ \dot{\mathcal{W}}^{1,2} (\Gamma)  }^2 )^{1/2},\\
& \Vert \psi_S \Vert_{ \mathcal{W}^{2,2} (\Gamma) } = ( \Vert \psi_S \Vert_{ \mathcal{L}^2 (\Gamma) }^2 + \Vert \psi_S \Vert_{ \dot{\mathcal{W}}^{1,2} (\Gamma)  }^2  + \Vert \psi_S \Vert_{ \mathring{\mathcal{W}}^{2,2} (\Gamma)  }^2)^{1/2},\\
& \Vert \psi_S \Vert_{ \mathcal{L}^1 (\Gamma) } = \int_{\mathbb{R}^2} \vert \psi_S (X) \vert \sqrt{G_S} { \ }d X,
\end{align*}
and
\begin{equation*}
\Vert \psi_S \Vert_{\mathcal{W}^{1,1} (\Gamma) } = \int_{\mathbb{R}^2} \vert \psi_S \vert \sqrt{G_S} { \ }dX + \sum_{j=1}^3 \int_{\mathbb{R}^2} \bigg\vert g_S^{\alpha \beta} \frac{\partial \hat{x}^S_j }{\partial X_\alpha} \frac{ \partial \psi_S }{\partial X_\beta} \bigg\vert \sqrt{G_S} { \ }dX.
\end{equation*}
Here $G_S =G_S (X) = 1 + \varphi_1^2 + \varphi_2^2$. See subsection \ref{subsec71} for $\hat{x}_j^S$ and $g_S^{\alpha \beta}$.
\end{definition}

\begin{remark}\label{rem34} $(\rm{i})$ From assertion $(\rm{iv})$ in Remark \ref{rem72}, and Lemma \ref{lem37}, we find that for all $\psi_S \in W^{2,2} (\mathbb{R}^2)$,
\begin{align*}
\Vert \psi_S \Vert_{\dot{\mathcal{W}}^{2,2}(\Gamma)} & = \bigg( \sum_{i,j=1}^3 \Vert \partial_i^\Gamma \partial_j^\Gamma \mathcal{P}_S [\psi_S] \Vert_{L^2(\Gamma)}^2 \bigg)^{1/2},\\
\Vert \psi_S \Vert_{\mathring{\mathcal{W}}^{2,2}(\Gamma)} & = \Vert \Delta_\Gamma \mathcal{P}_S [\Psi_S] \Vert_{L^2(\Gamma)},\\
\Vert \psi_S \Vert_{\mathcal{W}^{2,2}(\Gamma)} & = \bigg( \Vert \mathcal{P}_S[\psi_S] \Vert_{L^2(\Gamma)}^2 + \Vert \nabla_\Gamma \mathcal{P}_S[\psi_S] \Vert^2_{L^2(\Gamma)} + \Vert \Delta_\Gamma \mathcal{P}_S [\Psi_S] \Vert_{L^2(\Gamma)}^2 \bigg)^{1/2},
\end{align*}
and  that there is $C(\Vert \nabla_X^2 \varphi \Vert_{L^\infty (\mathbb{R}^2)}) >0$ such that for each $\psi_S \in W^{2,2} (\mathbb{R}^2)$,
\begin{equation*}
\Vert \psi_S \Vert_{\dot{\mathcal{W}}^{2,2}(\Gamma)} + \Vert \psi_S \Vert_{\mathring{\mathcal{W}}^{2,2}(\Gamma)} + \Vert \psi_S \Vert_{\mathcal{W}^{2,2}(\Gamma)} \leq C(\Vert \nabla_X^2 \varphi \Vert_{L^\infty (\mathbb{R}^2)}) \Vert \psi_S \Vert_{W^{2,2} (\mathbb{R}^2) }.
\end{equation*}
Here $\partial_j^\Gamma f = \partial_j f - n_j (n \cdot \nabla) f$, $\nabla_\Gamma = { }^t (\partial_1^\Gamma , \partial_2^\Gamma , \partial_3^\Gamma)$, $\Delta_\Gamma = (\partial_1^\Gamma)^2 + (\partial_2^\Gamma)^2 +  (\partial_3^\Gamma)^2$, and $n = { }^t(n_1,n_2,n_3)$ denotes the unit outer normal vector of $\Gamma$. See subsection \ref{subsec71} for details on differential operators $\partial_j^\Gamma$, $\nabla_\Gamma$, and $\Delta_\Gamma$.\\
$(\rm{ii})$ From assertion $(\rm{v})$ in Remark \ref{rem72}, we find that for $\psi_S \in W^{1,1} (\mathbb{R}^2)$,
\begin{equation*}
\Vert \psi_S \Vert_{\mathcal{W}^{1,1}(\Gamma)} = \Vert \mathcal{P}_S[\psi_S] \Vert_{L^1(\Gamma)} + \Vert \nabla_\Gamma \mathcal{P}_S[\psi_S] \Vert_{L^1(\Gamma)}.
\end{equation*}

\end{remark}

\begin{definition}[Norms and semi-norms (II)]\label{def35}
$(\rm{i})$ For $\psi_A \in L^1_{loc} (\mathbb{R}^3_+)$, we define
\begin{align*}
& \Vert \psi_A \Vert_{ \mathcal{L}^2 (\Omega_A) } = \bigg( \int_{\mathbb{R}^3_+} \vert \psi_A (y) \vert^2 { \ }d y \bigg)^{1/2},\\
& \Vert \psi_A \Vert_{ \dot{\mathcal{W}}^{1,2} (\Omega_A)} = \bigg( \int_{\mathbb{R}^3_+}  g_A^{i j} \frac{\partial \psi_A }{\partial y_i} \frac{ \partial \psi_A }{\partial y_j} { \ }dy \bigg)^{1/2},\\
& \Vert \psi_A \Vert_{ \dot{\mathcal{W}}^{2,2} (\Omega_A) } = \bigg( \sum_{\ell, \ell' =1}^3 \int_{\mathbb{R}^3_+} \bigg\vert g_A^{i' j'} \frac{\partial \tilde{x}^A_{\ell'} }{\partial y_{i'}} \frac{ \partial }{\partial y_{j'} } \bigg( g_A^{i j} \frac{\partial \tilde{x}^A_\ell }{\partial y_i} \frac{ \partial \psi_A }{\partial y_j}  \bigg) \bigg\vert^2 { \ }dy \bigg)^{1/2},\\
& \Vert \psi_A \Vert_{ \mathcal{W}^{1,2} (\Omega_A) } = ( \Vert \psi_A \Vert_{ \mathcal{L}^2 (\Omega_A) }^2 + \Vert \psi_A \Vert_{ \dot{\mathcal{W}}^{1,2} (\Omega_A)  }^2 )^{1/2},\\
& \Vert \psi_A \Vert_{ \mathcal{W}^{2,2} (\Omega_A) } = ( \Vert \psi_A \Vert_{ \mathcal{L}^2 (\Omega_A) }^2 + \Vert \psi_A \Vert_{ \dot{\mathcal{W}}^{1,2} (\Omega_A)  }^2  + \Vert \psi_A \Vert_{ \dot{\mathcal{W}}^{2,2} (\Omega_A)  }^2)^{1/2},\\
& \Vert \psi_A \Vert_{ \mathcal{L}^1 (\Omega_A) } = \int_{\mathbb{R}^3_+} \vert \psi_A (y) \vert { \ }d y,\\
& \Vert \psi_A \Vert_{\mathcal{W}^{1,1} (\Omega_A) } = \int_{\mathbb{R}^3_+} \vert \psi_A \vert { \ }dy + \sum_{\ell =1}^3 \int_{\mathbb{R}^3_+} \bigg\vert g_A^{i j} \frac{\partial \tilde{x}^A_\ell }{\partial y_i} \frac{ \partial \psi_A }{\partial y_j} \bigg\vert { \ }dy.
\end{align*}
$(\rm{ii})$ For $\psi_B \in L^1_{loc} (\mathbb{R}^3_-)$, we define
\begin{align*}
& \Vert \psi_B \Vert_{ \mathcal{L}^2 (\Omega_B) } = \bigg( \int_{\mathbb{R}^3_-} \vert \psi_B (z) \vert^2 { \ }d z \bigg)^{1/2},\\
& \Vert \psi_B \Vert_{ \dot{\mathcal{W}}^{1,2} (\Omega_B)} = \bigg( \int_{\mathbb{R}^3_-}  g_B^{i j} \frac{\partial \psi_B }{\partial z_i} \frac{ \partial \psi_B }{\partial z_j} { \ }dz \bigg)^{1/2},\\
& \Vert \psi_B \Vert_{ \dot{\mathcal{W}}^{2,2} (\Omega_B) } = \bigg( \sum_{\ell, \ell' =1}^3 \int_{\mathbb{R}^3_-} \bigg\vert g_B^{i' j'} \frac{\partial \check{x}^B_{\ell'} }{\partial z_{i'}} \frac{ \partial }{\partial z_{j'} } \bigg( g_B^{i j} \frac{\partial \check{x}^B_\ell }{\partial z_i} \frac{ \partial \psi_B }{\partial z_j}  \bigg) \bigg\vert^2 { \ }dz \bigg)^{1/2},\\
& \Vert \psi_B \Vert_{ \mathcal{W}^{1,2} (\Omega_B) } = ( \Vert \psi_B \Vert_{ \mathcal{L}^2 (\Omega_B) }^2 + \Vert \psi_B \Vert_{ \dot{\mathcal{W}}^{1,2} (\Omega_B)  }^2 )^{1/2},\\
& \Vert \psi_B \Vert_{ \mathcal{W}^{2,2} (\Omega_B) } = ( \Vert \psi_B \Vert_{ \mathcal{L}^2 (\Omega_B) }^2 + \Vert \psi_B \Vert_{ \dot{\mathcal{W}}^{1,2} (\Omega_B)  }^2  + \Vert \psi_B \Vert_{ \dot{\mathcal{W}}^{2,2} (\Omega_B)  }^2)^{1/2},\\
& \Vert \psi_B \Vert_{ \mathcal{L}^1 (\Omega_B) } = \int_{\mathbb{R}^3_-} \vert \psi_B (z) \vert { \ }d z,\\
& \Vert \psi_B \Vert_{\mathcal{W}^{1,1} (\Omega_B) } = \int_{\mathbb{R}^3_-} \vert \psi_B \vert { \ }dz + \sum_{\ell =1}^3 \int_{\mathbb{R}^3_-} \bigg\vert g_B^{i j} \frac{\partial \check{x}^B_\ell }{\partial z_i} \frac{ \partial \psi_B }{\partial z_j} \bigg\vert { \ }dz.
\end{align*}
See subsection \ref{subsec72} for $\tilde{x}_\ell^A$, $g_A^{ij}$ and subsection \ref{subsec73} for $\check{x}_\ell^B$, $g_B^{ij}$.
\end{definition}

\begin{remark}\label{rem36} $(\rm{i})$ From assertion $(\rm{iii})$ in Remark \ref{rem76}, assertion $(\rm{ii})$ in Remark \ref{rem7010}, and Lemma \ref{lem37}, we find that for $\psi_A \in W^{2,2} (\mathbb{R}^3_+)$ and $\psi_B \in W^{2,2} (\mathbb{R}^3_-)$,
\begin{align*}
\Vert \psi_A \Vert_{\mathcal{W}^{2,2}(\Omega_A)} & = \bigg( \Vert \mathcal{P}_A[\psi_A] \Vert_{L^2(\Omega_A)}^2 + \Vert \nabla \mathcal{P}_A[\psi_A] \Vert_{L^2(\Omega_A)}^2 + \Vert \nabla^2 \mathcal{P}_A [\psi_A] \Vert_{L^2(\Omega_A)}^2 \bigg)^{1/2},\\
\Vert \psi_B \Vert_{\mathcal{W}^{2,2}(\Omega_B)} & = \bigg( \Vert \mathcal{P}_B[\psi_B] \Vert_{L^2(\Omega_B)}^2 + \Vert \nabla \mathcal{P}_B[\psi_B] \Vert_{L^2(\Omega_B)}^2 + \Vert \nabla^2 \mathcal{P}_B [\psi_B] \Vert_{L^2(\Omega_B)}^2 \bigg)^{1/2},
\end{align*}
and that there is $C(\Vert \nabla_X^2 \varphi \Vert_{L^\infty (\mathbb{R}^2)}) >0$ for $\psi_A \in W^{2,2} (\mathbb{R}^3_+)$ and $\psi_B \in W^{2,2} (\mathbb{R}^3_-)$,
\begin{align*}
\Vert \psi_A \Vert_{\mathcal{W}^{2,2}(\Omega_A)} & \leq C(\Vert \nabla_X^2 \varphi \Vert_{L^\infty (\mathbb{R}^2)}) \Vert \psi_A \Vert_{W^{2,2} (\mathbb{R}^3_+)},\\
\Vert \psi_B \Vert_{\mathcal{W}^{2,2}(\Omega_B)} & \leq C(\Vert \nabla_X^2 \varphi \Vert_{L^\infty (\mathbb{R}^2)}) \Vert \psi_B \Vert_{W^{2,2} (\mathbb{R}^3_-)}.
\end{align*}
$(\rm{ii})$ From assertion $(\rm{iv})$ in Remark \ref{rem76} and assertion $(\rm{iii})$ in Remark \ref{rem7010}, we find that for $\psi_A \in W^{1,1} (\mathbb{R}^3_+)$ and $\psi_B \in W^{1,1} (\mathbb{R}^3_-)$,
\begin{align*}
\Vert \psi_A \Vert_{\mathcal{W}^{1,1}(\Omega_A)} & = \Vert \mathcal{P}_A[\psi_A] \Vert_{L^1(\Omega_A)} + \Vert \nabla \mathcal{P}_A[\psi_A] \Vert_{L^1(\Omega_A)},\\
\Vert \psi_B \Vert_{\mathcal{W}^{1,1}(\Omega_B)} & = \Vert \mathcal{P}_B[\psi_B] \Vert_{L^1(\Omega_B)} + \Vert \nabla \mathcal{P}_B[\psi_B] \Vert_{L^1(\Omega_B)}.
\end{align*}
\end{remark}

Let us check fundamental properties of our norms and semi-norms. From Lemmas \ref{lem73}, \ref{lem77}, \ref{lem7011} and Lemmas \ref{lem74}, \ref{lem78}, \ref{lem7012}, we have Lemma \ref{lem37} and Lemma \ref{lem38}, respectively.
\begin{lemma}\label{lem37}
$(\rm{i})$ For all $\psi_A \in L^2(\mathbb{R}^3_+)$, $\psi_B \in L^2(\mathbb{R}^3_-)$, and $\psi_S \in L^2 (\mathbb{R}^2)$,
\begin{align}
\Vert \psi_A \Vert_{L^2 (\mathbb{R}^3_+)} & = \Vert \psi_A \Vert_{\mathcal{L}^2(\Omega_A)},\label{eq34}\\
\Vert \psi_B \Vert_{L^2 (\mathbb{R}^3_-)} &= \Vert \psi_B \Vert_{\mathcal{L}^2(\Omega_B)},\\
\Vert \psi_S \Vert_{L^2 (\mathbb{R}^2)} &\leq \Vert \psi_S \Vert_{\mathcal{L}^2(\Gamma)} \leq \sqrt{2} \Vert \psi_S \Vert_{L^2(\mathbb{R}^2)}. \label{eq36}
\end{align}
$(\rm{ii})$ For all $\psi_A \in W^{1,2}(\mathbb{R}^3_+)$, $\psi_B \in W^{1,2}(\mathbb{R}^3_-)$, and $\psi_S \in W^{1,2} (\mathbb{R}^2)$,
\begin{align}
(1/3) \Vert \nabla_y \psi_A \Vert_{L^2 (\mathbb{R}^3_+)} & \leq \Vert \psi_A \Vert_{\dot{\mathcal{W}}^{1,2}(\Omega_A)} \leq 3 \Vert \nabla_y \psi_A \Vert_{L^2(\mathbb{R}^3_+)},\label{eq37}\\
(1/3) \Vert \nabla_z \psi_B \Vert_{L^2 (\mathbb{R}^3_-)} & \leq \Vert \psi_B \Vert_{\dot{\mathcal{W}}^{1,2}(\Omega_B)} \leq 3 \Vert \nabla_z \psi_B \Vert_{L^2(\mathbb{R}^3_-)},\\
(1/2)\Vert \nabla_X \psi_S \Vert_{L^2 (\mathbb{R}^2)} & \leq \Vert \psi_S \Vert_{\dot{\mathcal{W}}^{1,2}(\Gamma)} \leq 2 \Vert \nabla_X \psi_S \Vert_{L^2(\mathbb{R}^2)}.
\end{align}
$(\rm{iii})$ There is $C = C ( \Vert \nabla_X^2 \varphi \Vert_{L^\infty (\mathbb{R}^2)}) >0$ such that for all $\psi_A \in W^{2,2}(\mathbb{R}^3_+)$, $\psi_B \in W^{2,2}(\mathbb{R}^3_-)$, and $\psi_S \in W^{2,2} (\mathbb{R}^2)$,
\begin{align}
C^{-1} \Vert \psi_A \Vert_{W^{2,2}(\mathbb{R}^3_+)} & \leq \Vert \psi_A \Vert_{ \mathcal{W}^{2,2}(\Omega_A)}  \leq C \Vert \psi_A \Vert_{W^{2,2}(\mathbb{R}^3_+)},\label{eq3010}\\
C^{-1} \Vert \psi_B \Vert_{W^{2,2}(\mathbb{R}^3_-)} & \leq \Vert \psi_B \Vert_{ \mathcal{W}^{2,2}(\Omega_B)} \leq C \Vert \psi_B \Vert_{W^{2,2}(\mathbb{R}^3_-)},\\
C^{-1} \Vert \psi_S \Vert_{W^{2,2}(\mathbb{R}^2)} & \leq \Vert \psi_S \Vert_{ \mathcal{W}^{2,2}(\Gamma)} \leq C \Vert \psi_S \Vert_{W^{2,2}(\mathbb{R}^2)},\\
\Vert \psi_S \Vert_{ \dot{\mathcal{W}}^{2,2}(\Gamma)} & \leq C \Vert \psi_S \Vert_{W^{2,2}(\mathbb{R}^2)}.
\end{align}
\end{lemma}

\begin{lemma}\label{lem38}$(\rm{i})$ For all $\psi_A \in L^1(\mathbb{R}^3_+)$, $\psi_B \in L^1(\mathbb{R}^3_-)$, and $\psi_S \in L^1 (\mathbb{R}^2)$,
\begin{align*}
\Vert \psi_A \Vert_{L^1 (\mathbb{R}^3_+)} & = \Vert \psi_A \Vert_{\mathcal{L}^1(\Omega_A)},\\
\Vert \psi_B \Vert_{L^1 (\mathbb{R}^3_-)} &= \Vert \psi_B \Vert_{\mathcal{L}^1(\Omega_B)},\\
\Vert \psi_S \Vert_{L^1 (\mathbb{R}^2)} &\leq \Vert \psi_S \Vert_{\mathcal{L}^1(\Gamma)} \leq 2 \Vert \psi_S \Vert_{L^1(\mathbb{R}^2)}. 
\end{align*}
$(\rm{ii})$ For all $\psi_A \in W^{1,1} (\mathbb{R}^3_+)$, $\psi_B \in W^{1,1} (\mathbb{R}^3_-)$, and $\psi_S \in W^{1,1} (\mathbb{R}^2)$,
\begin{align*}
(1/2) \Vert \psi_A \Vert_{W^{1,1} (\mathbb{R}^3_+)} & \leq \Vert \psi_A \Vert_{\mathcal{W}^{1,1}(\Omega_A)} \leq 18 \Vert \psi_A \Vert_{W^{1,1}(\mathbb{R}^3_+)},\\
(1/2) \Vert \psi_B \Vert_{W^{1,1} (\mathbb{R}^3_-)} & \leq \Vert \psi_B \Vert_{\mathcal{W}^{1,1}(\Omega_B)} \leq 18 \Vert \psi_B \Vert_{W^{1,1}(\mathbb{R}^3_-)},\\
\Vert \psi_S \Vert_{W^{1,1} (\mathbb{R}^2)} & \leq \Vert \psi_S \Vert_{\mathcal{W}^{1,1}(\Gamma)} \leq 12 \Vert \psi_S \Vert_{W^{1,1}(\mathbb{R}^2)}.
\end{align*}
\end{lemma}
\noindent Note that we assume that $\Vert \nabla_X \varphi \Vert_{L^\infty(\mathbb{R}^2)} \leq 1/2$ in this section.

Let us define our function spaces. 
\begin{definition}[Function spaces $(\rm{I})$]\label{def39}
Let $k \in \{ 0,1,2 \}$. Define
\begin{align*}
C_0^k (\Gamma) &= \{ f_S :\Gamma \to \mathbb{R}; f_S = \mathcal{P}_S [\psi_S], \psi_S \in C_0^k(\mathbb{R}^2)  \},\\
L^2(\Gamma) &= \{ f_S :\Gamma \to \mathbb{R}; { \ }\Vert f_S \Vert_{L^2(\Gamma)} < \infty, f_S = \mathcal{P}_S [\psi_S], \psi_S \in L^2(\mathbb{R}^2) \},\\
W^{1,2}(\Gamma) &= \{ f_S \in L^2(\Gamma); { \ }\Vert f_S \Vert_{W^{1,2}(\Gamma)} < \infty,  f_S = \mathcal{P}_S [\psi_S], \psi_S \in W^{1,2}(\mathbb{R}^2)  \},\\
W^{2,2}(\Gamma) &= \{ f_S \in L^2(\Gamma); { \ }\Vert f_S \Vert_{W^{2,2}(\Gamma)} < \infty,  f_S = \mathcal{P}_S [\psi_S], \psi_S \in W^{2,2}(\mathbb{R}^2)  \},
\end{align*}
where
\begin{equation*}
\Vert f_S \Vert_{L^2(\Gamma)} := \Vert \psi_S \Vert_{\mathcal{L}^2 (\Gamma)},{ \ }\Vert f_S \Vert_{W^{1,2}(\Gamma)} := \Vert \psi_S \Vert_{\mathcal{W}^{1,2} (\Gamma)},{ \ }\Vert f_S \Vert_{W^{2,2}(\Gamma)} := \Vert \psi_S \Vert_{\mathcal{W}^{2,2} (\Gamma)}.
\end{equation*}
\end{definition}

\begin{definition}[Function spaces $(\rm{II})$]\label{def3010}
Let $k \in \{ 0,1,2 \}$. Define
\begin{align*}
C_0^k(\overline{\Omega_A}) &= \{ f_A :\overline{\Omega_A} \to \mathbb{R}; f_A = \mathcal{P}_A [\psi_A], \psi_A \in C_0^k(\overline{\mathbb{R}^3_+})  \},\\
C_0^k(\Omega_A) &= \{ f_A :\Omega_A \to \mathbb{R}; f_A = \mathcal{P}_A [\psi_A], \psi_A \in C_0^k(\mathbb{R}^3_+)  \},\\
L^2(\Omega_A) &= \{ f_A :\Omega_A \to \mathbb{R}; { \ }\Vert f_A \Vert_{L^2(\Omega_A)} < \infty, f_A = \mathcal{P}_A [\psi_A], \psi_A \in L^2(\mathbb{R}^3_+) \},\\
W^{1,2}(\Omega_A) &= \{ f_A \in L^2(\Omega_A); { \ }\Vert f_A \Vert_{W^{1,2}(\Omega_A)} < \infty,  f_A = \mathcal{P}_A [\psi_A], \psi_A \in W^{1,2}(\mathbb{R}^3_+)  \},\\
W_0^{1,2}(\Omega_A) &= \{ f_A \in L^2(\Omega_A); { \ }\Vert f_A \Vert_{W^{1,2}(\Omega_A)} < \infty,  f_A = \mathcal{P}_A [\psi_A], \psi_A \in W_0^{1,2}(\mathbb{R}^3_+)  \},\\
W^{2,2}(\Omega_A) &= \{ f_A \in L^2(\Omega_A); { \ }\Vert f_A \Vert_{W^{2,2}(\Omega_A)} < \infty,  f_A = \mathcal{P}_A [\psi_A], \psi_A \in W^{2,2}(\mathbb{R}^3_+)  \},\\
C_0^k(\overline{\Omega_B}) &= \{ f_B :\overline{\Omega_B} \to \mathbb{R}; f_B = \mathcal{P}_B [\psi_B], \psi_B \in C_0^k(\overline{\mathbb{R}^3_-})  \},\\
C_0^k(\Omega_B) &= \{ f_B :\Omega_B \to \mathbb{R}; f_B = \mathcal{P}_B [\psi_B], \psi_B \in C_0^k(\mathbb{R}^3_-)  \},\\
L^2(\Omega_B) &= \{ f_B :\Omega_B \to \mathbb{R}; { \ }\Vert f_B \Vert_{L^2(\Omega_B)} < \infty, f_B = \mathcal{P}_B [\psi_B], \psi_B \in L^2(\mathbb{R}^3_-) \},\\
W^{1,2}(\Omega_B) &= \{ f_B \in L^2(\Omega_B); { \ }\Vert f_B \Vert_{W^{1,2}(\Omega_B)} < \infty, f_B = \mathcal{P}_B [\psi_B], \psi_B \in W^{1,2}(\mathbb{R}^3_-)  \},\\
W_0^{1,2}(\Omega_B) &= \{ f_B \in L^2(\Omega_B); { \ }\Vert f_B \Vert_{W^{1,2}(\Omega_B)} < \infty, f_B = \mathcal{P}_B [\psi_B], \psi_B \in W_0^{1,2}(\mathbb{R}^3_-)  \},\\
W^{2,2}(\Omega_B) &= \{ f_B \in L^2(\Omega_B); { \ }\Vert f_B \Vert_{W^{2,2}(\Omega_B)} < \infty, f_B = \mathcal{P}_B [\psi_B], \psi_B \in W^{2,2}(\mathbb{R}^3_-)  \},
\end{align*}
where
\begin{multline*}
\Vert f_A \Vert_{L^2(\Omega_A)} := \Vert \psi_A \Vert_{\mathcal{L}^2 (\Omega_A)},{ \ }\Vert f_A \Vert_{W^{1,2}(\Omega_A)} := \Vert \psi_A \Vert_{\mathcal{W}^{1,2} (\Omega_A)},\\
\Vert f_A \Vert_{W^{2,2}(\Omega_A)} := \Vert \psi_A \Vert_{\mathcal{W}^{2,2} (\Omega_A)},{ \ }\Vert f_B \Vert_{L^2(\Omega_B)} := \Vert \psi_B \Vert_{\mathcal{L}^2 (\Omega_B)},\\
\Vert f_B \Vert_{W^{1,2}(\Omega_B)} := \Vert \psi_B \Vert_{\mathcal{W}^{1,2} (\Omega_B)},{ \ }\Vert f_B \Vert_{W^{2,2}(\Omega_B)} := \Vert \psi_B \Vert_{\mathcal{W}^{2,2} (\Omega_B)}.
\end{multline*}
\end{definition}

\begin{definition}[Function spaces $(\rm{III})$]\label{def3011} $(\rm{I})$
Define
\begin{align*}
L^1 (\Gamma) &= \{ f_S: \Gamma \to \mathbb{R} ; { \ }\Vert f_S \Vert_{L^1(\Gamma)} < \infty,  f_S = \mathcal{P}_S [\psi_S], \psi_S \in L^1(\mathbb{R}^2) \},\\
L^1 (\Omega_A) &= \{ f_A: \Omega_A \to \mathbb{R} ; { \ }\Vert f_A \Vert_{L^1(\Omega_A)} < \infty,  f_A = \mathcal{P}_A [\psi_A], \psi_A \in L^1(\mathbb{R}^3_+) \},\\
L^1 (\Omega_B) &= \{ f_B: \Omega_B \to \mathbb{R} ; { \ }\Vert f_B \Vert_{L^1(\Omega_B)} < \infty,  f_B = \mathcal{P}_B [\psi_B], \psi_B \in L^1(\mathbb{R}^3_-) \},\\
W^{1,1}(\Gamma) &= \{ f_S: \Gamma \to \mathbb{R} ; { \ }\Vert f_S \Vert_{W^{1,1}(\Gamma)} < \infty,  f_S = \mathcal{P}_S [\psi_S], \psi_S \in W^{1,1}(\mathbb{R}^2) \},\\
W^{1,1}(\Omega_A) &= \{ f_A: \Omega_A \to \mathbb{R} ; { \ }\Vert f_A \Vert_{W^{1,1}(\Omega_A)} < \infty,  f_A = \mathcal{P}_A [\psi_A], \psi_A \in W^{1,1}(\mathbb{R}^3_+) \},\\
W^{1,1}(\Omega_B) &= \{ f_B: \Omega_B \to \mathbb{R} ; { \ }\Vert f_B \Vert_{W^{1,1}(\Omega_B)} < \infty,  f_B = \mathcal{P}_B [\psi_B], \psi_B \in W^{1,1}(\mathbb{R}^3_-) \},
\end{align*}
where
\begin{multline*}
\Vert f_S \Vert_{L^1(\Gamma)} := \Vert \psi_S \Vert_{\mathcal{L}^1 (\Gamma)},{ \ }\Vert f_A \Vert_{L^1(\Omega_A)} := \Vert \psi_A \Vert_{\mathcal{L}^1 (\Omega_A)}, { \ }\Vert f_B \Vert_{L^1(\Omega_B)} := \Vert \psi_B \Vert_{\mathcal{L}^1 (\Omega_B)},\\
\Vert f_S \Vert_{W^{1,1}(\Gamma)} := \Vert \psi_S \Vert_{\mathcal{W}^{1,1} (\Gamma)},{ \ }\Vert f_A \Vert_{W^{1,1}(\Omega_A)} := \Vert \psi_A \Vert_{\mathcal{W}^{1,1} (\Omega_A)},\\
\Vert f_B \Vert_{W^{1,1}(\Omega_B)} := \Vert \psi_B \Vert_{\mathcal{W}^{1,1} (\Omega_B)}.
\end{multline*}
$(\rm{II})$ Define
\begin{align*}
L^\infty(\Gamma) &= \{ f_S: \Gamma \to \mathbb{R} ; { \ }\Vert f_S \Vert_{L^\infty (\Gamma)} < \infty,  f_S = \mathcal{P}_S [\psi_S], \psi_S \in L^\infty(\mathbb{R}^2) \},\\
L^\infty(\Omega_A) &= \{ f_A: \Omega_A \to \mathbb{R} ; { \ }\Vert f_A \Vert_{L^\infty(\Omega_A)} < \infty,  f_A = \mathcal{P}_A [\psi_A], \psi_A \in L^\infty (\mathbb{R}^3_+) \},\\
L^\infty (\Omega_B) &= \{ f_B: \Omega_B \to \mathbb{R} ; { \ }\Vert f_B \Vert_{L^\infty(\Omega_B)} < \infty,  f_B = \mathcal{P}_B [\psi_B], \psi_B \in L^\infty(\mathbb{R}^3_-) \},
\end{align*}
where
\begin{multline*}
\Vert f_S \Vert_{L^\infty(\Gamma)} := \Vert \psi_S \Vert_{L^\infty (\mathbb{R}^2)},{ \ }\Vert f_A \Vert_{L^\infty (\Omega_A)} := \Vert \psi_A \Vert_{L^\infty (\mathbb{R}^3_+)},\\
\Vert f_B \Vert_{L^\infty(\Omega_B)} := \Vert \psi_B \Vert_{L^\infty (\mathbb{R}^3_-)}.
\end{multline*}
\end{definition}

From Lemmas \ref{lem37} and \ref{lem38}, we see that we can define our function spaces.

\begin{definition}[Inner products]\label{def3012}
For all $f_A^\natural, f_A^\flat \in L^2(\Omega_A)$, $f_B^\natural, f_B^\flat \in L^2(\Omega_B)$, and $f_S^\natural, f_S^\flat \in L^2(\Gamma)$, we set
\begin{align*}
\dual{f_A^\natural , f_A^\flat}_{L^2(\Omega_A)} & =  \dual{\psi_A^\natural, \psi_A^\flat}_{L^2(\mathbb{R}^3_+)},\\
\dual{f_B^\natural , f_B^\flat}_{L^2(\Omega_B)} &=  \dual{\psi_B^\natural, \psi_B^\flat}_{L^2(\mathbb{R}^3_-)},\\
\dual{f_S^\natural , f_S^\flat}_{L^2(\Gamma)} & =  \dual{\psi_S^\natural (1 + \varphi_1^2 + \varphi_2^2)^{1/4}, \psi_S^\flat (1 + \varphi_1^2 + \varphi_2^2)^{1/4} }_{L^2(\mathbb{R}^2)},
\end{align*}
where $\psi_A^\natural, \psi_A^\flat \in L^2(\mathbb{R}^3_+)$, $\psi_B^\natural, \psi_B^\flat \in L^2(\mathbb{R}^3_-)$, and $\psi_S^\natural, \psi_S^\flat \in L^2(\mathbb{R}^2)$ such that $f_A^\natural = \mathcal{P}_A[\psi_A^\natural]$, $f_A^\flat = \mathcal{P}_A[\psi_A^\flat]$, $f_B^\natural = \mathcal{P}_B[\psi_B^\natural]$, $f_B^\flat = \mathcal{P}_B[\psi_B^\flat]$,
$f_S^\natural = \mathcal{P}_S[\psi_S^\natural]$, $f_S^\flat = \mathcal{P}_S[\psi_S^\flat]$, and $\dual{\cdot , \cdot }_{L^2(\mathbb{R}^3_+)}$, $\dual{\cdot , \cdot }_{L^2(\mathbb{R}^3_-)}$, $\dual{\cdot , \cdot }_{L^2(\mathbb{R}^2)}$ are the usual inner products on $L^2(\mathbb{R}^3_+)$, $L^2(\mathbb{R}^3_-)$, and $L^2(\mathbb{R}^2)$, respectively.
\end{definition}
Note that if $f_A \in L^2(\Omega_A)$, $f_B \in L^2(\Omega_B)$, $f_S \in L^2(\Gamma)$ then
\begin{multline*}
\Vert f_A \Vert_{L^2(\Omega_A)} = \dual{f_A , f_A}_{L^2(\Omega_A)}^{1/2}, { \ }\Vert f_B \Vert_{L^2(\Omega_B)} = \dual{f_B , f_B}_{L^2(\Omega_B)}^{1/2},\\
\Vert f_S \Vert_{L^2(\Gamma)} = \dual{f_S , f_S}_{L^2(\Gamma)}^{1/2}.
\end{multline*}

Using the following lemma, we discuss the uniqueness of strong solutions to our systems in Section \ref{sect4}.
\begin{lemma}\label{lem3013}Let $p \in \{ 1 , 2 \}$. Let $f_A \in L^p(\Omega_A)$, $f_B \in L^p(\Omega_B)$, and $f_S \in L^p(\Gamma)$. Assume that
\begin{equation*}
\Vert f_A \Vert_{L^p(\Omega_A)} + \Vert f_B \Vert_{L^p(\Omega_B)} + \Vert f_S \Vert_{L^p(\Gamma)} =0.
\end{equation*}
Then $f_A =0$, $f_B =0$, and $f_S =0$.
\end{lemma}

\begin{proof}[Proof of Lemma \ref{lem3013}]
Let $p \in \{ 1 , 2\}$. Let $f_A \in L^p(\Omega_A)$, $f_B \in L^p(\Omega_B)$, and $f_S \in L^p(\Gamma)$. By definitions, there are $\psi_A \in L^p(\mathbb{R}^3_+)$, $\psi_B \in L^p(\mathbb{R}^3_-)$, and $\psi_S \in L^p(\mathbb{R}^2)$ such that $f_A = \mathcal{P}_A [\psi_A]$, $f_B = \mathcal{P}_B [\psi_B]$, and $f_S = \mathcal{P}_S [\psi_S]$. By assumption and Lemmas \ref{lem37}, \ref{lem38}, we check that 
\begin{equation*}
\Vert \psi_A \Vert_{L^p(\mathbb{R}^3_+)} +\Vert \psi_B \Vert_{L^p(\mathbb{R}^3_-)} + \Vert \psi_S \Vert_{L^p(\mathbb{R}^2)}  \leq \Vert f_A \Vert_{L^p(\Omega_A)} + \Vert f_B \Vert_{L^p(\Omega_B)} + \Vert f_S \Vert_{L^p(\Gamma)} =0.
\end{equation*}
This implies that $\psi_A = 0$, $\psi_B = 0$, and $\psi_S =0$. Therefore, we conclude that $f_A = \mathcal{P}_A [\psi_A] =0$, $f_B = \mathcal{P}_B [\psi_B] =0$, and $f_S = \mathcal{P}_S [\psi_S] =0$. 
\end{proof}

From Lemmas \ref{lem37}, \ref{lem38}, and \ref{lem3013}, we find that $\Vert \cdot \Vert_{L^p(\Omega_A)} $, $\Vert \cdot \Vert_{L^p(\Omega_B)} $ $\Vert \cdot \Vert_{L^p(\Gamma)}$ are norms $(p \in \{ 1 , 2 \})$.

Let us characterize our function spaces.
\begin{proposition}\label{prop3014}
$(\rm{i})$ $L^2 (\Gamma ) = \overline{ C_0^2(\Gamma) }^{\Vert \cdot \Vert_{L^2(\Gamma)}}$,
 $(\rm{ii})$ $W^{1,2} (\Gamma ) = \overline{ C_0^2(\Gamma) }^{\Vert \cdot \Vert_{W^{1,2}(\Gamma)}}$,
  $(\rm{iii})$ $W^{2,2} (\Gamma ) = \overline{ C_0^2(\Gamma) }^{\Vert \cdot \Vert_{W^{2,2}(\Gamma)}}$,
   $(\rm{iv})$ $L^2 (\Omega_A ) = \overline{ C_0^2(\Omega_A) }^{\Vert \cdot \Vert_{L^2(\Omega_A)}}$,
    $(\rm{v})$ $W^{1 , 2} (\Omega_A ) = \overline{ C_0^2(\overline{\Omega_A}) }^{\Vert \cdot \Vert_{W^{1,2}(\Omega_A)}}$,
     $(\rm{vi})$ $W_0^{1 , 2} (\Omega_A ) = \overline{ C_0^2(\Omega_A) }^{\Vert \cdot \Vert_{W^{1,2}(\Omega_A)}}$,
      $(\rm{vii})$ $W^{2,2} (\Omega_A ) = \overline{ C_0^2(\overline{\Omega_A}) }^{\Vert \cdot \Vert_{W^{2,2}(\Omega_A)}}$,
       $(\rm{viii})$ $L^2 (\Omega_B ) = \overline{ C_0^2(\Omega_B) }^{\Vert \cdot \Vert_{L^2(\Omega_B)}}$,
        $(\rm{ix})$ $W^{1 , 2} (\Omega_B ) = \overline{ C_0^2(\overline{\Omega_B}) }^{\Vert \cdot \Vert_{W^{1,2}(\Omega_B)}}$,
         $(\rm{x})$ $W_0^{1 , 2} (\Omega_B ) = \overline{ C_0^2(\Omega_B) }^{\Vert \cdot \Vert_{W^{1,2}(\Omega_B)}}$,
          $(\rm{xi})$ $W^{2,2} (\Omega_B ) = \overline{ C_0^2(\overline{\Omega_B}) }^{\Vert \cdot \Vert_{W^{2,2}(\Omega_B)}}$,
           $(\rm{xii})$ $W^{1,1}(\Gamma) = \overline{ C_0^2(\Gamma) }^{\Vert \cdot \Vert_{W^{1,1}(\Gamma)}}$,
            $(\rm{xiii})$ $W^{1,1} (\Omega_A ) = \overline{ C_0^2(\overline{\Omega_A}) }^{\Vert \cdot \Vert_{W^{1,1}(\Omega_A)}}$,
             $(\rm{xiv})$ $W^{1,1} (\Omega_B ) = \overline{ C_0^2(\overline{\Omega_B}) }^{\Vert \cdot \Vert_{W^{1,1}(\Omega_B)}}$,
              $(\rm{xv})$ Each function space $L^2(\Gamma)$, $W^{1,2} (\Gamma)$, $W^{2,2} (\Gamma)$, $L^2(\Omega_A)$, $W^{1,2} (\Omega_A)$, $W_0^{1,2}(\Omega_A)$, $W^{2,2} (\Omega_A)$, $L^2(\Omega_B)$, $W^{1,2} (\Omega_B)$, $W_0^{1,2}(\Omega_B)$, $W^{2,2} (\Omega_B)$, $W^{1,1}(\Gamma)$, $W^{1,1} (\Omega_A)$, and $W^{1,1} (\Omega_B)$ is a Banach space,
               $(\rm{xvi})$ Each function space $L^2(\Omega_A)$, $L^2(\Omega_B)$, and $L^2(\Gamma)$ is a Hilbert space.
\end{proposition}

\begin{proof}[Proof of Proposition \ref{prop3014}]
We only show $(\rm{v})$. To this end, we show that
\begin{align}
W^{1,2} (\Omega_A) \subset \overline{ C_0^2(\overline{\Omega_A}) }^{\Vert \cdot \Vert_{W^{1,2}(\Omega_A)}},\label{eq3014}\\
 \overline{ C_0^2(\overline{\Omega_A}) }^{\Vert \cdot \Vert_{W^{1,2}(\Omega_A)}} \subset W^{1,2} (\Omega_A).\label{eq3015}
\end{align}

We first show \eqref{eq3014}. Let $f_A \in W^{1,2} (\Omega_A)$. By definition, there is $\psi_A \in W^{1,2} (\mathbb{R}^3_+)$ such that $f_A = \mathcal{P}[\psi_A]$. Since $C_0^2(\overline{\mathbb{R}^3_+})$ is dense in $W^{1,2} (\mathbb{R}^3_+)$, there is $\{ \psi_A^m \} \subset C_0^2( \overline{\mathbb{R}^3_+})$ such that
\begin{equation}\label{eq3016}
\lim_{m \to \infty} \Vert \psi_A - \psi_A^m \Vert_{W^{1,2} (\mathbb{R}^3_+)} = 0.
\end{equation}
By definition, we see that $f_A^m := \mathcal{P}_A[\psi_A^m] \in C_0^2(\overline{ \Omega_A })$. From \eqref{eq34}, \eqref{eq37}, and \eqref{eq3016}, we check that
\begin{align*}
\Vert f_A - f_A^m \Vert_{W^{1,2} (\Omega_A)} \leq 3 \Vert \psi_A - \psi_A^m \Vert_{W^{1,2}(\mathbb{R}^3_+)} \to 0 \text{ as } m \to \infty.
\end{align*}
This implies that $f_A \in \overline{C_0^2( \overline{\Omega_A})}^{ \Vert \cdot \Vert_{W^{2,2}(\Omega_A) }}$.

Next, we show \eqref{eq3015}. Let $f_A \in \overline{ C_0^2(\overline{\Omega_A}) }^{\Vert \cdot \Vert_{W^{1,2}(\Omega_A)}}$. By definition, there are $f_A^m \in C_0^2(\overline{\Omega_A})$ and $\psi_A^m \in C_0^2(\overline{\mathbb{R}^3_+})$ such that $f^m_A = \mathcal{P}_A [\psi_A^m]$ and
\begin{equation}\label{eq3017}
\lim_{m \to \infty} \Vert f_A - f_A^m \Vert_{W^{1,2} (\Omega_A)} = 0.
\end{equation}
By \eqref{eq3017}, \eqref{eq34}, and \eqref{eq37}, we find that
\begin{multline*}
\Vert \psi_A^m - \psi_A^{m'} \Vert_{W^{1,2} (\mathbb{R}^3_+)}  \leq 3 \Vert f_A^m - f_A^{m'} \Vert_{W^{1,2} (\Omega_A)}\\
 \leq 3 \Vert f_A^m - f_A \Vert_{W^{1,2} (\Omega_A)} + 3 \Vert f_A - f_A^{m'} \Vert_{W^{1,2} (\Omega_A)} \to 0 \text{ as }m,m' \to \infty.
\end{multline*}
Since $W^{1,2} (\mathbb{R}^3_+)$ is a Banach space, there is $\psi_A^\infty \in W^{1,2} (\mathbb{R}^3_+)$ such that
\begin{equation}\label{eq3018}
\lim_{m \to \infty} \Vert \psi_A^m - \psi_A^\infty \Vert_{W^{1,2} (\mathbb{R}^3_+)} = 0.
\end{equation}
Set $f_A^\infty = \mathcal{P}_A [\psi_A^\infty]$. By \eqref{eq3017} and \eqref{eq3018}, we check that
\begin{multline*}
\Vert f_A - f_A^\infty \Vert_{W^{1,2} (\Omega_A)}  \leq C \Vert f_A - f_A^m \Vert_{W^{1,2} (\Omega_A)} + C \Vert f_A^m - f_A^\infty \Vert_{W^{1,2}(\Omega_A)}\\
  \leq C \Vert f_A - f_A^m \Vert_{W^{1,2} (\Omega_A)} + C \Vert \psi_A^m - \psi_A^\infty \Vert_{W^{1,2}(\mathbb{R}^3_+)} \to 0 \text{ as }m \to \infty.
\end{multline*}
This implies that $f_A = f_A^\infty$, that is, $f_A \in W^{1,2} (\Omega_A)$. Therefore, we have \eqref{eq3015}. From \eqref{eq3014} and \eqref{eq3015}, we see $(\rm{v})$.
\end{proof}

Let us define and study our trace operators.
\begin{definition}[Trace Operators (I)]\label{def3015}{ \ }\\
$(\rm{i})$ Let $f_A \in W^{1,2} (\Omega_A)$. By definition, there is $\psi_A \in W^{1,2} (\mathbb{R}^3_+)$ such that
\begin{equation*}
f_A = \mathcal{P}_A[\psi_A].
\end{equation*}
We define the operator $\gamma_A: W^{1,2} (\Omega_A) \to L^2(\Gamma)$ by
\begin{equation*}
\gamma_A[f_A] = \mathcal{P}_S[ \gamma_+[\psi_A]].
\end{equation*}
Here $\gamma_+$ is the trace operator such that $\gamma_+: W^{1,2} (\mathbb{R}^3_+) \to L^2(\mathbb{R}^2)$. We call operator $\gamma_A$ a \emph{trace operator}.\\
$(\rm{ii})$ Let $f_B \in W^{1,2} (\Omega_B)$. By definition, there is $\psi_B \in W^{1,2} (\mathbb{R}^3_-)$ such that
\begin{equation*}
f_B = \mathcal{P}_B[\psi_B].
\end{equation*}
We define the operator $\gamma_B: W^{1,2} (\Omega_B) \to L^2(\Gamma)$ by
\begin{equation*}
\gamma_B[f_B] = \mathcal{P}_S[ \gamma_-[\psi_B]].
\end{equation*}
Here $\gamma_-$ is the trace operator such that $\gamma_-: W^{1,2} (\mathbb{R}^3_-) \to L^2(\mathbb{R}^2)$. We call operator $\gamma_B$ a \emph{trace operator}.
\end{definition}

\begin{proposition}[Properties of trace operators (I)]\label{prop3016}
$(\rm{i})$ Let $f_A \in W^{1,2} (\Omega_A)$ and $\psi_A \in W^{1,2} (\mathbb{R}^3_+)$ such that $f_A = \mathcal{P}_A [\psi_A]$. Then
\begin{align}
\Vert \gamma_A[f_A] \Vert_{L^2(\Gamma)} & \leq 12 \Vert f_A \Vert_{W^{1,2} (\Omega_A)},\label{eq3019}\\
\Vert \gamma_A[f_A] \Vert_{L^2(\Gamma)} & \leq 2 \Vert \gamma_+[\psi_A ] \Vert_{L^2(\mathbb{R}^2)}.\label{eq3020}
\end{align}
$(\rm{ii})$ If $f_A \in W_0^{1,2} (\Omega_A)$, then $\gamma_A[f_A] =0$.\\
$(\rm{iii})$ Let $f_B \in W^{1,2} (\Omega_B)$ and $\psi_B \in W^{1,2} (\mathbb{R}^3_-)$ such that $f_B = \mathcal{P}_B [\psi_B]$. Then
\begin{align*}
\Vert \gamma_B[f_B] \Vert_{L^2(\Gamma)} & \leq 12 \Vert f_B \Vert_{W^{1,2} (\Omega_B)},\\
\Vert \gamma_B[f_B] \Vert_{L^2(\Gamma)} & \leq 2 \Vert \gamma_-[\psi_B ] \Vert_{L^2(\mathbb{R}^2)}.
\end{align*}
$(\rm{iv})$ If $f_B \in W_0^{1,2} (\Omega_B)$, then $\gamma_B[f_B] =0$.
\end{proposition}

\begin{proof}[Proof of Proposition \ref{prop3016}]

We only show $(\rm{i})$ and $(\rm{ii})$. Let $f_A \in W^{1,2} (\Omega_A)$ and $\psi_A \in W^{1,2} (\mathbb{R}^3_+)$ such that $f_A = \mathcal{P}_A[\psi_A]$.
Since $W^{1,2} (\Omega_A) = \overline{ C_0^2(\overline{\Omega_A}) }^{\Vert \cdot \Vert_{W^{1,2}(\Omega_A)}}$ and $W^{1,2} (\mathbb{R}^3_+) = \overline{ C_0^2(\overline{\mathbb{R}^3_+}) }^{\Vert \cdot \Vert_{W^{1,2}(\mathbb{R}^3_+)}}$, there are $\{ f_A^m \} \subset C_0^2(\overline{\Omega_A})$ and $\{ \psi_A^m \} \subset C_0^2(\overline{\mathbb{R}^3_+})$ such that $f_A^m = \mathcal{P}_A[ \psi_A^m]$,
\begin{align}
\lim_{m \to \infty} \Vert f^m_A  - f_A \Vert_{W^{1,2}(\Omega_A)} & = 0,\label{eq3021}\\
\lim_{m \to \infty} \Vert \psi^m_A  - \psi_A \Vert_{W^{1,2}(\mathbb{R}^3_+)} & = 0.\label{eq3022}
\end{align}
By the definition of surface integral, we check that
\begin{multline}\label{eq3023}
\int_{\Gamma} \vert f_A^m(x) \vert^2 { \ }d \mathcal{H}_x^2 = \int_{\Gamma} \vert \mathcal{P}_A[\psi_A^m](x) \vert^2 { \ }d \mathcal{H}_x^2\\
 = \int_{\Gamma} \vert \mathcal{P}_A[\psi_A^m](X_1,X_2, \varphi (X_1,X_2)) \vert^2  \sqrt{1 + \varphi_1^2 + \varphi_2^2}{ \ }d X\\
= \int_{\mathbb{R}^2} \vert \psi^m_A(X_1,X_2, 0) \vert^2 \sqrt{1 + \varphi_1^2 + \varphi_2^2} { \ }dX.
\end{multline}
By $\Vert \nabla_X \varphi \Vert_{L^\infty (\mathbb{R}^2)} \leq 1/2$, \eqref{eq34}, and \eqref{eq37}, we see that
\begin{align}
\Vert \psi_A^m \vert_{y_3 =0} (1+\varphi_1^2 + \varphi_2^2 )^{1/4} \Vert_{L^2(\mathbb{R}^2)} &\leq 2 \Vert \gamma_+ [\psi_A^m] \Vert_{L^2(\mathbb{R}^2)}\notag\\
 & \leq 4 \Vert \psi_A^m \Vert_{W^{1,2}(\mathbb{R}^3_+)}\notag\\
& \leq 12 \Vert f_A^m \Vert_{W^{1,2}(\Omega_A)}.\label{eq3024}
\end{align}
Here we used the property of the trace operator $\gamma_+$ such that
\begin{equation*}
\Vert \gamma_+ [\psi_A^m] \Vert_{L^2(\mathbb{R}^2)} \leq 2 \Vert \psi_A^m \Vert_{W^{1,2} (\mathbb{R}^3_+)}.
\end{equation*}
From \eqref{eq3024} and \eqref{eq3021}, we observe that
\begin{equation}\label{eq3025}
\lim_{m \to \infty}\Vert \psi_A^m \vert_{y_3 =0} (1+\varphi_1^2 + \varphi_2^2 )^{1/4} \Vert_{L^2(\mathbb{R}^2)} \leq 12 \Vert f_A \Vert_{W^{1,2}(\Omega_A)}.
\end{equation}
By the definition of $\gamma_+$, we find that
\begin{multline}\label{eq3026}
\lim_{m \to \infty }\int_{\mathbb{R}^2} \vert \psi^m_A(X_1,X_2, 0) \vert^2 \sqrt{1 + \varphi_1^2 + \varphi_2^2} { \ }dX \\
= \int_{\mathbb{R}^2} \vert \gamma_+[\psi_A](X_1,X_2) \vert^2 \sqrt{1 + \varphi_1^2 + \varphi_2^2} { \ }dX = \Vert \mathcal{P}_S [\gamma_+[\psi_A]] \Vert_{L^2(\Gamma)}^2.
\end{multline}
From \eqref{eq3023}, \eqref{eq3025}, and \eqref{eq3026}, we see that
\begin{equation*}
\lim_{m \to \infty} \int_{\Gamma} \vert f_A^m(x) \vert^2 { \ }d \mathcal{H}_x^2 = \Vert \mathcal{P}_S [\gamma_+[\psi_A]] \Vert_{L^2(\Gamma)}^2,
\end{equation*}
and that
\begin{equation*}
\Vert \mathcal{P}_S [\gamma_+[\psi_A]] \Vert_{L^2(\Gamma)} \leq 12 \Vert f_A \Vert_{W^{1,2} (\Omega_A)}.
\end{equation*}
Thus, we have \eqref{eq3019}. From \eqref{eq3026} and \eqref{eq3023}, we observe that
\begin{multline*}
\Vert \mathcal{P}_S [\gamma_+[\psi_A]] \Vert_{L^2(\Gamma)} = \lim_{m \to \infty} \Vert \psi_A^m \vert_{y_3 =0} (1+\varphi_1^2 + \varphi_2^2 )^{1/4} \Vert_{L^2(\mathbb{R}^2)}\\
 \leq 2 \lim_{m \to \infty} \Vert \gamma_+ [\psi_A^m] \Vert_{L^2(\mathbb{R}^2)} = 2 \Vert \gamma_+ [\psi_A] \Vert_{L^2(\mathbb{R}^2)},
\end{multline*}
which is \eqref{eq3020}. Therefore, we see $(\rm{i})$.

Next, we consider $(\rm{ii})$. Assume that $f_A \in W_0^{1,2}(\Omega_A)$. By definition, there are $\psi_A \in W_0^{1,2} (\mathbb{R}^3_+)$ such that $f_A = \mathcal{P}_A[f_A]$. From $\gamma_+ [\psi_A] =0$ and \eqref{eq3020}, we find that $\gamma_A[f_A] =0$. Therefore, the lemma follows.
\end{proof}

Now we attack Proposition \ref{prop31}.
\begin{proof}[Proof of Proposition \ref{prop31}]
Let $\delta >0$. We only show $(\rm{i})$. Let $f_A \in W^{1,2}(\Omega_A)$ and $f_S \in W^{1,2}(\Gamma)$ such that $\gamma_A[f_A] = f_S$. By definitions, there $\psi_A \in W^{1,2} (\mathbb{R}^3_+)$, $\psi_S \in W^{1,2}(\mathbb{R}^2)$ such that $f_A = \mathcal{P}_A[\psi_A]$ and $f_S = \mathcal{P}_S[\psi_S]$. For almost all $y \in \mathbb{R}^3_+$, we set
\begin{equation*}
\tilde{\mathfrak{u}}_A = \tilde{\mathfrak{u}}_A(y) = \psi_A (y) - \psi_S (y_h) {\rm{e}}^{- \delta y_3}. 
\end{equation*}
We easily check that
\begin{align*}
\int_{\mathbb{R}^3_+} \vert \psi_S (y_h)  {\rm{e}}^{- \delta y_3} \vert^2 { \ }dy = \frac{1}{2 \delta} \Vert \psi_S \Vert_{L^2(\mathbb{R}^2)}^2 < \infty,\\
\int_{\mathbb{R}^3_+} \vert \partial_{y_1}\{ \psi_S (y_h)  {\rm{e}}^{- \delta y_3} \} \vert^2 { \ }dy = \frac{1}{2 \delta} \Vert \partial_{y_1}\psi_S \Vert_{L^2(\mathbb{R}^2)}^2 < \infty,\\
\int_{\mathbb{R}^3_+} \vert \partial_{y_2}\{\psi_S (y_h)  {\rm{e}}^{- \delta y_3} \} \vert^2 { \ }dy = \frac{1}{2 \delta} \Vert \partial_{y_2}\psi_S \Vert_{L^2(\mathbb{R}^2)}^2 < \infty,\\
\int_{\mathbb{R}^3_+} \vert \partial_{y_3}\{\psi_S (y_h)  {\rm{e}}^{- \delta y_3} \} \vert^2 { \ }dy = \frac{\delta}{2 } \Vert \psi_S \Vert_{L^2(\mathbb{R}^2)}^2 < \infty.
\end{align*}
Thus, we find that $\psi_S (y_h) {\rm{e}}^{- \delta y_3} \in W^{1,2} (\mathbb{R}^3_+)$. From $\psi_A, \psi_S (y_h) {\rm{e}}^{- \delta y_3} \in W^{1,2} (\mathbb{R}^3_+)$,
we see that $ \tilde{\mathfrak{u}}_A \in W^{1,2} (\mathbb{R}^3_+)$. Since $\psi_S = \hat{f}_S$, we find that $\mathfrak{u}_A = \mathcal{P}_A [ \tilde{\mathfrak{u}}_A]$, that is, $\mathfrak{u}_A \in W^{1,2} (\Omega_A)$. From assumption that $\gamma_A[f_A] = f_S$ and \eqref{eq3020}, we check that
\begin{align*}
\Vert \gamma_A [\mathfrak{u}_A] \Vert_{L^2(\Gamma)} & \leq 2 \Vert \gamma_+[ \tilde{\mathfrak{u}}_A ] \Vert_{L^2(\mathbb{R}^2)}\\
& = 2 \Vert \gamma_+[ \psi_A (y) - \psi_S (y_h) {\rm{e}}^{- \delta y_3}) ] \Vert_{L^2(\mathbb{R}^2)}\\
& = 2 \Vert \gamma_+ [\psi_A] - \psi_S \Vert_{L^2(\mathbb{R}^2)} = 2 \Vert \psi_S - \psi_S \Vert_{L^2(\mathbb{R}^2)} =0.
\end{align*}
Thus, we see that $\gamma_A [\mathfrak{u}_A] =0$. Therefore, Proposition \ref{prop31} is proved.
\end{proof}

Next we study properties of our trace operators. Since four trace operators $\gamma_A$, $\gamma_B$, $\gamma_+$, $\gamma_-$ are not closed operators, we need Proposition \ref{prop3017} and Corollary \ref{cor3018} to construct strong solutions of our systems.
\begin{proposition}[Properties of trace operators (II)]\label{prop3017}
Let $\mu_A , \mu_B \in \mathbb{R}$. Let $\{ f_A^m \}_{m \in \mathbb{N}} \subset W^{2,2}(\Omega_A)$, $f_A \in W^{2,2} (\Omega_A)$, $\{ f_B^m \}_{m \in \mathbb{N}} \subset W^{2,2} (\Omega_B)$, $f_B \in W^{2,2} (\Omega_B)$, and $f_S \in L^2(\Gamma)$. Assume that
\begin{align*}
&\lim_{m \to \infty} \Vert f_A^m- f_A \Vert_{W^{2,2} (\Omega_A)} =0,\\
&\lim_{m \to \infty} \Vert f_B^m - f_B \Vert_{W^{2,2} (\Omega_B)} =0,\\
&\lim_{m \to \infty} \Vert (\mu_A \gamma_A [(n \cdot \nabla ) f_A^m] + \mu_B \gamma_B [ (n \cdot \nabla) f_B^m]) - f_S \Vert_{L^2(\Gamma)} = 0.
\end{align*}
Then
\begin{equation}\label{eq3027}
f_S = \mu_A \gamma_A [ (n\cdot \nabla )f_A] + \mu_B \gamma_B [(n \cdot \nabla )f_B].
\end{equation}
\end{proposition}

\begin{proof}[Proof of Proposition \ref{prop3017}]
By assumptions, there are $\{ \psi_A^m\} \subset W^{2,2}(\mathbb{R}^3_+)$, $\psi_A \in W^{2,2} (\mathbb{R}^3_+)$ such that $f_A^m = \mathcal{P}_A[\psi_A^m] $, $f_A = \mathcal{P}_A[\psi_A]$, and
\begin{equation*}
\lim_{m \to \infty } \Vert \psi_A - \psi_A^m \Vert_{W^{2,2} (\mathbb{R}^3_+)} = 0.
\end{equation*}
Now we set for almost all $y \in \mathbb{R}$
\begin{equation*}
\Psi_A = \Psi_A (y) = - \frac{\varphi_1}{ \sqrt{G_S} } \frac{\partial \psi_A}{\partial y_1} - \frac{\varphi_2}{ \sqrt{G_S} } \frac{\partial \psi_A}{\partial y_2} + \frac{1 + \varphi_1^2 + \varphi_2^2}{ \sqrt{G_S} } \frac{\partial \psi_A}{\partial y_3},
\end{equation*}
where $G_S = G_S (y_h) = 1 + \varphi_1^2 + \varphi_2^2$, $\varphi_1 = \partial \varphi/{\partial y_1}$, $\varphi_2 = \partial \varphi/{\partial y_2}$. Since $\psi_A \in W^{2,2}(\mathbb{R}^3_+)$ and $\varphi \in BC^2(\mathbb{R}^2)$, we find that $\Psi_A \in W^{1,2} (\mathbb{R}^3_+)$. Moreover, we set for almost all $x \in \Omega_A$,
\begin{equation*}
F_A = F_A(x) = \mathcal{P}_A[\Psi_A].
\end{equation*}
By definition, we see that $F_A \in W^{1,2} (\Omega_A)$. From assertion $(\rm{i})$ in Remark \ref{rem76} and \eqref{eq7041}, we find that
\begin{multline*}
\int_{\Omega_A} \bigg( - \frac{\varphi_1}{ \sqrt{G_S} } \frac{\partial f_A}{\partial x_1} - \frac{\varphi_2}{ \sqrt{G_S} } \frac{\partial f_A}{\partial x_2} + \frac{1}{ \sqrt{G_S} } \frac{\partial f_A}{\partial x_3} \bigg) { \ }dx\\ = \int_{\mathbb{R}^3_+}  \bigg(   - \frac{\varphi_1}{ \sqrt{G_S} } \frac{\partial \psi_A}{\partial y_1} - \frac{\varphi_2}{ \sqrt{G_S} } \frac{\partial \psi_A}{\partial y_2} + \frac{1 + \varphi_1^2 + \varphi_2^2}{ \sqrt{G_S} } \frac{\partial \psi_A}{\partial y_3}   \bigg) { \ }dy.
\end{multline*}
This implies that
\begin{equation*}
F_A = F_A (x) = - \frac{\varphi_1}{ \sqrt{G_S} } \frac{\partial f_A}{\partial x_1} - \frac{\varphi_2}{ \sqrt{G_S} } \frac{\partial f_A}{\partial x_2} + \frac{1}{ \sqrt{G_S} } \frac{\partial f_A}{\partial x_3},
\end{equation*}
where $G_S = G_S (x_h) = 1 + \varphi_1^2 + \varphi_2$, $\varphi_1 = \partial \varphi/{\partial x_1}$, $\varphi_2 = \partial \varphi/{\partial x_2}$. From
\begin{equation*}
n(x_1,x_2,x_3) = \frac{1}{\sqrt{G_S}} \begin{pmatrix}
- \varphi_1\\
- \varphi_2\\
1
\end{pmatrix} \text{ at } x = { }^t (x_1,x_2,x_3) \in \Gamma, 
\end{equation*}
we see that $(n \cdot \nabla ) f_A = F_A $ for $x \in \Gamma$. Since $n(x)$ does not depend on $x_3$ and $\varphi \in BC^2(\mathbb{R}^2)$, we can consider $(n \cdot \nabla)f_A$ as a $W^{1,2}$-function in $\Omega_A$. By the definition of $\gamma_A$, we observe that
\begin{equation*}
\gamma_A [(n \cdot \nabla ) f_A] = \mathcal{P}_S \bigg[ - \frac{\varphi_1}{\sqrt{G_S}} \gamma_+ \bigg[ \frac{\partial \psi_A }{\partial y_1} \bigg] - \frac{\varphi_2}{\sqrt{G_S}} \gamma_+ \bigg[ \frac{\partial \psi_A }{\partial y_2} \bigg] +  \frac{1 + \varphi_1^2 + \varphi_2^2 }{\sqrt{G_S}} \gamma_+ \bigg[ \frac{\partial \psi_A }{\partial y_3} \bigg] \bigg].
\end{equation*}
Using \eqref{eq36}, fundamental properties of $\gamma_+$, and  \eqref{eq3010}, we check that
\begin{align*}
\Vert \gamma_A [(n \cdot \nabla ) f_A] \Vert_{L^2(\Gamma)} & \leq C \sum_{j=1}^3 \Vert \gamma_+[\partial_j \psi_A ] \Vert_{L^2(\mathbb{R}^2)}\\
& \leq C \Vert \psi_A \Vert_{W^{2,2} (\mathbb{R}^3_+)} \leq C ( \Vert \nabla_X^2 \varphi \Vert_{L^\infty (\mathbb{R}^2)} ) \Vert f_A \Vert_{W^{2,2} (\Omega_A)}. 
\end{align*}
Similarly, we observe that
\begin{equation}\label{eq3028}
\Vert \gamma_A [(n \cdot \nabla ) \{ f_A - f_A^m \}] \Vert_{L^2(\Gamma)} \leq C ( \Vert \nabla_X^2 \varphi \Vert_{L^\infty (\mathbb{R}^2)} ) \Vert f_A - f_A^m \Vert_{W^{2,2} (\Omega_A)},
\end{equation}
and that
\begin{equation}\label{eq3029}
\Vert \gamma_B [(n \cdot \nabla ) \{ f_B - f_B^m \}] \Vert_{L^2(\Gamma)} \leq C ( \Vert \nabla_X^2 \varphi \Vert_{L^\infty (\mathbb{R}^2)} ) \Vert f_B - f_B^m \Vert_{W^{2,2} (\Omega_B)}. 
\end{equation}
By \eqref{eq3028} and $\eqref{eq3029}$, we easily check that
\begin{multline*}
\Vert (\mu_A \gamma_A [(n \cdot \nabla )f_A] + \mu_B \gamma_B [(n \cdot \nabla)f_B]) - f_S \Vert_{L^2(\Gamma)}\\
 \leq \Vert (\mu_A \gamma_A [(n \cdot \nabla) f_A] + \mu_B \gamma_B [(n \cdot \nabla) f_B] )-  (\mu_A \gamma_A [(n \cdot \nabla) f_A^m] + \mu_B \gamma_B [(n \cdot \nabla) f_B^m] ) \Vert_{L^2(\Gamma)}\\ +  \Vert ( \mu_A \gamma_A [ (n \cdot \nabla) f_A^m] + \mu_B \gamma_B [(n \cdot \nabla) f_B^m]) - f_S \Vert_{L^2(\Gamma)}\\
\leq C(\mu_A) ( \Vert f_A^m - f_A \Vert_{W^{2,2} (\Omega_A)} +  C (\mu_B) \Vert f_B^m - f_B \Vert_{W^{2,2} (\Omega_B)} \\
+ \Vert ( \mu_A \gamma_A [(n \cdot \nabla ) f_A^m] + \mu_B \gamma_B [ (n \cdot \nabla ) f_B^m]) - f_S \Vert_{L^2(\Gamma)} \to 0 \text{ as } m \to \infty.
\end{multline*}
Therefore, we see \eqref{eq3027}.
\end{proof}

Applying an argument as in the proof of Proposition \ref{prop3017}, we have the following corollary.
\begin{corollary}\label{cor3018}
Let $\mu_A , \mu_B \in \mathbb{R}$ and $j \in \{ 1 , 2, 3 \}$. Let $\{ \psi_A^m \}_{m \in \mathbb{N}} \subset W^{2,2}(\mathbb{R}^3_+)$, $\psi_A \in W^{2,2} (\mathbb{R}^3_+)$, $\{ \psi_B^m \}_{m \in \mathbb{N}} \subset W^{2,2} (\mathbb{R}^3_-)$, $\psi_B \in W^{2,2} (\mathbb{R}^3_-)$, and $\psi_S \in L^2(\mathbb{R}^2)$. Assume that
\begin{align*}
&\lim_{m \to \infty} \Vert \psi_A^m- \psi_A \Vert_{W^{2,2} (\mathbb{R}^3_+)} =0,\\
&\lim_{m \to \infty} \Vert \psi_B^m - \psi_B \Vert_{W^{2,2} (\mathbb{R}^3_-)} =0,\\
&\lim_{m \to \infty} \Vert (\mu_A \gamma_+ [\partial_{y_j}\psi_A^m] + \mu_B \gamma_- [ \partial_{z_j} \psi_B^m]) - \psi_S \Vert_{L^2(\mathbb{R}^2)} = 0.
\end{align*}
Then
\begin{equation*}
\psi_S = \mu_A \gamma_+ [ \partial_{y_j} \psi_A] + \mu_B \gamma_- [\partial_{z_j} \psi_B].
\end{equation*}
\end{corollary}

Now, we derive divergence theorems and integration by parts formulas. To this end, we introduce trace operators on $W^{1,1}$-spaces. 
\begin{definition}[Trace Operators (II)]\label{lem3019}{ \ }\\
$(\rm{i})$ Let $f_A \in W^{1,1} (\Omega_A)$. By definition, there is $\psi_A \in W^{1,1} (\mathbb{R}^3_+)$ such that
\begin{equation*}
f_A = \mathcal{P}_A[\psi_A].
\end{equation*}
We define the operator $\mathring{\gamma}_A: W^{1,1} (\Omega_A) \to L^1(\Gamma)$ by
\begin{equation*}
\mathring{\gamma_A}[f_A] = \mathcal{P}_S[ \mathring{\gamma}_+[\psi_A]].
\end{equation*}
Here $\mathring{\gamma}_+$ is the trace operator such that $\mathring{\gamma}_+: W^{1,1} (\mathbb{R}^3_+) \to L^1(\mathbb{R}^2)$.\\
$(\rm{ii})$ Let $f_B \in W^{1,1} (\Omega_B)$. By definition, there is $\psi_B \in W^{1,1} (\mathbb{R}^3_-)$ such that
\begin{equation*}
f_B = \mathcal{P}_B[\psi_B].
\end{equation*}
We define the operator $\mathring{\gamma}_B: W^{1,1} (\Omega_B) \to L^1(\Gamma)$ by
\begin{equation*}
\mathring{\gamma}_B[f_B] = \mathcal{P}_S[ \mathring{\gamma}_-[\psi_B]].
\end{equation*}
Here $\mathring{\gamma}_-$ is the trace operator such that $\mathring{\gamma}_-: W^{1,1} (\mathbb{R}^3_-) \to L^1(\mathbb{R}^2)$.
\end{definition}

By the same arguments as in the proof of Proposition \ref{prop3016}, we have the lemma.
\begin{lemma}\label{lem3020}$(\rm{i})$ There is $C >0$ such that for each $f_A \in W^{1,1} (\Omega_A)$,
\begin{equation*}
\Vert \mathring{\gamma}_A[f_A] \Vert_{L^1(\Gamma)} \leq C \Vert f_A \Vert_{W^{1,1} (\Omega_A)}.
\end{equation*}
$(\rm{ii})$ There is $C >0$ such that for each $f_B \in W^{1,1} (\Omega_B)$,
\begin{equation*}
\Vert \mathring{\gamma}_B[f_B] \Vert_{L^1(\Gamma)} \leq C \Vert f_B \Vert_{W^{1,1} (\Omega_B)}.
\end{equation*}
\end{lemma}

\begin{proposition}[Divergence theorems and integration by parts formulas]\label{prop3021}{ \ }\\
$(\rm{i})$ For all $F_A = { }^t (F^A_1,F^A_2,F^A_3) \in W^{1,1}(\Omega_A)$, $F_B = { }^t (F^B_1,F^B_2,F^B_3) \in W^{1,1}(\Omega_B)$, and $F_S = { }^t (F^S_1,F^S_2,F^S_3) \in W^{1,1}(\Gamma)$,
\begin{align*}
\int_{\Omega_A } \nabla \cdot F_A { \ }d x &= - \int_\Gamma F_A \cdot n { \ } d\mathcal{H}_x^2,\\
\int_{\Omega_B } \nabla \cdot F_B { \ }d x &= \int_\Gamma F_B \cdot n { \ } d\mathcal{H}_x^2,\\
\int_{\Gamma } \nabla_\Gamma \cdot F_S { \ }d \mathcal{H}_x^2 & = - \int_\Gamma H_\Gamma (F_S \cdot n) { \ } d\mathcal{H}_x^2.
\end{align*}
Here $H_\Gamma (= - {\rm{div}}_\Gamma n)$ is the mean curvature.\\ 
$(\rm{ii})$ For all $f_A, \phi_A \in W^{1,2}(\Omega_A)$, $f_B, \phi_B \in W^{1,2}(\Omega_B)$, and $f_S, \phi_S \in W^{1,2}(\Gamma )$, 
\begin{align*}
\int_{\Omega_A } f_A (\partial_j \phi_A) { \ }d x & = - \int_{\Omega_A } (\partial_j f_A) \phi_A { \ }d x - \int_\Gamma f_A \phi_A n_j { \ } d\mathcal{H}_x^2,\\
\int_{\Omega_B } f_B (\partial_j \phi_B) { \ }d x & = - \int_{\Omega_B } (\partial_j f_B) \phi_B { \ }d x + \int_\Gamma f_B \phi_B n_j { \ } d\mathcal{H}_x^2,\\
\int_{\Gamma } f_S (\partial_j^\Gamma \phi_S) { \ }d \mathcal{H}_x^2 &= - \int_{\Gamma } (\partial_j^\Gamma f_S) \phi_S { \ }d \mathcal{H}_x^2 - \int_\Gamma H_\Gamma f_S \phi_S n_j { \ } d\mathcal{H}_x^2.
\end{align*}
$(\rm{iii})$ For all $f_A \in W^{2,2}(\Omega_A)$, $\phi_A \in W^{1,2}(\Omega_A)$, $f_B \in W^{2,2}( \Omega_B )$, $\phi_B \in W^{1,2}( \Omega_B )$, $f_S \in W^{2,2}(\Gamma )$, $\phi_S \in W^{1,2}(\Gamma )$,
\begin{align*}
-  \int_{\Omega_A} \nabla f_A \cdot \nabla \phi_A { \ }d x & = \int_{\Omega_A} ( \Delta f_A ) \phi_A { \ }d x +  \int_{\Gamma} \frac{\partial f_A}{\partial n} \phi_A { \ }d \mathcal{H}_x^2,\\
-  \int_{\Omega_B}  \nabla f_B \cdot \nabla \phi_B { \ }d x & = \int_{\Omega_B} ( \Delta f_B ) \phi_B { \ }d x -  \int_{\Gamma} \frac{\partial f_B}{\partial n} \phi_B { \ }d \mathcal{H}_x^2,\\
-  \int_{\Gamma} \nabla_\Gamma f_S \cdot \nabla_\Gamma \phi_S { \ }d \mathcal{H}_x^2 &= \int_{\Gamma} ( \Delta_\Gamma f_S ) \phi_S { \ }d \mathcal{H}_x^2.
\end{align*}
Here $\partial f/{\partial n} := (n \cdot \nabla)f $.
\end{proposition}

\begin{proof}[Proof of Proposition \ref{prop3021}]
Since $C_0^2(\overline{\Omega_A})$ is dense in $W^{1,1} (\Omega_A)$, $C_0^2(\overline{\Omega_B})$ is dense in $W^{1,1} (\Omega_B)$, $C_0^2(\Gamma)$ is dense in $W^{1,1} (\Gamma)$ from Proposition \ref{prop3014}, we apply Lemmas \ref{lem38}, \ref{lem3020}, \ref{lem7013} to derive $(\rm{i})$. In the same manner, we see $(\rm{ii})$ and $(\rm{iii})$ from Lemmas \ref{lem37}, \ref{lem7014}, \ref{lem7015}, and Propositions \ref{prop3014}, \ref{prop3016}. Therefore, Proposition \ref{prop3021} is proved.
\end{proof}

Finally, we introduce Banach-space-valued functions. Write
\begin{equation*}
\Omega_S = \Gamma,{ \ }U_A = \mathbb{R}^3_+, {  \ }U_B = \mathbb{R}^3_-,{ \ }U_S = \mathbb{R}^2, \text{ and }[0, \infty ] = [0 , \infty ).
\end{equation*}

\begin{definition}\label{def3022}
For $T \in (0, \infty]$, $\sharp \in \{ A,B,S \}$, and $p \in \{ 1 , 2 \}$, we define
\begin{equation*}
L^p_{loc}(\Omega_\sharp \times (0,T)) = \{ \Phi_\sharp : \Phi_\sharp = \mathcal{P}_\sharp [w_\sharp], w_\sharp \in L^p_{loc}(U_\sharp \times (0,T)) \},
\end{equation*}
\begin{align*}
C([0,T]; L^2 (\Omega_\sharp) ) = \{ \Phi_\sharp ; \Phi_\sharp = \mathcal{P}_\sharp [w_\sharp], w_\sharp \in C([0,T]; L^2 (U_\sharp)) \cap L^1_{loc}(U_\sharp \times (0,T)) \},
\end{align*}
\begin{multline*}
L^2(0,T; L^2(\Omega_\sharp) ) = \{ \Phi_\sharp ; \Vert \Phi_\sharp \Vert_{L^2(0,T;L^2(\Omega_\sharp))} < \infty,\\ \Phi_\sharp = \mathcal{P}_\sharp [w_\sharp], w_\sharp \in L^2(0,T;L^2(U_\sharp)) \cap  L_{loc}^1(U_\sharp \times (0,T))  \},
\end{multline*}
\begin{multline*}
L^2(0,T; W^{1,2}(\Omega_\sharp) ) = \{ \Phi_\sharp ; \Vert \Phi_\sharp \Vert_{L^2(0,T;W^{1,2}(\Omega_\sharp))} < \infty,\\ \Phi_\sharp = \mathcal{P}_\sharp [w_\sharp], w_\sharp \in L^2(0,T;W^{1,2}(U_\sharp)) \cap  L_{loc}^1(U_\sharp \times (0,T))  \},
\end{multline*}
\begin{multline*}
L^2(0,T; W^{2,2}(\Omega_\sharp) ) = \{ \Phi_\sharp ; \Vert \Phi_\sharp \Vert_{L^2(0,T;W^{2,2}(\Omega_\sharp))} < \infty,\\ \Phi_\sharp = \mathcal{P}_\sharp [w_\sharp], w_\sharp \in L^2(0,T;W^{2,2}(U_\sharp)) \cap  L_{loc}^1(U_\sharp \times (0,T))  \},
\end{multline*}
\begin{multline*}
W^{1,2}(0,T; L^2(\Omega_\sharp) ) = \{ \Phi_\sharp ; \Vert \Phi_\sharp \Vert_{W^{1,2}(0,T;L^2(\Omega_\sharp))} < \infty,\\ \Phi_\sharp = \mathcal{P}_\sharp [w_\sharp], w_\sharp \in W^{1,2}(0,T;L^2(U_\sharp)) \cap  L_{loc}^1(U_\sharp \times (0,T))  \}.
\end{multline*}
Here
\begin{align*}
\Vert \Phi_\sharp \Vert_{L^2(0,T;L^2(\Omega_\sharp))} & := \bigg( \int_0^T \Vert w_\sharp \Vert_{\mathcal{L}^2(\Omega_\sharp)}^2  { \ }dt \bigg)^{1/2},\\
\Vert \Phi_\sharp \Vert_{L^2(0,T;W^{1,2}(\Omega_\sharp))} & := \bigg( \int_0^T \Vert w_\sharp \Vert_{\mathcal{W}^{1,2}(\Omega_\sharp)}^2  { \ }dt \bigg)^{1/2},\\
\Vert \Phi_\sharp \Vert_{L^2(0,T;W^{2,2}(\Omega_\sharp))} & := \bigg( \int_0^T \Vert w_\sharp \Vert_{\mathcal{W}^{2,2}(\Omega_\sharp)}^2  { \ }dt \bigg)^{1/2},\\
\Vert \Phi_\sharp \Vert_{W^{1,2}(0,T;L^2(\Omega_\sharp))} & := \bigg( \int_0^T \Vert w_\sharp \Vert_{\mathcal{L}^2(\Omega_\sharp)}^2  { \ }dt + \int_0^T \Vert \partial_t w_\sharp \Vert_{\mathcal{L}^2(\Omega_\sharp)}^2  { \ }dt\bigg)^{1/2}.
\end{align*}
\end{definition}

\begin{remark}\label{rem3023}
Let $\sharp \in \{ A,B,S \}$ and $T \in (0, \infty]$. Let $\Phi_\sharp \in L^2(0,T; W^{2,2}(\Omega_\sharp) )$ and $w_\sharp \in L^2(0,T;W^{2,2} (U_\sharp ))$ such that $\Psi_\sharp = \mathcal{P}_\sharp [w_\sharp]$. From Lemma \ref{lem37}, we see that
\begin{align*}
\Vert w_\sharp \Vert_{L^2(0,T; L^2(U_\sharp))} & \leq \Vert \Phi_\sharp \Vert_{L^2(0,T; L^2(\Omega_\sharp))} \leq \sqrt{2} \Vert w_\sharp \Vert_{L^2(0,T; L^2(U_\sharp))},\\
(1/3)\Vert w_\sharp \Vert_{L^2(0,T; W^{1,2}(U_\sharp))} & \leq \Vert \Phi_\sharp \Vert_{L^2(0,T; W^{1,2}(\Omega_\sharp))} \leq 3 \Vert w_\sharp \Vert_{L^2(0,T; W^{1,2}(U_\sharp))},\\
(1/C)\Vert w_\sharp \Vert_{L^2(0,T; W^{2,2}(U_\sharp))} & \leq \Vert \Phi_\sharp \Vert_{L^2(0,T; W^{2,2}(\Omega_\sharp))} \leq C \Vert w_\sharp \Vert_{L^2(0,T; W^{2,2}(U_\sharp))}.
\end{align*}
Here $C = C ( \Vert \nabla_X^2 \varphi \Vert_{L^\infty (\mathbb{R}^2)} ) >0$.
\end{remark}

\section{On the Uniqueness and Existence of our Systems}\label{sect4}

In this section, we discuss the uniqueness and the existence of strong solutions to systems \eqref{eq15} and \eqref{eq29}. In Section \ref{sect5}, we prove the existence of strong solutions to these systems. Let $\varphi \in BC^2(\mathbb{R}^2)$ and $n= n(x) = { }^t (n_1,n_2,n_3)$ be the unit outer normal vector at $x \in \Gamma$ defined by \eqref{eq11}. Let $\rho_A , \rho_B, \rho_S , \kappa_A , \kappa_B, \kappa_S >0$ and $\delta >0$. Let $\mathcal{P}_A$, $\mathcal{P}_B$, and $\mathcal{P}_S$ be the three pullback operators defined by Definition \ref{def32}. Let $\gamma_A$ and $\gamma_B$ be the two trace operators defined by Definition \ref{def3015}, and $\gamma_+$ and $\gamma_-$ be the two trace operators such that $\gamma_+:W^{1,2} (\mathbb{R}^3_+) \to L^2(\mathbb{R}^2)$ and $\gamma_-:W^{1,2} (\mathbb{R}^3_-) \to L^2(\mathbb{R}^2)$. Let $\mathcal{F}_A$, $\mathcal{F}_B$, $\mathcal{F}_S$, $\mathring{\mathcal{F}}_1$, $\mathring{\mathcal{F}}_2$, $\mathring{\mathcal{F}}_3$, $\mathcal{F}_1$, $\mathcal{F}_2, \mathcal{F}_3$, and $\mathcal{A}_1$, $\mathcal{A}_2$, $\mathcal{A}_3$, $\mathcal{B}_1$, $\mathcal{B}_2$, $\mathcal{B}_3$ be the nine mappings and the six operators defined by Section \ref{sect2} (see also Sections \ref{sect5} and \ref{sect7}).

We first give the definition of a strong solution to systems \eqref{eq15} and \eqref{eq29}. Then we investigate the relationship between the strong solutions of the two systems. In particular, we apply the uniqueness of the strong solutions to \eqref{eq15} to derive the uniqueness of the strong solutions to \eqref{eq29} (see Proposition \ref{prop43} for details). This is one of the key ideas of this paper.

\begin{definition}[Strong solutions to system \eqref{eq15}]\label{def41}
Let ${ }^t (\theta_0^A , \theta_0^B , \theta_0^S ) \in W^{1,2} (\Omega_A) \times W^{1,2} (\Omega_B) \times W^{1,2} (\Gamma)$ such that $\gamma_A [\theta_0^A] = \theta_0^S$ and $\gamma_B [\theta_0^B] = \theta_0^S$.\\
$(\rm{i})$ {\rm{[Local-in-time strong solutions]}} Let $T >0$. Let $\theta_A \in L^1_{loc}( \Omega_A \times (0,T)) \cap C([0,T];L^2(\Omega_A))$, $\theta_B \in L^1_{loc}( \Omega_B \times (0,T)) \cap C([0,T];L^2(\Omega_B))$, and $\theta_S \in L^1_{loc}( \Gamma \times (0,T)) \cap C([0,T];L^2(\Gamma))$. We call ${ }^t (\theta_A , \theta_B , \theta_S)$ a \emph{(local-in-time) strong solution} to system \eqref{eq15} with initial data ${ }^t (\theta_0^A , \theta_0^B , \theta_0^S )$ if ${ }^t (\theta_A , \theta_B , \theta_S)$ satisfies the following properties:
\begin{align}
\theta_A \in L^2 (0,T; W^{2,2} (\Omega_A)) \cap W^{1,2} (0,T ; L^2(\Omega_A)),\label{eq41}\\
\theta_B \in L^2 (0,T; W^{2,2} (\Omega_B)) \cap W^{1,2} (0,T ; L^2(\Omega_B)),\\
\theta_S \in L^2 (0,T; W^{2,2} (\Gamma)) \cap W^{1,2} (0,T ; L^2(\Gamma)),
\end{align}
\begin{align}
& \Vert \partial_t \theta_A - \kappa_A \Delta \theta_A \Vert_{L^2(0,T;L^2(\Omega_A))} = 0 ,\\
& \Vert \partial_t \theta_B - \kappa_B \Delta \theta_B \Vert_{L^2(0,T; L^2(\Omega_B))} = 0 ,
\end{align}
\begin{equation}
\bigg\Vert \partial_t \theta_S - \kappa_S \Delta_\Gamma \theta_S - \frac{ \rho_A \kappa_A}{\rho_S} \gamma_A[ (n \cdot \nabla) \theta_A] + \frac{\rho_B \kappa_B}{\rho_S} \gamma_B[ (n \cdot \nabla) \theta_B] \bigg\Vert_{L^2(0,T;L^2(\Gamma ))} = 0 ,
\end{equation}
\begin{align}
& \Vert \gamma_A [\theta_A] - \theta_S \Vert_{L^2(0,T;L^2(\Gamma))} = 0,\\
& \Vert \gamma_B [ \theta_B] - \theta_S \Vert_{L^2(0,T;L^2(\Gamma))} =0,
\end{align}
and
\begin{multline}\label{eq49}
\lim_{t \to  0 + 0 }\Vert \theta_A (t) - \theta_0^A \Vert_{L^2(\Omega_A)} = 0,{ \ }\lim_{t \to  0 + 0 }\Vert \theta_B (t) - \theta_0^B \Vert_{L^2(\Omega_B)} = 0,\\
\lim_{t \to  0 + 0 }\Vert \theta_S (t) - \theta_0^S \Vert_{L^2(\Gamma)} = 0.
\end{multline}
$(\rm{ii})$ {\rm{[Global-in-time strong solutions]}} Let $\theta_A \in L^1_{loc}( \Omega_A \times (0,\infty)) \cap C([0, \infty);L^2(\Omega_A))$, $\theta_B \in L^1_{loc}( \Omega_B \times (0,\infty)) \cap C([0,\infty);L^2(\Omega_B))$, and $\theta_S \in L^1_{loc}( \Gamma \times (0,\infty)) \cap C([0,\infty);L^2(\Gamma))$. We call ${ }^t (\theta_A , \theta_B , \theta_S)$ a \emph{global-in-time strong solution} to system \eqref{eq15} with initial data ${ }^t (\theta_0^A , \theta_0^B , \theta_0^S )$ if ${ }^t (\theta_A , \theta_B , \theta_S)$ satisfies that for each fixed $T >0$, \eqref{eq41}-\eqref{eq49} hold.
\end{definition}

\begin{definition}[Strong solutions to system \eqref{eq29}]\label{def42}
Let ${ }^t (v_0^A , v_0^B , v_0^S ) \in W_0^{1,2} (\mathbb{R}^3_+) \times W_0^{1,2} (\mathbb{R}^3_-) \times W^{1,2} (\mathbb{R}^2)$.\\
$(\rm{i})$ {\rm{[Local-in-time strong solutions]}} Let $T >0$ and ${ }^t (v_A , v_B , v_S) \in C([0,T]; L^2(\mathbb{R}^3_+) ) \times C([0,T]; L^2(\mathbb{R}^3_-) ) \times C([0,T]; L^2(\mathbb{R}^2) )$. We call ${ }^t (v_A , v_B , v_S)$ a \emph{(local-in-time) strong solution} to system \eqref{eq29} with initial data ${ }^t (v_0^A , v_0^B , v_0^S )$ if ${ }^t (v_A , v_B , v_S)$ satisfies the following properties:
\begin{align}
v_A \in L^2 (0,T; W_0^{1,2} (\mathbb{R}^3_+ ) \cap W^{2,2} (\mathbb{R}^3_+)) \cap W^{1,2} (0,T ; L^2(\mathbb{R}^3_+)),\label{eq4010}\\
v_B \in L^2 (0,T; W_0^{1,2} (\mathbb{R}^3_-) \cap W^{2,2} (\mathbb{R}^3_-)) \cap W^{1,2} (0,T ; L^2(\mathbb{R}^3_-)),\\
v_S \in L^2 (0,T; W^{2,2} (\mathbb{R}^2)) \cap W^{1,2} (0,T ; L^2(\mathbb{R}^2)),
\end{align}
\begin{align}
\Vert dv_A/{dt} + \mathcal{A}_1 v_A + \mathcal{B}_1 v_A - \mathcal{F}_1(v_A,v_B,v_S) \Vert_{L^2(0,T;L^2(\mathbb{R}^3_+))} = 0,\\
\Vert dv_B/{dt} + \mathcal{A}_2 v_B + \mathcal{B}_2 v_B - \mathcal{F}_2(v_A,v_B,v_S) \Vert_{L^2(0,T;L^2(\mathbb{R}^3_-))} = 0,\\
\Vert dv_S/{dt} + \mathcal{A}_3 v_S + \mathcal{B}_3 v_S - \mathcal{F}_3(v_A,v_B,v_S) \Vert_{L^2(0,T;L^2(\mathbb{R}^2))} = 0,
\end{align}
and
\begin{multline}\label{eq4016}
\lim_{t \to  0 + 0 }\Vert v_A (t) - v_0^A \Vert_{L^2(\mathbb{R}^3_+)} = 0,{ \ }\lim_{t \to  0 + 0 }\Vert v_B (t) - v_0^B \Vert_{L^2(\mathbb{R}^3_-)} = 0,\\
\lim_{t \to  0 + 0 }\Vert v_S (t) - v_0^S \Vert_{L^2(\mathbb{R}^2)} = 0.
\end{multline}
Here
\begin{align*}
L^2 (0,T; W_0^{1,2} \cap W^{2,2} (\mathbb{R}^3_+)) := \{ v_A \in L^2 (0,T; W^{2,2} (\mathbb{R}^3_+))   ;  \Vert \gamma_+ [v_A] \Vert_{L^2(0,T ; L^2(\mathbb{R}^2))} = 0   \},\\
L^2 (0,T; W_0^{1,2} \cap W^{2,2} (\mathbb{R}^3_-)) := \{ v_B \in L^2 (0,T; W^{2,2} (\mathbb{R}^3_-))   ;  \Vert \gamma_- [v_B] \Vert_{L^2(0,T ; L^2(\mathbb{R}^2))} = 0   \}.
\end{align*}
$(\rm{ii})$ {\rm{[Global-in-time strong solutions]}} Let ${ }^t (v_A , v_B , v_S) \in C([0, \infty); L^2(\mathbb{R}^3_+) ) \times C([0,\infty); L^2(\mathbb{R}^3_-) ) \times C([0, \infty ); L^2(\mathbb{R}^2) )$. We call ${ }^t (v_A , v_B , v_S)$ a \emph{global-in-time strong solution} to system \eqref{eq29} with initial data ${ }^t (v_0^A , v_0^B , v_0^S )$ if ${ }^t (v_A , v_B , v_S)$ satisfies that for each fixed $T >0$, \eqref{eq4010}-\eqref{eq4016} hold.
\end{definition}
\noindent Remark that we define the solutions of systems \eqref{eq24}, \eqref{eq28}, \eqref{eq2010} in the same way as Definitions \ref{def41}, \ref{def42}.

Let us investigate the relationship between the strong solutions of systems \eqref{eq15} and \eqref{eq29}.
\begin{proposition}[Uniqueness and existence]\label{prop43}Let $v_0^A \in W_0^{1,2}(\mathbb{R}^3_+)$, $v_0^B \in W_0^{1,2} (\mathbb{R}^3_-)$, $v_0^S \in W^{1,2}(\mathbb{R}^2)$, and $T \in (0,\infty)$. Let $v_A \in C([0,T]; L^2(\mathbb{R}^3_+)) \cap L_{loc}^1 ( \mathbb{R}^3_+ \times (0,T))$, $v_B \in C([0,T]; L^2(\mathbb{R}^3_-)) \cap L_{loc}^1 ( \mathbb{R}^3_- \times (0,T))$, and $v_S \in C([0,T]; L^2(\mathbb{R}^2)) \cap L_{loc}^1 ( \mathbb{R}^2 \times (0,T))$. For $t \in (0,T]$, we set
\begin{align}
\label{eq4017}\begin{cases}
\theta_A = \theta_A (x,t) = \mathcal{P}_A [v_A (y,t) + {\rm{e}}^{- \delta y_3} v_S(y_h,t)](x,t),\\
\theta_B = \theta_B (x,t) = \mathcal{P}_B [v_B (z,t) + {\rm{e}}^{\delta z_3} v_S(z_h,t)](x,t),\\
\theta_S = \theta_S (x,t) = \mathcal{P}_S [v_S(X,t)](x,t),
\end{cases}\\
\label{eq4018}\begin{cases}
\theta_0^A = \theta_0^A (x) = \mathcal{P}_A [v_0^A (y) + {\rm{e}}^{- \delta y_3} v_0^S(y_h)](x),\\
\theta_0^B = \theta_0^B (x) = \mathcal{P}_B [v_0^B (z) + {\rm{e}}^{\delta z_3} v_0^S(z_h)](x),\\
\theta_0^S = \theta_0^S (x) = \mathcal{P}_S [v_0^S(X)](x).
\end{cases}
\end{align}
Assume that ${ }^t(v_A , v_B, v_S)$ is a strong solution to \eqref{eq29} with initial data ${ }^t(v_0^A , v_0^B , v_0^S)$. Then, the three assertions hold:\\
$(\rm{i})$ ${ }^t (\theta_0^A , \theta_0^B , \theta_0^S ) \in W^{1,2} (\Omega_A) \times W^{1,2} (\Omega_B) \times W^{1,2} (\Gamma)$, $\gamma_A [\theta_0^A] = \theta_0^S$, and $\gamma_B[\theta_0^B] = \theta_0^S$.\\
$(\rm{ii})$ ${ }^t(\theta_A , \theta_B , \theta_S)$ is a unique strong solution to \eqref{eq15} with initial data ${ }^t(\theta_0^A , \theta_0^B , \theta_0^S)$.\\
$(\rm{iii})$ ${ }^t (v_A , v_B, v_S)$ is a unique strong solution to \eqref{eq29} with initial data ${ }^t(v_0^A , v_0^B , v_0^S)$.
\end{proposition}

To prove Proposition \ref{prop43}, we prepare the lemma.
\begin{lemma}[Energy equality and uniqueness]\label{lem44}
Let $\theta_0^A \in W^{1,2}(\Omega_A)$, $\theta_0^B \in W^{1,2} (\Omega_B)$, $\theta_0^S \in W^{1,2}(\Gamma)$ such that $\gamma_A [\theta_0^A] = \theta_0^S$ and $\gamma_B[\theta_0^B] = \theta_0^S$, and $T \in (0,\infty)$. Let $\theta_A,\theta_A^\natural \in L^1_{loc}( \Omega_A \times (0,T)) \cap C([0,T];L^2(\Omega_A))$, $\theta_B,\theta_B^\natural \in L^1_{loc}( \Omega_B \times (0,T)) \cap C([0,T];L^2(\Omega_B))$, and $\theta_S, \theta_S^\natural \in L^1_{loc}( \Gamma \times (0,T)) \cap C([0,T];L^2(\Gamma))$. Assume that ${ }^t(\theta_A , \theta_B , \theta_S)$ and ${ }^t(\theta_A^\natural , \theta_B^\natural , \theta_S^\natural)$ are two strong solutions to system \eqref{eq15} with initial data ${ }^t(\theta_0^A, \theta_0^B , \theta_0^S)$. Then the two assertions hold:\\
$(\rm{i})$ The solution ${ }^t(\theta_A, \theta_B , \theta_S )$ satisfies that for each $0 \leq t_1 \leq t_2 \leq T$,
\begin{multline}\label{eq4019}
\rho_A \Vert \theta_A (t_2) \Vert_{L^2(\Omega_A)}^2 + \rho_B \Vert \theta_B(t_2) \Vert_{L^2(\Omega_B)}^2 + \rho_S \Vert \theta_S(t_2) \Vert_{L^2(\Gamma)}^2 \\
+ 2 \rho_A \kappa_A \int_{t_1}^{t_2} \Vert \nabla \theta_A (\tau) \Vert_{L^2(\Omega_A)}^2 { \ }d \tau + 2 \rho_B \kappa_B \int_{t_1}^{t_2} \Vert \nabla \theta_B (\tau) \Vert_{L^2(\Omega_B)}^2 { \ } d \tau\\
 + 2 \rho_S \kappa_S \int_{t_1}^{t_2} \Vert \nabla_\Gamma \theta_S (\tau) \Vert_{L^2(\Gamma)}^2{ \ }d \tau\\
= \rho_A \Vert \theta_A (t_1) \Vert_{L^2(\Omega_A)}^2 + \rho_B \Vert \theta_B(t_1) \Vert_{L^2(\Omega_B)}^2 + \rho_S \Vert \theta_S(t_1) \Vert_{L^2(\Gamma)}^2.
\end{multline}
$(\rm{ii})$ For all $t \in [0,T]$,
\begin{equation}\label{eq4020}
\Vert \theta_A (t) -  \theta_A^\natural (t) \Vert_{L^2(\Omega_A)} = 0, \Vert \theta_B (t) - \theta_B^\natural (t) \Vert_{L^2(\Omega_B)} = 0, \Vert \theta_S (t) - \theta_S^\natural (t) \Vert_{L^2(\Gamma)} = 0.
\end{equation}
\end{lemma}

\begin{proof}[Proof of Lemma \ref{lem44}]
 We first show $(\rm{i})$. Since ${ }^t(\theta_A , \theta_B , \theta_S)$ is a strong solution to system \eqref{eq15} with initial data ${ }^t(\theta_0^A, \theta_0^B , \theta_0^S)$, we use assertion $(\rm{iii})$ in Proposition \ref{prop3021} to check that for almost all $\tau \in (0,T)$,
\begin{multline}\label{eq4021}
\frac{1}{2}\frac{d}{d\tau} ( \rho_A \Vert \theta_A (\tau) \Vert_{L^2(\Omega_A)}^2 + \rho_B \Vert \theta_B(\tau) \Vert_{L^2(\Omega_B)}^2 + \rho_S \Vert \theta_S(\tau) \Vert_{L^2(\Gamma)}^2 )\\
= \dual{\rho_A \partial_\tau \theta_A, \theta_A }_{L^2(\Omega_A)} + \dual{\rho_B \partial_\tau \theta_B, \theta_B }_{L^2(\Omega_B)} + \dual{\rho_S \partial_\tau \theta_S, \theta_S }_{L^2(\Gamma)}\\
= \dual{\rho_A \kappa_A \Delta \theta_A, \theta_A }_{L^2(\Omega_A)} + \dual{\rho_B \kappa_B \Delta \theta_B, \theta_B }_{L^2(\Omega_B)}\\
 + \dual{\rho_S \kappa_S \Delta_\Gamma \theta_S + \rho_A \kappa_A \gamma_A[ (n \cdot \nabla) \theta_A] - \rho_B \kappa_B \gamma_B[ (n \cdot \nabla) \theta_B], \theta_S }_{L^2(\Gamma)}\\
 = - \rho_A \kappa_A \Vert \nabla \theta_A (\tau) \Vert_{L^2(\Omega_A)}^2 - \rho_B \kappa_B \Vert \nabla \theta_B (\tau) \Vert_{L^2(\Omega_B)}^2 - \rho_S \kappa_S \Vert \nabla \theta_S (\tau) \Vert_{L^2(\Gamma)}^2.
\end{multline}
Let $N_*$ be the null set of $[0,T]$. Fix $t_1, t_2 \in [0,T]$ such that $t_1 \leq t_2$. Since $[0,T] \setminus N_*$ is dense in $[0,T]$, there are $\{ t_1^m \}$, $\{ t_2^{m'} \} \subset [0,T] \setminus N_*$ such that $t_1^m \leq t_2^{m'}$, $t_1^m \to t_1$, $t_2^{m'} \to t_2$ as $m \to \infty$ and $m' \to \infty$. Integrating both sides of \eqref{eq4021} with respect to time $\tau$, we obtain
\begin{multline*}
\rho_A \Vert \theta_A (t_2^{m'}) \Vert_{L^2(\Omega_A)}^2 + \rho_B \Vert \theta_B(t_2^{m'}) \Vert_{L^2(\Omega_B)}^2 + \rho_S \Vert \theta_S(t_2^{m'}) \Vert_{L^2(\Gamma)}^2 \\
+ 2 \rho_A \kappa_A \int_{t_1^m}^{t_2^{m'}} \Vert \nabla \theta_A (\tau) \Vert_{L^2(\Omega_A)}^2 { \ }d\tau + 2 \rho_B \kappa_B \int_{t_1^m}^{t_2^{m'}} \Vert \nabla \theta_B (\tau) \Vert_{L^2(\Omega_B)}^2 { \ } d\tau\\
 + 2 \rho_S \kappa_S \int_{t_1^m}^{t_2^{m'}} \Vert \nabla_\Gamma \theta_S (\tau) \Vert_{L^2(\Gamma)}^2{ \ }d \tau\\
= \rho_A \Vert \theta_A (t_1^m) \Vert_{L^2(\Omega_A)}^2 + \rho_B \Vert \theta_B(t_1^m) \Vert_{L^2(\Omega_B)}^2 + \rho_S \Vert \theta_S(t_1^m) \Vert_{L^2(\Gamma)}^2.
\end{multline*}
Since $\theta_A \in C ([0,T];L^2(\Omega_A)) \cap L^2(0,T; W^{2,2} (\Omega_A) )$, $\theta_B \in C ([0,T];L^2(\Omega_B)) \cap L^2(0,T; W^{2,2} (\Omega_B) )$, and $\theta_S \in C ([0,T];L^2(\Gamma)) \cap L^2(0,T; W^{2,2} (\Gamma) )$, we let $m \to \infty$ and $m' \to \infty$ to have \eqref{eq4019}.

Next, we show $(\rm{ii})$. Set $\theta_A^\star = \theta_A - \theta_A^\natural$, $\theta_B^\star = \theta_B - \theta_B^\natural$, and $\theta_S^\star = \theta_S - \theta_S^\natural$. We easily check that $\theta_A^\star \in L^1_{loc}( \Omega_A \times (0,T)) \cap C([0,T];L^2(\Omega_A))$, $\theta_B^\star \in L^1_{loc}( \Omega_B \times (0,T)) \cap C([0,T];L^2(\Omega_B))$, $\theta_S^\star \in L^1_{loc}( \Gamma \times (0,T)) \cap C([0,T];L^2(\Gamma))$, and that ${ }^t(\theta_A^\star , \theta_B^\star, \theta_S^\star)$ is a strong solution to the following system:
\begin{equation*}
\begin{cases}
\rho_A \partial_t \theta^\star_A = \rho_A \kappa_A \Delta \theta^\star_A & \text{ in } \Omega_{A.T}  ,\\
\rho_B \partial_t \theta^\star_B = \rho_B \kappa_B \Delta \theta^\star_B & \text{ in } \Omega_{B,T},\\
\rho_S \partial_t \theta^\star_S = \rho_S \kappa_S \Delta_\Gamma \theta^\star_S + \rho_A \kappa_A \gamma_A[ (n \cdot \nabla) \theta^\star_A] - \rho_B \kappa_B \gamma_B[ (n \cdot \nabla) \theta^\star_B] & \text{ in } \Gamma_T,\\
\gamma_A [\theta^\star_A] = \gamma_B [\theta^\star_B] = \theta^\star_S & \text{ in } \Gamma_T,\\
\theta^\star_A \vert_{t=0} = 0{ \ } \text{ in } \Omega_A, { \ }\theta^\star_B \vert_{t = 0} = 0{ \ } \text{ in }\Omega_B, { \ }\theta^\star_S \vert_{t=0} = 0{ \ } \text{ in } \Gamma .&
\end{cases}
\end{equation*}
By the same argument to derive \eqref{eq4019}, we find that for $0 \leq t \leq T$,
\begin{multline*}
\rho_A \Vert \theta^\star_A (t) \Vert_{L^2(\Omega_A)}^2 + \rho_B \Vert \theta^\star_B(t) \Vert_{L^2(\Omega_B)}^2 + \rho_S \Vert \theta^\star_S(t) \Vert_{L^2(\Gamma)}^2\\
+ 2 \rho_A \kappa_A \int_{0}^{t} \Vert \nabla \theta^\star_A (\tau) \Vert_{L^2(\Omega_A)}^2 { \ }d\tau + 2 \rho_B \kappa_B \int_{0}^{t} \Vert \nabla \theta^\star_B (\tau) \Vert_{L^2(\Omega_B)}^2 { \ } d\tau\\
 + 2 \rho_S \kappa_S \int_{0}^{t} \Vert \nabla \theta^\star_S (\tau) \Vert_{L^2(\Gamma)}^2{ \ }d \tau = 0.
\end{multline*}
This implies that for all $t \in [0,T]$,
\begin{equation*}
\Vert \theta^\star_A (t) \Vert_{L^2(\Omega_A)}^2 + \Vert \theta^\star_B(t) \Vert_{L^2(\Omega_B)}^2 + \Vert \theta^\star_S(t) \Vert_{L^2(\Gamma)}^2 = 0,
\end{equation*}
which is \eqref{eq4020}. Therefore, Lemma \ref{lem44} is proved.
\end{proof}

Now, we prove Proposition \ref{prop43}.
\begin{proof}[Proof of Proposition \ref{prop43}]
Fix $v_0^A \in W_0^{1,2}(\mathbb{R}^3_+)$, $v_0^B \in W_0^{1,2} (\mathbb{R}^3_-)$, $v_0^S \in W^{1,2}(\mathbb{R}^2)$, and $T \in (0,\infty)$. Let $v_A \in C([0,T]; L^2(\mathbb{R}^3_+)) \cap L_{loc}^1 ( \mathbb{R}^3_+ \times (0,T))$, $v_B \in C([0,T]; L^2(\mathbb{R}^3_-)) \cap L_{loc}^1 ( \mathbb{R}^3_- \times (0,T))$, and $v_S \in C([0,T]; L^2(\mathbb{R}^2)) \cap L_{loc}^1 ( \mathbb{R}^2 \times (0,T))$. Assume that ${ }^t(v_A , v_B, v_S)$ is a strong solution to system \eqref{eq29} with initial data ${ }^t(v_0^A , v_0^B , v_0^S)$. Since $v_A \in L_{loc}^1 ( \mathbb{R}^3_+ \times (0,T))$, $v_B \in L_{loc}^1 ( \mathbb{R}^3_- \times (0,T))$, $v_S \in L_{loc}^1 ( \mathbb{R}^2 \times (0,T))$, and ${ }^t(v_A , v_B, v_S)$ is a strong solution to \eqref{eq29}, we find that ${ }^t(v_A , v_B, v_S)$ satisfies that
\begin{align}
\Vert \partial_t v_A + \mathcal{A}_1 v_A + \mathcal{B}_1 v_A - \mathring{\mathcal{F}}_1(v_A,v_B,v_S) \Vert_{L^2(0,T;L^2(\mathbb{R}^3_+))} = 0,\label{eq4022}\\
\Vert \partial_t v_B + \mathcal{A}_2 v_B + \mathcal{B}_2 v_B - \mathring{\mathcal{F}}_2(v_A,v_B,v_S) \Vert_{L^2(0,T;L^2(\mathbb{R}^3_-))} = 0,\\
\Vert \partial_t v_S + \mathcal{A}_3 v_S + \mathcal{B}_3 v_S - \mathring{\mathcal{F}}_3(v_A,v_B,v_S) \Vert_{L^2(0,T;L^2(\mathbb{R}^2))} = 0.\label{eq4024}
\end{align}
That is, ${ }^t (v_A, v_B, v_S)$ is a solution to system \eqref{eq28}. For $t >0$ we set
\begin{equation}\label{eq4025}
\begin{cases}
u_A  = u_A (x,t) = \mathcal{P}_A [v_A(y,t)](x,t),\\
u_B  = u_B (x,t) = \mathcal{P}_B [v_B(z,t)](x,t),\\
u_S  = u_S (x,t) = \mathcal{P}_S [v_S(X,t)](x,t),
\end{cases}{ \ }
\begin{cases}
u_0^A  = u_0^A (x) = \mathcal{P}_A [v_0^A(y)](x),\\
u_0^B  = u_0^B (x) = \mathcal{P}_B [v_0^B(z)](x),\\
u_0^S  = u_0^S (x) = \mathcal{P}_S [v_0^S(X)](x).
\end{cases}
\end{equation}
From Definitions \ref{def3022}, \ref{def39}, and \ref{def3010}, we check that
\begin{align*}
u_A \in L^1_{loc} ((0,T) \times \Omega_A) \cap C([0,T]; L^2(\Omega_A)) \cap L^2 (0,T;  W^{2,2} (\Omega_A)) \cap W^{1,2} (0,T ; L^2(\Omega_A)),\\
u_B \in L^1_{loc} ((0,T) \times \Omega_B) \cap C([0,T]; L^2(\Omega_B)) \cap L^2 (0,T;  W^{2,2} (\Omega_B)) \cap W^{1,2} (0,T ; L^2(\Omega_B)),\\
u_S \in L^1_{loc} ((0,T) \times \Gamma) \cap C([0,T]; L^2(\Gamma)) \cap L^2 (0,T; W^{2,2} (\Gamma)) \cap W^{1,2} (0,T ; L^2(\Gamma)),
\end{align*}
and that
\begin{equation*}
{ }^t(u_0^A , u_0^B , u_0^S ) \in W_0^{1,2} (\Omega_A) \times W_0^{1,2} (\Omega_B) \times W^{1,2} (\Gamma).
\end{equation*}
From Lemmas \ref{lem71}, \ref{lem75}, \ref{lem79}, Remarks \ref{rem72}, \ref{rem76}, \ref{rem7010} (see also Section \ref{sect2}), and \eqref{eq4022}-\eqref{eq4024}, we find that
\begin{multline}\label{eq4026}
\Vert \partial_t u_A - \kappa_A \Delta u_A - \mathcal{F}_A(u_A, u_B,v_S) \Vert_{L^2(0,T;L^2(\Omega_A))}\\
 = \Vert \partial_t v_A + \mathcal{A}_1 v_A + \mathcal{B}_1 v_A - \mathring{\mathcal{F}}_1(v_A,v_B,v_S) \Vert_{L^2(0,T;L^2(\mathbb{R}^3_+))} = 0,
\end{multline}
\begin{multline}\label{eq4027}
\Vert \partial_t u_B - \kappa_B \Delta u_B - \mathcal{F}_B(u_A, u_B,v_S) \Vert_{L^2(0,T;L^2(\Omega_B))}\\
 = \Vert \partial v_B + \mathcal{A}_2 v_B + \mathcal{B}_2 v_B - \mathring{\mathcal{F}}_2(v_A,v_B,v_S) \Vert_{L^2(0,T;L^2(\mathbb{R}^3_-))} = 0,
\end{multline}
\begin{multline}\label{eq4028}
\Vert \partial_t u_S - \kappa_S \Delta_\Gamma u_S - \mathcal{F}_S(u_A, u_B, v_S) \Vert_{L^2(0,T;L^2(\Gamma))}\\
 = \Vert \{ \partial_t v_S + \mathcal{A}_3 v_S + \mathcal{B}_3 v_S - \mathring{\mathcal{F}}_3(v_A,v_B,v_S) \} G_S^{1/4} \Vert_{L^2(0,T;L^2(\mathbb{R}^2))} = 0.
\end{multline}
Here $G_S = 1 + \varphi_1^2 + \varphi_2^2$. From Proposition \ref{prop3016}, we see that
\begin{align*}
& \Vert \gamma_A [u_A] \Vert_{L^2(0,T;L^2(\Gamma))} \leq 2 \Vert \gamma_+ [v_A] \Vert_{L^2(0,T;L^2(\mathbb{R}^2))} = 0,\\ 
& \Vert \gamma_B [u_B] \Vert_{L^2(0,T;L^2(\Gamma))} \leq 2 \Vert \gamma_- [v_B] \Vert_{L^2(0,T;L^2(\mathbb{R}^2))} = 0.
\end{align*}
Using Lemma \ref{lem37} and the definitions of our norms (Definitions \ref{def39}, \ref{def3010}), we observe that
\begin{align}
\Vert u_A (t) - u_0^A \Vert_{L^2(\Omega_A)} = \Vert v_A (t) - v_0^A \Vert_{L^2(\mathbb{R}^3_+)} \to 0 { \ }(\text{as }t \to 0 + 0),\label{eq4029}\\
\Vert u_B (t) - u_0^B \Vert_{L^2(\Omega_B)} = \Vert v_B (t) - v_0^B \Vert_{L^2(\mathbb{R}^3_-)} \to 0 { \ }(\text{as }t \to 0 + 0),\\
\Vert u_S (t) - u_0^S \Vert_{L^2(\Gamma)} \leq \sqrt{2} \Vert v_S (t) - v_0^S \Vert_{L^2(\mathbb{R}^2)} \to 0 { \ }(\text{as }t \to 0 + 0).\label{eq4031}
\end{align}
Thus, we find that ${ }^t (u_A, u_B, u_S)$ is a solution to system \eqref{eq24} with $\hat{\theta}_S = v_S$. Remark that $v_S$ is a function in $\mathbb{R}^2$, $u_S$ is a function in $\Gamma$, and that $u_S = v_S$ for $ x \in \Gamma$.

Now we show $(\rm{i})$ and $(\rm{ii})$. By the definitions of $\mathcal{P}_A$, $\mathcal{P}_B$, $\mathcal{P}_S$, \eqref{eq4017}, and \eqref{eq4018}, we find that
\begin{align}
\label{eq4032}\begin{cases}
\theta_A = \theta_A (x,t) = u_A (x,t) + {\rm{e}}^{\delta \{ \varphi(x_h) - x_3 \}}v_S (x_h,t),\\
\theta_B = \theta_B (x,t) = u_B (x,t) + {\rm{e}}^{\delta \{ x_3 -\varphi(x_h) \}}v_S (x_h,t),\\
\theta_S = \theta_S (x,t) = u_S (x,t),
\end{cases}\\
\label{eq4033}\begin{cases}
\theta_0^A = \theta_0^A (x) = u_0^A (x) + {\rm{e}}^{\delta \{ \varphi(x_h) - x_3 \}}v_0^S (x_h),\\
\theta_0^B = \theta_0^B (x) = u_0^B (x) + {\rm{e}}^{\delta \{ x_3 -\varphi(x_h) \}}v_0^S (x_h),\\
\theta_0^S = \theta_0^S (x) = u_0^S (x).
\end{cases}
\end{align}
We easily check that
\begin{equation*}
{ }^t(\theta_0^A , \theta_0^B , \theta_0^S ) \in W^{1,2} (\Omega_A) \times W^{1,2} (\Omega_B) \times W^{1,2} (\Gamma)
\end{equation*}
such that $\gamma_A [\theta_0^A] = \theta_0^S$ and $\gamma_B [\theta_0^B] = \theta_0^S$ (see Proposition \ref{prop31}). By \eqref{eq4032}, \eqref{eq4026}-\eqref{eq4028}, we find that
\begin{multline*}
\Vert \partial_t \theta_A - \kappa_A \Delta \theta_A \Vert_{L^2(0,T;L^2(\Omega_A))}\\
= \Vert \partial_t u_A - \kappa_A \Delta u_A - \mathcal{F}_A(u_A, u_B,v_S) \Vert_{L^2(0,T;L^2(\Omega_A))} =0,
\end{multline*}
\begin{multline*}
\Vert \partial_t \theta_B - \kappa_B \Delta \theta_B \Vert_{L^2(0,T;L^2(\Omega_B))}\\
= \Vert \partial_t u_B - \kappa_B \Delta u_B - \mathcal{F}_B(u_A, u_B,v_S) \Vert_{L^2(0,T;L^2(\Omega_B))} = 0,
\end{multline*}
\begin{multline*}
\bigg\Vert \partial_t \theta_S - \kappa_S \Delta_\Gamma \theta_S - \frac{ \rho_A \kappa_A}{\rho_S} \gamma_A[ (n \cdot \nabla) \theta_A] + \frac{\rho_B \kappa_B}{\rho_S} \gamma_B[ (n \cdot \nabla) \theta_B] \bigg\Vert_{L^2(0,T;L^2(\Gamma ))}\\
= \Vert \partial_t u_S - \kappa_S \Delta_\Gamma u_S - \mathcal{F}_S(u_A, u_B,v_S) \Vert_{L^2(0,T;L^2(\Gamma))} =0.
\end{multline*}
We also find that
\begin{align*}
\Vert \gamma_A [\theta_A] - \theta_S \Vert_{L^2(0,T; L^2(\Gamma))} = \Vert \gamma_A [u_A] \Vert_{L^2(0,T; L^2(\Gamma))} =0,\\
\Vert \gamma_B [\theta_B] - \theta_S \Vert_{L^2(0,T; L^2(\Gamma))} = \Vert \gamma_B [u_B] \Vert_{L^2(0,T; L^2(\Gamma))} =0.
\end{align*}
From \eqref{eq4032}, \eqref{eq4033}, \eqref{eq4029}, \eqref{eq4031}, we check that
\begin{multline*}
\Vert \theta_A (t) - \theta_0^A \Vert_{L^2(\Omega_A)} \leq \Vert u_A (t) - u_0^A \Vert_{L^2(\Omega_A)} + \Vert {\rm{e}}^{\delta \{ \varphi(x_h) - x_3 \}} (v_S(x_h,t) - v_0^S )  \Vert_{L^2(\Omega_A)}\\
= \Vert u_A (t) - u_0^A \Vert_{L^2(\Omega_A)} + \Vert {\rm{e}}^{\delta \{ \varphi(x_h) - x_3 \}} (v_S(y_h,t) - v_0^S(y_h) ) \Vert_{L^2(\mathbb{R}^3_+)}\\
\leq \Vert u_A (t) - u_0^A \Vert_{L^2(\Omega_A)} + C(\delta) \Vert v_S(t) - v_0^S \Vert_{L^2(\mathbb{R}^2)} \to 0 { \ }(\text{as }t \to 0 + 0).
\end{multline*}
Similarly, we see that
\begin{equation*}
\lim_{ t \to 0 + 0 } \Vert \theta_B - \theta_0^B \Vert_{L^2(\Omega_B)} = 0.
\end{equation*}
By \eqref{eq4031}, we observe that
\begin{align*}
\Vert \theta_S (t) - \theta_0^S \Vert_{L^2(\Gamma)} = \Vert u_S (t) - u_0^S \Vert_{L^2(\Gamma)} \to 0 { \ }(\text{as }t \to 0 + 0).
\end{align*}
Thus, we conclude that ${ }^t(\theta_A , \theta_B , \theta_S)$ is a strong solution to \eqref{eq15} with initial data ${ }^t(\theta_0^A , \theta_0^B , \theta_0^S)$. From assertion $(\rm{ii})$ in Lemma \ref{lem44}, we find that ${ }^t(\theta_A , \theta_B , \theta_S)$ is a unique strong solution to \eqref{eq15} with ${ }^t(\theta_0^A , \theta_0^B , \theta_0^S)$. Therefore, we see $(\rm{i})$ and $(\rm{ii})$.

Next, we show $(\rm{iii})$. Assume that ${ }^t(v^\natural_A , v^\natural_B, v^\natural_S)$ is a strong solution to system \eqref{eq29} with initial data ${ }^t(v_0^A , v_0^B , v_0^S)$. For $t>0$ we set
\begin{align*}
\theta_A^\natural = \theta_A^\natural (x,t) &= \mathcal{P}_A [v_A^\natural (y,t) + {\rm{e}}^{- \delta y_3} v^\natural_S(y_h,t)](x,t),\\
\theta_B^\natural = \theta_B^\natural (x,t) &= \mathcal{P}_B [v_B^\natural (z,t) + {\rm{e}}^{ \delta z_3} v_S^\natural (z_h,t)](x,t),\\
\theta_S^\natural = \theta_S^\natural (x,t) &= \mathcal{P}_S [v_S^\natural (X,t)](x,t).
\end{align*}
From assertions $(\rm{i})$ and $(\rm{ii})$, we see that ${ }^t(\theta_A^\natural , \theta_B^\natural , \theta_S^\natural)$ is a unique strong solution to \eqref{eq15} with initial data ${ }^t(\theta_0^A , \theta_0^B , \theta_0^S)$. Since ${ }^t(\theta_A , \theta_B , \theta_S)$ are ${ }^t(\theta_A^\natural , \theta_B^\natural , \theta_S^\natural)$ is are two strong solutions to \eqref{eq15} with initial data ${ }^t(\theta_0^A , \theta_0^B , \theta_0^S)$, it follow from assertion $(\rm{ii})$ in Lemma \ref{lem44} to see that for all $t \in [0,T]$,
\begin{align}
\Vert \theta_A (t) - \theta_A^\natural (t) \Vert_{L^2(\Omega_A)} = 0,\label{eq4034}\\
\Vert \theta_B (t) - \theta_B^\natural (t) \Vert_{L^2(\Omega_B)} = 0,\label{eq4035}\\
\Vert \theta_S (t) - \theta_S^\natural (t) \Vert_{L^2(\Gamma)} = 0.\label{eq4036}
\end{align}
From the fact that $G_S = 1 + \varphi_1^2 + \varphi_2^2 \geq 1$ and \eqref{eq4036}, we see that for $t \in (0,T]$
\begin{align*}
0 = \Vert \theta_S (t) - \theta_S^\natural (t) \Vert_{L^2(\Gamma)} & = \Vert \mathcal{P}_S[ v_S ](t) - \mathcal{P}_S[v_S^\natural] (t) \Vert_{L^2(\Gamma)}\\
& = \Vert (v_S (t) - v_S^\natural (t)) G_S^{1/4} \Vert_{L^2(\mathbb{R}^2)}\\
& \geq \Vert v_S (t) - v_S^\natural (t) \Vert_{L^2(\mathbb{R}^2)}.
\end{align*}
This implies that $v_S (t) = v_S^\natural (t)$ for $t \in (0,T]$. Using Lemma \ref{lem37}, \eqref{eq4034}, and $v_S (t) = v_S^\natural (t)$ on $(0,T]$, we see that
\begin{multline*}
\Vert v_A (t) - v_A^\natural (t) \Vert_{L^2(\mathbb{R}^3_+)}  = \Vert \mathcal{P}_A [v_A (y,t) - v_A^\natural (y,t)] \Vert_{L^2(\Omega_A)}\\
 \leq  \Vert \mathcal{P}_A [v_A (y,t) + {\rm{e}}^{- \delta y_3} v_S(y_h,t)] - \mathcal{P}_A[ v_A^\natural (y,t) + {\rm{e}}^{- \delta y_3} v^\natural_S(y_h,t)] \Vert_{L^2(\Omega_A)}\\
 + \Vert \mathcal{P}_A [ {\rm{e}}^{- \delta y_3} v_S(y_h,t) - {\rm{e}}^{- \delta y_3} v^\natural_S(y_h,t)] \Vert_{L^2(\Omega_A)}\\
 = \Vert \theta_A (t) - \theta_A^\natural (t) \Vert_{L^2(\Omega_A} + \Vert {\rm{e}}^{- \delta y_3} v_S(y_h,t) - {\rm{e}}^{- \delta y_3} v^\natural_S(y_h,t) \Vert_{L^2(\mathbb{R}^3_+)}\\
 \leq 0 + C (\delta) \Vert v_S(t) - v^\natural_S(t) \Vert_{L^2(\mathbb{R}^2)} = 0.
\end{multline*}
Thus, we find that $v_A (t) = v_A^\natural (t)$ for $t \in (0,T]$. Similarly, we see that $v_B (t) = v_B^\natural (t)$ for $t \in (0,T]$. Therefore, we find that ${ }^t(v_A , v_B, v_S)$ is a unique strong solution to \eqref{eq29} with initial data ${ }^t(v_0^A , v_0^B , v_0^S)$. Therefore, Proposition \ref{prop43} is proved.
\end{proof}

\section{Existence of Strong Solutions}\label{sect5}

In this section, we construct strong solutions to system \eqref{eq2010} to show the existence of strong solutions to systems \eqref{eq29} and \eqref{eq15}. In subsection \ref{subsec51}, we study fundamental properties of our Laplace operators and the elliptic regularity theorem for these operators. In subsection \ref{subsec52}, we derive the key estimates necessary to show the existence of strong solutions to our systems. In subsection \ref{subsec53}, we first apply both nice properties of our Laplace operators (Proposition \ref{prop51}) and the key estimates (Proposition \ref{prop53}) to construct strong solutions to system \eqref{eq2010} to prove Theorem \ref{thm24}. Then, we show the existence of a unique global-in-time strong solution to system \eqref{eq15} to prove Theorem \ref{thm11}.

Let $\varphi \in BC^2(\mathbb{R}^2)$ and $n= n(x) = { }^t (n_1,n_2,n_3)$ be the unit outer normal vector at $x \in \Gamma$ defined by \eqref{eq11}. Let $ \rho_A , \rho_B , \rho_S, \kappa_A, \kappa_B ,\kappa_S >0$ and $\delta >0$. Let $\mathcal{A}$ be the operator defined by Section \ref{sect1} (see also subsection \ref{subsec51}). Throughout this section, we assume that
\begin{equation*}
\Vert \nabla_X \varphi \Vert_{L^\infty (\mathbb{R}^2)} \leq \frac{1}{2}.
\end{equation*}
To construct a strong solution to system \eqref{eq2010}, we use the function spaces $H$ and $X_T$. Define
\begin{equation*}
H := \{ \psi = { }^t(\psi_A , \psi_B , \psi_S) \in L^2(\mathbb{R}^3_+) \times L^2(\mathbb{R}^3_-) \times L^2(\mathbb{R}^2); {  \ } \Vert \psi \Vert_H < \infty \}
\end{equation*}
with
\begin{equation*}
\Vert \psi \Vert_H : = ( \Vert \psi_A \Vert_{L^2(\mathbb{R}^3_+)}^2 + \Vert \psi_B \Vert_{L^2(\mathbb{R}^3_-)}^2 + \Vert \psi_S \Vert_{L^2(\mathbb{R}^2)}^2 )^{1/2}.
\end{equation*}
For $\psi^\natural = { }^t(\psi_A^\natural, \psi_B^\natural, \psi_S^\natural), \psi^\flat = { }^t(\psi_A^\flat, \psi_B^\flat , \psi_S^\flat) \in H$,
\begin{equation*}
\dual{\psi^\natural , \psi^\flat}_H := \dual{\psi_A^\natural , \psi_A^\flat}_{L^2(\mathbb{R}^3_+)} + \dual{\psi_B^\natural , \psi_B^\flat}_{L^2(\mathbb{R}^3_-)} + \dual{\psi_S^\natural , \psi_S^\flat}_{L^2(\mathbb{R}^2)}.
\end{equation*}
We easily check that $\dual{ \cdot , \cdot}_H$ is an inner product on $H$, and that $H$ is a Hilbert space. For $T >0$, we set
\begin{equation*}
X_T := \{ w \in C([0,T] ; H)  ; { \ } \Vert w \Vert_{X_T} < \infty \}
\end{equation*}
with
\begin{align}
& \Vert w \Vert_{X_T}  := \sup_{0 \leq t \leq T} \Vert w \Vert_H + \Vert dw/{dt} \Vert_{L^2(0,T ; H)} + \Vert \mathcal{A} w \Vert_{L^2(0,T;H)},\label{eq51}\\
& \Vert w \Vert_{L^2(0,T;H)} := \bigg( \int_0^T \Vert w(t) \Vert_H^2 { \ }dt \bigg)^{1/2}. \notag
\end{align}
See subsection \ref{subsec51} for $\mathcal{A}$ and $D (\mathcal{A})$.

\subsection{Operator $\mathcal{A}$}\label{subsec51}

Let us recall and study our Laplace operators. Define the four operators $\mathcal{A}_1$, $\mathcal{A}_2$, $\mathcal{A}_3$, and $\mathcal{A}$ on $L^2(\mathbb{R}^3_+)$, $L^2(\mathbb{R}^3_-)$, $L^2(\mathbb{R}^2)$, and $H$ as follows:
\begin{multline*}
\begin{cases}
\mathcal{A}_1 \psi_A = - \kappa_A \Delta_y \psi_A,\\
D ( \mathcal{A}_1 ) = W_0^{1,2} (\mathbb{R}^3_+) \cap W^{2,2} (\mathbb{R}^3_+),
\end{cases}
\begin{cases}
\mathcal{A}_2 \psi_B = - \kappa_B \Delta_z \psi_B,\\
D ( \mathcal{A}_2 ) = W_0^{1,2} (\mathbb{R}^3_-) \cap W^{2,2} (\mathbb{R}^3_-),
\end{cases}\\
\begin{cases}
\mathcal{A}_3 \psi_S = - \kappa_S \Delta_X \psi_S,\\
D ( \mathcal{A}_3 ) = W^{2,2} (\mathbb{R}^2),
\end{cases}{ \ }
\begin{cases}
\mathcal{A} \psi  = { }^t (\mathcal{A}_1 \psi_A, \mathcal{A}_2 \psi_B, \mathcal{A}_3 \psi_S),\\
D ( \mathcal{A} ) = D (\mathcal{A}_1) \times  D (\mathcal{A}_2) \times D (\mathcal{A}_3) .
\end{cases}
\end{multline*}
From \cite{Paz83}, \cite{DHP01}, we see that each operator $- \mathcal{A}_1$, $- \mathcal{A}_2$, and $- \mathcal{A}_3$ generates a bounded analytic semigroup on $L^2(\mathbb{R}^3_+)$, $L^2(\mathbb{R}^3_-)$, and $L^2(\mathbb{R}^2)$, respectively. We also see that each operator $- \mathcal{A}_1$, $- \mathcal{A}_2$, and $- \mathcal{A}_3$ generates a contraction $C_0$-semigroup on $L^2(\mathbb{R}^3_+)$, $L^2(\mathbb{R}^3_-)$, and $L^2(\mathbb{R}^2)$, respectively. From \cite{Des64}, we see that if a linear operator on a Hilbert space generates a bounded analytic semigroup on the Hilbert space then the linear operator has maximal $L^p$-regularity. This leads to the conclusion that each operator $\mathcal{A}_1$, $\mathcal{A}_2$, and $\mathcal{A}_3$ has maximal $L^p$-regularity (see also \cite{DHP01}, \cite{DHP03} and \cite{KW04}). Let ${\rm{e}}^{ - t \mathcal{A}_1}$, ${\rm{e}}^{ - t \mathcal{A}_2}$, and ${\rm{e}}^{ - t \mathcal{A}_3}$ be the three semigroups generated by operators $\mathcal{A}_1$, $\mathcal{A}_2$, and $\mathcal{A}_3$, respectively. It is well-known that for $\psi_A \in L^2(\mathbb{R}^3_+)$, $\psi_B \in L^2(\mathbb{R}^3_-)$, $\psi_S \in L^2(\mathbb{R}^2)$,
\begin{equation*}
{\rm{e}}^{ - t \mathcal{A}_1} \psi_A = \mathcal{G}_A *\psi_A{ \ },{\rm{e}}^{ - t \mathcal{A}_2} \psi_B = \mathcal{G}_B *\psi_B,{ \ }{\rm{e}}^{ - t \mathcal{A}_3} \psi_S = \mathcal{G}_S *\psi_S,
\end{equation*}
where $\mathcal{G}_A$, $\mathcal{G}_B$ denote two heat kernels for two half spaces $\mathbb{R}^3_+ , \mathbb{R}^3_-$ (see \cite{Uka87} for the representation formula of $\mathcal{G}_A$ and $\mathcal{G}_B$) and $\mathcal{G}_S$ is the heat kernel for the whole space $\mathbb{R}^2$ defined by
\begin{equation*}
\mathcal{G}_S = \mathcal{G}_S(X, t) = \frac{1}{4 \pi t} {\rm{e}}^{- \frac{\vert X \vert^2}{4 t \kappa_S} }.
\end{equation*}
From \cite[Chapter II.3]{Soh01}, we see that each operator $\mathcal{A}_1$, $\mathcal{A}_2$, and $\mathcal{A}_3$ is a positive selfadjoint operator. We also see that
\begin{equation*}
D(\mathcal{A}_1^{1/2}) = W_0^{1,2} (\mathbb{R}^3_+), { \ }D(\mathcal{A}_2^{1/2}) = W_0^{1,2} (\mathbb{R}^3_-), { \ }D(\mathcal{A}_3^{1/2}) = W^{1,2} (\mathbb{R}^2),
\end{equation*}
and that for each $\psi_A \in D (\mathcal{A}_1^{1/2})$, $\psi_B \in D (\mathcal{A}_2^{1/2})$, $\psi_S \in D (\mathcal{A}_3^{1/2})$,
\begin{multline*}
\Vert \mathcal{A}_1^{1/2} \psi_A \Vert_{L^2(\mathbb{R}^3_+)} = \sqrt{\kappa_A} \Vert \nabla_y \psi_A \Vert_{L^2(\mathbb{R}^3_+)},{ \ }\Vert \mathcal{A}_2^{1/2} \psi_B \Vert_{L^2(\mathbb{R}^3_-)} = \sqrt{\kappa_B} \Vert \nabla_z \psi_B \Vert_{L^2(\mathbb{R}^3_+)},\\
\Vert \mathcal{A}_3^{1/2} \psi_S \Vert_{L^2(\mathbb{R}^2)} = \sqrt{\kappa_S} \Vert \nabla_X \psi_S \Vert_{L^2(\mathbb{R}^2)}.
\end{multline*} 
We define $\mathcal{A}^{1/2}$ as follows:
\begin{equation*}
\begin{cases}
\mathcal{A}^{1/2} \psi = { }^t (\mathcal{A}_1^{1/2} \psi_A , \mathcal{A}_2^{1/2} \psi_B , \mathcal{A}_3^{1/2} \psi_S),\\
D (\mathcal{A}^{1/2}) = D (\mathcal{A}_1^{1/2}) \times D (\mathcal{A}_2^{1/2}) \times D (\mathcal{A}_3^{1/2}),
\end{cases}
\end{equation*}
that is, for $\psi = { }^t(\psi_A , \psi_B , \psi_S) \in D (\mathcal{A}^{1/2})$
\begin{equation*}
\Vert \mathcal{A}^{1/2} \psi \Vert_H = ( \kappa_A \Vert \nabla_y \psi_A \Vert_{L^2(\mathbb{R}^3_+)}^2 +  \kappa_B \Vert \nabla_z \psi_B \Vert_{L^2(\mathbb{R}^3_-)}^2 + \kappa_S \Vert \nabla_X \psi_S \Vert_{L^2(\mathbb{R}^2)}^2  )^{1/2} .
\end{equation*}
Therefore, we see that $\mathcal{A}$ is a positive selfadjoint operator, that $\mathcal{A}$ is a closed operator, and that $\mathcal{A}$ has the following properties.
\begin{proposition}\label{prop51}
Let $V_0 = { }^t (V_0^A , V_0^B , V_0^S) \in H$ and $\mathscr{F} \in L^2(0,T;H) \cap C^{\eta}_{loc} ((0,T); H)$ for some $0 < \eta <1$ and $T \in (0,\infty)$. Then system
\begin{equation*}
\begin{cases}
\frac{d}{dt} V + \mathcal{A} V = \mathscr{F} \text{ on } (0,T),\\
V \vert_{t =0} = V_0,
\end{cases}
\end{equation*}
admits a unique strong solution $V = { }^t (V_A, V_B ,V_S)$ in 
\begin{equation*}
C ([0, T] ; H ) \cap C((0,T);D (\mathcal{A})) \cap C^1 ((0,T);H).
\end{equation*}
Moreover, $v$ satisfies the following properties:\\
$(\rm{i})$ {\rm{[Representation formula]}} For all $0 < t \leq T$, $V(t)$ is written by
\begin{equation}\label{eq52}
V (t) = {\rm{e}}^{- t \mathcal{A}} V_0 + \int_0^t {\rm{e}}^{- ( t - \tau ) \mathcal{A}} \mathscr{F} ( \tau ) { \ }d \tau.
\end{equation}
Here ${\rm{e}}^{ - t \mathcal{A}}$ is the semigroup generated by operator $\mathcal{A}$.\\ 
$(\rm{ii})$ {\rm{[Initial condition]}}
\begin{equation*}
\lim_{t \to 0 +0} \Vert V(t) - V_0 \Vert_H = 0.
\end{equation*}
$(\rm{iii})$ {\rm{[H\"{o}lder continuity]}}
\begin{equation*}
V, dV/{dt}, \mathcal{A} V \in C_{loc}^{\eta} ((0,T); H).
\end{equation*}
$(\rm{iv})$ {\rm{[Contraction semigroup]}} For all $\psi \in H$,
\begin{equation}\label{eq53}
\Vert {\rm{e}}^{-t \mathcal{A}} \psi \Vert_H \leq \Vert \psi \Vert_H,
\end{equation}
and
\begin{equation}\label{eq54}
\sup_{0 \leq t \leq T} \Vert V (t) \Vert_{H} \leq \Vert V_0 \Vert_H + T^{1/2} \Vert \mathscr{F} \Vert_{L^2(0,T;H)}. 
\end{equation}
$(\rm{v})$ {\rm{[Heat semigroups and kernels]}}
For $\psi = { }^t (\psi_A , \psi_B , \psi_S) \in H$,
\begin{equation*}
 {\rm{e}}^{-t \mathcal{A}} \psi = { }^t ( {\rm{e}}^{-t \mathcal{A}_1} \psi_A,  {\rm{e}}^{-t \mathcal{A}_2} \psi_B , {\rm{e}}^{-t \mathcal{A}_3} \psi_S ) = { }^t (\mathcal{G}_A * \psi_A , \mathcal{G}_B * \psi_B , \mathcal{G}_S * \psi_S) =: \mathcal{G} * \psi.
\end{equation*}
$(\rm{vi})$ {\rm{[Maximal regularity]}} Assume in addition that $V_0 = { }^t (V_0^A , V_0^B , V_0^S) \in D (\mathcal{A}^{1/2})$. Then there is $C_\star = C_\star (\kappa_A , \kappa_B , \kappa_S) >0$ such that
\begin{equation}\label{eq55}
\bigg\Vert \frac{d}{dt} V \bigg\Vert_{L^2(0,T; H)} +  \Vert \mathcal{A} V \Vert_{L^2(0,T; H)} \leq 2 \Vert \mathcal{A}^{1/2} V_0 \Vert_H + C_\star \Vert \mathscr{F} \Vert_{L^2(0,T;H)}. 
\end{equation}
Moreover,
\begin{equation}\label{eq56}
\Vert V \Vert_{X_T} \leq \Vert V_0 \Vert_H + 2 \Vert \mathcal{A}^{1/2} V_0 \Vert_H + (T^{1/2} + C_\star ) \Vert \mathscr{F} \Vert_{L^2(0,T;H)}.
\end{equation}
Here $\Vert \cdot \Vert_{X_T}$ is defined by \eqref{eq51}.
\end{proposition}

\noindent Using the same arguments as in \cite[Section 4.3]{Paz83} and \cite[Section 4.3.]{Lun95}, we derive $(\rm{iii})$. Applying the H\"{o}lder inequality and \eqref{eq53} into \eqref{eq52}, we have \eqref{eq54}. By \eqref{eq51}, \eqref{eq54}, and \eqref{eq55}, we obtain \eqref{eq56}. See also \cite[Section 3]{Kob26}.

Next we introduce elliptic regularity theorem for our Laplace operators. From \cite[Chapters 8,9]{GT98} and an interpolation theory, we have the lemma.
\begin{lemma}\label{lem52}
$(\rm{i})$ For each $\zeta >0$ there is $C(\zeta)>0$ independent of $(\kappa_A,\kappa_B, \kappa_S)$ such that for all $\psi_A \in D (\mathcal{A}_1)$, $\psi_B \in D (\mathcal{A}_2)$, and $\psi_S \in D (\mathcal{A}_3)$,
\begin{align}
\Vert \nabla_y \psi_A \Vert_{L^2(\mathbb{R}^3_+)} & \leq \frac{\zeta}{\kappa_A} \Vert \mathcal{A}_1 \psi_A \Vert_{L^2 (\mathbb{R}^3_+)} + C (\zeta) \Vert \psi_A \Vert_{L^2(\mathbb{R}^3_+)},\label{eq57}\\
\Vert \nabla_z \psi_B \Vert_{L^2(\mathbb{R}^3_-)} & \leq \frac{\zeta}{\kappa_B} \Vert \mathcal{A}_2 \psi_B \Vert_{L^2 (\mathbb{R}^3_-)} + C(\zeta) \Vert \psi_B \Vert_{L^2(\mathbb{R}^3_-)},\label{eq58}\\
\Vert \nabla_X \psi_S \Vert_{L^2(\mathbb{R}^2)} & \leq \frac{\zeta}{\kappa_S} \Vert \mathcal{A}_3 \psi_S \Vert_{L^2 (\mathbb{R}^2)} + C(\zeta) \Vert \psi_S \Vert_{L^2(\mathbb{R}^2)}.\label{eq59}
\end{align}
$(\rm{ii})$ There is $C_\triangle>0$ independent of $(\kappa_A,\kappa_B, \kappa_S)$ such that for all $\psi_A \in D (\mathcal{A}_1)$, $\psi_B \in D (\mathcal{A}_2)$, $\psi_S \in D (\mathcal{A}_3)$,
\begin{align}
\Vert \nabla_y \psi_A \Vert_{L^2(\mathbb{R}^3_+)} + \Vert \nabla_y^2 \psi_A \Vert_{L^2(\mathbb{R}^3_+)} & \leq \frac{C_\triangle}{\kappa_A} \Vert \mathcal{A}_1 \psi_A \Vert_{L^2 (\mathbb{R}^3_+)} + C_\triangle \Vert \psi_A \Vert_{L^2(\mathbb{R}^3_+)},\label{eq5010}\\
\sup_{j \in \{ 1,2,3 \} }\Vert \gamma_+ [\partial_{y_j} \psi_A] \Vert_{L^2(\mathbb{R}^2)} & \leq \frac{C_\triangle}{\kappa_A} \Vert \mathcal{A}_1 \psi_A \Vert_{L^2 (\mathbb{R}^3_+)} + C_\triangle \Vert \psi_A \Vert_{L^2(\mathbb{R}^3_+)},\label{eq5011}\\
\Vert \nabla_z \psi_B \Vert_{L^2(\mathbb{R}^3_-)} + \Vert \nabla_z^2 \psi_B \Vert_{L^2(\mathbb{R}^3_-)} & \leq \frac{C_\triangle}{\kappa_B} \Vert \mathcal{A}_2 \psi_B \Vert_{L^2 (\mathbb{R}^3_-)} + C_\triangle \Vert \psi_B \Vert_{L^2(\mathbb{R}^3_-)},\label{eq5012}\\
\sup_{j \in \{ 1,2,3 \} } \Vert \gamma_- [\partial_{z_j} \psi_B] \Vert_{L^2(\mathbb{R}^2)} & \leq \frac{C_\triangle}{\kappa_B} \Vert \mathcal{A}_2 \psi_B \Vert_{L^2 (\mathbb{R}^3_-)} + C_\triangle \Vert \psi_B \Vert_{L^2(\mathbb{R}^3_-)},\label{eq5013}\\
\Vert \nabla_X \psi_S \Vert_{L^2(\mathbb{R}^2)} + \Vert \nabla_X^2 \psi_S \Vert_{L^2(\mathbb{R}^2)} & \leq \frac{C_\triangle}{\kappa_S} \Vert \mathcal{A}_3 \psi_S \Vert_{L^2 (\mathbb{R}^2)} + C_\triangle \Vert \psi_S \Vert_{L^2(\mathbb{R}^2)}.\label{eq5014}
\end{align}
Here $\gamma_+$ and $\gamma_-$ are the two trace operators such that $\gamma_+: W^{1,2}(\mathbb{R}^3_+) \to L^2(\mathbb{R}^2)$ and $\gamma_-: W^{1,2}(\mathbb{R}^3_-) \to L^2(\mathbb{R}^2)$.\\

\end{lemma}

\begin{proof}[Proof of Lemma \ref{lem52}]
We only derive \eqref{eq57}, \eqref{eq5010}, and \eqref{eq5011}. Fix $\psi_A \in D (\mathcal{A}_1)$.

We first show \eqref{eq57}. Fix $\zeta >0$. Using integration by parts the Cauchy-Shwarz inequality, we check that
\begin{align*}
\Vert \nabla_y \psi_A \Vert_{L^2 (\mathbb{R}^3_+)}^2 &= \vert \dual{\nabla_y \psi_A , \nabla_y \psi_A}_{L^2 (\mathbb{R}^3_+)} \vert = \vert \dual{ -\Delta_y \psi_A , \psi_A}_{L^2(\mathbb{R}^3_+)} \vert\\
& \leq \Vert - \Delta_y \psi_A \Vert_{L^2(\mathbb{R}^3_+)} \Vert \psi_A \Vert_{L^2(\mathbb{R}^3_+)}\\
& \leq \zeta^2 \Vert -\Delta_y \psi_A \Vert_{L^2(\mathbb{R}^3_+)}^2 + (1/4\zeta^2) \Vert \psi_A \Vert_{L^2(\mathbb{R}^3_+)}^2.
\end{align*}
This shows that
\begin{align*}
\Vert \nabla_y \psi_A \Vert_{L^2 (\mathbb{R}^3_+)} & \leq \frac{\zeta}{\kappa_A} \Vert - \kappa_A \Delta_y \psi_A \Vert_{L^2(\mathbb{R}^3_+)} + \frac{1}{2\zeta} \Vert \psi_A \Vert_{L^2(\mathbb{R}^3_+)}\\
& = \frac{\zeta}{\kappa_A} \Vert \mathcal{A}_1 \psi_A \Vert_{L^2(\mathbb{R}^3_+)} + \frac{1}{2\zeta} \Vert \psi_A\Vert_{L^2(\mathbb{R}^3_+)}.
\end{align*}
Therefore, we have \eqref{eq57}.

Next we derive \eqref{eq5010}. Set $\mathcal{A}_1^* \psi_A = (1 - \Delta_y)\psi_A$ and $D (\mathcal{A}_1^*) = W_0^{1,2} (\mathbb{R}^3_+) \cap W^{2,2} (\mathbb{R}^3_+) (= D (\mathcal{A}_1))$. Using the elliptic regularity theorem for operator $\mathcal{A}_1^*$, we find that there is $C>0$ such that
\begin{align*}
\Vert \nabla_y^2 \psi_A \Vert_{L^2(\mathbb{R}^3_+)} &\leq C \Vert (1 - \Delta_y ) \psi_A \Vert_{L^2(\mathbb{R}^3_+)}\\
& \leq \frac{C}{\kappa_A} \Vert - \kappa_A \Delta_y  \psi_A \Vert_{L^2(\mathbb{R}^3_+)} + C \Vert \psi_A \Vert_{L^2(\mathbb{R}^3_+)}\\
& \leq \frac{C}{\kappa_A}\Vert \mathcal{A}_1 \psi_A \Vert_{L^2(\mathbb{R}^3_+)} + C \Vert \psi_A \Vert_{L^2(\mathbb{R}^3_+)}.
\end{align*}
Combining the above inequalities and \eqref{eq57} with $\zeta =1$, we have \eqref{eq5010}.

Finally, we deduce \eqref{eq5011}. Using fundamental properties of trace operator $\gamma_+$ with \eqref{eq57} and \eqref{eq5010}, we find that
\begin{align*}
\Vert \gamma_+[ \partial_{y_j} \psi_A ] \Vert_{L^2(\mathbb{R}^2)} &\leq 2 \Vert \nabla_y \psi_A \Vert_{W^{1,2}(\mathbb{R}^3_+)}\\
& \leq 2 \Vert \nabla_y^2  \psi_A \Vert_{L^2(\mathbb{R}^3_+)} + 2 \Vert \nabla_y \psi_A \Vert_{L^2(\mathbb{R}^3_+)}\\
& \leq \frac{C}{\kappa_A}\Vert \mathcal{A}_1 \psi_A \Vert_{L^2(\mathbb{R}^3_+)} + C \Vert \psi_A \Vert_{L^2(\mathbb{R}^3_+)}.
\end{align*}
Thus, we obtain \eqref{eq5011}. Therefore, the lemma follows.
\end{proof}

\subsection{Properties of $\mathcal{B}$ and $\mathcal{F}$}\label{subsec52}

In this subsection, we investigate properties of $\mathcal{B}w$ and $\mathcal{F}(w)$ by using Proposition \ref{prop51} and Lemma \ref{lem52}. Let us first recall and redefine $\mathcal{B}w$ and $\mathcal{F}(w)$. Let $T>0$. For all $w = { }^t (w_A , w_B, w_S) \in X_T$, we define $\mathcal{B} w := { }^t (\mathcal{B}_1 w_A , \mathcal{B}_2 w_B, \mathcal{B}_3 w_S)$ and $\mathcal{F} (w) := { }^t (\mathcal{F}_1 ( w) , \mathcal{F}_2 (w), \mathcal{F}_3 (w))$, where 
\begin{multline}\label{eq5015}
\mathcal{B}_1 w_A\\
 := \kappa_A \bigg( 2 \varphi_1 \frac{\partial^2 w_A}{\partial y_1 \partial y_3} + 2 \varphi_2 \frac{\partial^2 w_A}{\partial y_2 \partial y_3} - (\varphi_1^2 + \varphi_2^2) \frac{\partial^2 w_A}{\partial y_3^2 } + (\varphi_{11} +\varphi_{22}) \frac{\partial w_A}{\partial y_3} \bigg),
\end{multline}
\begin{multline}\label{eq5016}
\mathcal{B}_2 w_B\\
 := \kappa_B \bigg(  2 \varphi_1 \frac{\partial^2 w_B}{\partial z_1 \partial z_3} + 2 \varphi_2 \frac{\partial^2 w_B}{\partial z_2 \partial z_3} - (\varphi_1^2 + \varphi_2^2) \frac{\partial^2 w_B}{\partial z_3^2 } + ( \varphi_{11} + \varphi_{22} ) \frac{\partial w_B}{\partial z_3} \bigg),
\end{multline}
\begin{multline}\label{eq5017}
\mathcal{B}_3 w_S := \kappa_S \bigg( \frac{\varphi_1^2}{G_S} \frac{\partial^2 w_S}{\partial X_1^2} + \frac{\varphi_2^2}{G_S}\frac{\partial^2 w_S}{\partial X_2^2} + \frac{2 \varphi_1 \varphi_2}{G_S}\frac{\partial^2 w_S}{\partial X_1 \partial X_2}\\
+ \frac{ \varphi_1 ( \varphi_{11} + \varphi_{22} + \varphi_2^2 \varphi_{11} +  \varphi_1^2 \varphi_{22} - 2 \varphi_1 \varphi_2 \varphi_{12})}{G^2_S} \frac{\partial w_S}{\partial X_1}\\
+ \frac{\varphi_2 ( \varphi_{11} + \varphi_{22} + \varphi_2^2 \varphi_{11} +  \varphi_1^2 \varphi_{22} - 2 \varphi_1 \varphi_2 \varphi_{12}) }{G^2_S}  \frac{\partial w_S}{\partial X_2} \bigg),
\end{multline}
\begin{multline}\label{eq5018}
\mathcal{F}_1 (w)  := - ( d w_S/{dt}) {\rm{e}}^{ - \delta y_3} + \kappa_A  (\Delta_{y_h} w_S) {\rm{e}}^{ - \delta y_3}\\ +  \kappa_A w_S \{ \delta^2 + \delta \varphi_{11} + \delta \varphi_{22} + (\delta \varphi_1)^2 + (\delta \varphi_2)^2 \}  {\rm{e}}^{ - \delta y_3}\\
+ 2 \kappa_A (\delta \varphi_1 \partial_{y_1} v_S + \delta \varphi_2 \partial_{y_2} v_S) {\rm{e}}^{ - \delta y_3},
\end{multline}
\begin{multline}\label{eq5019}
\mathcal{F}_2 (w)  := - (d w_S/{dt}) {\rm{e}}^{ \delta z_3} + \kappa_B  (\Delta_{z_h} w_S) {\rm{e}}^{ \delta z_3}\\ +  \kappa_B w_S \{ \delta^2 - \delta \varphi_{11} - \delta \varphi_{22} + (\delta \varphi_1)^2 + (\delta \varphi_2)^2 \}  {\rm{e}}^{ \delta z_3}\\
- 2 \kappa_B(\delta \varphi_1 \partial_{z_1} w_S + \delta \varphi_2 \partial_{z_2} w_S) {\rm{e}}^{ \delta z_3},
\end{multline}
and
\begin{multline}\label{eq5020}
\mathcal{F}_3 (w) := \frac{\rho_A \kappa_A}{\rho_S} \gamma_+ \bigg[   - \frac{\varphi_1}{\sqrt{G_S}} \frac{\partial w_A }{\partial y_1} - \frac{\varphi_2}{\sqrt{G_S}} \frac{\partial w_A }{\partial y_2} +  \frac{1 + \varphi_1^2 + \varphi_2^2 }{\sqrt{G_S}} \frac{\partial w_A }{\partial y_3}  \bigg]\\
 - \frac{\rho_B \kappa_B}{\rho_S} \gamma_- \bigg[  - \frac{\varphi_1}{\sqrt{G_S}} \frac{\partial w_B }{\partial z_1} - \frac{\varphi_2}{\sqrt{G_S}} \frac{\partial w_B }{\partial z_2} +  \frac{1 + \varphi_1^2 + \varphi_2^2 }{\sqrt{G_S}} \frac{\partial w_B }{\partial z_3}  \bigg]\\
- \bigg( \frac{ \rho_A \kappa_A -  \rho_B \kappa_B  }{\rho_S} \bigg) \bigg( \frac{\varphi_1}{\sqrt{G_S}} \frac{\partial w_S}{\partial X_1} + \frac{\varphi_2}{\sqrt{G_S}} \frac{\partial w_S}{\partial X_2} \bigg)\\
- \bigg( \frac{ \rho_A \kappa_A + \rho_B \kappa_B  }{\rho_S} \bigg) \bigg( \frac{\delta ( 1 + \varphi_1^2 + \varphi_2^2 )}{\sqrt{G_S}} w_S \bigg).
\end{multline}
Here $\varphi_\alpha = \partial \varphi/{\partial X_\alpha}$, $\varphi_{\alpha \beta} = \partial^2 \varphi/{\partial X_\alpha \partial X_\beta}$, $G_S = 1 + \varphi_1^2 + \varphi_2^2$ and $\gamma_+$, $\gamma_-$ are the two trace operators such that
\begin{align*}
\gamma_+[\cdot] : W^{1,2} (\mathbb{R}^3_+) \to L^2 (\mathbb{R}^2)(= L^2 (\partial \mathbb{R}^3_+)),\\
\gamma_-[\cdot] : W^{1,2} (\mathbb{R}^3_-) \to L^2 (\mathbb{R}^2)(= L^2(\partial \mathbb{R}^3_-)).
\end{align*}
See Sections \ref{sect2} and \ref{sect7} for $\mathcal{B}$ and $\mathcal{F}$.

The aim of this section is to prove Proposition \ref{prop53}.
\begin{proposition}[Properties of $\mathcal{F}$ and $\mathcal{B}$]\label{prop53}
$(\rm{i})$ Let $\epsilon >0$. Then there are $C_\triangledown > 0$ and $C_\Box = C_\Box (\epsilon , \delta, \Vert \nabla_X^2 \varphi \Vert_{L^\infty (\mathbb{R}^2)} ,  \rho_A, \rho_B,\rho_S, \kappa_A , \kappa_B,\kappa_S  ) >0$ such that for all $w \in X_T$
\begin{multline}\label{eq5021}
\Vert \mathcal{F} (w) - \mathcal{B} w \Vert_{L^2(0,T;H)}\\
\leq C_\triangledown \bigg(  \epsilon + \frac{1}{\sqrt{\delta }} \frac{\kappa_A + \kappa_B + \kappa_S}{\kappa_S} + \Vert \nabla_X \varphi \Vert_{L^\infty (\mathbb{R}^2)} + \frac{\rho_A + \rho_B}{\rho_S} +  C_\Box T^{1/2} \bigg) \Vert w \Vert_{X_T}.
\end{multline}
$(\rm{ii})$ Let $0 < \eta < 1$. Assume that $w \in X_T$ and that
\begin{equation}\label{eq5022}
w, \mathcal{A}w, dw/{dt} \in C_{loc}^\eta ((0,T) ; H).
\end{equation} 
Then
\begin{equation*}
F (w),  \mathcal{B}w,F (w) - \mathcal{B}w \in C_{loc}^\eta ((0,T) ; H).
\end{equation*} 
\end{proposition}
Let us prove Proposition \ref{prop53} by applying Proposition \ref{prop51} and Lemma \ref{lem52}.

\begin{proof}[Proof of Proposition \ref{prop53}]
We first show $(\rm{i})$. Fix $w = { }^t (w_A, w_B , w_S) \in X_T$ and $\epsilon >0$. Since
\begin{multline}\label{eq5023}
\Vert \mathcal{B} w \Vert_{L^2(0,T;H)}\\
 \leq \Vert \mathcal{B}_1 w_A \Vert_{L^2(0,T;L^2(\mathbb{R}^3_+))} + \Vert \mathcal{B}_2 w_B \Vert_{L^2(0,T;L^2(\mathbb{R}^3_-))} + \Vert \mathcal{B}_3 w_S \Vert_{L^2(0,T;L^2(\mathbb{R}^2))}
\end{multline}
and
\begin{multline}\label{eq5024}
\Vert \mathcal{F} (w) \Vert_{L^2(0,T;H)}\\
 \leq \Vert \mathcal{F}_1 (w) \Vert_{L^2(0,T;L^2(\mathbb{R}^3_+))} + \Vert \mathcal{F}_2 (w) \Vert_{L^2(0,T;L^2(\mathbb{R}^3_-))} + \Vert \mathcal{F}_3 (w) \Vert_{L^2(0,T;L^2(\mathbb{R}^2))},
\end{multline}
we consider $\Vert \mathcal{B}_1 w_A \Vert_{L^2(0,T;L^2(\mathbb{R}^3_+))}$, $\Vert \mathcal{B}_2 w_B \Vert_{L^2(0,T;L^2(\mathbb{R}^3_-))}$, $\Vert \mathcal{B}_3 w_S \Vert_{L^2(0,T;L^2(\mathbb{R}^2))}$,\\
 $\Vert \mathcal{F}_1 (w) \Vert_{L^2(0,T;L^2(\mathbb{R}^3_+))}$, $\Vert \mathcal{F}_2 (w) \Vert_{L^2(0,T;L^2(\mathbb{R}^3_-))}$, and $\Vert \mathcal{F}_3 (w) \Vert_{L^2(0,T;L^2(\mathbb{R}^2))}$.

Now we study $\Vert \mathcal{B} w \Vert_{L^2(0,T;H)}$. Applying the H\"{o}lder inequality, \eqref{eq5010}, and \eqref{eq57} with $\zeta = \epsilon/(2 \Vert \nabla_X^2 \varphi \Vert_{L^\infty (\mathbb{R}^2)} )$ into \eqref{eq5015}, we check that
\begin{multline}\label{eq5025}
\Vert \mathcal{B}_1 w_A \Vert_{L^2(\mathbb{R}^3_+) }
 \leq 6 \kappa_A \Vert \nabla_X \varphi \Vert_{L^\infty (\mathbb{R}^2)} \Vert \nabla_y^2 w_A \Vert_{L^2(\mathbb{R}^3_+)} + 2 \kappa_A \Vert \nabla_X^2 \varphi \Vert_{L^\infty (\mathbb{R}^2)} \Vert \nabla_y w_A \Vert_{L^2(\mathbb{R}^3_+)}\\
\leq ( 6 C_\triangle \Vert \nabla_X \varphi \Vert_{L^\infty (\mathbb{R}^2)} + \epsilon ) \Vert \mathcal{A}_1 w_A \Vert_{L^2(\mathbb{R}^3_+)} + C( \epsilon , \Vert \nabla_X^2 \varphi \Vert_{L^\infty(\mathbb{R}^2)} ) \Vert w_A \Vert_{L^2(\mathbb{R}^3_+)}.
\end{multline}
Here $C_\triangle$ is the positive constant appearing in Lemma \ref{lem52}. Since
\begin{equation*}
\Vert w_A \Vert_{L^2(0,T ;L^2(\mathbb{R}^3_+)) } \leq \sup_{0 \leq t \leq T} \Vert w_A ( t ) \Vert_{L^2(\mathbb{R}^3_+)} T^{1/2} ,
\end{equation*}
it follows from \eqref{eq5025} to see that
\begin{multline}\label{eq5026}
\Vert \mathcal{B}_1 w_A \Vert_{L^2(0,T; L^2(\mathbb{R}^3_+) ) }\\
 \leq C ( \Vert \nabla_X \varphi \Vert_{L^\infty (\mathbb{R}^2)} + \epsilon ) \Vert \mathcal{A}_1 w_A \Vert_{L^2(0,T;L^2(\mathbb{R}^3_+))} + C(\epsilon , \Vert \nabla_X^2 \varphi \Vert_{L^\infty(\mathbb{R}^2)}) \Vert w_A \Vert_{L^2(0,T;L^2(\mathbb{R}^3_+))}\\
 \leq C ( \Vert \nabla_X \varphi \Vert_{L^\infty (\mathbb{R}^2)} + \epsilon ) \Vert w \Vert_{X_T} + T^{1/2}C( \epsilon , \Vert \nabla_X^2 \varphi \Vert_{L^\infty(\mathbb{R}^2)}) \Vert w \Vert_{X_T}.
\end{multline}
Similarly, we have
\begin{multline}\label{eq5027}
\Vert \mathcal{B}_2 w_B \Vert_{L^2(\mathbb{R}^3_-) }\\
 \leq ( 6 C_\triangle \Vert \nabla_X \varphi \Vert_{L^\infty (\mathbb{R}^2)} + \epsilon ) \Vert \mathcal{A}_2 w_B \Vert_{L^2(\mathbb{R}^3_-)} + C( \epsilon , \Vert \nabla_X^2 \varphi \Vert_{L^\infty(\mathbb{R}^2)}) \Vert w_B \Vert_{L^2(\mathbb{R}^3_-)}
\end{multline}
and
\begin{multline}\label{eq5028}
\Vert \mathcal{B}_2 w_B \Vert_{L^2(0,T; L^2(\mathbb{R}^3_-) ) }\\
 \leq C( \Vert \nabla_X \varphi \Vert_{L^\infty (\mathbb{R}^2)} + \epsilon ) \Vert w \Vert_{X_T} + T^{1/2}C( \epsilon , \Vert \nabla_X^2 \varphi \Vert_{L^\infty(\mathbb{R}^2)}) \Vert w \Vert_{X_T}.
\end{multline}
From $G_S =  1 + \varphi_1^2 + \varphi_2^2$ and $\Vert \nabla_X \varphi \Vert_{L^2(\mathbb{R}^2)} \leq 1/2$, we find that
\begin{equation*}
\bigg\Vert \frac{\varphi_1^2}{G_S} \bigg\Vert_{L^\infty (\mathbb{R}^2)} + \bigg\Vert \frac{\varphi_2^2}{G_S} \bigg\Vert_{L^\infty (\mathbb{R}^2)} + \bigg\Vert \frac{\varphi_1 \varphi_2}{G_S} \bigg\Vert_{L^\infty (\mathbb{R}^2)} \leq  3 \Vert \nabla_X \varphi \Vert_{L^\infty (\mathbb{R}^2)},
\end{equation*}
\begin{equation*}
\bigg\Vert \frac{ \varphi_1 ( \varphi_{11} + \varphi_{22} + \varphi_2^2 \varphi_{11} + \varphi_1^2 \varphi_{22} - 2 \varphi_1 \varphi_2 \varphi_{12} )  }{G^2_S} \bigg\Vert_{L^\infty (\mathbb{R}^2)} \leq 12 \Vert \nabla_X^2 \varphi \Vert_{L^\infty (\mathbb{R}^2)},
\end{equation*}
and that
\begin{equation*}
\bigg\Vert \frac{ \varphi_2 ( \varphi_{11} + \varphi_{22} + \varphi_2^2 \varphi_{11} + \varphi_1^2 \varphi_{22} - 2 \varphi_1 \varphi_2 \varphi_{12} )   }{G^2_S} \bigg\Vert_{L^\infty (\mathbb{R}^2)} \leq 12 \Vert \nabla_X^2 \varphi \Vert_{L^\infty (\mathbb{R}^2)}.
\end{equation*}
Applying the above inequalities, \eqref{eq5014}, and \eqref{eq59} with $\zeta = \epsilon/(24 \Vert \nabla_X^2 \varphi \Vert_{L^\infty (\mathbb{R}^2)})$ into \eqref{eq5017}, we see that
\begin{multline}\label{eq5029}
\Vert \mathcal{B}_3 w_S \Vert_{L^2(\mathbb{R}^2) } \leq 6 \kappa_S \Vert \nabla_X \varphi \Vert_{L^\infty (\mathbb{R}^2)} \Vert \nabla_X^2 w_S \Vert_{L^2(\mathbb{R}^2)} + 24 \kappa_S \Vert \nabla_X^2 \varphi \Vert_{L^\infty (\mathbb{R}^2)} \Vert \nabla_X w_S \Vert_{L^2(\mathbb{R}^2)}\\
\leq ( 6 C_\triangle \Vert \nabla_X \varphi \Vert_{L^\infty (\mathbb{R}^2)} + \epsilon ) \Vert \mathcal{A}_3 w_S \Vert_{L^2(\mathbb{R}^2)} + C( \epsilon , \Vert \nabla_X^2 \varphi \Vert_{L^\infty(\mathbb{R}^2)}) \Vert w_S \Vert_{L^2(\mathbb{R}^2)}.
\end{multline}
Using \eqref{eq5029} and
\begin{equation*}
\Vert w_S \Vert_{L^2(0,T ;L^2(\mathbb{R}^2)) } \leq \sup_{0 \leq t \leq T} \Vert w_S ( t ) \Vert_{L^2(\mathbb{R}^2)} T^{1/2},
\end{equation*}
we check that
\begin{multline}\label{eq5030}
\Vert \mathcal{B}_3 w_S \Vert_{L^2(0,T; L^2(\mathbb{R}^2) ) }\\
 \leq C( \Vert \nabla_X \varphi \Vert_{L^\infty (\mathbb{R}^2)} + \epsilon ) \Vert \mathcal{A}_3 w_S \Vert_{L^2(0,T;L^2(\mathbb{R}^2))} +  C( \epsilon , \Vert \nabla_X^2 \varphi \Vert_{L^\infty(\mathbb{R}^2)}) \Vert w_S \Vert_{L^2(0,T;L^2(\mathbb{R}^2))}\\
\leq C( \Vert \nabla_X \varphi \Vert_{L^\infty (\mathbb{R}^2)} + \epsilon ) \Vert w \Vert_{X_T} + C( \epsilon , \Vert \nabla_X^2 \varphi \Vert_{L^\infty(\mathbb{R}^2)}) T^{1/2} \Vert w \Vert_{X_T }.
\end{multline}
By \eqref{eq5023}, \eqref{eq5026}, \eqref{eq5028}, and \eqref{eq5030}, we see that there are $C_1 >0$ and $C_2 = C_2( \epsilon , \Vert \nabla_X^2 \varphi \Vert_{L^\infty(\mathbb{R}^2)}) >0$ such that
\begin{equation}\label{eq5031}
\Vert \mathcal{B} w \Vert_{L^2(0,T;H)} \leq  ( C_1 \epsilon + C_1 \Vert \nabla_X \varphi \Vert_{L^\infty (\mathbb{R}^2)} +  C_2 T^{1/2} ) \Vert w \Vert_{X_T}.
\end{equation}

Next, we consider $\Vert \mathcal{F} (w) \Vert_{L^2(0,T;H)}$. To this end, we use the following facts that for each $\psi_S \in L^2(\mathbb{R}^2)$,
\begin{align}
\Vert \psi_S(y) {\rm{e}}^{ - \delta y_3} \Vert_{L^2( \mathbb{R}^3_+)} & = \frac{1}{\sqrt{2 \delta}} \Vert \psi_S \Vert_{L^2(\mathbb{R}^2)},\label{eq5032}\\
\Vert \psi_S(z) {\rm{e}}^{ \delta z_3} \Vert_{L^2( \mathbb{R}^3_-)} & = \frac{1}{\sqrt{2 \delta}} \Vert \psi_S \Vert_{L^2(\mathbb{R}^2)}.\notag
\end{align}
Applying the H\"{o}lder inequality, \eqref{eq5032}, $\Vert \nabla_X \varphi \Vert_{L^\infty (\mathbb{R}^2)} \leq 1/2$, and \eqref{eq59} with $\zeta = 1/(2\delta)$ into \eqref{eq5018}, we observe that
\begin{multline}\label{eq5033}
\Vert \mathcal{F}_1 (w) \Vert_{L^2 (\mathbb{R}^3_+)}  \leq \Vert (d w_S/{dt}) {\rm{e}}^{ - \delta y_3} \Vert_{L^2(\mathbb{R}^3_+)} + \frac{\kappa_A}{\kappa_S} \Vert (\kappa_S \Delta_{y_h} w_S) {\rm{e}}^{ - \delta y_3} \Vert_{L^2(\mathbb{R}^3_+)}\\ +  \kappa_A ( \delta^2 + 2 \delta \Vert \nabla^2_X \varphi \Vert_{L^\infty(\mathbb{R}^2)} + 2 \delta^2 \Vert \nabla_X \varphi \Vert_{L^\infty (\mathbb{R}^2)} \}  \Vert w_S {\rm{e}}^{ - \delta y_3} \Vert_{L^2(\mathbb{R}^3_+)}\\
+2 \kappa_A \delta \Vert \nabla_X \varphi \Vert_{L^\infty(\mathbb{R}^2)} \sum_{\alpha=1}^2 \Vert (\partial_{y_\alpha}w_S){\rm{e}}^{-\delta y_3} \Vert_{L^2(\mathbb{R}^3_+)} \\
\leq \frac{1}{\sqrt{2\delta}} \Vert d w_S/{dt} \Vert_{L^2(\mathbb{R}^2)} + \frac{2}{\sqrt{2\delta}}\frac{\kappa_A}{\kappa_S} \Vert \mathcal{A}_3 w_S \Vert_{L^2(\mathbb{R}^2)}\\
 + C (\delta, \kappa_A)  \Vert w_S \Vert_{L^2(\mathbb{R}^2)} + \frac{\kappa_A}{\sqrt{2\delta}} ( 3 \delta^2 + 2 \delta \Vert \nabla^2_X \varphi \Vert_{L^\infty(\mathbb{R}^2)} )  \Vert w_S \Vert_{L^2(\mathbb{R}^2)}.
\end{multline}
This gives
\begin{multline}\label{eq5034}
\Vert \mathcal{F}_1 (w) \Vert_{L^2(0,T;L^2 (\mathbb{R}^3_+))} \leq \frac{1}{\sqrt{2\delta}} \Vert d w_S/{dt} \Vert_{L^2(0,T;L^2 (\mathbb{R}^2))}  + \frac{2}{\sqrt{2 \delta}}\frac{\kappa_A}{\kappa_S} \Vert \mathcal{A}_3 w_S \Vert_{L^2(0,T;L^2 (\mathbb{R}^2))}\\ + C ( \delta , \Vert \nabla_X^2 \varphi \Vert_{L^{\infty}(\mathbb{R}^2)} ,\kappa_A)  \Vert w_S \Vert_{L^2(0,T;L^2 (\mathbb{R}^2))}\\
\leq \bigg( \frac{1}{\sqrt{2 \delta }} + \frac{2}{\sqrt{2 \delta }} \frac{\kappa_A}{\kappa_S} + C (\delta , \Vert \nabla_X^2 \varphi \Vert_{L^\infty (\mathbb{R}^2)} , \kappa_A) T^{1/2} \bigg) \Vert w \Vert_{X_T}.
\end{multline}
Similarly, we see that
\begin{multline}\label{eq5035}
\Vert \mathcal{F}_2 (w) \Vert_{L^2 (\mathbb{R}^3_-)} \leq \frac{1}{\sqrt{2\delta}} \Vert d w_S/{dt} \Vert_{L^2(\mathbb{R}^2)} + \frac{2}{\sqrt{2\delta}}\frac{\kappa_B}{\kappa_S} \Vert \mathcal{A}_3 w_S \Vert_{L^2(\mathbb{R}^2)}\\ +  C(\delta , \Vert \nabla^2_X \varphi \Vert_{L^\infty(\mathbb{R}^2)} ,\kappa_B )  \Vert w_S \Vert_{L^2(\mathbb{R}^2)},
\end{multline}
and
\begin{equation}\label{eq5036}
\Vert \mathcal{F}_2 (v) \Vert_{L^2(0,T;L^2 (\mathbb{R}^3_-))} \leq \bigg( \frac{ 1}{\sqrt{2 \delta }} + \frac{2}{\sqrt{2 \delta }} \frac{\kappa_B}{\kappa_S} + C (\delta , \Vert \nabla_X^2 \varphi \Vert_{L^\infty (\mathbb{R}^2)} , \kappa_B) T^{1/2} \bigg) \Vert w \Vert_{X_T}.
\end{equation}
Since $G_S = 1 + \varphi_1^2 + \varphi_2^2$ and $\Vert \nabla_X \varphi \Vert_{L^\infty (\mathbb{R}^2)} \leq 1/2$, we see that
\begin{equation}\label{eq5037}
\bigg\Vert \frac{\varphi_1}{\sqrt{G_S}} \bigg\Vert_{L^\infty (\mathbb{R}^2)} \leq 1,{ \ }\bigg\Vert \frac{\varphi_2}{\sqrt{G_S}} \bigg\Vert_{L^\infty (\mathbb{R}^2)} \leq 1,{ \ }\bigg\Vert \frac{ 1 + \varphi_1^2 + \varphi_2^2 }{\sqrt{G_S}} \bigg\Vert_{L^\infty (\mathbb{R}^2)} \leq 2.
\end{equation}
Applying the H\"{o}lder inequality, \eqref{eq5037}, \eqref{eq5014}, \eqref{eq59} with $\zeta = \rho_S \kappa_S \epsilon/(\rho_A \kappa_A  + \rho_B \kappa_B)$, \eqref{eq5011}, and \eqref{eq5013} into \eqref{eq5020}, we check that
\begin{multline}\label{eq5038}
\Vert \mathcal{F}_3 (w) \Vert_{L^2(\mathbb{R}^2)}
 \leq \frac{2 \rho_A \kappa_A}{\rho_S} \sum_{j=1}^3 \Vert \gamma_+ [\partial_{y_j} w_A] \Vert_{L^2(\mathbb{R}^2)} + \frac{2 \rho_B \kappa_B}{\rho_S} \sum_{j=1}^3 \Vert \gamma_- [\partial_{z_j} w_B] \Vert_{L^2(\mathbb{R}^2)}\\
+ \bigg( \frac{ \rho_A \kappa_A  + \rho_B \kappa_B  }{\rho_S} \bigg) \Vert \nabla_X w_S \Vert_{L^2(\mathbb{R}^2)} + 2 \delta \bigg( \frac{ \rho_A \kappa_A  + \rho_B \kappa_B }{\rho_S} \bigg) \Vert w_S \Vert_{L^2(\mathbb{R}^2)}\\
\leq 6C_\triangle \frac{\rho_A}{\rho_S} \Vert \mathcal{A}_1 w_A \Vert_{L^2(\mathbb{R}^3_+)} + C (\rho_A, \rho_S , \kappa_A) \Vert w_A \Vert_{L^2(\mathbb{R}^3_+)}\\
 + 6C_\triangle \frac{\rho_B }{\rho_S} \Vert \mathcal{A}_2 w_B \Vert_{L^2(\mathbb{R}^3_-)} + C (\rho_B , \rho_S  , \kappa_B ) \Vert w_B \Vert_{L^2(\mathbb{R}^3_-)}\\
  + \epsilon \Vert \mathcal{A}_3 w_S \Vert_{L^2(\mathbb{R}^2)} + C (\epsilon , \delta , \rho_A, \rho_B,\rho_S, \kappa_A , \kappa_B , \kappa_S ) \Vert w_S \Vert_{L^2(\mathbb{R}^2)}.
\end{multline}
By \eqref{eq5038}, we see that
\begin{multline}\label{eq5039}
\Vert \mathcal{F}_3 (w) \Vert_{L^2(0,T;L^2(\mathbb{R}^2))}\\
\leq \bigg( 6C_\triangle \frac{\rho_A + \rho_B}{\rho_S} + \epsilon + C (\epsilon , \delta,  \rho_A, \rho_B,\rho_S, \kappa_A , \kappa_B,\kappa_S ) T^{1/2} \bigg) \Vert w \Vert_{X_T}.
\end{multline}
From \eqref{eq5024}, \eqref{eq5034}, \eqref{eq5036}, and \eqref{eq5039}, we see that $C_3>0$ and\\
 $C_4 = C_4(\epsilon , \delta, \Vert \nabla_X^2 \varphi \Vert_{L^\infty (\mathbb{R}^2)} ,  \rho_A, \rho_B,\rho_S, \kappa_A , \kappa_B , \kappa_S  ) >0$ such that
\begin{equation}\label{eq5040}
\Vert \mathcal{F} (w) \Vert_{L^2(0,T;H)} \leq \bigg( C_3 \epsilon + \frac{C_3}{\sqrt{\delta }} \frac{\kappa_A + \kappa_B + \kappa_S}{\kappa_S} + C_3 \frac{\rho_A + \rho_B}{\rho_S} +  C_4 T^{1/2} \bigg) \Vert w \Vert_{X_T}.
\end{equation}

Since
\begin{equation*}
\Vert \mathcal{F} (w) - \mathcal{B} w \Vert_{L^2(0,T;H)} \leq \Vert \mathcal{B} w \Vert_{L^2(0,T;H)} + \Vert \mathcal{F} (w) \Vert_{L^2(0,T;H)},
\end{equation*}
it follows from \eqref{eq5031} and \eqref{eq5040} to find that there are $C_\triangledown > 0$ and\\
 $C_\Box = C_\Box (\epsilon , \delta, \Vert \nabla_X^2 \varphi \Vert_{L^\infty (\mathbb{R}^2)} ,  \rho_A, \rho_B,\rho_S, \kappa_A , \kappa_B , \kappa_S ) >0$ such that
\begin{multline*}
\Vert \mathcal{F} (w) - \mathcal{B} w \Vert_{L^2(0,T;H)}\\
\leq C_\triangledown \bigg(   \epsilon  + \frac{1}{\sqrt{\delta }} \frac{\kappa_A + \kappa_B + \kappa_S}{\kappa_S} + \Vert \nabla_X \varphi \Vert_{L^\infty (\mathbb{R}^2)}+ \frac{\rho_A + \rho_B}{\rho_S} + C_\Box T^{1/2} \bigg) \Vert w \Vert_{X_T}.
\end{multline*}
Therefore, we see $(\rm{i})$.

Next, we show $(\rm{ii})$. Let $\eta >0$ and $w \in X_T$. Fix $0 < \varepsilon  <T$. By assumption \eqref{eq5022}, there is $C >0$ such that for all $t_1,t_2(\varepsilon \leq t_1 \leq t_2 < T)$
\begin{multline}\label{eq5041}
\Vert w (t_2) - w (t_1) \Vert_{H} + \Vert \mathcal{A} w (t_2) - \mathcal{A} w (t_1) \Vert_{H} + \Vert {d w}/{dt} (t_2) - {d w}/{dt} (t_1) \Vert_{H}\\
 \leq C (t_2 - t_1)^{\eta} .
\end{multline}
Fix $t_1,t_2$ such that $\varepsilon \leq t_1 \leq t_2 <T$. By the definitions of $\Vert \cdot \Vert_H$, $\mathcal{A}$ and \eqref{eq5041}, we find that
\begin{multline}\label{eq5042}
\Vert w_A (t_2) - w_A (t_1) \Vert_{L^2(\mathbb{R}^3_+)} + \Vert \mathcal{A}_1 w_A (t_2) - \mathcal{A}_1 w_A (t_1) \Vert_{L^2(\mathbb{R}^3_+)}\\
 + \Vert dw_A/{dt} (t_2) - dw_A/{dt} (t_1) \Vert_{L^2(\mathbb{R}^3_+)} + \Vert w_B (t_2) - w_B (t_1) \Vert_{L^2(\mathbb{R}^3_-)}\\
  + \Vert \mathcal{A}_2 w_B (t_2) - \mathcal{A}_2 w_B (t_1) \Vert_{L^2(\mathbb{R}^3_-)} + \Vert dw_B/{dt} (t_2) - dw_B/{dt} (t_1) \Vert_{L^2(\mathbb{R}^3_-)}\\ 
+ \Vert w_S (t_2) - w_S (t_1) \Vert_{L^2(\mathbb{R}^2)} + \Vert \mathcal{A}_3 w_S (t_2) - \mathcal{A}_3 w_S (t_1) \Vert_{L^2(\mathbb{R}^2)}\\
 + \Vert dw_S/{dt} (t_2) - dw_S/{dt} (t_1) \Vert_{L^2(\mathbb{R}^2)} \leq C (t_2 - t_1)^\eta. 
\end{multline}
Applying \eqref{eq5042} into \eqref{eq5025}, \eqref{eq5027}, \eqref{eq5030}, \eqref{eq5033}, \eqref{eq5035}, \eqref{eq5038} when $\epsilon =1$, we check that
\begin{equation*}
\Vert \mathcal{B} (w(t_2)) - \mathcal{B} (w(t_1)) \Vert_H + \Vert \mathcal{F} (w(t_2)) - \mathcal{F} (w(t_1)) \Vert_H  \leq C (t_2-t_1)^\eta.
\end{equation*}
Therefore, we find that $\mathcal{F} (w), \mathcal{B} w \in C^\eta ([\varepsilon , T);H)$ for each fixed $\varepsilon \in (0,T)$. This implies that $\mathcal{F}(w), \mathcal{B}w, \mathcal{F}(w) - \mathcal{B}w \in C^\eta_{loc}((0,T); H)$. Therefore, Proposition \ref{prop53} is proved.
 \end{proof}

\subsection{Existence of Strong Solutions}\label{subsec53}
In this section, we show the existence of strong solutions to systems \eqref{eq15}, \eqref{eq29}, and \eqref{eq2010}. We first introduce some function spaces and tools to construct strong solutions to our systems. Secondly, we construct local-in-time strong solutions to system \eqref{eq2010}. Thirdly, we show the existence of a unique global-in-time strong solution of the system to prove Theorem \ref{thm24}. Finally, we prove the existence of a unique global-in-time strong solution to system \eqref{eq15} to prove Theorem \ref{thm11}.

Let $\mathcal{A}$, $\mathcal{B}$, and $\mathcal{F}$ be the two operators and the mapping defined by subsections \ref{subsec51}, \ref{subsec52}. Let us introduce function spaces and tools.

\begin{definition}\label{def54}
For $T \in (0,\infty)$, we define
\begin{equation*}
W^{1,2}(0,T; H) = \{ w = { }^t (w_A, w_B , w_S ) \in L^2(0,T; H);{ \ }\Vert w \Vert_{W^{1,2}(0,T;H)} < \infty  \},
\end{equation*}
\begin{multline*}
L^2(0,T; H_0^1 \cap H^2) = \{ w = { }^t (w_A, w_B , w_S ) \in L^2(0,T; H);\\
\Vert \gamma_+[w_A] \Vert_{L^2(0,T;L^2(\mathbb{R}^2))} =0,{ \ }\Vert \gamma_-[w_B] \Vert_{L^2(0,T;L^2(\mathbb{R}^2))} =0, { \ }\Vert w \Vert_{L^2(0,T;H^2)} < \infty  \},
\end{multline*}
\begin{multline*}
L^2((0,T) \times \mathbb{R}^3_{+,-,0} ) := \{ w = { }^t (w_A , w_B , w_S);w_A \in L^2(\mathbb{R}^3_+ \times (0,T) ),\\
w_B \in L^2(\mathbb{R}^3_- \times (0,T) ),{ \ } w_S \in L^2(\mathbb{R}^2 \times (0,T) ), { \ } \Vert w \Vert_{L^2((0,T) \times \mathbb{R}^3_{+,-,0} )} < \infty \}.
\end{multline*}
Here
\begin{equation*}
\Vert w \Vert_{W^{1,2}(0,T;H)} = (\Vert w \Vert_{L^2(0,T; H)}^2 + \Vert dw/{dt} \Vert_{L^2(0,T; H)}^2)^{1/2}, 
\end{equation*}
\begin{equation*}
\Vert w \Vert_{L^2(0,T;H^2)} = (\Vert w_A \Vert_{L^2(0,T;W^{2,2} (\mathbb{R}^3_+))}^2 + \Vert w_B \Vert_{L^2(0,T;W^{2,2} (\mathbb{R}^3_-))}^2 + \Vert w_S \Vert_{L^2(0,T;L^2(\mathbb{R}^2))}^2)^{1/2},
\end{equation*}
\begin{equation*}
\Vert w \Vert_{L^2((0,T) \times \mathbb{R}^3_{+,-,0} )} = \Vert w_A \Vert_{L^2( \mathbb{R}^3_+ \times (0,T))} + \Vert w_B \Vert_{L^2( \mathbb{R}^3_- \times (0,T))} + \Vert w_S \Vert_{L^2( \mathbb{R}^2 \times (0,T))}.
\end{equation*}
\end{definition}

Combining Propositions \ref{prop51} and \ref{prop53}, we have the following proposition.
\begin{proposition}\label{prop55}
Let $T>0$. Let $v_0 = { }^t (v_0^A , v_0^B , v_0^S) \in D (\mathcal{A}^{1/2})$ and $w = { }^t (w_A , w_B , w_S ) \in X_T$. Assume that
\begin{equation*}
\mathcal{F} (w) - \mathcal{B} w \in C^{\eta}_{loc} ((0,T); H) \text{ for some }0 < \eta < 1.
\end{equation*}
Then system
\begin{equation*}
\begin{cases}
\frac{d}{dt} v + \mathcal{A} v = \mathcal{F} (w) -\mathcal{B} w \text{ on } (0,T),\\
v \vert_{t =0} = v_0,
\end{cases}
\end{equation*}
admits a unique strong solution $v$ in 
\begin{equation*}
C ([0, T] ; H ) \cap C((0,T);D (\mathcal{A})) \cap C^1 ((0,T);H).
\end{equation*}
Moreover, $v$ satisfies the following properties:\\
$(\rm{i})$ {\rm{[Representation formula]}} For each $0 < t \leq T$, $v$ is written by
\begin{equation*}
v (t) = {\rm{e}}^{- t \mathcal{A}} v_0 + \int_0^t {\rm{e}}^{- ( t - \tau ) \mathcal{A}} \{ \mathcal{F} ( w ( \tau )) - \mathcal{B} w(\tau) \} { \ }d \tau.
\end{equation*}
$(\rm{ii})$ {\rm{[Initial condition]}}
\begin{equation*}
\lim_{t \to 0 +0} \Vert v(t) - v_0 \Vert_H = 0.
\end{equation*}
$(\rm{iii})$ {\rm{[H\"{o}lder continuity]}}
\begin{equation*}
v, dv/{dt}, \mathcal{A} v \in C_{loc}^{\eta} ((0,T); H).
\end{equation*}
$(\rm{iv})$ {\rm{[Estimates (I)]}} Let $\epsilon >0$. If $T \leq 1$, then there is\\ $C_\diamondsuit = C_\diamondsuit (\epsilon , \delta, \Vert \nabla_X^2 \varphi \Vert_{L^\infty (\mathbb{R}^2)} ,  \rho_A, \rho_B,\rho_S, \kappa_A , \kappa_B , \kappa_S  ) >0$ such that 
\begin{multline}\label{eq5043}
\Vert v \Vert_{X_T} \leq \Vert v_0 \Vert_{H} + 2 \Vert \mathcal{A}^{1/2} v_0 \Vert_{H} + C_\diamondsuit T^{1/2} \Vert w \Vert_{X_T}\\
 + C_\star C_\triangledown \bigg( \epsilon + \frac{1}{\sqrt{\delta }} \frac{\kappa_A + \kappa_B + \kappa_S}{\kappa_S} + \Vert \nabla_X \varphi \Vert_{L^\infty (\mathbb{R}^2)}  + \frac{\rho_A + \rho_B}{\rho_S} \bigg) \Vert w \Vert_{X_T}.
\end{multline}
Here $C_\star$ and $C_\triangledown$ are the two positive constants appearing in Propositions \ref{prop51} and \ref{prop53}, respectively.\\
$(\rm{v})$ {\rm{[Estimates (II)]}} Write
\begin{align}
\mathcal{M}_0 &= \mathcal{M}_0 ( \kappa_A , \kappa_B , \kappa_S) = \frac{1}{10 C_\star C_\triangledown},\label{eq5044}\\
\delta_0 & = \delta_0 ( \kappa_A , \kappa_B , \kappa_S) = \frac{400 C_\star^2 C_\triangledown^2 ( \kappa_A^2 + \kappa_B^2 + \kappa_S^2)}{\kappa_S^2}.\label{eq5045}
\end{align}
Set $\delta = \delta_0$. Assume that
\begin{equation}\label{eq5046}
\Vert \nabla_X \varphi \Vert_{L^\infty (\mathbb{R}^2)} + \frac{\rho_A + \rho_B}{\rho_S} \leq \mathcal{M}_0.
\end{equation}
If $T \leq 1$, then there is $C_* = C_* ( \Vert \nabla_X^2 \varphi \Vert_{L^\infty (\mathbb{R}^2)} ,  \rho_A, \rho_B,\rho_S, \kappa_A , \kappa_B, \kappa_S ) >0$ such that
\begin{equation}\label{eq5047}
\Vert v \Vert_{X_T} \leq \Vert v_0 \Vert_{H} + 2 \Vert \mathcal{A}^{1/2} v_0 \Vert_{H} + \bigg( \frac{1}{5} + C_* T^{1/2} \bigg) \Vert w \Vert_{X_T}.
\end{equation}

\end{proposition}

\begin{proof}[Proof of Proposition \ref{prop55}]
Let $T>0$. Fix $v_0 = { }^t (v_0^A , v_0^B , v_0^S) \in D (\mathcal{A}^{1/2})$ and $w = { }^t (w_A , w_B , w_S ) \in X_T$. From Proposition \ref{prop53}, we find that $\mathcal{F} ( w) - \mathcal{B} w \in L^2(0,T;H)$. Since $\mathcal{F} ( w) - \mathcal{B} w \in L^2(0,T;H) \cap C_{loc}^\eta ((0,T);H )$, we apply Proposition \ref{prop51} to deduce $(\rm{i})$-$(\rm{iii})$.

We now prove $(\rm{iv})$. Assume that $T \leq 1$. From \eqref{eq56} and \eqref{eq5021}, we observe that
\begin{multline*}
\Vert v \Vert_{X_T} \leq \Vert v_0 \Vert_H + 2 \Vert \mathcal{A}^{1/2} v_0 \Vert_H + (T^{1/2} + C_\star ) \Vert \mathcal{F}(w) - \mathcal{B} w \Vert_{L^2(0,T;H)}\\
\leq \Vert v_0 \Vert_H + 2 \Vert \mathcal{A}^{1/2} v_0 \Vert_H\\
 + (T^{1/2} + C_\star ) C_\triangledown \bigg( \epsilon +  \frac{1}{\sqrt{\delta }} \frac{\kappa_A + \kappa_B + \kappa_S}{\kappa_S} + \Vert \nabla_X \varphi \Vert_{L^\infty (\mathbb{R}^2)} + \frac{\rho_A + \rho_B}{\rho_S} + C_\Box T^{1/2} \bigg) \Vert w \Vert_{X_T}\\\leq \Vert v_0 \Vert_H + 2 \Vert \mathcal{A}^{1/2} v_0 \Vert_H + C_\diamondsuit T^{1/2} \Vert w \Vert_{X_T}\\
  + C_\star  C_\triangledown \bigg( \epsilon + \frac{1}{\sqrt{\delta }} \frac{\kappa_A + \kappa_B + \kappa_S}{\kappa_S} + \Vert \nabla_X \varphi \Vert_{L^\infty (\mathbb{R}^2)} + \frac{\rho_A + \rho_B}{\rho_S} \bigg) \Vert w \Vert_{X_T}.
\end{multline*}
Here 
\begin{multline*}
C_\diamondsuit = C_\diamondsuit (\epsilon , \delta, \Vert \nabla_X^2 \varphi \Vert_{L^\infty (\mathbb{R}^2)} ,  \rho_A, \rho_B,\rho_S, \kappa_A , \kappa_B,\kappa_S  )\\
:= C_\star C_\triangledown C_\Box + C_\triangledown \bigg( \epsilon +  \frac{1}{\sqrt{\delta }} \frac{\kappa_A + \kappa_B + \kappa_S}{\kappa_S} + \Vert \nabla_X \varphi \Vert_{L^\infty (\mathbb{R}^2)} + \frac{\rho_A + \rho_B}{\rho_S} + C_\Box \bigg).
\end{multline*}
Thus, we see $(\rm{iv})$.

Finally, we show $(\rm{v})$. Assume that $T \leq 1$ and that \eqref{eq5046} holds. Set
\begin{align}
\epsilon &= \frac{1}{20 C_\star C_\triangledown},\label{eq5048}\\ 
\delta & = \delta_0 = \frac{400 C_\star^2 C_\triangledown^2 ( \kappa_A^2 + \kappa_B^2 + \kappa_S^2)}{\kappa_S^2}.\notag
\end{align}
By \eqref{eq5046}, \eqref{eq5048}, and \eqref{eq5045}, we see that
\begin{equation}\label{eq5049}
C_\star C_\triangledown \bigg( \epsilon + \frac{1}{\sqrt{\delta }} \frac{\kappa_A + \kappa_B + \kappa_S}{\kappa_S} + \Vert \nabla_X \varphi \Vert_{L^\infty (\mathbb{R}^2)} + \frac{\rho_A + \rho_B}{\rho_S} \bigg)
\leq \frac{1}{20} + \frac{1}{20} + \frac{1}{10} = \frac{1}{5}. 
\end{equation}
From \eqref{eq5043} and \eqref{eq5049}, there is $C_* = C_* ( \Vert \nabla_X^2 \varphi \Vert_{L^\infty (\mathbb{R}^2)} ,  \rho_A, \rho_B,\rho_S, \kappa_A , \kappa_B, \kappa_S ) >0$ such that
\begin{multline*}
\Vert v \Vert_{X_T} \leq \Vert v_0 \Vert_{H} + 2 \Vert \mathcal{A}^{1/2} v_0 \Vert_{H} + C_* T^{1/2} \Vert w \Vert_{X_T}\\
 + C_\star C_\triangledown \bigg( \epsilon + \frac{1}{\sqrt{\delta }} \frac{\kappa_A + \kappa_B + \kappa_S}{\kappa_S} + \Vert \nabla_X \varphi \Vert_{L^\infty (\mathbb{R}^2)} + \frac{\rho_A + \rho_B}{\rho_S} \bigg) \Vert w \Vert_{X_T}\\
 \leq \Vert v_0 \Vert_{H} + 2 \Vert \mathcal{A}^{1/2} v_0 \Vert_{H} + (1/5 + C_* T^{1/2} ) \Vert w \Vert_{X_T}.
\end{multline*}
Thus, we have \eqref{eq5047}. Note that $C_* = C_\diamondsuit$ when $\epsilon$ and $\delta$ satisfy \eqref{eq5048} and \eqref{eq5045}. Therefore, Proposition \ref{prop55} is proved.
\end{proof}

Let us now construct a local-in-time strong solution to system \eqref{eq2010}.

\begin{proposition}[Existence of a local-in-time strong solution]\label{prop56}
Let $\mathcal{M}_0 = \mathcal{M}_0(\kappa_A , \kappa_B ,\kappa_S)$ and $\delta_0 = \delta_0(\kappa_A , \kappa_B, \kappa_S)$ be the two positive constants defined by \eqref{eq5044} and \eqref{eq5045}, respectively. Set $\delta = \delta_0$. Assume that 
\begin{equation}\label{eq5050}
\Vert \nabla_X \varphi \Vert_{L^\infty (\mathbb{R}^2)} + \frac{\rho_A + \rho_B}{\rho_S} \leq \mathcal{M}_0.
\end{equation}
Then for each $v_0 = { }^t (v_0^A , v_0^B , v_0^S) \in D (\mathcal{A}^{1/2})$, system \eqref{eq2010} admits a unique local-in-time strong solution $v$ in 
\begin{equation*}
C ([0, T_*] ; H ) \cap L^2 (0, T_*; H_0^1 \cap H^2 ) \cap W^{1,2} (0, T_* ; H) \cap L^2((0,T_*) \times \mathbb{R}^3_{+,-,0} ),
\end{equation*}
satisfying
\begin{equation*}
\lim_{t \to 0 + 0} v (t) = v_0 \text{ in }H.
\end{equation*}
Here $T_* = T_*( \Vert \nabla_X^2 \varphi \Vert_{L^\infty (\mathbb{R}^2)} ,  \rho_A, \rho_B,\rho_S, \kappa_A , \kappa_B, \kappa_S ) >0$.
\end{proposition}

\begin{proof}[Proof of Proposition \ref{prop56}]
Assume that $T \leq 1$ and \eqref{eq5050} holds. Fix $v_0 = { }^t (v_0^A, v_0^B, v_0^S) \in D (\mathcal{A}^{1/2})$. For each $m \in \mathbb{N}$, set $v_1 =v_1 (t)= { }^t (v_1^A, v_1^B, v_1^S)$ and $v_{m+1} = v_{m+1}(t) = { }^t (v_{m+1}^A , v_{m+1}^B , v_{m+1}^S)$ as follows:
\begin{equation}\label{eq5051}
\begin{cases}
\frac{d}{d t} v_1 + \mathcal{A} v_1 = 0 \text{ on }(0,T),\\
v_1 \vert_{t = 0 } = v_0,
\end{cases}
\end{equation}
\begin{equation}\label{eq5052}
\begin{cases}
\frac{d}{d t} v_{m+1} + \mathcal{A} v_{m+1} = \mathcal{F} (v_m) - \mathcal{B} v_m \text{ on }(0,T),\\
v_{m + 1 } \vert_{t = 0 } = v_0.
\end{cases}
\end{equation}

We first consider system \eqref{eq5051}. From Proposition \ref{prop55}, we see that system \eqref{eq5051} admits a unique strong solution $v_1$ such that
\begin{align*}
& v_1 \in C ([0,T];H) \cap C((0,T);D(\mathcal{A})) \cap C^1((0,T);H),\\
& v_1 (t) = {\rm{e}}^{ - t \mathcal{A} } v_0 { \ }(0 < t \leq T),\\
& \lim_{t \to 0 + 0} v_1 (t) = v_0 \text{ in }H,\\
&\Vert v_1 \Vert_{X_T} \leq \Vert v_0 \Vert_H + 2 \Vert \mathcal{A}^{1/2} v_0 \Vert_H,
\end{align*}
and
\begin{equation}\label{eq5053}
 v_1 ,{ \ } d v_1/{dt},{ \ }\mathcal{A} v_1 \in C^{1/2}_{loc}((0,T);H).
\end{equation}
For the readers, we give a sketch of the proof to derive \eqref{eq5053}. Fix $\varepsilon \in (0,T)$. Let $t_1,t_2 >0 $ such that $\varepsilon \leq t_1 \leq t_2 <T$. Since $- \mathcal{A}$ generates a bounded analytic semigroup on $H$ and $\mathcal{A}$ is a positive selfadjoint operator, we check that
\begin{multline*}
\Vert v_1 (t_2) - v_1(t_1) \Vert_H + \Vert \mathcal{A} v_1(t_2) - \mathcal{A}v_1(t_1) \Vert_H\\
 = \Vert {\rm{e}}^{- t_2 \mathcal{A}} v_0 -  {\rm{e}}^{- t_1 \mathcal{A}} v_0 \Vert_H + \Vert \mathcal{A} {\rm{e}}^{- t_2 \mathcal{A}} v_0 - \mathcal{A} {\rm{e}}^{- t_1 \mathcal{A}} v_0 \Vert_H\\ 
 = \Vert ({\rm{e}}^{- (t_2 - t_1) \mathcal{A}} - 1 ) {\rm{e}}^{- t_1 \mathcal{A}} v_0 \Vert_H +  \Vert ({\rm{e}}^{- (t_2 - t_1) \mathcal{A}} - 1 ) \mathcal{A}^{1/2} {\rm{e}}^{- t_1 \mathcal{A}} \mathcal{A}^{1/2}v_0 \Vert_H\\
 \leq C(T) (t_2-t_1)^{1/2} \Vert \mathcal{A}^{1/2} {\rm{e}}^{- t_1 \mathcal{A}} v_0 \Vert_H + C (T) (t_2 - t_1)^{1/2} \Vert \mathcal{A} {\rm{e}}^{- t_1 \mathcal{A}} \mathcal{A}^{1/2}v_0 \Vert_H\\ 
 \leq \frac{C(T)}{\varepsilon^{1/2}} (t_2-t_1)^{1/2} \Vert v_0 \Vert_H + \frac{C (T)}{\varepsilon} (t_2 - t_1)^{1/2} \Vert \mathcal{A}^{1/2} v_0 \Vert_H.
\end{multline*}
Thus, we find that $ v_1,\mathcal{A} v_1 \in C^{1/2}_{loc}((0,T);H)$. Since $d v_1/{dt} = - \mathcal{A} v_1$, we see \eqref{eq5053}. See \cite[Chapter 2]{Paz83} for details.

Next, we consider system \eqref{eq5052}. We prove that for each $m \in \mathbb{N}$ system \eqref{eq5052} admits a unique strong solution $v_{m+1}$ such that
\begin{equation}\label{eq5054}
v_{m+1} \in C ([0,T];H) \cap C((0,T);D(\mathcal{A})) \cap C^1((0,T);H),
\end{equation}
\begin{equation}
v_{m+1} (t) = {\rm{e}}^{ - t \mathcal{A}} v_0 + \int_0^t {\rm{e}}^{- (t - \tau )L} \{ \mathcal{F} ( v_m (\tau)) - \mathcal{B}v_m(\tau) \} { \ }d \tau { \ }(0 < t \leq T),
\end{equation}
\begin{equation}
\lim_{t \to 0 + 0} v_{m+1} (t) = v_0 \text{ in }H,
\end{equation}
\begin{equation}
v_{m+1} ,{ \ } d v_{m+1}/{dt},{ \ }\mathcal{A} v_{m+1} \in C^{1/2 }_{loc}((0,T);H), 
\end{equation}
and
\begin{equation}\label{eq5058}
\Vert v_{m+1} \Vert_{X_T} \leq \Vert v_0 \Vert_H + 2 \Vert \mathcal{A}^{1/2} v_0 \Vert_H + \left( \frac{1}{5} + C_* T^{1/2} \right) \Vert v_m \Vert_{X_T}.
\end{equation}
Here $C_*$ is the positive constant appearing in assertion $(\rm{v})$ in Proposition \ref{prop55}.

We now consider the case when $m=1$, that is,
\begin{equation}\label{eq5059}
\begin{cases}
\frac{d}{d t} v_2 + \mathcal{A} v_2 = \mathcal{F} ( v_1) - \mathcal{B} v_1 \text{ on }(0,T),\\
v_2 \vert_{t = 0 } = v_0.
\end{cases}
\end{equation}
From $v_1 \in X_T$, \eqref{eq5053}, and Proposition \ref{prop53}, we find that
\begin{equation*}
\mathcal{F} ( v_1) - \mathcal{B}v_1 \in L^2(0,T;H) \cap C^{1/2}_{loc}((0,T);H).
\end{equation*}
Since $v_0 \in D (\mathcal{A}^{1/2})$ and $\mathcal{F} ( v_1) - \mathcal{B}v_1 \in L^2(0,T;H) \cap C^{1/2}_{loc}((0,T);H)$, it follows from Proposition \ref{prop55} to see that system \eqref{eq5059} admits a unique strong solution $v_2$ such that
\begin{align*}
& v_2 \in C ([0,T];H) \cap C((0,T);D(\mathcal{A})) \cap C^1((0,T);H),\\
& v_2 (t) = {\rm{e}}^{ - t \mathcal{A} } v_0 + \int_0^t {\rm{e}}^{- ( t - \tau) \mathcal{A}} \{ \mathcal{F} ( v_1 (\tau)) - \mathcal{B}v_1(\tau) \}{ \ }d \tau { \ }(0 < t \leq T),\\
& \lim_{t \to 0 + 0} v_2 (t) = v_0 \text{ in }H,
\end{align*}
\begin{equation}\label{eq5060}
v_2 ,{ \ } d v_2/{dt},{ \ }\mathcal{A} v_2 \in C^{1/2}_{loc}((0,T);H),
\end{equation}
and
\begin{equation}
\Vert v_2 \Vert_{X_T} \leq \Vert v_0 \Vert_H + 2 \Vert \mathcal{A}^{1/2} v_0 \Vert_H + \left( \frac{1}{5} + C_* T^{1/2} \right) \Vert v_1 \Vert_{X_T}.\label{eq5061}
\end{equation}
From \eqref{eq5060}, \eqref{eq5061} and Proposition \ref{prop53}, we see that
\begin{equation*}
\mathcal{F} ( v_2) - \mathcal{B} v_2 \in L^2(0,T;H) \cap C^{1/2}_{loc}((0,T);H).
\end{equation*}

We now consider the case when $m=2$, that is,
\begin{equation}\label{eq5062}
\begin{cases}
\frac{d}{d t} v_3 + \mathcal{A} v_3 = \mathcal{F} ( v_2) - \mathcal{B} v_2 \text{ on }(0,T),\\
v_3 \vert_{t = 0 } = v_0.
\end{cases}
\end{equation}
Since $v_0 \in D (\mathcal{A}^{1/2})$ and $\mathcal{F} ( v_2) - \mathcal{B} v_2 \in L^2(0,T;H) \cap C^{1/2}_{loc}((0,T);H)$, it follow from Proposition \ref{prop55} to see that system \eqref{eq5062} admits a unique strong solution $v_3$ such that
\begin{align*}
& v_ 3\in C ([0,T];H) \cap C((0,T);D(\mathcal{A})) \cap C^1((0,T);H),\\
& v_3 (t) = {\rm{e}}^{ - t \mathcal{A} } v_0 + \int_0^t {\rm{e}}^{- ( t - \tau) \mathcal{A}} \{ \mathcal{F} ( v_2 (\tau)) - \mathcal{B} v_2(\tau) \}{ \ }d \tau { \ }(0 < t \leq T),\\
& \lim_{t \to 0 + 0} v_3 (t) = v_0 \text{ in }H,
\end{align*}
\begin{equation}\label{eq5063}
v_3 ,{ \ } d v_3/{dt},{ \ }\mathcal{A} v_3 \in C^{1/2}_{loc}((0,T);H),
\end{equation}
and
\begin{equation}\label{eq5064}
\Vert v_3 \Vert_{X_T} \leq \Vert v_0 \Vert_H + 2 \Vert \mathcal{A}^{1/2} v_0 \Vert_H + \left( \frac{1}{5} + C_{\star} T^{1/2}\right) \Vert v_2 \Vert_{X_T}.
\end{equation}
From \eqref{eq5063}, \eqref{eq5064}, and Proposition \ref{prop53}, we see that
\begin{equation*}
\mathcal{F} (v_3) - \mathcal{B} v_3 \in L^2(0,T;H) \cap C^{1/2}_{loc}((0,T);H).
\end{equation*}
By induction, we see that for each $m \in \mathbb{N}$ system \eqref{eq5052} admits a unique strong solution $v_{m+1}$ satisfying \eqref{eq5054}-\eqref{eq5058}.

Next, we prove that for each $m \in \mathbb{N}$
\begin{equation}\label{eq5065}
\Vert v_{m+2} - v_{m+1} \Vert_{X_T} \leq \left( \frac{1}{5} + C_* T^{1/2} \right) \Vert v_{m+1} - v_m \Vert_{X_T}.
\end{equation}
From
\begin{align*}
\begin{cases}
\frac{d}{d t} v_{m+2} + \mathcal{A} v_{m+2} = \mathcal{F} (v_{m+1}) - \mathcal{B} v_{m+1} \text{ on }(0,T),\\
v_{m + 2 } \vert_{t = 0 } = v_0,
\end{cases}\\
\begin{cases}
\frac{d}{d t} v_{m+1} + \mathcal{A} v_{m+1} = \mathcal{F} (v_m) - \mathcal{B} v_m \text{ on } (0,T),\\
v_{m + 1 } \vert_{t = 0 } = v_0,
\end{cases}
\end{align*}
we have
\begin{equation*}
\begin{cases}
\frac{d}{d t} (v_{m+2} - v_{m+1}) + \mathcal{A} (v_{m+2} - v_{m+1}) = \mathcal{F} ( v_{m+1} - v_m) - \mathcal{B}(v_{m+1} - v_m) \text{ on }(0,T),\\
(v_{m + 2 } - v_{m+1}) \vert_{t = 0 } = 0.
\end{cases}
\end{equation*}
Note that $\mathcal{F} ( v_{m+1}) - \mathcal{F}(v_m) = \mathcal{F} ( v_{m+1} - v_m)$. Since $\mathcal{F} ( v_{m+1} - v_m) - \mathcal{B}(v_{m+1} - v_m) \in L^2(0,T;H) \cap C^{1/2}_{loc}((0,T);H)$, it follows from \eqref{eq5047} in Proposition \ref{prop55} to derive \eqref{eq5065}.

Now we choose $T_* >0$ such that
\begin{equation*}
T_* = \min \bigg\{ 1 , \frac{9}{100 C_*^2} \bigg\}.
\end{equation*}
Note that $C_* = C_* ( \Vert \nabla_X^2 \varphi \Vert_{L^\infty (\mathbb{R}^2)}, \rho_A , \rho_B ,\rho_S , \kappa_A , \kappa_B , \kappa_S)$. From \eqref{eq5065} and \eqref{eq5058}, we have 
\begin{equation*}
\Vert v_{m+2} - v_{m+1} \Vert_{X_{T_*}} \leq \frac{1}{2} \Vert v_{m+1} - v_m \Vert_{X_{T_*}}
\end{equation*}
and
\begin{equation*}
\Vert v_{m+1} \Vert_{X_{T_*}} \leq \Vert v_0 \Vert_H + 2 \Vert \mathcal{A}^{1/2} v_0 \Vert_H + \frac{1}{2} \Vert v_m \Vert_{X_{T_*}}.
\end{equation*}
Since $\Vert v_1 \Vert_{X_{T_*}} \leq \Vert v_0 \Vert_H + 2 \Vert \mathcal{A}^{1/2} v_0 \Vert_H$, we use induction to see that for each $m \in \mathbb{N}$
\begin{equation*}
\Vert v_{m+1} \Vert_{X_{T_*}} \leq 2 \Vert v_0 \Vert_H + 4 \Vert \mathcal{A}^{1/2} v_0 \Vert_H.
\end{equation*}
We also see that
\begin{align*}
\Vert v_{m + 2} - v_{m+1} \Vert_{X_{T_*}} & \leq \bigg(\frac{1}{2}\bigg)^{m} \Vert v_2 - v_1 \Vert_{X_{T_*}},\\
 & \leq \bigg( \frac{1}{2} \bigg)^{m} ( 3 \Vert v_0 \Vert_H + 6 \Vert \mathcal{A}^{1/2} v_0 \Vert_H ) .
\end{align*}
From a fixed-point argument, we have a unique function $v = { }^t (v_A,v_B,v_S)$ in $X_{T_*}$ satisfying
\begin{align}
& \lim_{m \to \infty} \Vert v_m - v \Vert_{X_{T_*}} = 0,\label{eq5066}\\
& \Vert v \Vert_{X_{T_*}} \leq 2 \Vert v_0 \Vert_H + 4 \Vert \mathcal{A}^{1/2} v_0 \Vert_H. \label{eq5067}
\end{align}
Now we check that
\begin{align}
\Vert \gamma_+[v_A] \Vert_{L^2(0,T_*;L^2(\mathbb{R}^2))} =0,\label{eq5068}\\
\Vert \gamma_+[v_B] \Vert_{L^2(0,T_*;L^2(\mathbb{R}^2))} =0.\label{eq5069}
\end{align}
Since $v^A_m(t) \in W^{1,2}_0(\mathbb{R}^3_+) \cap W^{2,2} (\mathbb{R}^3_+)$, we find that
\begin{equation*}
\Vert \gamma_+[v_m^A] \Vert_{L^2(0,T_*;L^2(\mathbb{R}^2))} =0.
\end{equation*}
Applying Lemma \ref{lem52} and \eqref{eq5066}, we observe that
\begin{align*}
\Vert \gamma_+ [ v_A] - \gamma_+[v_m^A] \Vert_{L^2(0,T_*;L^2(\mathbb{R}^2))} & \leq C \Vert v_A - v^A_m \Vert_{L^2(0,T_*;W^{1,2} (\mathbb{R}^3_+) )}\\
& \leq C \Vert v - v_m \Vert_{X_{T_*}} \to 0 \text{ as }m \to \infty .
\end{align*}
Therefore, we see \eqref{eq5068}. Similarly, we have \eqref{eq5069}. Since $\mathcal{A}$ is a closed operator, it follows from \eqref{eq51}, \eqref{eq5067}, \eqref{eq5068}, \eqref{eq5069}, and Lemma \ref{lem52} to find that
\begin{equation*}
v \in C([0,T_*];H) \cap L^2(0,T_*; H_0^1 \cap H^2) \cap W^{1,2} (0,T_*; H).
\end{equation*}
From \eqref{eq5066}, \eqref{eq5067}, and assertion $(\rm{i})$ in Proposition \ref{prop53}, we see that
\begin{equation}\label{eq5070}
\mathcal{F} (v) - \mathcal{B} v \in L^2(0,T_* ;H),
\end{equation}
and that
\begin{equation}\label{eq5071}
\Vert (\mathcal{F} (v)- \mathcal{F} (v_m)) - ( \mathcal{B} v - \mathcal{B} v_m) \Vert_{ L^2(0,T_* ;H)} \leq C \Vert v - v_m \Vert_{X_{T_*}} \to 0 \text{ as }m\to \infty. 
\end{equation}
From \eqref{eq5066}, Lemma \ref{lem52}, Corollary \ref{cor3018}, and Proposition \ref{prop3017} we also see that
\begin{align}
& \Vert v - v_m \Vert_{L^2(0,T_*; H^2)} \to 0 \text{ as } m \to \infty,\label{eq5072}\\
& \Vert \mathcal{F}_3(v) - \mathcal{F}_3 (v_m) \Vert_{L^2(0,T_*; L^2(\mathbb{R}^2))} \to 0 \text{ as } m \to \infty. \label{eq5073}
\end{align}
Since $v_{m+1}$ satisfies
\begin{equation*}
\begin{cases}
\frac{d}{dt} v_{m+1} + \mathcal{A} v_{m+1} = \mathcal{F} (v_m) - \mathcal{B} v_m \text{ on }(0,T_*),\\
v_{m+1}\vert_{t = 0} = v_0,
\end{cases}
\end{equation*}
and
\begin{equation*}
v_{m+1} (t) = {\rm{e}}^{- t \mathcal{A}}v_0 + \int_0^t {\rm{e}}^{- (t - \tau )\mathcal{A}} \{ \mathcal{F} (v_m(\tau)) - \mathcal{B} v_m (\tau) \}{ \ }d \tau { \ } (0 < t \leq T_*),
\end{equation*}
we apply \eqref{eq5066}, \eqref{eq5071}, \eqref{eq5072}, \eqref{eq5073}, \eqref{eq53}, and \eqref{eq54} to see that $v$ satisfies that
\begin{equation*}
\left\Vert \frac{d}{dt}v + \mathcal{A} v - \mathcal{F} (v) + \mathcal{B} v \right\Vert_{L^2(0,T_*;H)} = 0
\end{equation*}
and that
\begin{equation*}
v (t) = {\rm{e}}^{- t \mathcal{A}}v_0 + \int_0^t {\rm{e}}^{- (t - \tau )\mathcal{A}} \{ \mathcal{F} (v(\tau)) - \mathcal{B} v(\tau) \} { \ }d \tau { \ } (0 < t \leq T_*).
\end{equation*}
Since ${\rm{e}}^{- t \mathcal{A}}$ is a $C_0$-semigroup on $H$, we use the Cauchy-Schwarz inequality and \eqref{eq5070} to observe that
\begin{align*}
\Vert v(t) - v_0 \Vert_H & \leq \Vert {\rm{e}}^{- t \mathcal{A}} v_0 - v_0 \Vert_H + t^{1/2} \Vert \mathcal{F} (v) - \mathcal{B} v \Vert_{L^2(0,T;H)}\\
& \to 0 { \ }(t \to 0 + 0).
\end{align*}
Therefore, we see that $v$ is a strong solution to system \eqref{eq2010} with initial data $v_0$.

Finally, we discuss the uniqueness of the solution $v$. To this end, we show that
\begin{equation}\label{eq5074}
v \in L^2((0,T_*) \times \mathbb{R}^3_{+,-,0}).
\end{equation}
From Proposition \ref{prop53}, we consider $v_1$ and $v_{m+1}$ as follows:
\begin{align*}
v_1 (x,t) & = \mathcal{G}* v_0,\\
v_{m+1} (x,t) &= \mathcal{G}*v_0 + \int_0^t \mathcal{G}*\{ \mathcal{F} (v_m(x, \tau)) - \mathcal{B}v_m(x, \tau) \} { \ }d \tau. 
\end{align*}
Here $\mathcal{G}$ denotes the heat kernels (see Proposition \ref{prop51}). We easily check that for each $m \in \mathbb{N}$, $v_m \in L^2((0,T_*) \times \mathbb{R}^3_{+,-,0})$. Using the H\"{o}lder inequality and \eqref{eq5066}, we see that
\begin{multline*}
\Vert v_{m+1} - v_m \Vert_{L^2((0,T_*) \times \mathbb{R}^3_{+,-,0} )}\\
 = \Vert v_{m+1}^A - v_m^A \Vert_{L^2( \mathbb{R}^3_+ \times (0,T_*))} + \Vert v_{m+1}^B - v_m^B \Vert_{L^2( \mathbb{R}^3_- \times (0,T_*))} + \Vert v_{m+1}^S - v_m^S \Vert_{L^2( \mathbb{R}^2 \times (0,T_*))}\\
\leq \Vert v_{m+1}^A - v_m^A \Vert_{L^\infty(0,T_*: L^2( \mathbb{R}^3_+))} T_*^{1/2} + \Vert v_{m+1}^B - v_m^B \Vert_{L^\infty(0,T_*;L^2( \mathbb{R}^3_-))}T_*^{1/2}\\
 + \Vert v_{m+1}^S - v_m^S \Vert_{L^\infty (0,T_*;L^2( \mathbb{R}^2))}T_*^{1/2}\\
 \leq T_*^{1/2}\Vert v_{m+1} - v_m \Vert_{X_{T_*}} \to 0 \text{ as }m \to \infty.
\end{multline*}
This implies that $v_\infty = { }^t (v_\infty^A, v_\infty^B , v_\infty^S) \in L^2 ((0,T_*) \times \mathbb{R}^3_{+,-,0})$ such that 
\begin{equation*}
\lim_{m \to \infty} \Vert v_m - v_\infty \Vert_{L^2((0,T_*) \times \mathbb{R}^3_{+,-,0} )} = 0.
\end{equation*}
It is easy check that
\begin{equation*}
\Vert v_\infty - v \Vert_{L^2(0,T_*; H)} \leq \Vert v_\infty - v_m \Vert_{L^2(0,T_*; H)} + \Vert v_m - v \Vert_{L^2(0,T_*; H)} \to 0 { \ }(\text{ as } m \to \infty). 
\end{equation*}
Thus, we see \eqref{eq5074}. Since ${ }^t (v_A, v_B , v_S)$ is a strong solution to system \eqref{eq29} with initial data ${ }^t (v_0^A,v_0^B , v_0^S)$ and $ { }^t (v_A ,v_B,v_S) \in L^2((0,T_*) \times \mathbb{R}^3_{+,-,0})$, it follows from Proposition \ref{prop43}(the uniqueness of the strong solutions to system \eqref{eq29}) to see that $v$ is a unique local-in-time strong solution to \eqref{eq2010} with initial data $v_0$. Therefore, Proposition \ref{prop56} is proved.
  \end{proof}

Let us prove Theorem \ref{thm24}.
\begin{proof}[Proof of Theorem \ref{thm24}]
Let $\varphi \in BC^2(\mathbb{R}^2)$ and $\rho_A, \rho_B, \rho_S, \kappa_A , \kappa_B , \kappa_S >0$. Let $\mathcal{M}_0 = \mathcal{M}_0(\kappa_A , \kappa_B ,\kappa_S)$ and $\delta_0 = \delta_0(\kappa_A , \kappa_B, \kappa_S)$ be the two positive constants defined by \eqref{eq5044} and \eqref{eq5045}, respectively. Set $\delta = \delta_0$. Assume that $\Vert \nabla_X \varphi \Vert_{L^\infty (\mathbb{R}^2)} \leq 1/2$ and
\begin{equation*}
\Vert \nabla_X \varphi \Vert_{L^\infty (\mathbb{R}^2)} + \frac{\rho_A + \rho_B}{\rho_S} \leq \mathcal{M}_0.
\end{equation*}
Fix $v_0 \in D (\mathcal{A}^{1/2})$. Let $T_* = T_*( \Vert \nabla_X^2 \varphi \Vert_{L^\infty (\mathbb{R}^2)} ,  \rho_A, \rho_B,\rho_S, \kappa_A , \kappa_B, \kappa_S )$ be the positive constant appearing in Proposition \ref{prop56}.

We first consider the following system:
\begin{equation}\label{eq5075}
\begin{cases}
\frac{d}{dt}v^1 + \mathcal{A} v^1 = \mathcal{F} (v^1) - \mathcal{B} v^1 \text{ on } (0,T_*),\\
v^1 \vert_{t=0} = v_0.
\end{cases}
\end{equation}
From Proposition \ref{prop56}, we see that there exists a unique strong solution $v^1$ to system \eqref{eq5075} with initial data $v_0$.

Let $T_1 \in [T_*/2 , T_*)$ such that $v^1(T_1) \in D (\mathcal{A})$. Since $v^1(T_1) \in D (\mathcal{A}^{1/2})$ and $T_*$ does not depend on initial data, it follows from Proposition \ref{prop56} to see that there exists a unique strong solution $v^2$ of system
\begin{equation*}
\begin{cases}
\frac{d}{dt}v^2 + \mathcal{A} v^2 = \mathcal{F} (v^2) - \mathcal{B} v^2 \text{ on } (T_1,T_1 + T_*),\\
v^2 \vert_{t=T_1} = v^1(T_1).
\end{cases}
\end{equation*}

Let $T_2 \in [T_1 + T_*/2 , T_1 + T_*)$ such that $v^2(T_2) \in D (\mathcal{A})$. Since $v^2(T_2) \in D (\mathcal{A}^{1/2})$ and $T_*$ does not depend on initial data, it follows from Proposition \ref{prop56} to see that there exists a unique strong solution $v^3$ of system
\begin{equation*}
\begin{cases}
\frac{d}{dt}v^3 + \mathcal{A} v^3 = \mathcal{F} (v^3) - \mathcal{B} v^3 \text{ on } (T_2, T_2 + T_*),\\
v^3 \vert_{t = T_2} = v^2 (T_2).
\end{cases}
\end{equation*}

We easily check that
\begin{equation*}
T_2 + T_* \geq (T_1 + T_*/2) + T_* \geq T_*/2 + T_*/2 + T_* = 2 T_*.
\end{equation*}
Now we set
\begin{equation*}
v = v (t) =
\begin{cases}
v^1 &\text{ on }(0,T_1],\\
v^2 & \text{ on }(T_1,  T_2],\\
v^3 & \text{ on }(T_2,  2 T_*).
\end{cases}
\end{equation*}
From Proposition \ref{prop43}, we see that $v^1 = v^2$ on $[T_1 , T_*)$ and $v^2 = v^3$ on $[T_2, T_1 + T_*)$. Thus, we find that
\begin{equation*}
v \in C([0,2T_*]; H) \cap L^2(0, 2T_*; H_0^1 \cap H^2 ) \cap W^{1,2} (0,  2T_* ; H) \cap L^2( (0, 2 T_*) \times \mathbb{R}^3_{+,-,0}),
\end{equation*}
and that $v$ is a strong solution to system
\begin{equation*}
\begin{cases}
\frac{d}{dt}v + \mathcal{A} v = \mathcal{F} (v) - \mathcal{B} v \text{ on } (0, 2T_*),\\
v \vert_{t=0} = v_0.
\end{cases}
\end{equation*}
Let $T >1$. Since $T_*$ does not depend on initial data, we repeat the same argument above to see that there is a function $v$ such that
\begin{equation*}
v \in C([0, T); H) \cap L^2(0,T; H_0^1 \cap H^2 ) \cap W^{1,2} (0, T ; H) \cap L^2( (0,T) \times \mathbb{R}^3_{+,-,0} )
\end{equation*}
and that $v$ is a strong solution to system
\begin{equation*}
\begin{cases}
\frac{d}{dt}v + \mathcal{A} v = \mathcal{F} (v) - \mathcal{B} v \text{ on } (0, T),\\
v \vert_{t=0} = v_0.
\end{cases}
\end{equation*}
Since we can choose $T$ to be any positive number, we find that there is $v$ such that
\begin{equation*}
v \in C([0, \infty); H) \cap L_{loc}^2(0, \infty; H_0^1 \cap H^2 ) \cap W_{loc}^{1,2} (0, \infty ; H) \cap L^2_{loc}( \mathbb{R}_+ \times \mathbb{R}^3_{+,-,0})
\end{equation*}
and that $v$ is a strong solution to system
\begin{equation*}
\begin{cases}
\frac{d}{dt}v + \mathcal{A} v = \mathcal{F} (v) - \mathcal{B} v \text{ on } (0, \infty),\\
v \vert_{t=0} = v_0.
\end{cases}
\end{equation*}
From Proposition \ref{prop43}, we see that $v$ is a unique global-in-time strong solution of system \eqref{eq2010} with initial data $v_0$. Therefore, Theorem \ref{thm24} is proved.
\end{proof}

Finally, we prove Theorem \ref{thm11}.
\begin{proof}[Proof of Theorem \ref{thm11}]
Let $\varphi \in BC^2(\mathbb{R}^2)$ and $\rho_A, \rho_B, \rho_S, \kappa_A , \kappa_B , \kappa_S >0$. Let $\mathcal{M}_0 = \mathcal{M}_0(\kappa_A , \kappa_B ,\kappa_S)$ and $\delta_0 = \delta_0(\kappa_A , \kappa_B, \kappa_S)$ be the two positive constants defined by \eqref{eq5044} and \eqref{eq5045}, respectively. Set $\delta = \delta_0$. Assume that $\Vert \nabla_X \varphi \Vert_{L^\infty (\mathbb{R}^2)} \leq 1/2$ and
\begin{equation*}
\Vert \nabla_X \varphi \Vert_{L^\infty (\mathbb{R}^2)} + \frac{\rho_A + \rho_B}{\rho_S} \leq \mathcal{M}_0.
\end{equation*}
Let $\theta_0^A \in W^{1,2} (\Omega_A)$, $\theta_0^B \in W^{1,2}( \Omega_B)$, $\theta_0^S \in W^{1,2} (\Gamma)$ satisfying $\gamma_A[\theta_0^A] = \theta_0^S$ and $\gamma_B[\theta_0^B] = \theta_0^S$. By definition, there is $v_0^S \in W^{1,2} (\mathbb{R}^2)$ such that $\theta_0^S = \mathcal{P}_S[v_0^S]$. Set $\delta$ by \eqref{eq5045}, and
\begin{equation*}
\begin{cases}
u_0^A = u_0^A (x) = \theta_0^A(x) - v_0^S(x_h) {\rm{e}}^{ \delta \{ \varphi (x_h ) -x_3 \} },\\
u_0^B = u_0^B (x) = \theta_0^B(x)  - v_0^S(x_h) {\rm{e}}^{ \delta \{ x_3 - \varphi (x_h )  \} } .
\end{cases}
\end{equation*}
From Proposition \ref{prop31}, we find that $u_0^A \in W_0^{1,2} (\Omega_A)$ and $u_0^B \in W_0^{1,2} (\Omega_B)$. By definition, there are $v_0^A \in W_0^{1,2} (\mathbb{R}^3_+)$ and $v_0^B \in W_0^{1,2}(\mathbb{R}^3_-)$ such that $u_0^A = \mathcal{P}_A [v_0^A]$ and $u_0^B = \mathcal{P}_B [v_0^B]$. Set $v_0 = { }^t (v_0^A ,v_0^B , v_0^S)$. From Theorem \ref{thm24}, we find that there exists a unique global-in-time strong solution $v = { }^t (v_A, v_B , v_S)$ to system \eqref{eq2010} with initial data $v_0 = { }^t(v_0^A , v_0^B , v_0^S)$. Since ${ }^t (v_A, v_B , v_A)$ is a strong solution to system \eqref{eq29} with initial data ${ }^t (v_0^A,v_0^B , v_0^S)$ and $ { }^t (v_A ,v_B,v_S) \in L_{loc}^2( \mathbb{R}_+ \times \mathbb{R}^3_{+,-,0})$, it follows from Proposition \ref{prop43} to see that there exists a unique global-in-time strong solution $\theta = { }^t (\theta_A, \theta_B , \theta_S)$ to system \eqref{eq15} with initial data ${ }^t (\theta_0^A , \theta_0^B ,\theta_0^S)$. By the same argument as in the proof of Lemma \ref{lem44}, we see that the solution $\theta$ satisfies \eqref{eq16}. Therefore, Theorem \ref{thm11} is proved.
\end{proof}

\section{Appendix $(\rm{I})$: Derivation of Heat Equations in two unbounded domains $\Omega_A$, $\Omega_B$ and the interface $\Gamma$}\label{sect6}

In this section, we derive the heat equations \eqref{eq12} in two unbounded domains $\Omega_A$, $\Omega_B$ and the interface $\Gamma  (=\partial \Omega_A \cap \partial \Omega_B)$ by applying an energetic variational approach based on \cite{Kob20,Kob23b}. We consider the temperatures of three phase problems such as a soap bubble flying in the air, oil floating on water, or a melting ice (see Figure \ref{Fig1}) from an energetic point of view. 

Let us first introduce our settings. Let $\varphi \in BC^2(\mathbb{R}^2)$ and $n= n(x) = { }^t (n_1,n_2,n_3)$ the unit outer normal vector at $x \in \Gamma$ defined by \eqref{eq11}. Let $T \in (0, \infty]$. Define $\Omega_{A,T} = \Omega_A \times (0,T)$, $\Omega_{B,T} = \Omega_B \times (0,T) $, $\Omega_{S,T} = \Gamma \times (0,T)$, $\overline{\Omega}_{A,T} = \overline{ \Omega_A} \times [0,T)$, $\overline{\Omega}_{B,T} = \overline{ \Omega_B} \times [0,T)$, and $\overline{\Omega}_{S,T} = \Gamma \times [0,T)$. We assume that there is a fluid (or a substance) in domain $\Omega_A$, domain $\Omega_B$, and surface $\Gamma$, respectively. For $\sharp = A , B, S$, let $\varrho_\sharp = \varrho_\sharp (x,t)$, $v_\sharp = v_\sharp (x,t) = { }^t (v^\sharp_1, v^\sharp_2 , v^\sharp_3)$, $\theta_\sharp = \theta_\sharp (x,t)$, $\mu_\sharp = \mu_\sharp (x,t)$, and $\mathcal{C}_\sharp = \mathcal{C}_\sharp (x,t)$ be the \emph{density}, the \emph{velocity}, the \emph{temperature}, the \emph{thermal conductivity}, and the \emph{specific heat} of the fluid in $\Omega_{\sharp, T}$, respectively. For each $\sharp \in \{ A,B,S \}$, we define
\begin{align*}
& C_0^2 (\overline{\mathbb{R}^4_+}) = \{ f : \mathbb{R}^3 \times [0,\infty ) \to \mathbb{R} ; { \ } f = \mathfrak{f} \vert_{\mathbb{R}^3 \times [0,\infty )}, \mathfrak{f} \in C_0^2 (\mathbb{R}^4) \},\\
& C_0^2 (\overline{\Omega}_{\sharp,T}) = \{ f ;{ \ } f = \mathfrak{f}  \vert_{\overline{\Omega}_{\sharp,T} }, \mathfrak{f}  \in C_0^2 (\overline{\mathbb{R}^4_+}) \},\\
& C^2 (\overline{\Omega}_{\sharp,T}) = \{ f ;{ \ } f = \mathfrak{f}  \vert_{\overline{\Omega}_{\sharp,T} }, \mathfrak{f}  \in C^2 (\overline{\mathbb{R}^4_+}) \}.
\end{align*}
We assume that $\varrho_A$, $\mathcal{C}_A \in C^2 (\overline{\Omega}_{A,T})$, $v_1^A$, $v_2^A$, $v_3^A$, $\theta_A \in C_0^2 ( \overline{\Omega}_{A,T})$, $\varrho_B$, $\mathcal{C}_B \in C^2 (\overline{\Omega}_{B,T})$, $v_1^B$, $v_2^B$, $v_3^B$, $\theta_B \in C_0^2 ( \overline{ \Omega }_{B,T})$, $\varrho_S$, $\mathcal{C}_S \in C^2 (\overline{\Omega}_{S,T})$, $v_1^S$, $v_2^S$, $v_3^S$, $\theta_S \in C_0^2 ( \overline{\Omega}_{S,T})$, and that $\mu_A$, $\mu_B$, $\mu_S$ are three positive constants. We assume that for each $0 \leq t<T$, $\theta_A (\cdot , t) \in W^{2,2} (\Omega_A)$, $\theta_B (\cdot ,t) \in W^{2,2} (\Omega_B)$, and $\theta_S (\cdot ,t) \in W^{2,2} (\Gamma)$.

We first introduce the transport theorems.
\begin{definition}[Velocity fields, Transport theorems]\label{def61}{ \ }\\ We say that $(\Omega_{A,T}, \Omega_{B,T}, \Omega_{S,T})$ is \emph{flowed by the velocity fields} $(v_A,v_B ,v_S)$ if for each $0< t <T$, $\Phi_A \in C^1 (\Omega_{A,T})$, $\Phi_B \in C^1(\Omega_{B,T})$, $\Phi_S \in C^1 (\Omega_{S,T})$, and $\Lambda \subset \mathbb{R}^3$,
\begin{align}
 \frac{d}{d t} \int_{\Omega_A \cap \Lambda} \Phi_A (x,t) { \ }d x & = \int_{\Omega_A \cap \Lambda}\{ D_t^A \Phi_A + (\nabla \cdot v_A ) \Phi_A \} { \ }dx,\label{eq61}\\ 
 \frac{d}{d t} \int_{\Omega_B \cap \Lambda} \Phi_B (x,t) { \ }d x & = \int_{\Omega_B \cap \Lambda} \{ D_t^B  \Phi_B + (\nabla \cdot v_B ) \Phi_B \} { \ }dx,\label{eq62}\\ 
 \frac{d}{d t} \int_{\Gamma \cap \Lambda} \Phi_S (x,t) { \ }d\mathcal{H}_x^2 & = \int_{\Gamma \cap \Lambda} \{ D_t^S \Phi_S + (\nabla_\Gamma \cdot v_S ) \Phi_S \} { \ }d\mathcal{H}_x^2. \label{eq63}
\end{align}
Here $D_t^A f := \partial_t f + (v_A \cdot \nabla)f$, $D_t^B f := \partial_t f + (v_B \cdot \nabla)f$, and $D_t^S f := \partial_t f + (v_S \cdot \nabla_\Gamma )f$.
\end{definition}
\noindent We often call the above three equalities the \emph{transport theorems}. In particular, we call \eqref{eq63} the \emph{surface transport theorem}. The derivation of the surface transport theorem can be founded in \cite{Bet86, GSW89, DE07, Kob23a}.

Throughout this section, we assume that $(\Omega_{A,T}, \Omega_{B,T}, \Omega_{S,T})$ is \emph{flowed by the velocity fields} $(v_A,v_B ,v_S)$. From the transport theorems, we admit that the densities $(\varrho_A, \varrho_B, \varrho_S)$ of our model satisfy
\begin{equation}\label{eq64}
\begin{cases}
D_t^A  \varrho_A + (\nabla \cdot v_A ) \varrho_A = 0 & \text{ in } \Omega_A \times (0, T),\\
D_t^B \varrho_B + (\nabla \cdot v_B ) \varrho_B = 0 & \text{ in } \Omega_B \times (0, T),\\
D_t^S \varrho_S + (\nabla_\Gamma \cdot v_S ) \varrho_S = 0 & \text{ in } \Gamma \times (0, T).
\end{cases}
\end{equation}

Let us now make our model by an energetic variational approach. To this end, we consider the variation of energies dissipation due to thermal diffusion. For $t \in (0,T)$, we set
\begin{multline*}
E_{TD} = E_{TD}[\theta_A , \theta_B , \theta_S] = E_{TD}[\theta_A , \theta_B , \theta_S](t)\\
= - \int_{\Omega_A} \frac{\mu_A}{2} \vert \nabla \theta_A (x, t) \vert^2 d x - \int_{\Omega_B} \frac{\mu_B}{2} \vert \nabla \theta_B (x,t) \vert^2 d x - \int_{\Gamma} \frac{\mu_S}{2} \vert \nabla_\Gamma  \theta_S (x,t) \vert^2 d\mathcal{H}_x^2.
\end{multline*}
We call $E_{TD}$ our \emph{dissipation energies}. We study the variation of the dissipation energies $E_{TD}$ under the restrictions that
\begin{equation}\label{eq65}
\theta_A \vert_{\Gamma} = \theta_B \vert_{\Gamma} = \theta_S \text{ on } (0,T).
\end{equation}
Fix $t \in (0,T)$. For $- 1 < \varepsilon <1$, $\phi_A \in C_0^2 (\overline{ \Omega_A }) \cap W^{1,2} (\Omega_A) $, $\phi_B \in C_0^2 ( \overline{ \Omega_B } ) \cap W^{1,2} (\Omega_B)$, $\phi_S \in C_0^2 ( \Gamma) \cap W^{1,2} (\Gamma)$, we set $\theta_A^\varepsilon = \theta_A + \varepsilon \phi_A$, $\theta_B^\varepsilon = \theta_B + \varepsilon \theta_B$, $\theta_S^\varepsilon = \theta_S + \varepsilon \phi_S$. Based on \eqref{eq65}, we assume that $(\theta_A^\varepsilon , \theta_B^\varepsilon , \theta_S^\varepsilon)$ satisfies that for every $- 1 < \varepsilon < 1$
\begin{equation}\label{eq66}
\theta^\varepsilon_A \vert_{\Gamma} = \theta^\varepsilon_B \vert_{\Gamma} = \theta^\varepsilon_S.
\end{equation}
From \eqref{eq65} and \eqref{eq66}, we find that
\begin{equation}\label{eq67}
\phi_A \vert_{\Gamma} = \phi_B \vert_{\Gamma} = \phi_S.
\end{equation}
A direct calculation gives
\begin{multline*}
\frac{d}{d \varepsilon} \bigg\vert_{\varepsilon = 0} E_{TD}[\theta_A^\varepsilon , \theta_B^\varepsilon , \theta_S^\varepsilon] = - \int_{\Omega_A} \mu_A \nabla \theta_A \cdot \nabla \phi_A { \ }d x - \int_{\Omega_B} \mu_B \nabla \theta_B \cdot \nabla \phi_B { \ }d x\\ - \int_{\Gamma} \mu_S \nabla_\Gamma  \theta_S \cdot \nabla_\Gamma \phi_S { \ }d\mathcal{H}_x^2.
\end{multline*}
Using integration by parts (Lemma \ref{lem7015} and Proposition \ref{prop3021}) with \eqref{eq67}, we see that
\begin{multline}\label{eq68}
\frac{d}{d \varepsilon} \bigg\vert_{\varepsilon = 0} E_{TD}[\theta_A^\varepsilon , \theta_B^\varepsilon , \theta_S^\varepsilon] = \int_{\Omega_A} (\mu_A \Delta \theta_A ) \phi_A { \ }d x + \int_{\Omega_B} ( \mu_B \Delta \theta_B ) \phi_B {  \ } d x\\ + \int_{\Gamma} ( \mu_S \Delta_\Gamma \theta_S + \mu_A (n \cdot \nabla) \theta_A\vert_{x_3 =0} - \mu_B (n \cdot \nabla) \theta_B\vert_{x_3 =0} ) \phi_S { \ }d\mathcal{H}_x^2.
\end{multline}
Therefore, we have the following proposition.
\begin{proposition}[Variation of Dissipation Energies]\label{prop62}
Let $t \in (0,T)$ and $q_A, q_B,q_S \in C (\mathbb{R}^3)$. Assume that for every $\phi_A \in C_0^2 (\overline{\Omega_A}) \cap W^{1,2} (\Omega_A) $, $\phi_B \in C_0^2 (\overline{ \Omega_B}) \cap W^{1,2} (\Omega_B)$ and $\phi_S \in C_0^2 (\Gamma) \cap W^{1,2} (\Gamma)$ satisfying \eqref{eq67},
\begin{equation*}
\frac{d}{d \varepsilon} \bigg\vert_{\varepsilon = 0} E_{TD}[\theta_A^\varepsilon , \theta_B^\varepsilon , \theta_S^\varepsilon] = \int_{\Omega_A} q_A \phi_A { \ }d x + \int_{\Omega_B} q_B \phi_B { \ }d x + \int_{\Gamma} q_S \phi_S { \ }d\mathcal{H}_x^2.
\end{equation*}
Then
\begin{equation*}
\begin{cases}
q_A = \mu_A \Delta \theta_A & \text{ in }\Omega_A,\\
q_B = \mu_B \Delta \theta_B & \text{ in } \Omega_B,\\
q_S = \mu_S \Delta_\Gamma \theta_S + \mu_A (n \cdot \nabla) \theta_A \vert_{\Gamma} - \mu_B (n \cdot \nabla) \theta_B \vert_{\Gamma} & \text{ in } \Gamma.
\end{cases}
\end{equation*}
\end{proposition}

\begin{proof}[Proof of Proposition \ref{prop62}]
Let $t \in (0,T)$ and $q_A, q_B,q_S \in C (\mathbb{R}^3)$. We first consider the case when $\phi_S \equiv 0$ in $\Gamma$. By assumption and \eqref{eq68}, we see that for all $\phi_A \in C_0^2 (\Omega_A)$, $\phi_B \in C_0^2 ( \Omega_B)$,
\begin{equation*}
\int_{\Omega_A} (\mu_A \Delta \theta_A ) \phi_A { \ }d x + \int_{\Omega_B} ( \mu_B \Delta \theta_B ) \phi_B {  \ } d x = \int_{\Omega_A} q_A \phi_A { \ }d x + \int_{\Omega_B} q_B \phi_B { \ }d x.
\end{equation*}
This implies that $q_A = \mu_A \Delta \theta_A$ in $\Omega_A$ and $q_B = \mu_B \Delta \theta_B$ in $\Omega_B$.

Next we consider the case when $\phi_S \in C_0^2 (\Gamma)$. Let $\phi_A \in C_0^2 (\overline{\Omega_A})$, $\phi_B \in C_0^2 (\overline{ \Omega_B})$, $\phi_S \in C_0^2 (\Gamma)$ satisfying \eqref{eq67}. Since $q_A = \mu_A \Delta \theta_A$ in $\Omega_A$ and $q_B = \mu_B \Delta \theta_B$ in $\Omega_B$, it follows from \eqref{eq68} to find that 
\begin{equation*}
 \int_{\Gamma} ( \mu_S \Delta_\Gamma \theta_S + \mu_A (n \cdot \nabla) \theta_A\vert_{\Gamma} - \mu_B (n \cdot \nabla) \theta_B\vert_{\Gamma} ) \phi_S { \ }d\mathcal{H}_x^2 = \int_{\Gamma} q_S \phi_S { \ }d\mathcal{H}_x^2.
\end{equation*}
Since the above quality holds for all $\phi_S \in C_0^2 (\Gamma)$, we see that $q_S = \mu_S \Delta_\Gamma \theta_S + \mu_A (n \cdot \nabla) \theta_A \vert_{\Gamma} - \mu_B (n \cdot \nabla) \theta_B \vert_{\Gamma}$. Therefore, Proposition \ref{prop62} is proved.
\end{proof}

Let us derive equations \eqref{eq12}. From Proposition \ref{prop62} we set
\begin{equation}\label{eq69}
\begin{cases}
Q_A = Q_A(x,t) = \mu_A \Delta \theta_A,\\
Q_B = Q_B (x,t) = \mu_B \Delta \theta_B,\\
Q_S = Q_S ( x , t) = \mu_S \Delta_\Gamma \theta_S + \mu_A (n \cdot \nabla) \theta_A \vert_{\Gamma} - \mu_B (n \cdot \nabla) \theta_B \vert_{\Gamma}.
\end{cases}
\end{equation}

\noindent Now we assume that the time rate of change of the heat energy is equal to the force derived from the variation of energy dissipation due to thermal diffusion, that is, suppose that for every $0<t<T$ and $\Lambda \subset \mathbb{R}^3$,
\begin{align*}
\frac{d}{d t}\int_{\Omega_A \cap \Lambda} \varrho_A \mathcal{C}_A \theta_A { \ }d x = \int_{\mathbb{R}_{+}^3 \cap \Lambda} Q_A { \ }d x,\\
\frac{d}{d t}\int_{\Omega_B \cap \Lambda} \varrho_B \mathcal{C}_B \theta_B { \ }d x = \int_{\mathbb{R}_{-}^3 \cap \Lambda} Q_B { \ }d x,\\
\frac{d}{d t}\int_{\Gamma \cap \Lambda} \varrho_S \mathcal{C}_S \theta_S { \ }d\mathcal{H}_x^2 = \int_{\Gamma \cap \Lambda} Q_S { \ }d\mathcal{H}_x^2.
\end{align*}
Then we apply the transport theorems \eqref{eq61}-\eqref{eq63} and \eqref{eq64} to derive
\begin{equation}\label{eq6010}
\begin{cases}
\varrho_A D_t^A (\mathcal{C}_A \theta_A) = Q_A & \text{ in } \Omega_A \times (0, T),\\
\varrho_B D_t^B (\mathcal{C}_B \theta_B) = Q_B & \text{ in } \Omega_B \times (0, T),\\
\varrho_S D_t^S ( \mathcal{C}_S \theta_S) = Q_S & \text{ in } \Gamma \times (0, T).
\end{cases}
\end{equation}
Now we assume that $v_\sharp \equiv { }^t (0,0,0)$, $\mathcal{C}_\sharp$, $\varrho_\sharp$ are positive constants, and that $\varrho_\sharp \mathcal{C}_\sharp \equiv \rho_\sharp$ for some $\rho_\sharp \in \mathbb{R}_+$ $(\sharp \in \{ A , B , S \})$. Combining \eqref{eq69}, \eqref{eq6010}, and \eqref{eq65}, we obtain
\begin{equation*}
\begin{cases}
\rho_A \partial_t \theta_A = \mu_A \Delta \theta_A & \text{ in } \Omega_A \times (0, T),\\
\rho_B \partial_t \theta_B = \mu_B \Delta \theta_B & \text{ in } \Omega_B \times (0, T),\\
\rho_S \partial_t \theta_S = \mu_S \Delta_\Gamma \theta_S + \mu_A (n \cdot \nabla) \theta_A \vert_{\Gamma} - \mu_B (n \cdot \nabla) \theta_B \vert_{\Gamma} & \text{ in } \Gamma \times (0, T),\\
\theta_A \vert_{\Gamma} = \theta_B \vert_{\Gamma} = \theta_S & \text{ in }\Gamma \times (0, T).
\end{cases}
\end{equation*}
Therefore, we have our heat equations \eqref{eq12}.

\section{Appendix $(\rm{II})$: Representation Formulas for Differential Operators}\label{sect7}

In this section, we study several representation formulas for differential operators in two unbounded domains $\Omega_A,\Omega_B$ and the interface $\Gamma$. Let $\varphi \in BC^2(\mathbb{R}^2)$. Let us recall that $\alpha,\beta, \alpha' , \beta' \in \{ 1,2\}$, $i,j,\ell, i',j',\ell' \in \{ 1,2,3 \}$, $\varphi_\alpha = \varphi_\alpha(x_h) = \partial \varphi/{\partial x_\alpha}$, $\varphi_{\alpha \beta} = \varphi_{\alpha \beta}(x_h) = \partial^2 \varphi/{\partial x_\alpha \partial x_\beta}$,
\begin{align*}
\Gamma &= \{ x \in \mathbb{R}^3; x_3 = \varphi (x_1,x_2) \},{ \ }\mathbb{R}^2 = \{ X = { }^t(X_1,X_2); X_1,X_2 \in \mathbb{R} \},\\
\Omega_A &= \{ x \in \mathbb{R}^3; x_3 > \varphi (x_1,x_2) \},{ \ }\mathbb{R}^3_+ = \{ y \in \mathbb{R}^3; y_3 > 0 \},\\
\Omega_B &= \{ x \in \mathbb{R}^3; x_3 < \varphi (x_1,x_2) \},{ \ }\mathbb{R}^3_- = \{ z \in \mathbb{R}^3; z_3 < 0 \},\\
\overline{\Omega_A} &= \{ x \in \mathbb{R}^3; x_3 \geq \varphi (x_1,x_2) \},{ \ } \overline{\mathbb{R}^3_+} = \{ y \in \mathbb{R}^3; y_3 \geq 0 \},\\
\overline{\Omega_B} &= \{ x \in \mathbb{R}^3; x_3 \leq \varphi (x_1,x_2) \},{ \ } \overline{\mathbb{R}^3_-} = \{ z \in \mathbb{R}^3; z_3 \leq 0 \}.
\end{align*}
One goal of this section is to derive the following three divergence theorems:
\begin{align*}
\int_{\Omega_A } \nabla \cdot F_A { \ }d x &= - \int_\Gamma F_A \cdot n { \ } d\mathcal{H}_x^2,\\
\int_{\Omega_B } \nabla \cdot F_B { \ }d x &= \int_\Gamma F_B \cdot n { \ } d\mathcal{H}_x^2,\\
\int_{\Gamma } \nabla_\Gamma \cdot F_S { \ }d \mathcal{H}_x^2 & = - \int_\Gamma H_\Gamma (F_S \cdot n) { \ } d\mathcal{H}_x^2,
\end{align*}
where $F_A = { }^t (F_1^A , F_2^A, F_3^A) \in [C_0^1 (\overline{\Omega_A})]^3$, $F_B= { }^t (F_1^B , F_2^B, F_3^B)  \in [C_0^1 ( \overline{\Omega_B})]^3$, $F_S= { }^t (F_1^S , F_2^S, F_3^S)  \in [C_0^1 (\Gamma)]^3$, the symbol $d \mathcal{H}_x^2$ denotes the $2$-dimensional Hausdorff measure, and $n = n(x)= { }^t(n_1,n_2,n_3)$ is the unit outer normal vector at $x \in \Gamma$ defined by \eqref{eq71}. See subsection \ref{subsec71} for differential operator $\nabla_\Gamma$ and the mean curvature $H_\Gamma$.

In subsections \ref{subsec71}, \ref{subsec72}, and \ref{subsec73}, we study representation formulas for the surface $\Gamma$, the domain $\Omega_A$, and the domain $\Omega_B$, respectively. In subsection \ref{subsec74}, we derive the three divergence theorems, and introduce their applications. Note that we often use the Einstein summation convention in this section, i.e.
\begin{equation*}
g_{\alpha \beta} g^\beta = \sum_{\beta = 1}^2 g_{\alpha \beta}g^\beta \text{ and } g^{i j} g_{i j} = \sum_{i,j = 1}^3 g^{i j} g_{i j}. 
\end{equation*}

\subsection{Representation Formulas for the Surface $\Gamma$}\label{subsec71}
Let us study some representation formulas for the surface $\Gamma$. For all $X = { }^t(X_1,X_2) \in \mathbb{R}^2$, we set
\begin{equation*}
\hat{x}_S = \hat{x}_S (X) =
\begin{pmatrix}
\hat{x}_1^S\\
\hat{x}_2^S\\
\hat{x}_3^S
\end{pmatrix} :=
\begin{pmatrix}
X_1\\
X_2\\
\varphi (X_1,X_2)
\end{pmatrix}.
\end{equation*}
It is clear that the mapping $\hat{x}_S: \mathbb{R}^2 \to \Gamma$ is bijective and that $\Gamma = \{ x \in \mathbb{R}^3; x = \hat{x}_S(X), X \in \mathbb{R}^2 \}$ (see Lemma \ref{lem22}). Define $g^S_1 =g^S_1(X)$, $g^S_2=g^S_2(X)$ by
\begin{equation*}
g^S_1(X) := \frac{\partial \hat{x}_S}{\partial X_1} = \begin{pmatrix}
1\\
0\\
\varphi_1
\end{pmatrix},{ \ }
g^S_2(X) := \frac{\partial \hat{x}_S}{\partial X_2} = \begin{pmatrix}
0\\
1\\
\varphi_2
\end{pmatrix}.
\end{equation*}
Here $\varphi_1 = \varphi_1(X) = \partial \varphi/{\partial X_1}$, $\varphi_2 = \varphi_2(X) = \partial \varphi/{\partial X_2}$. Direct calculations give
\begin{equation*}
g_ 1^S \times g_2^S = 
\begin{pmatrix}
- \varphi_1\\
- \varphi_2\\
1
\end{pmatrix},{ \ }\vert g_1^S \times g_2^S \vert = \sqrt{1 + \varphi_1^2 + \varphi_2^2}.
\end{equation*}
Set $g^S_{\alpha \beta} =g^S_{\alpha \beta}(X) := g^S_\alpha \cdot g^S_\beta$, $(g_S^{\alpha \beta})_{2 \times 2} := (g^S_{\alpha \beta})_{2 \times 2}^{-1}$, and $G_S = G_S (X) := {\rm{det}} (g^S_{\alpha \beta})_{2 \times 2}$. We easily check that
\begin{align*}
(g^S_{\alpha \beta})_{2 \times 2 } & =
\begin{pmatrix}
g^S_{11} & g^S_{12}\\
g^S_{21} & g^S_{22}
\end{pmatrix} =
\begin{pmatrix}
1 + \varphi_1^2 & \varphi_1 \varphi_2\\
\varphi_1 \varphi_2 & 1 + \varphi_2^2
\end{pmatrix},\\
(g_S^{\alpha \beta})_{2 \times 2 }  & =
\begin{pmatrix}
g_S^{11} & g_S^{12}\\
g_S^{21} & g_S^{22}
\end{pmatrix}= \frac{1}{1 + \varphi_1^2 + \varphi_2^2}
\begin{pmatrix}
1 + \varphi_2^2 & - \varphi_1 \varphi_2\\
-\varphi_1 \varphi_2 & 1 + \varphi_1^2
\end{pmatrix},
\end{align*}
and that $G_S = 1 + \varphi_1^2 + \varphi_2^2 = \vert g_1^S \times g_2^S \vert^2$. Set
\begin{equation*}
g_S^\alpha = g_S^\alpha(X) := g_S^{\alpha \beta}g^S_\beta \bigg(= \sum_{\beta=1}^2 g_S^{\alpha \beta}g^S_\beta \bigg).
\end{equation*}
It is clear that
\begin{equation*}
g_S^1(X) = \frac{1}{1 + \varphi_1^2 + \varphi_2^2}\begin{pmatrix}
1 + \varphi_2^2\\
- \varphi_1 \varphi_2\\
\varphi_1
\end{pmatrix},{ \ }
g_S^2(X) = \frac{1}{1 + \varphi_1^2 + \varphi_2^2}\begin{pmatrix}
- \varphi_1 \varphi_2\\
1 + \varphi_1^2\\
\varphi_2
\end{pmatrix},
\end{equation*}
and $g^{\alpha \beta}_S = g_S^\alpha \cdot g_S^\beta$. Let $n = n(x)= { }^t (n_1,n_2,n_3)$ be the unit outer normal vector at $x = { }^t (x_1,x_2,x_3) \in \Gamma$ defined by
\begin{equation}\label{eq71}
n(x_1,x_2, x_3) = \frac{ g_1^S \times g_2^S }{ \vert g_1^S \times g_2^S \vert } = \frac{1}{\sqrt{1 + \varphi_1^2 + \varphi_2^2}} 
\begin{pmatrix}
- \varphi_1\\
- \varphi_2\\
1
\end{pmatrix}.
\end{equation}
For all $f_S \in C^1 (\mathbb{R}^3)$, $F_S = { }^t(F_1^S,F_2^S ,F_3^S) \in [C^1(\mathbb{R}^3)]^3$, and $x \in \Gamma$, we define $\partial_j^\Gamma f_S := \partial_j f - n_j (n \cdot \nabla) f_S$ and ${\rm{div}}_\Gamma F_S := \partial_1^\Gamma F_1^S + \partial_2^\Gamma F_2^S + \partial_3^\Gamma F_3^S$. Write $\nabla_\Gamma = { }^t(\partial_1^\Gamma , \partial_2^\Gamma , \partial_3^\Gamma)$ and $\Delta_\Gamma = (\partial_1^\Gamma)^2 +(\partial_2^\Gamma)^2 +(\partial_3^\Gamma)^2$. From \cite[Lemma 3.1 and Theorem 2.4]{Kob23a}(see also \cite{Jos11} and \cite{Cia05}), we have the following representation formulas.
\begin{lemma}\label{lem71}
$(\rm{i})$ For all $f_S , f_S^\natural , f_S^\flat \in C^2 (\mathbb{R}^3)$,
\begin{align}
\int_\Gamma f_S(x) { \ }d \mathcal{H}^2_x &= \int_{\mathbb{R}^2} \hat{f}_S(X) \sqrt{G_S(X)} { \ }dX,\\
\int_\Gamma \partial^\Gamma_j f_S { \ }d \mathcal{H}^2_x &= \int_{\mathbb{R}^2} g_S^{\alpha \beta} \frac{\partial \hat{x}^S_j }{\partial X_\alpha} \frac{ \partial \hat{f}_S }{\partial X_\beta} \sqrt{G_S} { \ }dX,\\
\int_\Gamma \partial_i^\Gamma \partial^\Gamma_j f_S { \ }d \mathcal{H}^2_x &= \int_{\mathbb{R}^2}g_S^{\alpha' \beta'} \frac{\partial \hat{x}^S_i }{\partial X_{\alpha'}}  \frac{\partial}{\partial X_{\beta'}}\bigg( g_S^{\alpha \beta} \frac{\partial \hat{x}^S_j }{\partial X_\alpha} \frac{ \partial \hat{f}_S }{\partial X_\beta} \bigg) \sqrt{G_S} { \ }dX,\\
\int_\Gamma \nabla_\Gamma f_S^\natural \cdot \nabla_\Gamma f_S^\flat { \ }d \mathcal{H}^2_x & = \int_{\mathbb{R}^2} g_S^{\alpha \beta} \frac{\partial \hat{f}_S^\natural }{\partial X_\alpha} \frac{ \partial \hat{f}_S^\flat }{\partial X_\beta} \sqrt{G_S} { \ }dX,\\
\int_\Gamma \Delta_\Gamma f_S { \ }d \mathcal{H}^2_x & = \int_{\mathbb{R}^2} \frac{1}{\sqrt{G_S}} \frac{\partial}{\partial X_\alpha} \bigg( \sqrt{G_S} g_S^{\alpha \beta} \frac{\partial \hat{f}_S}{\partial X_\beta} \bigg)  \sqrt{G_S} { \ }dX,
\end{align}
where $\hat{f}_S = \hat{f}_S (X) = f_S (\hat{x}_S(X) )$.\\
$(\rm{ii})$ For all $F_S = { }^t (F_1^S, F_2^S,F_3^S) \in [C^1 (\mathbb{R}^3)]^3$ and $\Phi_S \in C^1 (\mathbb{R}^4)$,
\begin{align}
\int_\Gamma \nabla_\Gamma  \cdot F_S { \ }d \mathcal{H}^2_x &= \int_{\mathbb{R}^2} g_S^{\alpha} \cdot \frac{\partial \hat{F}_S }{\partial X_\alpha} \sqrt{G_S} { \ }dX,\label{eq77}\\
\int_\Gamma \frac{\partial \Phi_S}{\partial t} { \ }d \mathcal{H}_x^2 &= \int_{\mathbb{R}^2} \frac{\partial \hat{\Phi}_S}{\partial t} \sqrt{G_S} { \ }d X. 
\end{align}
Here $\hat{F}_S = \hat{F}_S (X) = F_S (\hat{x}_S(X) )$ and $\hat{\Phi}_S = \hat{\Phi}_S (X,t) = \Phi_S (\hat{x}_S(X),t)$.
\end{lemma}

\begin{remark}\label{rem72}
$(\rm{i})$ We easily check that for $\psi_S \in C^1(\mathbb{R}^2)$,
\begin{align*}
g_S^{\alpha \beta} \frac{\partial \hat{x}_1^S}{\partial X_\alpha} \frac{\partial \psi_S}{\partial X_\beta} & = \frac{1 + \varphi_2^2}{G_S} \frac{\partial \psi_S}{\partial X_1} -  \frac{ \varphi_1 \varphi_2}{G_S} \frac{\partial \psi_S}{\partial X_2},\\
g_S^{\alpha \beta} \frac{\partial \hat{x}_2^S}{\partial X_\alpha} \frac{\partial \psi_S}{\partial X_\beta} & = - \frac{ \varphi_1 \varphi_2}{G_S} \frac{\partial \psi_S}{\partial X_1} +  \frac{ 1 + \varphi_1^2 }{G_S} \frac{\partial \psi_S}{\partial X_2},\\
g_S^{\alpha \beta} \frac{\partial \hat{x}_3^S}{\partial X_\alpha} \frac{\partial \psi_S}{\partial X_\beta} & = \frac{ \varphi_1}{G_S} \frac{\partial \psi_S}{\partial X_1} +  \frac{ \varphi_2}{G_S} \frac{\partial \psi_S}{\partial X_2}.
\end{align*}
$(\rm{ii})$ Let $\kappa_S >0$. For $\psi_S \in C^2 (\mathbb{R}^2)$, we set
\begin{equation}\label{eq79}
\mathscr{L}_3 \psi_S := -\frac{\kappa_S}{\sqrt{G_S}} \frac{\partial}{\partial X_\alpha} \bigg(\sqrt{G_S} g_S^{\alpha \beta} \frac{\partial \psi_S}{\partial X_\beta} \bigg).
\end{equation}
A direct calculation gives
\begin{equation*}
\mathscr{L}_3 \psi_S  = - \kappa_S \bigg( g_S^{\alpha \beta} \frac{\partial^2 \psi_S}{\partial X_\alpha \partial X_\beta} + \frac{\partial g_S^{\alpha \beta}}{\partial X_\alpha} \frac{\partial \psi_S}{\partial X_\beta} + \frac{1}{2} \frac{g_S^{\alpha \beta}}{G_S} \frac{\partial G_S }{\partial X_\alpha} \frac{\partial \psi_S}{\partial X_\beta} \bigg).
\end{equation*}
It is easy to check that
\begin{align*}
g_S^{\alpha \beta} \frac{\partial^2 \psi_S}{\partial X_\alpha \partial X_\beta} &= \frac{1 + \varphi_2^2}{G_S} \frac{\partial^2 \psi_S}{\partial X_1^2} + \frac{1+ \varphi_1^2}{G_S}\frac{\partial^2 \psi_S}{\partial X_2^2} - \frac{2 \varphi_1 \varphi_2}{G_S}\frac{\partial^2 \psi_S}{\partial X_1 \partial X_2}\\
&= \Delta_X \psi_S - \frac{\varphi_1^2}{G_S} \frac{\partial^2 \psi_S}{\partial X_1^2} - \frac{\varphi_2^2}{G_S}\frac{\partial^2 \psi_S}{\partial X_2^2} - \frac{2 \varphi_1 \varphi_2}{G_S}\frac{\partial^2 \psi_S}{\partial X_1 \partial X_2},
\end{align*}
and that
\begin{multline*}
\frac{\partial g_S^{\alpha \beta}}{\partial X_\alpha} \frac{\partial \psi_S}{\partial X_\beta} + \frac{1}{2}\frac{g_S^{\alpha \beta}}{G_S} \frac{\partial G_S }{\partial X_\alpha} \frac{\partial \psi_S}{\partial X_\beta}
 = \frac{\varphi_1 ( 2 \varphi_1 \varphi_2 \varphi_{12} - \varphi_{11} - \varphi_{22} -  \varphi_2^2 \varphi_{11} -  \varphi_1^2 \varphi_{22}) }{G^2_S} \frac{\partial \psi_S}{\partial X_1}\\
  + \frac{ \varphi_2 ( 2 \varphi_1 \varphi_2 \varphi_{12} - \varphi_{11} - \varphi_{22} -  \varphi_2^2 \varphi_{11} -  \varphi_1^2 \varphi_{22}) }{G^2_S}  \frac{\partial \psi_S}{\partial X_2},
\end{multline*}
where $\varphi_{\alpha \beta} = \partial^2 \varphi/{\partial X_\alpha \partial X_\beta }$. Thus, we find that
\begin{equation}\label{eq7010}
\mathscr{L}_3 \psi_S = -\kappa_S \Delta_X \psi_S + \mathscr{B}_3 \psi_S.
\end{equation}
Here
\begin{multline}\label{eq7011}
\mathscr{B}_3 \psi_S := \kappa_S \bigg( \frac{\varphi_1^2}{G_S} \frac{\partial^2 \psi_S}{\partial X_1^2} + \frac{\varphi_2^2}{G_S}\frac{\partial^2 \psi_S}{\partial X_2^2} + \frac{2 \varphi_1 \varphi_2}{G_S}\frac{\partial^2 \psi_S}{\partial X_1 \partial X_2}\\
+ \frac{ \varphi_1 ( \varphi_{11} + \varphi_{22} + \varphi_2^2 \varphi_{11} +  \varphi_1^2 \varphi_{22} - 2 \varphi_1 \varphi_2 \varphi_{12}) }{G^2_S} \frac{\partial \psi_S}{\partial X_1}\\
+ \frac{ \varphi_2 ( \varphi_{11} + \varphi_{22} + \varphi_2^2 \varphi_{11} +  \varphi_1^2 \varphi_{22} - 2 \varphi_1 \varphi_2 \varphi_{12}) }{G^2_S}  \frac{\partial \psi_S}{\partial X_2} \bigg).
\end{multline}
$(\rm{iii})$ Let $H_\Gamma = H_\Gamma (x)$ be the mean curvature at $x \in \Gamma$ in the direction $n$ defined by $H_\Gamma = - {\rm{div}}_\Gamma n$ at $x \in \Gamma$. From \eqref{eq71} and \eqref{eq77}, we observe that
\begin{align*}
\int_\Gamma H_\Gamma { \ }d \mathcal{H}_x^2 & = \int_{\Gamma} (- {\rm{div}}_\Gamma n ) { \ }d \mathcal{H}_x^2\\
 &= \int_{\mathbb{R}^2} - g_S^\alpha \cdot \frac{\partial}{\partial X_\alpha} \bigg( \frac{g_1^S \times g_2^S}{ \vert g_1^S \times g_2^S \vert} \bigg) \sqrt{G_S} { \ }d X\\
 &= \int_{\mathbb{R}^2}\bigg( \frac{\varphi_{11} + \varphi_{22} + \varphi_2^2 \varphi_{11} + \varphi_1^2 \varphi_{22} -  2 \varphi_1 \varphi_2 \varphi_{12}}{ G_S \sqrt{G_S}} \bigg) \sqrt{G_S} { \ }d X.
\end{align*}
Set $\hat{H}_\Gamma = \hat{H}_\Gamma (X)$ as follows:
\begin{equation}\label{eq7012}
\hat{H}_\Gamma = \frac{\varphi_{11} + \varphi_{22} + \varphi_2^2 \varphi_{11} + \varphi_1^2 \varphi_{22} -  2 \varphi_1 \varphi_2 \varphi_{12}}{ G_S \sqrt{G_S}}.
\end{equation}
We also call $\hat{H}_\Gamma$ the \emph{mean curvature}. We easily check that
\begin{equation*}
\sup_{X \in \mathbb{R}^2} \vert \hat{H}_\Gamma (X) \vert \leq C (\Vert \nabla_X \varphi \Vert_{W^{1,\infty} (\mathbb{R}^2)}) < + \infty.
\end{equation*}
$(\rm{iv})$ From Lemma \ref{lem71}, we find that for each $f_S \in C^2(\mathbb{R}^3)$,
\begin{align*}
\int_\Gamma \vert f_S (x) \vert^2 { \ }d \mathcal{H}^2_x &= \int_{\mathbb{R}^2} \vert \hat{f}_S(X) \vert^2 \sqrt{G_S(X)} { \ }dX,\\
\int_\Gamma \vert \nabla_\Gamma f_S \vert^2 { \ }d \mathcal{H}^2_x & = \int_{\mathbb{R}^2} g_S^{\alpha \beta} \frac{\partial \hat{f}_S}{\partial X_\alpha} \frac{ \partial \hat{f}_S}{\partial X_\beta} \sqrt{G_S} { \ }dX,\\
\int_\Gamma \vert \partial_i^\Gamma \partial^\Gamma_j f_S \vert^2 { \ }d \mathcal{H}^2_x &= \int_{\mathbb{R}^2} \bigg\vert g_S^{\alpha' \beta'} \frac{\partial \hat{x}^S_i }{\partial X_{\alpha'}}  \frac{\partial}{\partial X_{\beta'}}\bigg( g_S^{\alpha \beta} \frac{\partial \hat{x}^S_j }{\partial X_\alpha} \frac{ \partial \hat{f}_S }{\partial X_\beta} \bigg) \bigg\vert^2 \sqrt{G_S} { \ }dX,\\
\int_\Gamma \vert \Delta_\Gamma f_S \vert^2 { \ }d \mathcal{H}^2_x & = \int_{\mathbb{R}^2} \bigg\vert \frac{1}{\sqrt{G_S}} \frac{\partial}{\partial X_\alpha} \bigg( \sqrt{G_S} g_S^{\alpha \beta} \frac{\partial \hat{f}_S}{\partial X_\beta} \bigg) \bigg\vert^2 \sqrt{G_S} { \ }dX,
\end{align*}
where $\hat{f}_S = \hat{f}_S(X) = f_S (\hat{x}_S(X))$. For $\psi_S \in L_{loc}^1 (\mathbb{R}^2)$, we define
\begin{align}
\Vert \psi_S \Vert_{\mathcal{L}^2(\Gamma)} &:= \bigg( \int_{\mathbb{R}^2} \vert \psi_S \vert^2 \sqrt{G_S} { \ }dX\bigg)^{1/2},\\
\Vert \psi_S \Vert_{\dot{\mathcal{W}}^{1,2} (\Gamma) } & := \bigg( \int_{\mathbb{R}^2} g_S^{\alpha \beta} \frac{\partial \psi_S}{\partial X_\alpha} \frac{ \partial \psi_S}{\partial X_\beta} \sqrt{G_S} { \ }dX \bigg)^{1/2},\\
\Vert \psi_S \Vert_{\mathring{\mathcal{W}}^{2,2} (\Gamma) }  & := \bigg( \int_{\mathbb{R}^2} \bigg\vert \frac{1}{\sqrt{G_S}} \frac{\partial}{\partial X_\alpha} \bigg( \sqrt{G_S} g_S^{\alpha \beta} \frac{\partial \psi_S }{\partial X_\beta} \bigg) \bigg\vert^2 \sqrt{G_S} { \ }dX \bigg)^{1/2},\label{eq7015}\\
\Vert \psi_S \Vert_{\mathcal{W}^{2,2} (\Gamma) } & := \bigg( \Vert \psi_S \Vert_{\mathcal{L}^2(\Gamma)}^2 +\Vert \psi_S \Vert_{\dot{\mathcal{W}}^{1,2} (\Gamma) }^2 + \Vert \psi_S \Vert_{\mathring{\mathcal{W}}^{2,2} (\Gamma) }^2 \bigg)^{1/2},
\end{align}
and
\begin{equation}\label{eq7017}
\Vert \psi_S \Vert_{\dot{\mathcal{W}}^{2,2} (\Gamma) } := \bigg( \sum_{i,j=1}^3 \int_{\mathbb{R}^2} \bigg\vert g_S^{\alpha' \beta'} \frac{\partial \hat{x}^S_i }{\partial X_{\alpha'}}  \frac{\partial}{\partial X_{\beta'}}\bigg( g_S^{\alpha \beta} \frac{\partial \hat{x}^S_j }{\partial X_\alpha} \frac{ \partial \psi_S }{\partial X_\beta} \bigg) \bigg\vert^2 \sqrt{G_S} { \ }dX \bigg)^{1/2}.
\end{equation}
$(\rm{v})$ From Lemma \ref{lem71}, we find that for each $f_S \in C^1(\mathbb{R}^3)$,
\begin{align*}
\int_\Gamma \vert f_S (x) \vert { \ }d \mathcal{H}^2_x &= \int_{\mathbb{R}^2} \vert \hat{f}_S \vert \sqrt{G_S} { \ }dX,\\
\int_\Gamma \vert \partial^\Gamma_j f_S \vert { \ }d \mathcal{H}^2_x &= \int_{\mathbb{R}^2} \bigg\vert g_S^{\alpha \beta} \frac{\partial \hat{x}^S_j }{\partial X_\alpha} \frac{ \partial \hat{f}_S }{\partial X_\beta} \bigg\vert \sqrt{G_S} { \ }dX,
\end{align*}
where $\hat{f}_S = \hat{f}_S(X) = f_S (\hat{x}(X))$. For $\psi_S \in L_{loc}^1 (\mathbb{R}^2)$, we define
\begin{align}
\Vert \psi_S \Vert_{\mathcal{L}^1(\Gamma)} &:=  \int_{\mathbb{R}^2} \vert \psi_S \vert \sqrt{G_S} { \ }dX,\label{eq7018}\\
\Vert \psi_S \Vert_{\dot{\mathcal{W}}^{1,1} (\Gamma) } & := \sum_{j=1}^3 \int_{\mathbb{R}^2} \bigg\vert g_S^{\alpha \beta} \frac{\partial \hat{x}^S_j }{\partial X_\alpha} \frac{ \partial \psi_S }{\partial X_\beta} \bigg\vert \sqrt{G_S} { \ }dX,\\
\Vert \psi_S \Vert_{\mathcal{W}^{1,1} (\Gamma) } & := \Vert \psi_S \Vert_{\mathcal{L}^1 (\Gamma) } + \Vert \psi_S \Vert_{\dot{\mathcal{W}}^{1,1} (\Gamma) }.
\end{align}
\end{remark}

Let us investigate the relationships between $\Vert \cdot \Vert_{\mathcal{L}^2(\Gamma)}$ and $\Vert \cdot \Vert_{L^2(\mathbb{R}^2)}$, between $\Vert \cdot \Vert_{\dot{\mathcal{W}}^{1,2}(\Gamma)}$ and $\Vert \nabla_X \cdot \Vert_{L^2(\mathbb{R}^2)}$, and between $\Vert \cdot \Vert_{\mathcal{W}^{2,2}(\Gamma)}$ and $\Vert \cdot \Vert_{W^{2,2}(\mathbb{R}^2)}$.
\begin{lemma}\label{lem73}Assume that $\Vert \nabla_X \varphi \Vert_{L^\infty (\mathbb{R}^2)} \leq 1/2$. Then,\\
$(\rm{i})$ For all $\psi_S \in L^2(\mathbb{R}^2)$
\begin{equation}\label{eq7021}
\Vert \psi_S \Vert_{L^2(\mathbb{R}^2)}^2 \leq  \Vert \psi_S \Vert_{\mathcal{L}^2(\Gamma)}^2 \leq 2 \Vert \psi_S \Vert_{L^2(\mathbb{R}^2)}^2.
\end{equation}
$(\rm{ii})$ For all $\psi_S \in W^{1,2}(\mathbb{R}^2)$
\begin{equation}\label{eq7022}
\frac{1}{4} \Vert \nabla_X \psi_S \Vert_{L^2(\mathbb{R}^2)}^2 \leq \Vert \psi_S \Vert_{\dot{\mathcal{W}}^{1,2}(\Gamma)}^2  \leq 4 \Vert \nabla_X \psi_S \Vert_{L^2(\mathbb{R}^2)}^2. 
\end{equation}
$(\rm{iii})$ There is $C = C( \Vert \nabla_X^2 \varphi \Vert_{L^{\infty} (\mathbb{R}^2)} ) > 0 $ such that for all $\psi_S \in W^{2,2}(\mathbb{R}^2)$
\begin{align}
\Vert \psi_S \Vert_{\dot{\mathcal{W}}^{2,2}(\Gamma)} & \leq C \Vert \psi_S \Vert_{W^{2,2} (\mathbb{R}^2)},\label{eq7023}\\
\Vert \psi_S \Vert_{\mathring{\mathcal{W}}^{2,2}(\Gamma)} & \leq C \Vert \psi_S \Vert_{W^{2,2} (\mathbb{R}^2)}.\label{eq7024}
\end{align}
$(\rm{iv})$ There is $C = C( \Vert \nabla_X^2 \varphi \Vert_{L^{\infty} (\mathbb{R}^2)} ) > 0 $ such that for all $\psi_S \in W^{2,2}(\mathbb{R}^2)$
\begin{equation}\label{eq7025}
\Vert \nabla_X^2 \psi_S \Vert_{L^2 (\mathbb{R}^2)} \leq C( \Vert \psi_S \Vert_{\dot{\mathcal{W}}^{1,2}(\Gamma)} + \Vert \psi_S \Vert_{\mathring{\mathcal{W}}^{2,2}(\Gamma)} ).
\end{equation}
$(\rm{v})$ There is $C = C( \Vert \nabla_X^2 \varphi \Vert_{L^{\infty} (\mathbb{R}^2)} ) > 0 $ such that for all $\psi_S \in W^{2,2}(\mathbb{R}^2)$
\begin{equation*}
C^{-1} \Vert \psi_S \Vert_{W^{2,2} (\mathbb{R}^2)} \leq \Vert \psi_S \Vert_{\mathcal{W}^{2,2}(\Gamma)} \leq C \Vert \psi_S \Vert_{W^{2,2} (\mathbb{R}^2)} .
\end{equation*}
\end{lemma}
\noindent See also Lemma \ref{lem37}.

\begin{proof}[Proof of Lemma \ref{lem73}]
Fix $\psi_S \in L^2(\mathbb{R}^2)$. Assume that $\Vert \nabla_X \varphi \Vert_{L^\infty (\mathbb{R}^2)} \leq 1/2$. By assumption, we see that for every $X \in \mathbb{R}^2$,
\begin{equation*}
1 \leq 1 + \varphi_1^2 + \varphi_2^2 \leq 4.
\end{equation*}
This shows that for each $X \in \mathbb{R}^2$,
\begin{align}
1 \leq \sqrt{G_S} \leq 2,\label{eq7026}\\
\frac{1}{4} \leq \frac{1}{G_S} \leq 1. \label{eq7027}
\end{align}

We first show $(\rm{i})$. By \eqref{eq7026}, we check that for almost all $X \in \mathbb{R}^2$
\begin{equation*}
\vert \psi_S (X) \vert^2 \leq \vert \psi_S (X) \vert^2 \sqrt{G_S(X)} \leq 2 \vert \psi_S (X) \vert^2.
\end{equation*}
Integrating over $\mathbb{R}^2$ with respect to $X$, we have \eqref{eq7021}.

Next, we prove $(\rm{ii})$. Assume that $\psi_S \in W^{1,2} (\mathbb{R}^2)$. By the definition of $g_S^{\alpha \beta}$, we find that
\begin{multline}\label{eq7028}
\int_{\mathbb{R}^2} g_S^{\alpha \beta} \frac{\partial \psi_S }{\partial X_\alpha} \frac{ \partial \psi_S }{\partial X_\beta} \sqrt{G_S} { \ }dX\\
= \int_{\mathbb{R}^2} \bigg( \frac{1 + \varphi_2^2}{G_S} \bigg\vert \frac{\partial \psi_S}{\partial X_1} \bigg\vert^2 + \frac{1 + \varphi_1^2}{G_S} \bigg\vert \frac{\partial \psi_S}{\partial X_2} \bigg\vert^2   - \frac{2 \varphi_1 \varphi_2}{G_S} \frac{\partial \psi_S }{\partial X_1} \frac{ \partial \psi_S }{\partial X_2} \bigg) \sqrt{G_S} { \ }dX\\
 = \int_{\mathbb{R}^2}\bigg\{ \vert \nabla_X \psi_S \vert^2 + \bigg( \varphi_2 \frac{\partial \psi_S}{\partial X_1} - \varphi_1 \frac{\partial \psi_S}{\partial X_2} \bigg)^2  \bigg\} \frac{\sqrt{G_S}}{G_S} { \ }dX.
\end{multline}
Applying \eqref{eq7026}, \eqref{eq7027}, and \eqref{eq7028}, we check that
\begin{equation}\label{eq7029}
\int_{\mathbb{R}^2} g_S^{\alpha \beta} \frac{\partial \psi_S }{\partial X_\alpha} \frac{ \partial \psi_S }{\partial X_\beta} \sqrt{G_S} { \ }dX\\ \geq \frac{1}{4} \int_{\mathbb{R}^2} \vert \nabla_X \psi_S \vert^2 { \ }dX.
\end{equation}
We also check that
\begin{equation}\label{eq7030}
\int_{\mathbb{R}^2} g_S^{\alpha \beta} \frac{\partial \psi_S }{\partial X_\alpha} \frac{ \partial \psi_S }{\partial X_\beta} \sqrt{G_S} { \ }dX\\ \leq 4 \int_{\mathbb{R}^2} \vert \nabla_X \psi_S \vert^2 { \ }dX.
\end{equation}
Here we used the fact that $(a-b)^2 \leq 2a^2+ 2 b^2$ and $\Vert \nabla_X \varphi \Vert_{L^\infty (\mathbb{R}^2)} \leq 1/2$. From \eqref{eq7029} and \eqref{eq7030}, we have \eqref{eq7022}.

Thirdly, we show $(\rm{iii})$. Assume that $\psi_S \in W^{2,2} (\mathbb{R}^2)$. Direct calculations give
\begin{align*}
\frac{\partial}{\partial X_\alpha} \bigg( \frac{1}{G_S} \bigg) & = \frac{-2 \varphi_1 \varphi_{1 \alpha} - 2 \varphi_2 \varphi_{\alpha 2} }{G_S^2},\\
\frac{\partial}{\partial X_\beta} \bigg( \frac{\varphi_\alpha}{G_S} \bigg) & = \frac{\varphi_{\alpha \beta} (1 + \varphi_1^2 + \varphi_2^2)  - 2 \varphi_\alpha (\varphi_1 \varphi_{1\beta} + \varphi_2 \varphi_{\beta 2})  }{G_S^2},\\
\frac{\partial}{\partial X_1} \bigg( \frac{\varphi_\alpha \varphi_\beta}{G_S} \bigg) &=  \frac{ (\varphi_\beta \varphi_{1 \alpha} + \varphi_\alpha \varphi_{1 \beta})(1 + \varphi_1^2 + \varphi_2^2) - 2\varphi_\alpha \varphi_\beta ( \varphi_1 \varphi_{11} + \varphi_2 \varphi_{12}) }{G_S^2},\\
\frac{\partial}{\partial X_2} \bigg( \frac{\varphi_\alpha \varphi_\beta}{G_S} \bigg) &=  \frac{ (\varphi_\beta \varphi_{ \alpha 2} + \varphi_\alpha \varphi_{\beta 2})(1 + \varphi_1^2 + \varphi_2^2) - 2\varphi_\alpha \varphi_\beta ( \varphi_1 \varphi_{12} + \varphi_2 \varphi_{22}) }{G_S^2}.
\end{align*}
Since $1/{G_S} \leq 1$, $\varphi_1^2/{G_S} \leq 1$, $\varphi_2^2/{G_S} \leq 1$, $\varphi_1 \varphi_2/{G_S} \leq 1$ $(X \in \mathbb{R}^2)$, and $\Vert \nabla_X \varphi \Vert_{L^\infty (\mathbb{R}^2)} \leq 1/2$, we find that there is $C> 0$ such that for each $\alpha , \beta \in \{ 1,2\}$,
\begin{multline}\label{eq7031}
\bigg\Vert \frac{\partial}{\partial X_{\alpha}} \bigg( \frac{1}{G_S} \bigg) \bigg\Vert_{L^\infty (\mathbb{R}^2)} + \bigg\Vert \frac{\partial}{\partial X_\beta} \bigg( \frac{\varphi_\alpha}{G_S} \bigg) \bigg\Vert_{L^\infty (\mathbb{R}^2)} + \bigg\Vert \frac{\partial}{\partial X_1} \bigg( \frac{\varphi_\alpha \varphi_\beta}{G_S} \bigg) \bigg\Vert_{L^\infty (\mathbb{R}^2)}\\
  + \bigg\Vert \frac{\partial}{\partial X_2} \bigg( \frac{\varphi_\alpha \varphi_\beta}{G_S} \bigg) \bigg\Vert_{L^\infty (\mathbb{R}^2)} \leq C \Vert \nabla_X^2 \varphi \Vert_{L^\infty (\mathbb{R}^2)}.
\end{multline}
Applying assertion $(\rm{i})$ in Remark \ref{rem72} and \eqref{eq7031} into \eqref{eq7015} and \eqref{eq7017}, we see \eqref{eq7023} and \eqref{eq7024}.

Fourthly, we prove $(\rm{iv})$. To this end, we show that for each $\psi_S \in W^{2,2} (\mathbb{R}^2)$,
\begin{equation}\label{eq7032}
\Vert \nabla_X^2 \psi_S \Vert_{L^2(\mathbb{R}^2)} \leq \Vert \Delta_X \psi_S \Vert_{L^2(\mathbb{R}^2)}.
\end{equation}
Fix $\psi_S \in W^{2,2} (\mathbb{R}^2)$. By definition, we obtain
\begin{equation*}
\Vert \Delta_X \psi_S \Vert_{L^2(\mathbb{R}^2)}^2 = \int_{\mathbb{R}^2} \bigg( \bigg\vert \frac{\partial^2 \psi_S}{\partial X_1^2} \bigg\vert^2 + 2  \frac{\partial^2 \psi_S}{\partial X_1^2} \frac{\partial^2 \psi_S}{\partial X_2^2} + \bigg\vert \frac{\partial^2 \psi_S}{\partial X_2^2} \bigg\vert^2 \bigg) { \ }dX.
\end{equation*}
Since $C_0^\infty (\mathbb{R}^2)$ is dense in $W^{2,2}(\mathbb{R}^2)$, we apply integration by parts to find that
\begin{equation*}
\int_{\mathbb{R}^2} \frac{\partial^2 \psi_S}{\partial X_1^2} \frac{\partial^2 \psi_S}{\partial X_2^2} { \ }dX = \int_{\mathbb{R}^2} \bigg\vert \frac{\partial^2 \psi_S}{\partial X_1 \partial X_2} \bigg\vert^2 { \ }dX .
\end{equation*}
Thus, we observe that
\begin{align*}
\Vert \Delta_X \psi_S \Vert_{L^2(\mathbb{R}^2)}^2 & = \Vert \partial^2_{X_1}\psi_S \Vert_{L^2(\mathbb{R}^2)}^2 + 2 \Vert \partial_{X_1} \partial_{X_2} \psi_S \Vert_{L^2(\mathbb{R}^2)}^2 + \Vert \partial^2_{X_2}\psi_S \Vert_{L^2(\mathbb{R}^2)}^2\\
& \geq \Vert \nabla_X^2 \psi_S \Vert_{L^2(\mathbb{R}^2)}^2,
\end{align*}
which is \eqref{eq7032}. Using \eqref{eq7032}, \eqref{eq7010} when $\kappa_S =1$, and the H\"{o}lder inequality, we see that
\begin{multline}\label{eq7033}
\Vert \nabla_X^2 \psi_S \Vert_{L^2(\mathbb{R}^2)}  \leq \Vert - \Delta_X \psi_S \Vert_{L^2(\mathbb{R}^2)}\\
  \leq \Vert \mathscr{B}_3 \psi_S - \Delta_X \psi_S \Vert_{L^2(\mathbb{R}^2)} + \Vert \mathscr{B}_3 \psi_S \Vert_{L^2(\mathbb{R}^2)}\\
 \leq \Vert \mathscr{L}_3 \psi_S \Vert_{L^2(\mathbb{R}^2)} + \frac{1}{2} \Vert \nabla_X^2 \psi_S \Vert_{L^2(\mathbb{R}^2)} + C (\Vert \nabla_X^2 \varphi \Vert_{L^\infty(\mathbb{R}^2)}) \Vert \nabla_X \psi_S \Vert_{L^2(\mathbb{R}^2)}.
\end{multline}
Here we used the fact that
\begin{multline*}
\Vert \mathscr{B}_3 \psi_S \Vert_{L^2(\mathbb{R}^2)}\\
 \leq \bigg\Vert \frac{\varphi_1^2}{G_S} \frac{\partial^2 \psi_S}{\partial X_1^2} + \frac{\varphi_2^2}{G_S}\frac{\partial^2 \psi_S}{\partial X_2^2} + \frac{2 \varphi_1 \varphi_2}{G_S}\frac{\partial^2 \psi_S}{\partial X_1 \partial X_2}  \bigg\Vert_{L^\infty(\mathbb{R}^2)} + C( \Vert \nabla_X^2 \varphi \Vert_{L^\infty (\mathbb{R}^2)}) \Vert \nabla_X \psi_S \Vert_{L^2(\mathbb{R}^2)}\\
 \leq 2 \Vert \nabla_X \varphi \Vert_{L^2(\mathbb{R}^2)}^2 \Vert \nabla_X^2 \psi_S \Vert_{L^2(\mathbb{R}^2)} + C (\Vert \nabla_X^2 \varphi \Vert_{L^\infty(\mathbb{R}^2)}) \Vert \nabla_X \psi_S \Vert_{L^2(\mathbb{R}^2)}\\
 \leq \frac{1}{2} \Vert \nabla_X^2 \psi_S \Vert_{L^2(\mathbb{R}^2)} + C (\Vert \nabla_X^2 \varphi \Vert_{L^\infty(\mathbb{R}^2)}) \Vert \nabla_X \psi_S \Vert_{L^2(\mathbb{R}^2)}.
\end{multline*}
By \eqref{eq7033}, we have
\begin{equation}\label{eq7034}
\Vert \nabla_X^2 \psi_S \Vert_{L^2(\mathbb{R}^2)} \leq 2 \Vert \mathscr{L}_3 \psi_S \Vert_{L^2(\mathbb{R}^2)} + C (\Vert \nabla_X^2 \varphi \Vert_{L^\infty(\mathbb{R}^2)}) \Vert \nabla_X \psi_S \Vert_{L^2(\mathbb{R}^2)}.
\end{equation}
Applying \eqref{eq7021} and \eqref{eq7022} into \eqref{eq7034}, we observe that
\begin{align*}
\Vert \nabla_X^2 \psi_S \Vert_{L^2(\mathbb{R}^2)} & \leq 2 \Vert \mathscr{L}_3 \psi_S (G_S^{1/4}) \Vert_{L^2(\mathbb{R}^2)} + C (\Vert \nabla_X^2 \varphi \Vert_{L^\infty(\mathbb{R}^2)})  \Vert \psi_S \Vert_{\dot{\mathcal{W}}^{1,2}(\Gamma)}\\
& = 2 \Vert \psi_S \Vert_{\mathring{\mathcal{W}}^{2,2}(\Gamma)} + C (\Vert \nabla_X^2 \varphi \Vert_{L^\infty(\mathbb{R}^2)})  \Vert \psi_S \Vert_{\dot{\mathcal{W}}^{1,2}(\Gamma)}.
\end{align*}
Therefore, we see $(\rm{iv})$. Note that $\Vert \mathscr{L}_3 \psi_S (G_S^{1/4}) \Vert_{L^2(\mathbb{R}^2)} = \Vert \psi_S \Vert_{\mathring{\mathcal{W}}^{2,2}(\Gamma)} $ when $\kappa_S =1$.

Finally, we show $(\rm{v})$. Fix $\psi_S \in W^{2,2} (\mathbb{R}^2)$. From assertions $(\rm{i})$-$(\rm{iv})$, we find that there is $C= C( \Vert \nabla^2 \varphi \Vert_{L^\infty (\mathbb{R}^2)})$ such that
\begin{align*}
\Vert \psi_S \Vert_{W^{2,2} (\mathbb{R}^2)} & =( \Vert \psi_S \Vert_{L^2(\mathbb{R}^2)}^2 + \Vert \nabla_X \psi_S \Vert_{L^2(\mathbb{R}^2)}^2  + \Vert \nabla_X^2 \psi_S \Vert_{L^2(\mathbb{R}^2)}^2 )^{1/2}\\
&\leq C (\Vert \psi_S \Vert_{\mathcal{L}^2(\Gamma)} + \Vert \psi_S \Vert_{\dot{\mathcal{W}}^{1,2}(\Gamma)} + \Vert \psi_S \Vert_{\mathring{ \mathcal{W} }^{2,2}(\Gamma)}),
\end{align*}
and that
\begin{align*}
\Vert \psi_S \Vert_{\mathcal{L}^2(\Gamma)} + \Vert \psi_S \Vert_{\dot{\mathcal{W}}^{1,2}(\Gamma)} + \Vert \psi_S \Vert_{\mathring{\mathcal{W}}^{2,2}(\Gamma)} \leq C \Vert \psi_S \Vert_{W^{2,2} (\mathbb{R}^2)} .
\end{align*}
Thus, we see $(\rm{v})$. Therefore, the lemma follows.
\end{proof}

Finally, we study the relationship between $\Vert \cdot \Vert_{\mathcal{W}^{1,1}(\Gamma) }$ and $\Vert \cdot \Vert_{W^{1,1} (\mathbb{R}^2)}$.
\begin{lemma}\label{lem74}
Assume that $\Vert \nabla_X \varphi \Vert_{L^\infty (\mathbb{R}^2)} \leq 1/2$. Then,\\
$(\rm{i})$ For all $\psi_S \in L^1(\mathbb{R}^2)$,
\begin{equation}\label{eq7035}
\Vert \psi_S \Vert_{L^1(\mathbb{R}^2)} \leq \Vert \psi_S \Vert_{\mathcal{L}^1(\Gamma)} \leq 2 \Vert \psi_S \Vert_{L^1(\mathbb{R}^2)}.
\end{equation}
$(\rm{ii})$ For all $\psi_S \in W^{1,1}(\mathbb{R}^2)$,
\begin{equation}\label{eq7036}
\Vert \nabla_X \psi_S \Vert_{L^1(\mathbb{R}^2)} \leq \Vert \psi_S \Vert_{\dot{\mathcal{W}}^{1,1}(\Gamma)} \leq 12 \Vert \nabla_X \psi_S \Vert_{L^1(\mathbb{R}^2)}.
\end{equation}
Here
\begin{equation*}
\Vert \nabla_X \psi_S \Vert_{L^1(\mathbb{R}^2)} := \Vert \partial_{X_1} \psi_S \Vert_{L^1(\mathbb{R}^2)} + \Vert \partial_{X_2} \psi_S \Vert_{L^1(\mathbb{R}^2)}. 
\end{equation*}
$(\rm{iii})$ For all $\psi_S \in W^{1,1}(\mathbb{R}^2)$,
\begin{equation}\label{eq7037}
\Vert \psi_S \Vert_{W^{1,1}(\mathbb{R}^2)} \leq \Vert \psi_S \Vert_{\mathcal{W}^{1,1}(\Gamma)} \leq 12 \Vert \psi_S \Vert_{W^{1,1}(\mathbb{R}^2)}.
\end{equation}
\end{lemma}
\noindent See also Lemma \ref{lem38}.

\begin{proof}[Proof of Lemma \ref{lem74}]
Let $\psi_S \in L^1(\mathbb{R}^2)$. Using \eqref{eq7018} and \eqref{eq7026}, we have \eqref{eq7035}. Now we show $(\rm{ii})$. Assume that $\psi_S \in W^{1,1} (\mathbb{R}^2)$. From $(\rm{i})$ in Remark \ref{rem72}, we find that
\begin{align*}
\frac{\partial \psi_S}{\partial X_1} = g_S^{\alpha \beta} \frac{\partial \hat{x}_1^S}{\partial X_\alpha} \frac{\partial \psi_S}{\partial X_\beta} + \varphi_1 g_S^{\alpha \beta} \frac{\partial \hat{x}_3^S}{\partial X_\alpha} \frac{\partial \psi_S}{\partial X_\beta},\\
\frac{\partial \psi_S}{\partial X_2} = g_S^{\alpha \beta} \frac{\partial \hat{x}_2^S}{\partial X_\alpha} \frac{\partial \psi_S}{\partial X_\beta} + \varphi_2 g_S^{\alpha \beta} \frac{\partial \hat{x}_3^S}{\partial X_\alpha} \frac{\partial \psi_S}{\partial X_\beta}.
\end{align*}
By $\Vert \nabla_X \varphi \Vert_{L^\infty (\mathbb{R}^2)} \leq 1/2$ and $1 \leq \sqrt{G_S}$, we check that for almost all $X \in \mathbb{R}^2$,
\begin{multline*}
\bigg\vert \frac{\partial \psi_S}{\partial X_1} \bigg\vert + \bigg\vert \frac{\partial \psi_S}{\partial X_2} \bigg\vert \leq \bigg\vert g_S^{\alpha \beta} \frac{\partial \hat{x}_1^S}{\partial X_\alpha} \frac{\partial \psi_S}{\partial X_\beta} \bigg\vert + \bigg\vert g_S^{\alpha \beta} \frac{\partial \hat{x}_2^S}{\partial X_\alpha} \frac{\partial \psi_S}{\partial X_\beta} \bigg\vert + \bigg\vert g_S^{\alpha \beta} \frac{\partial \hat{x}_3^S}{\partial X_\alpha} \frac{\partial \psi_S}{\partial X_\beta} \bigg\vert\\
 \leq \bigg( \bigg\vert g_S^{\alpha \beta} \frac{\partial \hat{x}_1^S}{\partial X_\alpha} \frac{\partial \psi_S}{\partial X_\beta} \bigg\vert + \bigg\vert g_S^{\alpha \beta} \frac{\partial \hat{x}_2^S}{\partial X_\alpha} \frac{\partial \psi_S}{\partial X_\beta} \bigg\vert + \bigg\vert g_S^{\alpha \beta} \frac{\partial \hat{x}_3^S}{\partial X_\alpha} \frac{\partial \psi_S}{\partial X_\beta} \bigg\vert \bigg)\sqrt{G_S}.
\end{multline*}
Integrating over $\mathbb{R}^2$ with respect to $X$, we have 
\begin{equation}\label{eq7038}
\Vert \nabla_X \psi_S \Vert_{L^1(\mathbb{R}^2)} \leq \Vert \psi_S \Vert_{\dot{\mathcal{W}}^{1,1}(\Gamma)}. 
\end{equation}
Since $\sqrt{G_S} \leq 2$, $\vert g_S^{\alpha \beta} \vert \leq 1$, $\vert \partial \hat{x}_j^S/{\partial X_\alpha} \vert \leq 1 $ for $X \in \mathbb{R}^2$, we see that for almost all $X \in \mathbb{R}^2$,
\begin{equation*}
\bigg( \bigg\vert g_S^{\alpha \beta} \frac{\partial \hat{x}_1^S}{\partial X_\alpha} \frac{\partial \psi_S}{\partial X_\beta} \bigg\vert + \bigg\vert g_S^{\alpha \beta} \frac{\partial \hat{x}_2^S}{\partial X_\alpha} \frac{\partial \psi_S}{\partial X_\beta} \bigg\vert + \bigg\vert g_S^{\alpha \beta} \frac{\partial \hat{x}_3^S}{\partial X_\alpha} \frac{\partial \psi_S}{\partial X_\beta} \bigg\vert \bigg)\sqrt{G_S} \leq 12 \bigg( \bigg\vert \frac{\partial \psi_S}{\partial X_1} \bigg\vert + \bigg\vert \frac{\partial \psi_S}{\partial X_2} \bigg\vert \bigg).
\end{equation*}
Integrating over $\mathbb{R}^2$ with respect to $X$, we have 
\begin{equation}\label{eq7039}
\Vert \psi_S \Vert_{\dot{\mathcal{W}}^{1,1}(\Gamma)} \leq 12 \Vert \nabla_X \psi_S \Vert_{L^1(\mathbb{R}^2)}. 
\end{equation}
By \eqref{eq7038} and \eqref{eq7039}, we have \eqref{eq7036}. From assertions $(\rm{i})$ and $(\rm{ii})$, we see $(\rm{iii})$. Therefore, the lemma follows.
\end{proof}

\subsection{Representation Formulas for the Domain $\Omega_A$}\label{subsec72}
Let us study some representation formulas for the domain $\Omega_A$. For all $y= { }^t(y_1,y_2,y_3) \in \overline{\mathbb{R}^3_+}$, we set
\begin{equation*}
\tilde{x}_A = \tilde{x}_A (y)  =
\begin{pmatrix}
\tilde{x}_1^A\\
\tilde{x}_2^A\\
\tilde{x}_3^A
\end{pmatrix}
:=
\begin{pmatrix}
y_1\\
y_2\\
y_3 + \varphi (y_1,y_2)
\end{pmatrix}.
\end{equation*}
It is clear that the mapping $\tilde{x}_A: \overline{\mathbb{R}^3_+} \to \overline{\Omega_A}$ is bijective, that $\overline{\Omega_A} = \{ x \in \mathbb{R}^3; x= \tilde{x}_A(y), y \in \overline{\mathbb{R}_+^3} \}$, and that $\Omega_A = \{ x \in \mathbb{R}^3; x= \tilde{x}_A(y), y \in \mathbb{R}_+^3 \}$ (see Lemma \ref{lem22}). Define $g^A_1 =g^A_1(y)$, $g^A_2=g^A_2(y)$, $g^A_3=g^A_3(y)$ by
\begin{equation*}
g^A_1(y) := \frac{\partial \tilde{x}_A}{\partial y_1} = \begin{pmatrix}
1\\
0\\
\varphi_1
\end{pmatrix},{ \ }
g^A_2(y) := \frac{\partial \tilde{x}_A}{\partial y_2} = \begin{pmatrix}
0\\
1\\
\varphi_2
\end{pmatrix},{ \ }
 g^A_3(y) := \frac{\partial \tilde{x}_A}{\partial y_3} = \begin{pmatrix}
0\\
0\\
1
\end{pmatrix}.
\end{equation*}
Here $\varphi_1 = \varphi_1(y_h) = \partial \varphi/{\partial y_1}$ and $\varphi_2 = \varphi_2(y_h) = \partial \varphi/{\partial y_2}$. Set $g^A_{ij} =g^A_{ij}(y) := g^A_i \cdot g^A_j$, $(g_A^{ij})_{3 \times 3} := (g^A_{ij})_{3 \times 3}^{-1}$, and $G_A = G_A (y) :={\rm{det}} (g^A_{ij})_{3 \times 3}$. We easily check that
\begin{equation*}
(g^A_{ij})_{3 \times 3 } =
\begin{pmatrix}
1 + \varphi_1^2 & \varphi_1 \varphi_2 & \varphi_1\\
\varphi_1 \varphi_2 & 1 + \varphi_2^2 & \varphi_2 \\
\varphi_1 & \varphi_2 &  1
\end{pmatrix},{ \ }
(g_A^{ij})_{3 \times 3 } =
\begin{pmatrix}
1 & 0 & - \varphi_1\\
0 & 1 & - \varphi_2 \\
- \varphi_1 & - \varphi_2 & 1 + \varphi_1^2 + \varphi_2^2
\end{pmatrix},
\end{equation*}
and that $G_A \equiv 1$. Set
\begin{equation*}
g_A^i = g_A^i(y)  := g_A^{ij}g^A_j \bigg( = \sum_{j=1}^3 g_A^{ij}g^A_j \bigg).
\end{equation*}
It is clear that
\begin{equation*}
g_A^1(y) = \begin{pmatrix}
1\\
0\\
0
\end{pmatrix},{ \ }
g_A^2(y) = \begin{pmatrix}
0\\
1\\
0
\end{pmatrix},{ \ }
g_A^3(y) = \begin{pmatrix}
-\varphi_1\\
-\varphi_2\\
1
\end{pmatrix},
\end{equation*}
and $g^{ij}_A = g_A^i \cdot g_A^j$. From \cite[Lemma 3.3]{KS18}, we have the following representation formulas.
\begin{lemma}\label{lem75}
$(\rm{i})$ For all $f_A, f_A^\natural, f_A^\flat \in C^2 (\mathbb{R}^3)$,
\begin{align}
\int_{\Omega_A} f_A (x) { \ }d x &= \int_{\mathbb{R}^3_+} \tilde{f}_A (y) { \ }dy,\\
\int_{\Omega_A} \partial_\ell f_A { \ }dx &= \int_{\mathbb{R}^3_+} g_A^{i j} \frac{\partial \tilde{x}^A_\ell }{\partial y_i} \frac{ \partial \tilde{f}_A }{\partial y_j} { \ }dy,\label{eq7041}\\
\int_{\Omega_A} \partial_{\ell'} \partial_\ell f_A { \ }dx &= \int_{\mathbb{R}^3_+} g_A^{i' j'} \frac{\partial \tilde{x}^A_{\ell'} }{\partial y_{i'}} \frac{ \partial }{\partial y_{j'} } \bigg( g_A^{i j} \frac{\partial \tilde{x}^A_\ell }{\partial y_i} \frac{ \partial \tilde{f}_A }{\partial y_j} \bigg) { \ }dy,\label{eq7042}\\
\int_{\Omega_A} \nabla f_A^\natural \cdot \nabla f_A^\flat { \ }dx & = \int_{\mathbb{R}^3_+} g_A^{i j} \frac{\partial \tilde{f}_A^\natural }{\partial y_i} \frac{ \partial \tilde{f}_A^\flat }{\partial y_j} { \ }dy,\label{eq7043}\\
\int_{\Omega_A} \Delta f_A { \ }d x & = \int_{\mathbb{R}^3_+} \frac{\partial}{\partial y_i} \bigg( g_A^{i j} \frac{\partial \tilde{f}_A}{\partial y_j} \bigg) { \ }dy,\label{eq7044}
\end{align}
where $\tilde{f}_A = \tilde{f}_A (y) = f_A (\tilde{x}_A(y) )$.\\
$(\rm{ii})$ For all $F_A = { }^t (F_1^A, F_2^A,F_3^A) \in [C^1 (\mathbb{R}^3)]^3$ and $\Phi_A \in C^1 (\mathbb{R}^4)$,
\begin{align}
\int_{\Omega_A} \nabla \cdot F_A { \ }d x &= \int_{\mathbb{R}^3_+} g_A^i \cdot \frac{\partial \tilde{F}_A }{\partial y_i} { \ }dy,\label{eq7045}\\
\int_{\Omega_A} \frac{\partial \Phi_A}{\partial t} { \ }d x &= \int_{\mathbb{R}^3_+} \frac{\partial \tilde{\Phi}_A }{\partial t} { \ }d y. 
\end{align}
Here $\tilde{F}_A = \tilde{F}_A (y) = F_A (\tilde{x}_A(y) )$ and $\tilde{\Phi}_A = \tilde{\Phi}_A (y,t) = \Phi_A (\tilde{x}_A(y),t)$. 
\end{lemma}
\noindent From \eqref{eq7041} and \eqref{eq7042}, we have \eqref{eq7044} and \eqref{eq7045} (see also Remark \ref{rem76} and the proof of Lemma \ref{lem77}).

\begin{remark}\label{rem76}$(\rm{i})$ We easily check that for $\psi_A \in C^1(\overline{\mathbb{R}^3_+})$,
\begin{align*}
g_A^{i j} \frac{\partial \tilde{x}_1^A}{\partial y_i} \frac{\partial \psi_A}{\partial y_j} & = \frac{\partial \psi_A}{\partial y_1} -  \varphi_1\frac{\partial \psi_A}{\partial y_3},\\
g_A^{i j} \frac{\partial \tilde{x}_2^A}{\partial y_i} \frac{\partial \psi_A}{\partial y_j} & = \frac{\partial \psi_A}{\partial y_2} - \varphi_2 \frac{\partial \psi_A}{\partial y_3},\\
g_A^{i j} \frac{\partial \tilde{x}_3^A}{\partial y_i} \frac{\partial \psi_A}{\partial y_j} & = \frac{\partial \psi_A}{\partial y_3},
\end{align*}
and for $\Psi_A = { }^t (\Psi_1^A , \Psi_2^A, \Psi_3^A) \in [C^1(\overline{\mathbb{R}_+^3})]^3$,
\begin{equation}\label{eq7047}
g_A^i \cdot \frac{\partial \Psi_A }{\partial y_i} = \frac{\partial \Psi_1^A}{\partial y_1 } + \frac{\partial \Psi_2^A}{\partial y_2 } -\varphi_1 \frac{\partial \Psi_1^A}{\partial y_3 } - \varphi_2 \frac{\partial \Psi_2^A}{\partial y_3 } + \frac{\partial \Psi_3^A}{\partial y_3 }   .
\end{equation}
$(\rm{ii})$ Let $\kappa_A >0$. For $\psi_A \in C^2 (\overline{\mathbb{R}^3_+})$, we set
\begin{equation}\label{eq7048}
\mathscr{L}_1 \psi_A := - \kappa_A \frac{\partial}{\partial y_i} \bigg( g_A^{ij} \frac{\partial \psi_A}{\partial y_j} \bigg).
\end{equation}
We easily check that
\begin{equation}\label{eq7049}
\mathscr{L}_1 \psi_A = -\kappa_A \Delta_y \psi_A + \mathscr{B}_1 \psi_A,
\end{equation}
where
\begin{equation}\label{eq7050}
\mathscr{B}_1 \psi_A := \kappa_A \bigg( 2 \varphi_1 \frac{\partial^2 \psi_A}{\partial y_1 \partial y_3} + 2 \varphi_2 \frac{\partial^2 \psi_A}{\partial y_2 \partial y_3} - (\varphi_1^2 + \varphi_2^2) \frac{\partial^2 \psi_A}{\partial y_3^2 } + (\varphi_{11} +\varphi_{22}) \frac{\partial \psi_A}{\partial y_3} \bigg).
\end{equation}
Here $\varphi_{\alpha \beta} = \partial^2 \varphi/{\partial y_\alpha \partial y_\beta}$.\\
$(\rm{iii})$ From Lemma \ref{lem75}, we find that for each $f_A \in C^2(\mathbb{R}^3)$,
\begin{align*}
\int_{\Omega_A} \vert f_A \vert^2 { \ }dx &=  \int_{\mathbb{R}^3_+} \vert \tilde{f}_A \vert^2 { \ }d y,\\
\int_{\Omega_A} \vert \nabla f_A \vert^2 { \ }dx & = \int_{\mathbb{R}^3_+} g_A^{i j} \frac{\partial \tilde{f}_A }{\partial y_i} \frac{ \partial \tilde{f}_A }{\partial y_j} { \ }dy,\\
\int_{\Omega_A} \vert \partial_{\ell'} \partial_ {\ell} f_A \vert^2 { \ }dx &= \int_{\mathbb{R}^3_+} \bigg\vert g_A^{i' j'} \frac{\partial \tilde{x}^A_{\ell'} }{\partial y_{i'}} \frac{ \partial }{\partial y_{j'} } \bigg( g_A^{i j} \frac{\partial \tilde{x}^A_\ell }{\partial y_i} \frac{ \partial \tilde{f}_A }{\partial y_j} \bigg) \bigg\vert^2 { \ }dy,
\end{align*}
where $\tilde{f}_A = \tilde{f}_A(y) = f_A (\tilde{x}_A (y))$. For $\psi_A \in L_{loc}^1 (\mathbb{R}^3_+)$, we define
\begin{align}
\Vert \psi_A \Vert_{\mathcal{L}^2(\Omega_A)} & :=  \bigg( \int_{\mathbb{R}^3_+} \vert \psi_A \vert^2 { \ }d y \bigg)^{1/2},\\
\Vert \psi_A \Vert_{\dot{\mathcal{W}}^{1,2}(\Omega_A)}  & := \bigg( \int_{\mathbb{R}^3_+}  g_A^{i j} \frac{\partial \psi_A }{\partial y_i} \frac{ \partial \psi_A }{\partial y_j}  { \ }dy \bigg)^{1/2},\\
\Vert \psi_A \Vert_{\dot{\mathcal{W}}^{2,2}(\Omega_A)}  & := \bigg( \sum_{\ell, \ell' =1}^3 \int_{\mathbb{R}^3_+} \bigg\vert g_A^{i' j'} \frac{\partial \tilde{x}^A_{\ell'} }{\partial y_{i'}} \frac{ \partial }{\partial y_{j'} } \bigg( g_A^{i j} \frac{\partial \tilde{x}^A_\ell }{\partial y_i} \frac{ \partial \psi_A }{\partial y_j} \bigg) \bigg\vert^2 {  \ }dy \bigg)^{1/2},\label{eq7053}\\
\Vert \psi_A \Vert_{ \mathcal{W}^{2,2}(\Omega_A)} & := ( \Vert \psi_A \Vert_{\mathcal{L}^2(\Omega_A)}^2 + \Vert \psi_A \Vert_{\dot{\mathcal{W}}^{1,2}(\Omega_A)}^2 + \Vert \psi_A \Vert_{\dot{\mathcal{W}}^{2,2}(\Omega_A)}^2 )^{1/2},\\
\Vert \psi_A \Vert_{\dot{W}^{2,2}(\mathbb{R}^3_+)}  & := \bigg( \sum_{\ell,\ell'=1}^3 \Vert \partial_{y_{\ell'} } \partial_{y_\ell} \psi_A \Vert_{L^2(\mathbb{R}^3_+)}^2 \bigg)^{1/2}.
\end{align}
$(\rm{iv})$ From Lemma \ref{lem75}, we find that for each $f_A \in C^2(\mathbb{R}^3)$,
\begin{align*}
\int_{\Omega_A} \vert f_A \vert { \ }d x &= \int_{\mathbb{R}^3_+} \vert \tilde{f}_A \vert { \ }dy,\\
\int_{\Omega_A} \vert \partial_\ell f_A \vert { \ }dx &= \int_{\mathbb{R}^3_+} \bigg\vert g_A^{i j} \frac{\partial \tilde{x}^A_\ell }{\partial y_i} \frac{ \partial \tilde{f}_A }{\partial y_j} \bigg\vert { \ }dy.
\end{align*}
where $\tilde{f}_A = \tilde{f}_A(y) = f_A (\tilde{x}_A (y))$. For $\psi_A \in L_{loc}^1 (\mathbb{R}^3_+)$, we define
\begin{align}
\Vert \psi_A \Vert_{\mathcal{L}^1(\Omega_A)} &:= \int_{\mathbb{R}^3_+} \vert \psi_A \vert { \ }dy,\\
\Vert \psi_A \Vert_{\dot{\mathcal{W}}^{1,1} (\Omega_A) } &:= \sum_{\ell =1}^3 \int_{\mathbb{R}^3_+} \bigg\vert g_A^{i j} \frac{\partial \tilde{x}^A_\ell }{\partial y_i} \frac{ \partial \psi_A }{\partial y_j} \bigg\vert { \ }dy,\\
\Vert \psi_A \Vert_{\mathcal{W}^{1,1} (\Omega_A) } & := \Vert \psi_A \Vert_{\mathcal{L}^1 (\Omega_A) } + \Vert \psi_A \Vert_{\dot{\mathcal{W}}^{1,1} (\Omega_A) }.
\end{align}
\end{remark}

Let us investigate the relationships between $\Vert \cdot \Vert_{\dot{\mathcal{W}}^{1,2} (\Omega_A)}$ and $\Vert \nabla_y \cdot \Vert_{L^2 (\mathbb{R}^3_+)}$, and between $\Vert \cdot \Vert_{\mathcal{W}^{2,2} (\Omega_A)}$ and $\Vert \cdot \Vert_{W^{2,2} (\mathbb{R}^3_+)}$.
\begin{lemma}\label{lem77}
Assume that $\Vert \nabla_X \varphi \Vert_{L^\infty (\mathbb{R}^2)} \leq 1/2$. Then,\\
$(\rm{i})$ For all $\psi_A \in W^{1,2}(\mathbb{R}^3_+)$,
\begin{equation}\label{eq7059}
\frac{1}{9}\Vert \nabla_y \psi_A \Vert_{L^2(\mathbb{R}^3_+)}^2 \leq \Vert \psi_A \Vert_{\dot{\mathcal{W}}^{1,2}(\Omega_A)}^2 \leq 9 \Vert \nabla_y \psi_A \Vert_{L^2(\mathbb{R}^3_+)}^2.
\end{equation}
$(\rm{ii})$ There is $C = C( \Vert \nabla_X^2 \varphi \Vert_{L^{\infty} (\mathbb{R}^2)} ) > 0 $ such that for all $\psi_A \in W^{2,2}(\mathbb{R}^3_+)$,
\begin{equation}
\Vert \psi_A \Vert_{\dot{\mathcal{W}}^{2,2}(\Omega_A)} \leq C \Vert \psi_A \Vert_{W^{2,2} (\mathbb{R}^3_+)}.
\end{equation}
$(\rm{iii})$ There is $C = C( \Vert \nabla_X^2 \varphi \Vert_{L^{\infty} (\mathbb{R}^2)} ) > 0 $ such that for all $\psi_A \in W^{2,2}(\mathbb{R}^3_+)$,
\begin{equation}
\Vert \psi_A \Vert_{\dot{W}^{2,2} (\mathbb{R}^3_+)} \leq C( \Vert \psi_A \Vert_{\dot{\mathcal{W}}^{1,2}(\Omega_A)} + \Vert \psi_A \Vert_{\dot{\mathcal{W}}^{2,2}(\Omega_A)} ) .
\end{equation}
$(\rm{iv})$ There is $C = C( \Vert \nabla_X^2 \varphi \Vert_{L^{\infty} (\mathbb{R}^2)} ) > 0 $ such that for all $\psi_A \in W^{2,2}(\mathbb{R}^3_+)$,
\begin{equation*}
C^{-1} \Vert \psi_A \Vert_{W^{2,2} (\mathbb{R}^3_+)} \leq \Vert \psi_A \Vert_{\mathcal{W}^{2,2}(\Omega_A)} \leq C \Vert \psi_A \Vert_{W^{2,2} (\mathbb{R}^3_+)} .
\end{equation*}
\end{lemma}
\noindent See also Lemma \ref{lem37}.

\begin{proof}[Proof of Lemma \ref{lem77}]

Assume that $\Vert \nabla_X \varphi \Vert_{L^\infty (\mathbb{R}^2)} \leq 1/2$. Fix $\psi_A \in W^{1,2} (\mathbb{R}^3_+)$.

We first show $(\rm{i})$. By definition, we check that
\begin{multline}\label{eq7062}
\int_{\mathbb{R}^3_+} g_A^{i j} \frac{\partial \psi_A }{\partial y_i} \frac{ \partial \psi_A }{\partial y_j} { \ }dy\\
 = \int_{\mathbb{R}^3_+}\bigg\{ \bigg(\frac{ \partial \psi_A }{\partial y_1} -  \varphi_1 \frac{ \partial \psi_A }{\partial y_3} \bigg)^2 +\bigg(\frac{ \partial \psi_A }{\partial y_2} -  \varphi_2 \frac{ \partial \psi_A }{\partial y_3} \bigg)^2 + \bigg\vert \frac{\partial \psi_A }{\partial y_3} \bigg\vert^2 \bigg\} { \ }dy.
\end{multline}
From \eqref{eq7062}, we obtain
\begin{equation}\label{eq7063}
\int_{\mathbb{R}^3_+} g_A^{i j} \frac{\partial \psi_A }{\partial y_i} \frac{ \partial \psi_A }{\partial y_j} { \ }dy \geq \int_{\mathbb{R}^3_+} \bigg\vert \frac{\partial \psi_A }{\partial y_3} \bigg\vert^2 { \ }dy.
\end{equation}
Applying $(a-b)^2 \leq 2 a^2 + 2b^2$ and $\Vert \nabla_X \varphi \Vert_{L^\infty (\mathbb{R}^2)} \leq 1/2$ into \eqref{eq7062}, we find that
\begin{equation}\label{eq7064}
\int_{\mathbb{R}^3_+} g_A^{i j} \frac{\partial \psi_A }{\partial y_i} \frac{ \partial \psi_A }{\partial y_j} { \ }dy \leq 9 \int_{\mathbb{R}^3_+} \vert \nabla_y \psi_A \vert^2 { \ }dy.
\end{equation}
Since $a^2 \leq 2 (a-b)^2 + 2b^2$, we use \eqref{eq7063} and \eqref{eq7062} to see that
\begin{multline}\label{eq7065}
\Vert \nabla_y \psi_A \Vert_{L^2(\mathbb{R}^3_+)}^2 = \int_{\mathbb{R}^3_+} \bigg( \bigg\vert \frac{\partial \psi_A }{\partial y_1} \bigg\vert^2 + \bigg\vert \frac{\partial \psi_A }{\partial y_2} \bigg\vert^2 + \bigg\vert \frac{\partial \psi_A }{\partial y_3} \bigg\vert^2 \bigg) { \ }dy\\
 \leq \int_{\mathbb{R}^3_+} \bigg\{ 2 \bigg(\frac{ \partial \psi_A }{\partial y_1} -  \varphi_1 \frac{ \partial \psi_A }{\partial y_3} \bigg)^2 + 2 \bigg(\frac{ \partial \psi_A }{\partial y_2} -  \varphi_2 \frac{ \partial \psi_A }{\partial y_3} \bigg)^2 + 2 (1 + \varphi_1^2 + \varphi_2^2) \bigg\vert \frac{\partial \psi_A }{\partial y_3} \bigg\vert^2 \bigg\} { \ }dy\\
\leq 9 \int_{\mathbb{R}^3_+} g_A^{i j} \frac{\partial \psi_A }{\partial y_i} \frac{ \partial \psi_A }{\partial y_j} { \ }dy.
\end{multline}
By \eqref{eq7064} and \eqref{eq7065}, we see \eqref{eq7059}.

Now we prove $(\rm{ii})$-$(\rm{iv})$. Assume that $\psi_A \in W^{2,2} (\mathbb{R}^3_+)$. Using assertion $(\rm{i})$ in Remark \ref{rem76}, we find that
\begin{align*}
g_A^{i' j'} \frac{\partial \tilde{x}^A_1 }{\partial y_{i'}} \frac{ \partial }{\partial y_{j'}} \bigg( g_A^{i j} \frac{\partial \tilde{x}^A_1 }{\partial y_i} \frac{ \partial \psi_A }{\partial y_j} \bigg) & = \frac{\partial^2 \psi_A}{\partial y_1^2} -  \varphi_{11} \frac{\partial \psi_A}{\partial y_3} - 2 \varphi_1 \frac{\partial^2 \psi_A }{\partial y_1  \partial y_3} +  \varphi_1^2 \frac{\partial^2 \psi_A }{\partial y_3^2},\\
g_A^{i' j'} \frac{\partial \tilde{x}^A_1 }{\partial y_{i'}} \frac{ \partial }{\partial y_{j'}} \bigg( g_A^{i j} \frac{\partial \tilde{x}^A_2 }{\partial y_i} \frac{ \partial \psi_A }{\partial y_j} \bigg) & = \frac{\partial^2 \psi_A}{\partial y_1 \partial y_2} -  \varphi_{12} \frac{\partial \psi_A}{\partial y_3} - \varphi_2 \frac{\partial^2 \psi_A }{\partial y_1  \partial y_3} - \varphi_1 \frac{\partial^2 \psi_A }{\partial y_2  \partial y_3}+  \varphi_1 \varphi_2 \frac{\partial^2 \psi_A }{\partial y_3^2},\\
g_A^{i' j'} \frac{\partial \tilde{x}^A_1 }{\partial y_{i'}} \frac{ \partial }{\partial y_{j'}} \bigg( g_A^{i j} \frac{\partial \tilde{x}^A_3 }{\partial y_i} \frac{ \partial \psi_A }{\partial y_j} \bigg) & = \frac{\partial^2 \psi_A}{\partial y_1 \partial y_3} -  \varphi_1 \frac{\partial^2 \psi_A }{\partial y_3^2},\\
g_A^{i' j'} \frac{\partial \tilde{x}^A_2 }{\partial y_{i'}} \frac{ \partial }{\partial y_{j'}} \bigg( g_A^{i j} \frac{\partial \tilde{x}^A_1 }{\partial y_i} \frac{ \partial \psi_A }{\partial y_j} \bigg) & = \frac{\partial^2 \psi_A}{\partial y_1 \partial y_2} -  \varphi_{12} \frac{\partial \psi_A}{\partial y_3} - \varphi_2 \frac{\partial^2 \psi_A }{\partial y_1  \partial y_3} - \varphi_1 \frac{\partial^2 \psi_A }{\partial y_2  \partial y_3}+  \varphi_1 \varphi_2 \frac{\partial^2 \psi_A }{\partial y_3^2},\\
g_A^{i' j'} \frac{\partial \tilde{x}^A_2 }{\partial y_{i'}} \frac{ \partial }{\partial y_{j'}} \bigg( g_A^{i j} \frac{\partial \tilde{x}^A_2 }{\partial y_i} \frac{ \partial \psi_A }{\partial y_j} \bigg) & = \frac{\partial^2 \psi_A}{\partial y_2^2} -  \varphi_{22} \frac{\partial \psi_A}{\partial y_3} - 2 \varphi_2 \frac{\partial^2 \psi_A }{\partial y_2  \partial y_3} +  \varphi_2^2 \frac{\partial^2 \psi_A }{\partial y_3^2},\\
g_A^{i' j'} \frac{\partial \tilde{x}^A_2 }{\partial y_{i'}} \frac{ \partial }{\partial y_{j'}} \bigg( g_A^{i j} \frac{\partial \tilde{x}^A_3 }{\partial y_i} \frac{ \partial \psi_A }{\partial y_j} \bigg) & = \frac{\partial^2 \psi_A}{\partial y_2 \partial y_3} -  \varphi_2 \frac{\partial^2 \psi_A }{\partial y_3^2},\\
g_A^{i' j'} \frac{\partial \tilde{x}^A_3 }{\partial y_{i'}} \frac{ \partial }{\partial y_{j'}} \bigg( g_A^{i j} \frac{\partial \tilde{x}^A_1 }{\partial y_i} \frac{ \partial \psi_A }{\partial y_j} \bigg) & = \frac{\partial^2 \psi_A}{\partial y_1 \partial y_3} -  \varphi_1 \frac{\partial^2 \psi_A }{\partial y_3^2},\\
g_A^{i' j'} \frac{\partial \tilde{x}^A_3 }{\partial y_{i'}} \frac{ \partial }{\partial y_{j'}} \bigg( g_A^{i j} \frac{\partial \tilde{x}^A_2 }{\partial y_i} \frac{ \partial \psi_A }{\partial y_j} \bigg) & = \frac{\partial^2 \psi_A}{\partial y_2 \partial y_3} -  \varphi_2 \frac{\partial^2 \psi_A }{\partial y_3^2},\\
g_A^{i' j'} \frac{\partial \tilde{x}^A_3 }{\partial y_{i'}} \frac{ \partial }{\partial y_{j'}} \bigg( g_A^{i j} \frac{\partial \tilde{x}^A_3 }{\partial y_i} \frac{ \partial \psi_A }{\partial y_j} \bigg) & = \frac{\partial^2 \psi_A }{\partial y_3^2}.
\end{align*}
Applying the H\"{o}lder inequality and the above equalities into \eqref{eq7053}, we see $(\rm{ii})$. Next we show $(\rm{iii})$. Since
\begin{equation*}
\frac{\partial^2 \psi_A }{\partial y_3^2} = g_A^{i' j'} \frac{\partial \tilde{x}^A_3 }{\partial y_{i'}} \frac{ \partial }{\partial y_{j'}} \bigg( g_A^{i j} \frac{\partial \tilde{x}^A_3 }{\partial y_i} \frac{ \partial \psi_A }{\partial y_j} \bigg), 
\end{equation*}
we find that
\begin{equation}\label{eq7066}
\bigg\Vert \frac{\partial^2 \psi_A }{\partial y_3^2} \bigg\Vert_{L^2(\mathbb{R}^3_+)} = \bigg\Vert g_A^{i' j'} \frac{\partial \tilde{x}^A_3 }{\partial y_{i'}} \frac{ \partial }{\partial y_{j'}} \bigg( g_A^{i j} \frac{\partial \tilde{x}^A_3 }{\partial y_i} \frac{ \partial \psi_A }{\partial y_j} \bigg) \bigg\Vert_{L^2(\mathbb{R}^3_+)} \leq \Vert \psi_A \Vert_{\dot{\mathcal{W}}^{2,2} (\Omega_A)}.
\end{equation}
It is clear that
\begin{align*}
\frac{\partial^2 \psi_A}{\partial y_1 \partial y_3} &= \bigg( \frac{\partial^2 \psi_A}{\partial y_1 \partial y_3} -  \varphi_1 \frac{\partial^2 \psi_A }{\partial y_3^2} \bigg) +  \varphi_1 \frac{\partial^2 \psi_A }{\partial y_3^2},\\
\frac{\partial^2 \psi_A}{\partial y_2 \partial y_3} &= \bigg( \frac{\partial^2 \psi_A}{\partial y_2 \partial y_3} -  \varphi_2 \frac{\partial^2 \psi_A }{\partial y_3^2} \bigg) +  \varphi_2 \frac{\partial^2 \psi_A }{\partial y_3^2}.
\end{align*}
By \eqref{eq7066} and $\Vert \nabla_X \varphi \Vert_{L^\infty (\mathbb{R}^2)} \leq 1/2$, we check that
\begin{multline}\label{eq7067}
\bigg\Vert \frac{\partial^2 \psi_A}{\partial y_1 \partial y_3} \bigg\Vert_{L^2(\mathbb{R}^3_+)} =  \bigg\Vert g_A^{i' j'} \frac{\partial \tilde{x}^A_1 }{\partial y_{i'}} \frac{ \partial }{\partial y_{j'}} \bigg( g_A^{i j} \frac{\partial \tilde{x}^A_3 }{\partial y_i} \frac{ \partial \psi_A }{\partial y_j} \bigg) + \varphi_1 \frac{\partial^2 \psi_A }{\partial y_3^2}\bigg\Vert_{L^2(\mathbb{R}^3_+)}\\
\leq  \bigg\Vert g_A^{i' j'} \frac{\partial \tilde{x}^A_1 }{\partial y_{i'}} \frac{ \partial }{\partial y_{j'}} \bigg( g_A^{i j} \frac{\partial \tilde{x}^A_3 }{\partial y_i} \frac{ \partial \psi_A }{\partial y_j} \bigg)\bigg\Vert_{L^2(\mathbb{R}^3_+)} + \bigg\Vert \varphi_1 \frac{\partial^2 \psi_A }{\partial y_3^2}\bigg\Vert_{L^2(\mathbb{R}^3_+)} \leq 2 \Vert \psi_A \Vert_{\dot{\mathcal{W}}^{2,2} (\Omega_A)}.
\end{multline}
Similarly, we see that
\begin{equation}\label{eq7068}
\bigg\Vert \frac{\partial^2 \psi_A}{\partial y_2 \partial y_3} \bigg\Vert_{L^2(\mathbb{R}^3_+)} \leq 2 \Vert \psi_A \Vert_{\dot{\mathcal{W}}^{2,2} (\Omega_A)}.
\end{equation}
It is easy to check that
\begin{multline}\label{eq7069}
\frac{\partial^2 \psi_A}{\partial y_1^2} = \bigg( \frac{\partial^2 \psi_A}{\partial y_1^2} -  \varphi_{11} \frac{\partial \psi_A}{\partial y_3} - 2 \varphi_1 \frac{\partial^2 \psi_A }{\partial y_1  \partial y_3} +  \varphi_1^2 \frac{\partial^2 \psi_A }{\partial y_3^2} \bigg)\\
 + \bigg( \varphi_{11} \frac{\partial \psi_A}{\partial y_3} + 2 \varphi_1 \frac{\partial^2 \psi_A }{\partial y_1  \partial y_3} -  \varphi_1^2 \frac{\partial^2 \psi_A }{\partial y_3^2} \bigg),\\
= g_A^{i' j'} \frac{\partial \tilde{x}^A_1 }{\partial y_{i'}} \frac{ \partial }{\partial y_{j'}} \bigg( g_A^{i j} \frac{\partial \tilde{x}^A_1 }{\partial y_i} \frac{ \partial \psi_A }{\partial y_j} \bigg)
  + \bigg( \varphi_{11} \frac{\partial \psi_A}{\partial y_3} + 2 \varphi_1 \frac{\partial^2 \psi_A }{\partial y_1  \partial y_3} -  \varphi_1^2 \frac{\partial^2 \psi_A }{\partial y_3^2} \bigg).
\end{multline}
From \eqref{eq7059}, we have
\begin{equation}\label{eq7070}
\bigg\Vert \frac{\partial \psi_A}{\partial y_3} \bigg\Vert_{L^2(\mathbb{R}^3_+)} \leq 3 \Vert \psi_A \Vert_{\dot{\mathcal{W}}^{1,2} (\Omega_A) }.
\end{equation}
Applying \eqref{eq7067}, \eqref{eq7068}, and \eqref{eq7070} into \eqref{eq7069}, we see that
\begin{equation}\label{eq7071}
\bigg\Vert \frac{\partial^2 \psi_A}{\partial y_1^2} \bigg\Vert_{L^2(\mathbb{R}^3_+)} \leq C  \Vert \psi_A \Vert_{\dot{\mathcal{W}}^{2,2} (\Omega_A) } + C (\Vert \nabla_X^2 \varphi \Vert_{L^\infty (\mathbb{R}^2)}) \Vert \psi_A \Vert_{\dot{\mathcal{W}}^{1,2} (\Omega_A) }.
\end{equation}
In the same manner, we find that
\begin{align}
\bigg\Vert \frac{\partial^2 \psi_A}{\partial y_2^2} \bigg\Vert_{L^2(\mathbb{R}^3_+)} \leq C  \Vert \psi_A \Vert_{\dot{\mathcal{W}}^{2,2} (\Omega_A) } + C (\Vert \nabla_X^2 \varphi \Vert_{L^\infty (\mathbb{R}^2)}) \Vert \psi_A \Vert_{\dot{\mathcal{W}}^{1,2} (\Omega_A) },\label{eq7072}\\
\bigg\Vert \frac{\partial^2 \psi_A}{\partial y_1 \partial y_2} \bigg\Vert_{L^2(\mathbb{R}^3_+)} \leq C  \Vert \psi_A \Vert_{\dot{\mathcal{W}}^{2,2} (\Omega_A) } + C (\Vert \nabla_X^2 \varphi \Vert_{L^\infty (\mathbb{R}^2)}) \Vert \psi_A \Vert_{\dot{\mathcal{W}}^{1,2} (\Omega_A) }.\label{eq7073}
\end{align}
Thus, we see $(\rm{iii})$ from \eqref{eq7066}-\eqref{eq7073}. Finally, we show $(\rm{iv})$. Since $\Vert \psi_A \Vert_{\mathcal{L}^2(\Omega_A)} = \Vert \psi_A \Vert_{L^2(\mathbb{R}^3_+)}$, we apply assertions $(\rm{i})$-$(\rm{iii})$ to derive $(\rm{iv})$. Therefore, the lemma follows.
\end{proof}

Finally, we study the relationship between $\Vert \cdot \Vert_{\mathcal{W}^{1,1}(\Omega_A) }$ and $\Vert \cdot \Vert_{W^{1,1} (\mathbb{R}^3_+)}$.
\begin{lemma}\label{lem78}
Assume that $\Vert \nabla_X \varphi \Vert_{L^\infty (\mathbb{R}^2)} \leq 1/2$. Then,\\
$(\rm{i})$ For all $\psi_A \in W^{1,1}(\mathbb{R}^3_+)$,
\begin{equation}\label{eq7074}
\frac{1}{2} \Vert \nabla_y \psi_A \Vert_{L^1(\mathbb{R}^3_+)} \leq \Vert \psi_A \Vert_{\dot{\mathcal{W}}^{1,1}(\Omega_A)} \leq 18 \Vert \nabla_y \psi_A \Vert_{L^1(\mathbb{R}^3_+)}.
\end{equation}
Here
\begin{equation*}
\Vert \nabla_y \psi_A \Vert_{L^1(\mathbb{R}^3_+)} := \Vert \partial_{y_1} \psi_A \Vert_{L^1(\mathbb{R}^3_+)} + \Vert \partial_{y_2} \psi_A \Vert_{L^1(\mathbb{R}^3_+)}  + \Vert \partial_{y_3} \psi_A \Vert_{L^1(\mathbb{R}^3_+)}. 
\end{equation*}
$(\rm{ii})$ For all $\psi_A \in W^{1,1}(\mathbb{R}^3_+)$,
\begin{equation}\label{eq7075}
\frac{1}{2} \Vert \psi_A \Vert_{W^{1,1}(\mathbb{R}^3_+)} \leq \Vert \psi_A \Vert_{\mathcal{W}^{1,1}(\Omega_A)} \leq 18 \Vert \psi_A \Vert_{W^{1,1}(\mathbb{R}^3_+)}.
\end{equation}
\end{lemma}
\noindent See also Lemma \ref{lem38}.

\begin{proof}[Proof of Lemma \ref{lem78}]
Let $\psi_A \in W^{1,1} (\mathbb{R}^3_+)$. From $(\rm{i})$ in Remark \ref{rem76}, we find that
\begin{align}
\frac{\partial \psi_A}{\partial y_3} & = g_A^{i j} \frac{\partial \tilde{x}_3^A}{\partial y_i} \frac{\partial \psi_A}{\partial y_j},\label{eq7076}\\
\frac{\partial \psi_A}{\partial y_1} & = g_A^{i j} \frac{\partial \tilde{x}_1^A}{\partial y_i} \frac{\partial \psi_A}{\partial y_j} + \varphi_1 g_A^{i j} \frac{\partial \tilde{x}_3^A}{\partial y_i} \frac{\partial \psi_A}{\partial y_j} ,\label{eq7077}\\
\frac{\partial \psi_A}{\partial y_2} &=g_A^{i j} \frac{\partial \tilde{x}_2^A}{\partial y_i} \frac{\partial \psi_A}{\partial y_j}  + \varphi_2 g_A^{i j} \frac{\partial \tilde{x}_3^A}{\partial y_i} \frac{\partial \psi_A}{\partial y_j}.\label{eq7078}
\end{align}
By $\Vert \nabla_X \varphi \Vert_{L^\infty (\mathbb{R}^2)} \leq 1/2$, $\Vert g_A^{ij} \Vert_{L^\infty (\mathbb{R}^3_+)} \leq 2$, $\Vert \partial \tilde{x}_j^A/{\partial y_i} \Vert_{L^\infty (\mathbb{R}^3_+)} \leq 1$, \eqref{eq7076}-\eqref{eq7078}. we find that for almost all $y \in \mathbb{R}^3_+$,
\begin{multline*}
\bigg\vert \frac{\partial \psi_A}{\partial y_1} \bigg\vert + \bigg\vert \frac{\partial \psi_A}{\partial y_2} \bigg\vert + \bigg\vert \frac{\partial \psi_A}{\partial y_3} \bigg\vert \leq 2 \bigg\vert g_A^{i j} \frac{\partial \tilde{x}_1^A}{\partial y_i} \frac{\partial \psi_A}{\partial y_j} \bigg\vert + 2 \bigg\vert g_A^{i j} \frac{\partial \tilde{x}_2^A}{\partial y_i} \frac{\partial \psi_A}{\partial y_j} \bigg\vert +2 \bigg\vert g_A^{i j} \frac{\partial \tilde{x}_3^A}{\partial y_i} \frac{\partial \psi_A}{\partial y_j} \bigg\vert\\
\leq 36 \bigg( \bigg\vert \frac{\partial \psi_A}{\partial y_1} \bigg\vert + \bigg\vert \frac{\partial \psi_A}{\partial y_2} \bigg\vert + \bigg\vert \frac{\partial \psi_A}{\partial y_3} \bigg\vert \bigg).
\end{multline*}
Integrating over $\mathbb{R}^3_+$ with respect to $y$, we have \eqref{eq7074}. By $\Vert \psi_A \Vert_{\mathcal{L}^1(\Omega_A)} = \Vert \psi \Vert_{L^1(\mathbb{R}^3_+)}$ and \eqref{eq7074}, we see \eqref{eq7075}. Therefore, the lemma follows.
\end{proof}

\subsection{Representation Formulas for the Domain $\Omega_B$}\label{subsec73}
Let us study some representation formulas for the domain $\Omega_B$. For all $z = { }^t (z_1,z_2,z_3) \in \overline{\mathbb{R}^3_-}$, we set
\begin{equation*}
\check{x}_B = \check{x}_B (z) =
\begin{pmatrix}
\check{x}_1^B\\
\check{x}_2^B\\
\check{x}_3^B
\end{pmatrix} :=
\begin{pmatrix}
z_1\\
z_2\\
z_3 + \varphi (z_1,z_2)
\end{pmatrix}.
\end{equation*}
It is clear that the mapping $\check{x}_B: \overline{\mathbb{R}^3_-} \to \overline{\Omega_B}$ is bijective, that $\overline{\Omega_B} = \{ x \in \mathbb{R}^3; x= \check{x}_B(z), z \in \overline{\mathbb{R}_-^3} \}$, and that $\Omega_B = \{ x \in \mathbb{R}^3; x= \check{x}_B(z), z \in \mathbb{R}_-^3 \}$ (see Lemma \ref{lem22}). Define $g^B_1 =g^B_1(z)$, $g^B_2=g^B_2(z)$, $g^B_3=g^B_3(z)$ by
\begin{equation*}
g^B_1(z) := \frac{\partial \check{x}_B}{\partial z_1} = \begin{pmatrix}
1\\
0\\
\varphi_1
\end{pmatrix},{ \ }
g^B_2(z) := \frac{\partial \check{x}_B}{\partial z_2} = \begin{pmatrix}
0\\
1\\
\varphi_2
\end{pmatrix},{ \ }
 g^B_3(z) := \frac{\partial \check{x}_B}{\partial z_3} = \begin{pmatrix}
0\\
0\\
1
\end{pmatrix}.
\end{equation*}
Here $\varphi_1 = \varphi_1 (z_h) = \partial \varphi/{\partial z_1}$ and $\varphi_2 = \varphi_2 (z_h) = \partial \varphi/{\partial z_2}$. Set $g^B_{ij} =g^B_{ij}(z) := g^B_i \cdot g^B_j$, $(g_B^{ij})_{3 \times 3} := (g^B_{ij})_{3 \times 3}^{-1}$, and $G_B = G_B (z) :={\rm{det}} (g^B_{ij})_{3 \times 3}$. We easily check that
\begin{equation*}
(g^B_{ij})_{3 \times 3 } =
\begin{pmatrix}
1 + \varphi_1^2 & \varphi_1 \varphi_2 & \varphi_1\\
\varphi_1 \varphi_2 & 1 + \varphi_2^2 & \varphi_2 \\
\varphi_1 & \varphi_2 &  1
\end{pmatrix},{ \ }
(g_B^{ij})_{3 \times 3 } =
\begin{pmatrix}
1 & 0 & - \varphi_1\\
0 & 1 & - \varphi_2 \\
- \varphi_1 & - \varphi_2 & 1 + \varphi_1^2 + \varphi_2^2
\end{pmatrix},
\end{equation*}
and that $G_B \equiv 1$. Set
\begin{equation*}
g_B^i = g_B^i(z) := g_B^{ij}g^B_j \bigg( = \sum_{j=1}^3 g_B^{ij}g^B_j \bigg).
\end{equation*}
It is clear that
\begin{equation*}
g_B^1(z) = \begin{pmatrix}
1\\
0\\
0
\end{pmatrix},{ \ }
g_B^2(z) = \begin{pmatrix}
0\\
1\\
0
\end{pmatrix},{ \ }
g_B^3(z) = \begin{pmatrix}
-\varphi_1\\
-\varphi_2\\
1
\end{pmatrix},
\end{equation*}
and $g_B^{ij} = g_B^i \cdot g_B^j$. From \cite[Lemma 3.3]{KS18}, we have the following representation formulas.
\begin{lemma}\label{lem79}
$(\rm{i})$ For all $f_B, f_B^\natural, f_B^\flat \in C^2 (\mathbb{R}^3)$,
\begin{align*}
\int_{\Omega_B} f_B (x) { \ }d x &= \int_{\mathbb{R}^3_-} \check{f}_B (z) { \ }dz,\\
\int_{\Omega_B} \partial_\ell f_B { \ }dx &= \int_{\mathbb{R}^3_-} g_B^{i j} \frac{\partial \check{x}^B_\ell }{\partial z_i} \frac{ \partial \check{f}_B }{\partial z_j} { \ }dz,\\
\int_{\Omega_B} \partial_{\ell'} \partial_\ell f_B { \ }dx &= \int_{\mathbb{R}^3_-}  g_B^{i' j'} \frac{\partial \check{x}^B_{\ell'} }{\partial z_{i'}} \frac{ \partial }{\partial z_{j'}} \bigg( g_B^{i j} \frac{\partial \check{x}^B_\ell }{\partial z_i} \frac{ \partial \check{f}_B }{\partial z_j} \bigg) { \ }dz,\\
\int_{\Omega_B} \nabla f_B^\natural \cdot \nabla f_B^\flat { \ }dx & = \int_{\mathbb{R}^3_-} g_B^{i j} \frac{\partial \check{f}_B^\natural }{\partial z_i} \frac{ \partial \check{f}_B^\flat }{\partial z_j} { \ }dz,\\
\int_{\Omega_B} \Delta f_B { \ }d x & = \int_{\mathbb{R}^3_-} \frac{\partial}{\partial z_i} \bigg( g_B^{i j} \frac{\partial \check{f}_B}{\partial z_j} \bigg) { \ }dz,
\end{align*}
where $\check{f}_B = \check{f}_B (z) = f_B (\check{x}_B(z) )$.\\
$(\rm{ii})$ For all $F_B = { }^t (F_1^B, F_2^B ,F_3^B) \in [C^1 (\mathbb{R}^3)]^3$ and $\Phi_B \in C^1 (\mathbb{R}^4)$,
\begin{align*}
\int_{\Omega_B} \nabla \cdot F_B { \ }d x &= \int_{\mathbb{R}^3_-} g_B^i \cdot \frac{\partial \check{F}_B }{\partial z_i} { \ }dz,\\
\int_{\Omega_B} \frac{\partial \Phi_B}{\partial t} { \ }d x &= \int_{\mathbb{R}^3_-} \frac{\partial \check{\Phi}_B}{\partial t} { \ }d z. 
\end{align*}
Here $\check{F}_B = \check{F}_B (z) = F_B (\check{x}_B(z) )$ and $\check{\Phi}_B = \check{\Phi}_B (z,t) = \Phi_B (\check{x}_B(z),t)$. 
\end{lemma}

\begin{remark}\label{rem7010}  $(\rm{i})$ Let $\kappa_B >0$. For $\psi_B \in C^2 (\overline{\mathbb{R}^3_-})$, we set
\begin{equation}\label{eq7079}
\mathscr{L}_2 \psi_B := - \kappa_B \frac{\partial}{\partial z_i} \bigg( g_B^{ij} \frac{\partial \psi_B}{\partial z_j} \bigg).
\end{equation}
We easily check that
\begin{equation}\label{eq7080}
\mathscr{L}_2 \psi_B = - \kappa_B \Delta_z \psi_B + \mathscr{B}_2 \psi_B,
\end{equation}
where
\begin{equation}\label{eq7081}
\mathscr{B}_2 \psi_B := \kappa_B \bigg(  2 \varphi_1 \frac{\partial^2 \psi_B}{\partial z_1 \partial z_3} + 2 \varphi_2 \frac{\partial^2 \psi_B}{\partial z_2 \partial z_3} - (\varphi_1^2 + \varphi_2^2) \frac{\partial^2 \psi_B}{\partial z_3^2 } + ( \varphi_{11} + \varphi_{22} ) \frac{\partial \psi_B}{\partial z_3} \bigg).
\end{equation}
Here $\varphi_{\alpha \beta} = \partial^2 \varphi/{\partial z_\alpha \partial z_\beta}$.\\
$(\rm{ii})$ For $\psi_B \in L_{loc}^1 (\mathbb{R}^3_-)$, we define
\begin{align*}
\Vert \psi_B \Vert_{\mathcal{L}^2(\Omega_B)} & :=  \bigg( \int_{\mathbb{R}^3_-} \vert \psi_B \vert^2 { \ }d z \bigg)^{1/2},\\
\Vert \psi_B \Vert_{\dot{\mathcal{W}}^{1,2}(\Omega_B)}  & := \bigg( \int_{\mathbb{R}^3_-} g_B^{i j} \frac{\partial \psi_B }{\partial z_i} \frac{ \partial \psi_B }{\partial z_j}  { \ }dz \bigg)^{1/2},\\
\Vert \psi_B \Vert_{\dot{\mathcal{W}}^{2,2}(\Omega_B)}  & := \bigg( \sum_{\ell, \ell' =1}^3 \int_{\mathbb{R}^3_-} \bigg\vert g_B^{i' j'} \frac{\partial \check{x}^B_{\ell'} }{\partial z_{i'}} \frac{ \partial }{\partial z_{j'} } \bigg( g_B^{i j} \frac{\partial \check{x}^B_\ell }{\partial z_i} \frac{ \partial \psi_B }{\partial z_j} \bigg) \bigg\vert^2 { \ }dz \bigg)^{1/2},\\
\Vert \psi_B \Vert_{ \mathcal{W}^{2,2}(\Omega_B)} & := ( \Vert \psi_B \Vert_{\mathcal{L}^2(\Omega_B)}^2 + \Vert \psi_B \Vert_{\dot{\mathcal{W}}^{1,2}(\Omega_B)}^2 + \Vert \psi_B \Vert_{\dot{\mathcal{W}}^{2,2}(\Omega_B)}^2 )^{1/2}.
\end{align*}
$(\rm{iii})$ For $\psi_B \in L_{loc}^1 (\mathbb{R}^3_+)$, we define
\begin{align*}
\Vert \psi_B \Vert_{\mathcal{L}^1(\Omega_B)} &:= \int_{\mathbb{R}^3_-} \vert \psi_B \vert { \ }dz,\\
\Vert \psi_B \Vert_{\dot{\mathcal{W}}^{1,1} (\Omega_B) } &:= \sum_{\ell =1}^3 \int_{\mathbb{R}^3_+} \bigg\vert g_B^{i j} \frac{\partial \check{x}^B_\ell }{\partial z_i} \frac{ \partial \psi_B }{\partial z_j} \bigg\vert { \ }dz,\\
\Vert \psi_B \Vert_{\mathcal{W}^{1,1} (\Omega_B) } &: = \Vert \psi_B \Vert_{\mathcal{L}^1 (\Omega_B) } + \Vert \psi_B \Vert_{\dot{\mathcal{W}}^{1,1} (\Omega_B) }.
\end{align*}
\end{remark}

In the same manner in the proof of Lemmas \ref{lem77} and \ref{lem78}, we have the following two lemmas.
\begin{lemma}\label{lem7011}
Assume that $\Vert \nabla_X \varphi \Vert_{L^\infty(\mathbb{R}^2)} \leq 1/2$. Then,{ \ }\\
$(\rm{i})$ For all $\psi_B \in W^{1,2}(\mathbb{R}^3_-)$
\begin{equation*}
\frac{1}{9}\Vert \nabla_z \psi_B \Vert_{L^2(\mathbb{R}^3_-)}^2 \leq \Vert \psi_B \Vert_{\dot{\mathcal{W}}^{1,2}(\Omega_B)}^2 \leq 9 \Vert \nabla_z \psi_B \Vert_{L^2(\mathbb{R}^3_-)}^2.
\end{equation*}
$(\rm{ii})$ There is $C = C( \Vert \nabla_X^2 \varphi \Vert_{L^{\infty} (\mathbb{R}^2)} ) > 0 $ such that for all $\psi_B \in W^{2,2}(\mathbb{R}^3_-)$
\begin{equation*}
C^{-1} \Vert \psi_B \Vert_{W^{2,2} (\mathbb{R}^3_-)} \leq \Vert \psi_B \Vert_{\mathcal{W}^{2,2}(\Omega_B)} \leq C \Vert \psi_B \Vert_{W^{2,2} (\mathbb{R}^3_-)}.
\end{equation*}
\end{lemma}

\begin{lemma}\label{lem7012}
Assume that $\Vert \nabla_X \varphi \Vert_{L^\infty (\mathbb{R}^2)} \leq 1/2$. Then,\\
$(\rm{i})$ For all $\psi_B \in W^{1,1}(\mathbb{R}^3_-)$,
\begin{equation*}
\frac{1}{2} \Vert \nabla_z \psi_B \Vert_{L^1(\mathbb{R}^3_-)} \leq \Vert \psi_B \Vert_{\dot{\mathcal{W}}^{1,1}(\Omega_B)} \leq 18 \Vert \nabla_z \psi_B \Vert_{L^1(\mathbb{R}^3_-)}.
\end{equation*}
Here
\begin{equation*}
\Vert \nabla_z \psi_B \Vert_{L^1(\mathbb{R}^3_-)} := \Vert \partial_{z_1} \psi_B \Vert_{L^1(\mathbb{R}^3_-)} + \Vert \partial_{z_2} \psi_B \Vert_{L^1(\mathbb{R}^3_-)}  + \Vert \partial_{z_3} \psi_B \Vert_{L^1(\mathbb{R}^3_-)}. 
\end{equation*}
$(\rm{ii})$ For all $\psi_B \in W^{1,1}(\mathbb{R}^3_-)$,
\begin{equation*}
\frac{1}{2} \Vert \psi_B \Vert_{W^{1,1}(\mathbb{R}^3_-)} \leq \Vert \psi_B \Vert_{\mathcal{W}^{1,1}(\Omega_B)} \leq 18 \Vert \psi_B \Vert_{W^{1,1}(\mathbb{R}^3_-)}.
\end{equation*}
\end{lemma}
\noindent See also Lemmas \ref{lem37} and \ref{lem38}.

\subsection{Divergence Theorems and their Applications}\label{subsec74}

In this subsection, we derive the three divergence theorems, and state their applications. Let $\mathcal{P}_A$, $\mathcal{P}_B$, $\mathcal{P}_S$ be the three pullback operators defined by Definition \ref{def32}, and $C_0^k (\Gamma)$, $C_0^k(\overline{\Omega_A})$, $C_0^k( \overline{\Omega_B})$ $(k = 0,1,2)$ be the function spaces defined by Section \ref{sect3}.

Using tools ($\tilde{x}_A$, $\check{x}_B$, $\hat{x}_S$, $g_A^i$, $g_B^i$, $g_S^\alpha$) provided in subsections \ref{subsec71}, \ref{subsec72}, and \ref{subsec73}, we prove the divergence theorems. 
\begin{lemma}[Divergence theorems]\label{lem7013} For all $F_A = { }^t (F^A_1,F^A_2,F^A_3) \in C_0^1(\overline{\Omega_A})$, $F_B = { }^t (F^B_1,F^B_2,F^B_3) \in C_0^1(\overline{\Omega_B})$, and $F_S = { }^t (F^S_1,F^S_2,F^S_3) \in C_0^1(\Gamma)$,
\begin{align}
\int_{\Omega_A } \nabla \cdot F_A { \ }d x &= - \int_\Gamma F_A \cdot n { \ } d\mathcal{H}_x^2,\label{eq7082}\\
\int_{\Omega_B } \nabla \cdot F_B { \ }d x &= \int_\Gamma F_B \cdot n { \ } d\mathcal{H}_x^2,\label{eq7083}\\
\int_{\Gamma } \nabla_\Gamma \cdot F_S { \ }d \mathcal{H}_x^2 & = - \int_\Gamma H_\Gamma (F_S \cdot n) { \ } d\mathcal{H}_x^2.\label{eq7084}
\end{align}
Here $H_\Gamma$ is the mean curvature defined by $H_\Gamma = - {\rm{div}}_\Gamma n$.
\end{lemma}
\begin{proof}[Proof of Lemma \ref{lem7013}]
We first show \eqref{eq7082}. Fix $F_A = { }^t (F^A_1,F^A_2,F^A_3) \in C_0^1(\overline{\Omega_A})$. By the definition of $C_0^1(\overline{\Omega_A})$, there is $\Psi_A = { }^t(\Psi^A_1, \Psi^A_2, \Psi^A_3) \in C_0^1( \overline{\mathbb{R}^3_+})$ such that
\begin{equation*}
F_A = \mathcal{P}_A[\Psi_A].
\end{equation*}
Since $F_A (x) = \Psi_A (x_1,x_2, x_3 - \varphi (x_1,x_2) )$, we find that
\begin{equation*}
F_A (\tilde{x}_A (y)) = \Psi_A (y).
\end{equation*}
Using \eqref{eq7045}, \eqref{eq7047}, and the divergence theorem for half space $\mathbb{R}^3_+$, we check that
\begin{align}
\int_{\Omega_A} \nabla \cdot F_A { \ } dx & = \int_{\Omega_A} \nabla \cdot \mathcal{P}_A[\Psi_A] { \ } dx= \int_{\mathbb{R}^3_+} g_A^i \cdot \frac{\partial \Psi_A}{\partial y_i} { \ }dy\notag\\
& = \int_{\mathbb{R}^3_+} \bigg( \frac{\partial \Psi_1^A}{\partial y_1 } + \frac{\partial \Psi_2^A}{\partial y_2 } -\varphi_1 \frac{\partial \Psi_1^A}{\partial y_3 } - \varphi_2 \frac{\partial \Psi_2^A}{\partial y_3 } + \frac{\partial \Psi_3^A}{\partial y_3 }  \bigg) { \ }dy\notag\\
& = \int_{\mathbb{R}^2} \bigg( \frac{\varphi_1}{\sqrt{G_S}} \Psi_1^A + \frac{\varphi_2}{\sqrt{G_S}} \Psi_2^A - \frac{1}{\sqrt{G_S}} \Psi_3^A  \bigg) \bigg\vert_{y_3=0}\sqrt{G_S} { \ }dy_h.\label{eq7085}
\end{align}
By the definition of $\mathcal{P}_A$ and \eqref{eq71}, we find that
\begin{align}
- \int_{\Gamma} (n \cdot F_A ) { \ }d \mathcal{H}_x^2 &= \int_{\Gamma} (-n \cdot \mathcal{P}_A[\Psi_A] ) { \ }d \mathcal{H}_x^2\notag\\
&= \int_{\Gamma} (-n(x) \cdot \Psi_A(x_1,x_2,x_3 - \varphi (x_1,x_2)) ) { \ }d \mathcal{H}_x^2\notag\\    
& = \int_{\mathbb{R}^2} \bigg( \frac{\varphi_1}{\sqrt{G_S}} \Psi_1^A + \frac{\varphi_2}{\sqrt{G_S}} \Psi_2^A - \frac{1}{\sqrt{G_S}} \Psi_3^A  \bigg) \bigg\vert_{y_3=0}\sqrt{G_S} { \ }dy_h.\label{eq7086}
\end{align}
From \eqref{eq7085} and \eqref{eq7086}, we see \eqref{eq7082}. Similarly, we have \eqref{eq7083}.

Next, we derive \eqref{eq7084}. Fix $F_S = { }^t (F^S_1,F^S_2,F^S_3) \in C_0^1(\Gamma)$. By the definition of $C_0^1 (\Gamma)$, there is $\psi_S = { }^t(\Psi^S_1, \Psi^S_2, \Psi^S_3) \in C_0^1( \mathbb{R}^2)$ such that
\begin{equation*}
F_S = \mathcal{P}_S[\Psi_S].
\end{equation*}
By the definitions of $\mathcal{P}_S$ and $\hat{x}_S$, we find that
\begin{equation*}
F_S (\hat{x}_S (X)) = \Psi_S (X).
\end{equation*}
Applying \eqref{eq77} and integration by parts, we check that
\begin{align*}
\int_{\Gamma } {\rm{div}}_\Gamma F_S { \ }d \mathcal{H}_x^2 & = \int_{\Gamma } {\rm{div}}_\Gamma \mathcal{P}_S [\Psi_S] { \ }d \mathcal{H}_x^2 = \int_{\mathbb{R}^2} g_S^\alpha \cdot \frac{\partial \Psi_S}{\partial X_\alpha} \sqrt{G_S} { \ }d X\\
& = - \int_{\mathbb{R}^2} \bigg( \frac{\partial g_S^\alpha}{\partial X_\alpha} + \frac{g_S^\alpha}{2 G_S} \frac{\partial G_S}{\partial X_\alpha} \bigg) \cdot \Psi_S \sqrt{G_S} { \ }d X\\
& = - \int_{\mathbb{R}^2} \hat{H}_\Gamma \bigg( \frac{g_1^S \times g_2^S}{ \vert g_1^S \times g_2^S \vert} \bigg) \cdot \Psi_S \sqrt{G_S} { \ }d X\\
& = - \int_{\Gamma } H_\Gamma (n \cdot \mathcal{P}_S[\Psi_S]) { \ }d\mathcal{H}_x^2= - \int_{\Gamma } H_\Gamma (n \cdot F_S) { \ }d\mathcal{H}_x^2.
\end{align*}
Here $\hat{H}_\Gamma$ is \eqref{eq7012}. Thus, we see \eqref{eq7084}. Therefore, the lemma follows.
\end{proof}

Now we introduce several applications of Lemma \ref{lem7013}.
\begin{lemma}[Integration by parts formulas]\label{lem7014} For each $f_A, \phi_A \in C_0^1(\overline{\Omega_A})$, $f_B, \phi_B \in C_0^1(\overline{\Omega_B} )$, $f_S, \phi_S \in C_0^1(\Gamma )$, and $j \in \{ 1,2,3\}$,
\begin{align}
\int_{\Omega_A } f_A (\partial_j \phi_A) { \ }d x & = - \int_{\Omega_A } (\partial_j f_A) \phi_A { \ }d x - \int_\Gamma f_A \phi_A n_j { \ } d\mathcal{H}_x^2,\label{eq7087}\\
\int_{\Omega_B } f_B (\partial_j \phi_B) { \ }d x & = - \int_{\Omega_B } (\partial_j f_B) \phi_B { \ }d x + \int_\Gamma f_B \phi_B n_j { \ } d\mathcal{H}_x^2,\label{eq7088}\\
\int_{\Gamma } f_S (\partial_j^\Gamma \phi_S) { \ }d \mathcal{H}_x^2 &= - \int_{\Gamma } (\partial_j^\Gamma f_S) \phi_S { \ }d \mathcal{H}_x^2 - \int_\Gamma H_\Gamma f_S \phi_S n_j { \ } d\mathcal{H}_x^2.\label{eq7089}
\end{align}
\end{lemma}
\begin{proof}[Proof of Lemma \ref{lem7014}]
We only show \eqref{eq7089} when $j=1$. Fix $f_S, \phi_S \in C_0^1(\Gamma )$. Set $F_S ={ }^t (f_S \phi_S ,0,0)$. Since $F_S \in C_0^1 (\Gamma)$, we use \eqref{eq7084} to find that
\begin{equation*}
\int_{\Gamma } (\partial_1^\Gamma f_S) \phi_S { \ }d \mathcal{H}_x^2 + \int_{\Gamma } f_S (\partial_1^\Gamma \phi_S) { \ }d \mathcal{H}_x^2  =  - \int_\Gamma H_\Gamma f_S \phi_S n_1 { \ } d\mathcal{H}_x^2.
\end{equation*}
Thus, we see \eqref{eq7089} when $j=1$. Other cases are similar. Therefore, the lemma follows.
\end{proof}

\begin{lemma}\label{lem7015}
For all $f_A \in C_0^2(\overline{\Omega_A})$, $\phi_A \in C_0^1(\overline{\Omega_A})$, $f_B \in C_0^2(\overline{\Omega_B} )$, $\phi_B \in C_0^1(\overline{\Omega_B} )$, $f_S \in C_0^2(\Gamma )$, $\phi_S \in C_0^1(\Gamma )$, and $\mu_A,\mu_B , \mu_S \in \mathbb{R}$,
\begin{align*}
-  \int_{\Omega_A} \mu_A \nabla f_A \cdot \nabla \phi_A { \ }d x & = \int_{\Omega_A} ( \mu_A \Delta f_A ) \phi_A { \ }d x +  \int_{\Gamma} \mu_A \frac{\partial f_A}{\partial n} \phi_A { \ }d \mathcal{H}_x^2,\\
-  \int_{\Omega_B} \mu_B \nabla f_B \cdot \nabla \phi_B { \ }d x & = \int_{\Omega_B} ( \mu_B \Delta f_B ) \phi_B { \ }d x -  \int_{\Gamma} \mu_B \frac{\partial f_B}{\partial n} \phi_B { \ }d \mathcal{H}_x^2,\\
-  \int_{\Gamma} \mu_S \nabla_\Gamma f_S \cdot \nabla_\Gamma \phi_S { \ }d \mathcal{H}_x^2 &= \int_{\Gamma} ( \mu_S \Delta_\Gamma f_S ) \phi_S { \ }d \mathcal{H}_x^2.
\end{align*}
Here $\partial f/{\partial n} := (n \cdot \nabla)f $.
\end{lemma}

\begin{proof}[Proof of Lemma \ref{lem7015}]
Fix $f_A \in C_0^2(\overline{\Omega_A})$, $\phi_A \in C_0^1(\overline{\Omega_A})$, $f_B \in C_0^2(\overline{\Omega_B} )$, $\phi_B \in C_0^1(\overline{\Omega_B} )$, $f_S \in C_0^2(\Gamma )$, $\phi_S \in C_0^1(\Gamma )$, and $\mu_A,\mu_B , \mu_S \in \mathbb{R}$. Direct calculations give
\begin{align*}
\int_{\Omega_A} \nabla \cdot \{ \mu_A (\nabla f_A)\phi_A \} { \ }d x &= \int_{\Omega_A} ( \mu_A \Delta f_A ) \phi_A { \ }d x +  \int_{\Omega_A} \mu_A \nabla f_A \cdot \nabla \phi_A { \ }d x,\\
\int_{\Omega_B} \nabla \cdot \{ \mu_B (\nabla f_B)\phi_B \} { \ }d x &= \int_{\Omega_B} ( \mu_B \Delta f_B ) \phi_B { \ }d x +  \int_{\Omega_B} \mu_B \nabla f_B \cdot \nabla \phi_B { \ }d x,\\
\int_{\Gamma} \nabla_\Gamma \cdot \{ \mu_S (\nabla_\Gamma f_S)\phi_S \} { \ }d \mathcal{H}_x^2 &= \int_{\Gamma} ( \mu_S \Delta_\Gamma f_S ) \phi_S { \ }d \mathcal{H}_x^2 +  \int_{\Gamma} \mu_S \nabla_\Gamma f_S \cdot \nabla_\Gamma \phi_S { \ }d\mathcal{H}_x^2.
\end{align*}
Using \eqref{eq7082}, \eqref{eq7083}, and \eqref{eq7084}. we see that
\begin{align*}
- \int_{\Gamma} \mu_A \frac{\partial f_A}{\partial n} \phi_A { \ }d \mathcal{H}_x^2 =  \int_{\Omega_A} ( \mu_A \Delta f_A ) \phi_A { \ }d x + \int_{\Omega_A} \mu_A \nabla f_A \cdot \nabla \phi_A { \ }d x ,\\
 \int_{\Gamma} \mu_B \frac{\partial f_B}{\partial n} \phi_B { \ }d \mathcal{H}_x^2 = \int_{\Omega_B} ( \mu_B \Delta f_B ) \phi_B { \ }d x + \int_{\Omega_B} \mu_B \nabla f_B \cdot \nabla \phi_B { \ }d x ,
\end{align*}
and
\begin{equation*}
0 = \int_{\Gamma} ( \mu_S \Delta_\Gamma f_S ) \phi_S { \ }d \mathcal{H}_x^2 + \int_{\Gamma} \mu_S \nabla_\Gamma f_S \cdot \nabla_\Gamma \phi_S { \ }d \mathcal{H}_x^2.
\end{equation*}
Here we used the fact that $(n \cdot \nabla_\Gamma) f_S =0 $. Therefore, the lemma follows.
\end{proof}

Finally, we study representation formulas for the Neumann and Dirichlet conditions.
\begin{lemma}\label{lem7016}
$(\rm{i})$ Let $f_A \in C_0^2 (\overline{\Omega_A})$ and $\psi_A \in C_0^2(\overline{\mathbb{R}^3_+})$ such that $f_A = \mathcal{P}_A[\psi_A]$. Then
\begin{equation}
\int_{\Gamma} f_A { \ }d \mathcal{H}^2_x = \int_{\mathbb{R}^2} \psi_A \vert_{y_3=0} \sqrt{G_S}{ \ }d y_h,\label{eq7090}
\end{equation}
\begin{multline}\label{eq7091}
\int_{\Gamma} (n \cdot \nabla ) f_A { \ }d\mathcal{H}_x^2\\
 = \int_{\mathbb{R}^2} \bigg\{ - \frac{\varphi_1}{\sqrt{G_S}} \frac{\partial \psi_A }{\partial y_1} - \frac{\varphi_2}{\sqrt{G_S}} \frac{\partial \psi_A }{\partial y_2} +  \frac{ 1 +\varphi_1^2 + \varphi_2^2 }{\sqrt{G_S}} \frac{\partial \psi_A }{\partial y_3} \bigg\} \bigg\vert_{y_3 =0} \sqrt{G_S}{ \ }dy_h.
\end{multline}

$(\rm{ii})$ Let $f_B \in C_0^2 (\overline{\Omega_B})$ and $\psi_A \in C_0^2(\overline{\mathbb{R}^3_-})$ such that $f_B = \mathcal{P}_B[\psi_B]$. Then
\begin{align*}
& \int_{\Gamma} f_B { \ }d \mathcal{H}^2_x = \int_{\mathbb{R}^2} \psi_B \vert_{z_3=0} \sqrt{G_S}{ \ }d z_h,\\
& \int_{\Gamma} (n \cdot \nabla ) f_B { \ }d\mathcal{H}_x^2 = \int_{\mathbb{R}^2} \bigg\{ - \frac{\varphi_1}{\sqrt{G_S}} \frac{\partial \psi_B }{\partial z_1} - \frac{\varphi_2}{\sqrt{G_S}} \frac{\partial \psi_B }{\partial z_2} +  \frac{ 1 +\varphi_1^2 + \varphi_2^2 }{\sqrt{G_S}} \frac{\partial \psi_B }{\partial z_3} \bigg\} \bigg\vert_{z_3 =0}  \sqrt{G_S}{ \ }dz_h.
\end{align*}
\end{lemma}

\begin{proof}[Proof of Lemma \ref{lem7016}]

We only show \eqref{eq7090} and \eqref{eq7091}. Let $f_A \in C_0^2 (\overline{\Omega_A})$ and $\psi_A \in C_0^2(\overline{\mathbb{R}^3_+})$ such that $f_A = \mathcal{P}_A[\psi_A]$. By definition, we find that
\begin{align*}
\int_{\Gamma} f_A (x) { \ }d \mathcal{H}_x^2 & = \int_{\Gamma} \mathcal{P}_A[\psi_A](x) { \ }d \mathcal{H}_x^2\\
& = \int_{\mathbb{R}^2} \mathcal{P}_A[ \psi_A] (X_1,X_2, \varphi (X_1,X_2)) \sqrt{G_S} { \ }dX\\
 & = \int_{\mathbb{R}^2} \psi_A (X_1,X_2,0) \sqrt{G_S} { \ }dX\\
 &= \int_{\mathbb{R}^2} \psi_A \vert_{y_3=0} \sqrt{G_S} { \ }dy_h,
\end{align*}
which is \eqref{eq7090}. Since $\Delta f_A= \nabla \cdot \nabla f_A$, we use \eqref{eq7082} to have
\begin{equation}\label{eq7092}
\int_{\Omega_A} \Delta f_A { \ }dx = - \int_{\Gamma} \nabla f_A \cdot n { \ }d \mathcal{H}_x^2.
\end{equation}
Using \eqref{eq7044}, \eqref{eq7049} when $\kappa_A =1$, integration by parts, and the divergence theorem for half space $\mathbb{R}^3_+$, we see that
\begin{multline}\label{eq7093}
 \int_{\Omega_A } \Delta f_A { \ }dx= \int_{\Omega_A } \Delta \mathcal{P}_A[ \psi_A] { \ }dx\\
 =  \int_{\mathbb{R}^3_+}\Delta_y \psi_A - 2 \varphi_1 \frac{\partial^2 \psi_A}{\partial y_1 \partial y_3} - 2 \varphi_2 \frac{\partial^2 \psi_A}{\partial y_2 \partial y_3} + (\varphi_1^2 + \varphi_2^2) \frac{\partial^2 \psi_A}{\partial y_3^2 } - (\varphi_{11} + \varphi_{22}) \frac{\partial \psi_A}{\partial y_3} { \ } dy\\
 =  \int_{\mathbb{R}^3_+}\Delta_y \psi_A - \varphi_1 \frac{\partial^2 \psi_A}{\partial y_1 \partial y_3}- \varphi_2 \frac{\partial^2 \psi_A}{\partial y_2 \partial y_3} + (\varphi_1^2 + \varphi_2^2) \frac{\partial^2 \psi_A}{\partial y_3^2 } dy\\
  = \int_{\mathbb{R}^2} \bigg\{ \varphi_1 \frac{\partial \psi_A}{\partial y_1} + \varphi_2 \frac{\partial \psi_A}{\partial y_2} -  (1 + \varphi_1^2 + \varphi_2^2 ) \frac{\partial \psi_A}{\partial y_3} \bigg\} \bigg\vert_{y_3 =0} \frac{\sqrt{G_S}}{\sqrt{G_S}} { \ }dy_h.
\end{multline}
By \eqref{eq7092} and \eqref{eq7093}, we see \eqref{eq7091}. Note that
\begin{multline*}
\int_{\mathbb{R}^2} \bigg\{ \varphi_1 \frac{\partial \psi_A}{\partial y_1} + \varphi_2 \frac{\partial \psi_A}{\partial y_2} -  ( 1 + \varphi_1^2 + \varphi_2^2 ) \frac{\partial \psi_A}{\partial y_3} \bigg\}\bigg\vert_{y_3 =0} \frac{\sqrt{G_S}}{\sqrt{G_S}} { \ }dy_h\\
   = \int_{\mathbb{R}^2} (\varphi_1 , \varphi_2 , - 1) \cdot \bigg( \frac{\partial \psi_A}{\partial y_1} -\varphi_1 \frac{\partial \psi_A}{\partial y_3}, \frac{\partial \psi_A}{\partial y_2} - \varphi_2 \frac{\partial \psi_A}{\partial y_3},  \frac{\partial \psi_A}{\partial y_3} \bigg) \bigg\vert_{y_3 =0} \frac{\sqrt{G_S}}{\sqrt{G_S}} { \ }dy_h\\
= \int_{\Gamma} (- n \cdot \nabla ) \mathcal{P}_A[\psi_A] { \ }d\mathcal{H}_x^2 = - \int_{\Gamma} (n \cdot \nabla ) f_A { \ }d\mathcal{H}_x^2.
\end{multline*}
Therefore, the lemma follows.
\end{proof}

\begin{remark}\label{rem7017}
From Lemma \ref{lem7016}, we see that for all $f_A,f_B \in C^2 (\mathbb{R}^3)$,
\begin{equation*}
\int_{\Gamma} (n \cdot \nabla ) f_A { \ }d\mathcal{H}_x^2 = \int_{\mathbb{R}^2} \bigg\{ - \frac{\varphi_1}{\sqrt{G_S}} \frac{\partial \tilde{f}_A }{\partial y_1} - \frac{\varphi_2}{\sqrt{G_S}} \frac{\partial \tilde{f}_A }{\partial y_2} +  \frac{1 + \varphi_1^2 + \varphi_2^2 }{\sqrt{G_S}} \frac{\partial \tilde{f}_A }{\partial y_3} \bigg\} \bigg\vert_{y_3 =0} \sqrt{G_S}{ \ }dy_h,
\end{equation*}
and
\begin{equation*}
\int_{\Gamma} (n \cdot \nabla ) f_B { \ }d\mathcal{H}_x^2 = \int_{\mathbb{R}^2} \bigg\{ - \frac{\varphi_1}{\sqrt{G_S}} \frac{\partial \check{f}_B }{\partial z_1} - \frac{\varphi_2}{\sqrt{G_S}} \frac{\partial \check{f}_B }{\partial z_2} +  \frac{1 + \varphi_1^2 + \varphi_2^2 }{\sqrt{G_S}} \frac{\partial \check{f}_B }{\partial z_3} \bigg\} \bigg\vert_{z_3 =0}  \sqrt{G_S}{ \ }dz_h.
\end{equation*}
Here $\tilde{f}_A = \tilde{f}_A (y) = f_A (\tilde{x}_A (y))$ and $\check{f}_B = \check{f}_B(z) = f_B (\check{x}_B (z) )$.
\end{remark}

\end{document}